\documentclass[11pt,oneside]{article}
\usepackage{amssymb,amsmath,latexsym,amsfonts,amscd,amsthm,multirow,ctable,colortbl,mathdots}
\usepackage{fancyhdr}
\usepackage[headings]{fullpage}
\usepackage{stmaryrd}
\usepackage[all]{xy}
\usepackage{rotating}

\usepackage{setspace}

\usepackage{hyperref}

\newcommand{\bea}{\begin{eqnarray}} 
\newcommand{\eea}{\end{eqnarray}} 
\newcommand{\bee}{\begin{eqnarray*}} 
\newcommand{\eee}{\end{eqnarray*}} 
\newcommand{\al}{\begin{align*}} 
\newcommand{\eal}{\end{align*}} 
\newcommand{\be}{\begin{equation}} 
\newcommand{\ee}{\end{equation}} 
\newcommand{\eq}[1]{(\ref{#1})} 
\newcommand{\bem}{\begin{pmatrix}} 
\newcommand{\eem}{\end{pmatrix}}

\def\c{\gamma}

\def\e{\epsilon}    
\def\f{\phi}

\def\h{\eta}

\def\inf{\infty}

\def\l{\lambda} 
\def\m{\mu} 
\def\n{\nu} 
 
\def\o{\omega}  
 
\def\p{\pi}    
\def\pa{\partial}        

\def\r{\rho}                  
\def\s{\sigma}            
\def\t{\tau} 
\def\th{\theta} 
\def\til{\tilde}

\def\D{\Delta}

\def\L{\Lambda} 
\def\O{\Omega}

\newcolumntype{R}{ >{$}r <{$}}
\newcolumntype{C}{ >{$}c <{$}}

\def\ll{\ell}
\def\LL{\Lambda}

\newcommand{\mc}[1]{\mathcal{#1}}

\newcommand{\comment}[1]{}

\newcommand{\RR}{{\mathbb R}}
\newcommand{\CC}{{\mathbb C}}
\newcommand{\ZZ}{{\mathbb Z}}
\newcommand{\QQ}{{\mathbb Q}}
\newcommand{\HH}{{\mathbb H}}
\newcommand{\FF}{{\mathbb F}}

\newcommand{\kk}{{\mathbf k}}

\newcommand{\ii}{{\bf i}}

\newcommand{\tpi}{2\pi\ii}

\newcommand{\tc}{\!\!:\!\!}     
\newcommand{\lab}{{\langle}}    
\newcommand{\rab}{{\rangle}}    

\newcommand{\Span}{\operatorname{Span}}

\def\jac{\operatorname{jac}}

\newcommand{\tr}{\operatorname{{tr}}}
\newcommand{\str}{\operatorname{{str}}}

\newcommand{\Sym}{{\textsl{Sym}}}
\newcommand{\Alt}{{\textsl{Alt}}}
\newcommand{\Dih}{{\textsl{Dih}}}
\newcommand{\HO}{{\textsl{Oct}}}

\newcommand{\ex}{\operatorname{e}} 

\newcommand{\PSL}{\operatorname{\textsl{PSL}}}    
\newcommand{\SL}{\operatorname{\textsl{SL}}}      
\newcommand{\Sp}{\operatorname{\textsl{Sp}}}      
\newcommand{\PGL}{\operatorname{\textsl{PGL}}}    
\newcommand{\AGL}{{\textsl{AGL}}}    
\newcommand{\GL}{{\textsl{GL}}}      
\newcommand{\SU}{\operatorname{\textsl{SU}}}    
\newcommand{\SO}{\operatorname{\textsl{SO}}}    

\newcommand{\G}{\Gamma}	
\newcommand{\g}{\gamma}	

\newcommand{\wttwo}{{ F}}
\newcommand{\MM}{\mathbb{M}}

\newtheorem{thm}{Theorem}[section]

\newtheorem{lem}[thm]{Lemma}
\newtheorem{prop}[thm]{Proposition}
\newtheorem{conj}[thm]{Conjecture}

\theoremstyle{definition}

\theoremstyle{remark}
\newtheorem{rmk}[thm]{Remark}

\numberwithin{equation}{section}

\begin{document}

\setstretch{1.4}

\title{
\vspace{-35pt}
    \textsc{\huge{{U}mbral {M}oonshine} }
    }

\author{
	Miranda C. N. Cheng\footnote{
	Universit\'e Paris 7, UMR CNRS 7586 and LPTHE, Universit\'e Paris 6, Paris, France.
	\newline\indent\indent
	{\em E-mail:} {\tt chengm@math.jussieu.fr}
	}\\
	John F. R. Duncan\footnote{
         Department of Mathematics,
         Case Western Reserve University,
         Cleveland, OH 44106,
         U.S.A.
         \newline\indent\indent
         {\em E-mail:} {\tt john.duncan@case.edu}
         }\\
	Jeffrey A. Harvey\footnote{
	Enrico Fermi Institute and Department of Physics,
         University of Chicago,
         Chicago, IL 60637,
         U.S.A.
         \newline\indent\indent
         {\em E-mail:} {\tt j-harvey@uchicago.edu}
          }
}

\date{2013 October 13}

\maketitle

\abstract{
We describe surprising relationships between automorphic forms of various kinds, imaginary quadratic number fields and a certain system of six finite groups that are parameterised naturally by the divisors of twelve. The Mathieu group correspondence recently discovered by Eguchi--Ooguri--Tachikawa is recovered as a special case. We introduce a notion of extremal Jacobi form and prove that it characterises the Jacobi forms arising by establishing a connection to critical values of Dirichlet series attached to modular forms of weight two. These extremal Jacobi forms are closely related to certain vector-valued mock modular forms studied recently by Dabholkar--Murthy--Zagier in connection with the physics of quantum black holes in string theory. In a manner similar to monstrous moonshine the automorphic forms we identify constitute evidence for the existence of infinite-dimensional graded modules for the six groups in our system. We formulate an umbral moonshine conjecture that is in direct analogy with the monstrous moonshine conjecture of Conway--Norton. Curiously, we find a number of Ramanujan's mock theta functions appearing as McKay--Thompson series. A new feature not apparent in the monstrous case is a property which allows us to predict the fields of definition of certain homogeneous submodules for the groups involved. For four of the groups in our system we find analogues of both the classical McKay correspondence and McKay's monstrous Dynkin diagram observation manifesting simultaneously and compatibly.
}

\clearpage

\tableofcontents

\clearpage

\section{Introduction}\label{sec:intro}

The term {\em monstrous moonshine} was coined by Conway \cite{conway_norton} in order to describe the unexpected and mysterious connections between the representation theory of the largest sporadic group---the {\em Fischer--Griess monster}, $\MM$---and modular functions that stemmed from McKay's observation that $196883+1=196884$, where the summands on the left are degrees of irreducible representations of $\MM$ and the number on the right is the coefficient of $q$ in the Fourier expansion of the {\em elliptic modular invariant}
\begin{gather}\label{eqn:intro:FouExpJ}
J(\t)=\sum_{m \ge -1} a(m) q^m= q^{-1} + 196884 q + 21493760 q^2 + 864299970 q^3 + \cdots.
\end{gather}
Thompson expanded upon McKay's observation in \cite{Tho_NmrlgyMonsEllModFn} and conjectured the existence of an infinite-dimensional monster module
\be
V=\bigoplus_{m \ge -1} V_m
\ee
with $\dim  V_m=a(m)$ for all $m$. He also proposed \cite{Tho_FinGpsModFns} to consider the series, now known as {\em McKay--Thompson series}, given by 
\be
T_g(\tau)= \sum_{m\ge -1} {\tr}_{V_m}(g )\,q^m
\ee
for $g\in\MM$, and detailed explorations \cite{conway_norton} by Conway--Norton led to the astonishing {\em moonshine conjecture}: 
\begin{quote}
For each $g\in\MM$ the function $T_g$ is a principal modulus for some genus zero group $\G_g$.
\end{quote}
(A discrete group $\G<\PSL_2(\RR)$ is said to have {\em genus zero} if the Riemann surface $\G\backslash\HH$ is isomorphic to the Riemann sphere minus finitely many points, and a holomorphic function $f$ on $\HH$ is called a {\em principal modulus} for a genus zero group $\G$ if it generates the field of $\G$-invariant functions on $\HH$.)

Thompson's conjecture was verified by Atkin, Fong and Smith (cf. \cite{MR604633,MR822245}). A more constructive verification was obtained by Frenkel--Lepowsky--Meurman \cite{FLMPNAS,FLMBerk} with the explicit construction of a monster module $V=V^{\natural}$ with graded dimension given by the Fourier expansion (\ref{eqn:intro:FouExpJ}) of the elliptic modular invariant. They used {\em vertex operators}---structures originating in the dual resonance theory of particle physics and finding contemporaneous application \cite{FreKac_AffLieDualRes,Seg_UtyRpsInfDimGps} to affine Lie algebras---to recover the non-associative {\em Griess algebra} structure (developed in the first proof \cite{Gri_FG} of the existence of the monster) from a subspace of $V^{\natural}$. Borcherds found a way to attach vertex operators to every element of $V^{\natural}$ and determined the precise sense in which these operators could be given a commutative associative composition law, and thus arrived at the notion of {\em vertex algebra} \cite{Bor_PNAS}, an axiomatisation of the operator product expansion of chiral conformal field theory. The closely related notion of {\em vertex operator algebra (VOA)} was subsequently introduced by Frenkel--Lepowsky--Meurman \cite{FLM} and they established that the monster is precisely the group of automorphisms of a VOA structure on $V^{\natural}$; the Frenkel--Lepowsky--Meurman construction of $V^{\natural}$ would ultimately prove to furnish the first example of an {\em orbifold conformal field theory}. 

Borcherds introduced the notion of {\em generalised Kac--Moody algebra} in \cite{Bor_GKM} and by using the VOA structure on $V^{\natural}$ was able to construct a particular example---the monster Lie algebra---and use the corresponding equivariant denominator identities to arrive at a proof \cite{borcherds_monstrous} of the Conway--Norton moonshine conjectures. Thus by 1992 monstrous moonshine had already become a phenomenon encompassing elements of finite group theory, modular forms, vertex algebras and generalised Kac--Moody algebras, as well as aspects of conformal field theory and string theory.

Recently Eguchi--Ooguri--Tachikawa have presented evidence \cite{Eguchi2010} for a new kind of moonshine involving the elliptic genus of $K3$ surfaces (the elliptic genus is a topological invariant and therefore independent of the choice of $K3$ surface) and the {\em largest Mathieu group}
$M_{24}$ (cf. \S\ref{sec:grps:spec}). This connection between $K3$ surfaces and $M_{24}$ becomes apparent only after decomposing the elliptic genus into characters
of the $N=4$ superconformal algebra (cf. \cite{Eguchi1987,Eguchi1988,MR2060475}). This decomposition process (cf. \S\ref{sec:forms:sca}) reveals the presence of a mock modular form of weight $1/2$ (cf. \S\ref{sec:forms:mock}) satisfying
\be\label{eqn:intro:H2Fou}
H^{(2)}(\tau)= \sum_{n=0}^\infty c^{(2)}_1(n-1/8) q^{n-1/8} = 2 q^{-1/8}(-1 + 45 q + 231 q^2 + 770 q^3+2277 q^4 + \cdots )
\ee
and one recognises here the dimensions of several irreducible representations of $M_{24}$ (cf. Table \ref{tab:chars:irr:2}). 

One is soon led to follow the path forged by Thompson in the case of the monster: to suspect the existence of a graded infinite-dimensional $M_{24}$-module
\be
K^{(2)}= \bigoplus_{n=0}^\infty K^{(2)}_{n-1/8}
\ee
with $\dim  K^{(2)}_{n-1/8}= c^{(2)}(n-1/8)$ for $n\geq 1$, and to study the analogues $H^{(2)}_{g}$ of the monstrous McKay--Thompson series obtained by replacing $c^{(2)}(n-1/8)=\dim  K^{(2)}_{n-1/8}$ with
$\tr_{K^{(2)}_{n-1/8}}(g)$ in (\ref{eqn:intro:H2Fou}). This idea has been implemented successfully in \cite{Cheng2010_1,Gaberdiel2010,Gaberdiel2010a,Eguchi2010a} and provides strong evidence for the existence of such an $M_{24}$-module $K^{(2)}$. A proof of the existence of $K^{(2)}$ has now been established in \cite{Gannon:2012ck} although no explicit construction is yet known. In particular there is as yet no known analogue of the vertex operator algebra structure which conjecturally characterises \cite{FLM} the monster module $V^{\natural}$.

The strong evidence in support of the $M_{24}$ analogue of Thompson's conjecture invites us to consider the $M_{24}$ analogue of the Conway--Norton moonshine conjectures---this will justify the use of the term moonshine in the $M_{24}$ setting---except that it is not immediately obvious what the analogue should be. Whilst the McKay--Thompson series $H^{(2)}_{g}$ is a mock modular form of weight $1/2$ on some $\G_g<\SL_2(\ZZ)$ for every $g$ in $M_{24}$ \cite{Eguchi2010a}, it is not the case that $\G_g$ is a genus zero group for every $g$, and even if it were, there is no obvious sense in which one mock modular form of weight $1/2$ on some group can ``generate'' all the others, and thus no obvious analogue of the principal modulus property.

A solution to this problem---the formulation of the moonshine conjecture for $M_{24}$---was found in \cite{Cheng2011} (see also \S\ref{sec:conj:moon}) via an extension of the program that was initiated in \cite{DunFre_RSMG}; the antecedents of which include Rademacher's pioneering work \cite{Rad_FuncEqnModInv} on the elliptic modular invariant $J(\t)$, quantum gravity in three dimensions \cite{Witten2007,MalWit_QGPtnFn3D,LiSonStr_ChGrav3D}, the AdS/CFT correspondence in physics \cite{MaldacenaAdv.Theor.Math.Phys.2:231-2521998,Gubser:1998bc,Witten:1998qj}, and the application of {\em Rademacher sums} to these and other settings in string theory \cite{Dijkgraaf2007,Moore2007,BoerJHEP0611:0242006,KrausJHEP0701:0022007,Denef2007,Sen2007c,Dabholkar:2011ec} (and in particular \cite{Manschot2007}). To explain the formulation of the moonshine conjecture for $M_{24}$ we recall that in \cite{DunFre_RSMG} a Rademacher sum $R_{\G}(\t)$ is defined for each discrete group $\G<\PSL_2(\RR)$ commensurable with the modular group in such a way as to naturally generalise Rademacher's Poincar\'e series-like expression for the elliptic modular invariant derived in \cite{Rad_FuncEqnModInv}. It is then shown in \cite{DunFre_RSMG} that a holomorphic function on the upper-half plane (with invariance group commensurable with $\PSL_2(\ZZ)$) is the principal modulus for its invariance group if and only if it coincides with the Rademacher sum attached to this group. Thus the genus property of monstrous moonshine may be reformulated:
\begin{quote}
For each $g$ in $\MM$ we have $T_g=R_{\G_g}$ where $\G_g$ is the invariance group of $T_g$.
\end{quote}
Write $R^{(2)}_{\G}$ to indicate a weight $1/2$ generalisation of the (weight $0$) Rademacher sum construction $R_{\G}$ studied in \cite{DunFre_RSMG}. (Note that a choice of multiplier system on $\G$ is also required.) Then a natural $M_{24}$-analogue of the Conway--Norton moonshine conjecture comes into view:
\begin{quote}
For each $g$ in $M_{24}$ we have $H^{(2)}_g=R^{(2)}_{\G_g}$ where $\G_g$ is the invariance group of $H^{(2)}_{g}$.
\end{quote}
This statement is confirmed in \cite{Cheng2011} for the functions $H^{(2)}_g$ that are, at this point, conjecturally attached to $M_{24}$ via the conjectural $M_{24}$-module $K^{(2)}$. 

Through these results we come to envisage the possibility that both monstrous moonshine and the $M_{24}$ observation of Eguchi--Ooguri--Tachikawa will eventually be understood as aspects of one underlying structure which will include finite groups, various kinds of automorphic forms, extended algebras and string theory, and will quite possibly be formulated in terms of the AdS/CFT correspondence
in the context of some higher-dimensional string or gravitational theory. 

In fact we can expect this {\em moonshine structure} to encompass more groups beyond the monster and $M_{24}$: In this paper we identify the $M_{24}$ observation as one of a family of correspondences between finite groups and (vector-valued) mock modular forms, with each member in the family admitting a natural analogue of the Conway--Norton moonshine conjecture according to the philosophy of \cite{DunFre_RSMG,Cheng2011} (see also \cite{MR2985326}). For each $\ll$ in $\LL=\{2,3,4,5,7,13\}$---the set of positive integers $\ll$ such that $\ll-1$ divides $12$---we identify a distinguished Jacobi form $Z^{(\ll)}$, a finite group $G^{(\ll)}$ and a family of vector-valued mock modular forms $H^{(\ll)}_g$ for $g\in G^{(\ll)}$. 
The Jacobi forms $Z^{(\ll)}$ satisfy an {\em extremal} condition formulated in terms of unitary irreducible characters of the $N=4$ superconformal algebra (cf. \S\ref{sec:forms:wtzero}), the mock modular form $H^{(\ll)}=H^{(\ll)}_e$ is related to $Z^{(\ll)}$ as $H^{(2)}$ is to the elliptic genus of a $K3$ surface (cf. \S\S\ref{sec:forms:sca},\ref{sec:forms:wtzero}), the Fourier coefficients of the {\em McKay--Thompson series} $H^{(\ll)}_g$ support the existence of an infinite-dimensional graded module $K^{(\ll)}$ for $G^{(\ll)}$ playing a r\^ole analogous to that of $K^{(2)}$ for $G^{(2)}\simeq M_{24}$ (cf. \S\ref{sec:conj:mod}), and 
the following {\em umbral moonshine conjecture} is predicted to hold where $R^{(\ll)}_{\G}$ is an $(\ll-1)$-vector-valued generalisation (cf. \S\ref{sec:conj:moon}) of the Rademacher sum construction $R^{(2)}_{\G}$ studied in \cite{Cheng2011}.
\begin{quote}
For each $g$ in $G^{(\ll)}$ we have $H^{(\ll)}_g=R^{(\ll)}_{\G_g}$ where $\G_g$ is the invariance group of $H^{(\ll)}_{g}$.
\end{quote}
In addition to the above properties with monstrous analogues we find the following new {\em discriminant property}: The exponents of the powers of $q$ having non-vanishing coefficient in the Fourier development of $H^{(\ll)}$ determine certain imaginary quadratic number fields and predict the existence of dual pairs of irreducible representations of $G^{(\ll)}$ that are defined over these fields and irreducible over $\CC$. Moreover, these dual pairs consistently 
appear as irreducible constituents in homogeneous $G^{(\ll)}$-submodules of $K^{(\ll)}$ in such a way that the degree of the submodule determines the discriminant of the corresponding quadratic field (cf. \S\ref{sec:conj:disc}).

A main result of this paper is Theorem \ref{thm:forms:wtzero:ext}, which states that the extremal condition formulated in \S\ref{sec:forms:wtzero} characterises the Jacobi forms $Z^{(\ll)}$ for $\ll\in\{2,3,4,5,7,13\}$. To achieve this we establish a result of independent interest, which also serves to illustrate the depth of the characterisation problem: We show in Theorem \ref{lem:forms:wtzero:constvan} that the existence of an extremal Jacobi of index $m-1$ implies the vanishing of $L(f,1)$ for all new forms $f$ of weight $2$ and level $m$, where $L(f,s)$ is the Dirichlet series naturally attached to $f$. According to the Birch--Swinnerton-Dyer conjecture the vanishing of $L(f,1)$ implies that the elliptic curve $E_f$, attached to $f$ by Eichler--Shimura theory, has rational points of infinite order. Thus it is extremely unexpected that an extremal Jacobi form can exist for all but finitely many values of $m$, and using estimates \cite{MR2176151} due to Ellenberg together with some explicit computations we are able to verify that the only possible values are those for which no non-zero weight $2$ forms exist; i.e. $\G_0(m)$ has genus zero. The last step in our proof of Theorem \ref{thm:forms:wtzero:ext} is to check the finitely many corresponding finite dimensional spaces of Jacobi forms, for which useful bases have been determined by Gritsenko in \cite{Gri_EllGenCYMnflds}. In this we way obtain that the $Z^{(\ll)}$ for $\ll\in\{2,3,4,5,7,13\}$ are precisely the unique, up to scale, weak Jacobi forms of weight zero satisfying the extremal condition (\ref{eqn:forms:wtzero:ext}).

In contrast to the monstrous case the McKay--Thompson series $H^{(\ll)}_g$ arising here are mock modular forms, and these are typically not in fact modular but become so after {\em completion} with respect to a {\em shadow} function $S^{(\ll)}_g$. It turns out that all the McKay--Thompson series $H^{(\ll)}_g$ for fixed $\ll\in\LL$ have shadows that are (essentially) proportional to a single vector-valued {\em unary theta function} $S^{(\ll)}$ (cf. \S\ref{sec:forms:jac}) and so it is in a sense the six {\em moonlight shadows} $S^{(\ll)}$ for $\ll\in \LL$ that provide the irreducible information required to uncover the structure that we reveal in this paper. We therefore refer to the phenomena investigated here as  {\em umbral moonshine}.

According to the Oxford English Dictionary, a flame or light that is {\em lambent} is playing ``lightly upon or gliding over a surface without burning it, like a `tongue of fire'; shining with a soft clear light and without fierce heat.'' And since the light of umbral moonshine is apparently of this nature, we call the six values in $\LL=\{2,3,4,5,7,13\}$ {\em lambent}, and we refer to the index $\ll\in\LL$ as the {\em lambency} of the connections relating the {\em umbral group} $G^{(\ll)}$ to the {\em umbral forms} $Z^{(\ll)}$ and $H^{(\ll)}_g$.

The rest of this paper is organised as follows. In \S \ref{sec:forms} we discuss properties of Jacobi forms, Siegel forms, mock modular forms and mock theta functions. We explain two closely related ways in which Jacobi forms determine mock modular forms, one involving the decomposition into characters of the $N=4$ superconformal algebra and the other involving a decomposition of meromorphic Jacobi forms into mock modular forms following \cite{zwegers} and  \cite{Dabholkar:2012nd}.   In \S \ref{sec:forms:wtzero} we introduce the Jacobi forms $Z^{(\ell)}$ of weight $0$ and index $\ell-1$ for lambent $\ll$ and their  associated vector-valued mock modular forms $H^{(\ell)}$. We prove (Theorem \ref{thm:forms:wtzero:ext}) that these functions are characterized by the extremal property (\ref{eqn:forms:wtzero:ext}) and we establish the connection (Theorem \ref{thm:forms:wtzero:ccv}) to critical values of automorphic $L$-functions. We note that the coefficients in the $q$-expansions of the mock modular forms $H^{(\ll)}$ appear to be connected to the dimensions of irreducible representations of groups $G^{(\ell)}$ which we introduce and study in \S\ref{sec:grps}. In \S\S\ref{sec:grps:dynI},\ref{sec:grps:dynII} we discuss the remarkable fact that some of these groups manifest both the McKay correspondence relating ADE Dynkin diagrams to finite subgroups of $\SU(2)$ as well as a generalisation of his monstrous $E_8$ observation. In \S \ref{sec:mckay} we discuss analytic properties of the McKay--Thompson series $H^{(\ell)}_{g}$ which are obtained by twisting the mock modular forms $H^{(\ell)}$ by elements $g \in G^{(\ell)}$. We determine the proposed McKay--Thompson series $H^{(\ll)}_{g}=\big(H^{(\ll)}_{g,r}\big)$ precisely in terms of modular forms of weight $2$ for all but a few $g$ occurring for $\ll\in\{7,13\}$, and we find that we can identify many of the component functions $H^{(\ll)}_{g,r}$ either with ratios of products of eta functions or with classical mock theta functions introduced by Ramanujan (and others). In \S \ref{sec:conj} we collect our observations into a set of conjectures. These include the analogue of Thompson's conjecture (cf. \S\ref{sec:conj:mod}), the umbral counterpart to the Conway--Norton moonshine conjecture (cf. \S\ref{sec:conj:moon}), and a precise formulation of the discriminant property mentioned above (cf. \S\ref{sec:conj:disc}).  In \S \ref{sec:conj:geomphys} we mention some possible connections between our results, the geometry of complex surfaces, and string theory.  

Our conventions for modular forms appear in \S\ref{sec:modforms}, the character tables of the umbral groups $G^{(\ll)}$ appear in \S\ref{sec:chars}, tables of Fourier coefficients of low degree for all the proposed McKay--Thompson series $H^{(\ll)}_g$ appear in \S\ref{sec:coeffs}, and tables describing the $G^{(\ll)}$-module structures implied (for low degree) by the $H^{(\ll)}_g$ are collected in \S\ref{sec:decompositions}. 

It is important to mention that much of the data presented in the tables of \S\ref{sec:coeffs} was first derived using certain vector-valued generalisations of the Rademacher sum construction that was applied to the functions of monstrous moonshine in \cite{DunFre_RSMG}, and in \cite{Cheng2011} to the functions attached to $M_{24}$ via the observation of Eguchi--Ooguri--Tachikawa. In particular, these vector-valued Rademacher sums played an indispensable r\^ole in helping us arrive at the groups $G^{(\ll)}$ specified in \S\ref{sec:grps}, especially for $\ll>3$, and also allowed us to formulate and test hypotheses regarding the modularity of the (vector-valued) functions $H^{(\ll)}_{g}$, including eta product expressions and the occurrences of classical mock theta functions; considerations which ultimately developed into the discussion of \S\ref{sec:mckay}. A detailed discussion of the Rademacher construction is beyond the scope of this article but a full treatment is to be the focus of forthcoming work. The fact that the Rademacher sum approach proved so powerful may be taken as strong evidence in support of the {\em umbral moonshine conjecture}, Conjecture \ref{conj:conj:moon}.

\section{Automorphic Forms}\label{sec:forms}

In this section we discuss the modular objects that play a r\^ole in the connection  between mock modular forms and finite groups that we will develop later in the paper. We also establish our notation and describe various relationships between Jacobi forms, theta functions, and vector-valued mock modular forms, including mock theta functions. 

In what follows we take 
$\tau$ in the upper half-plane $\HH$ and $z \in \CC$, and adopt
the shorthand notation $e(x) = e^{2\p ix}$. We also define $q=e(\tau)$ and $y=e(z)$ and write 
\be
 \g\t = \frac{a\t+b}{c\t+d}, \quad \g= 
 	\begin{pmatrix}
	a&b\\
	c&d
	\end{pmatrix}
	\in\SL_2(\ZZ)
\ee
for the natural action of $\SL_2(\ZZ)$ on $\HH$
and
\be
 \g (\t,z) = \left(\frac{a\t+b}{c\t+d},\frac{z}{c\t+d} \right)
\ee
for the action of $\SL_2(\ZZ)$ on $\HH \times \CC$. 

\subsection{Mock Modular Forms}\label{sec:forms:mock}

Mock theta functions were first introduced in 1920 by S. Ramanujan in his  last letter to Hardy. This letter contained $17$ examples divided into four of {\em order $3$}, ten of
{\em order $5$} and three of  {\em order $7$}. Ramanujan did not define what he meant by the term order and to this day there seems to be no universally agreed upon definition. In this paper we use the term order only as a historical label. 
Ramanujan wrote his mock theta functions as what he termed ``Eulerian series'' that today
would be recognised as specialisations of $q$-hypergeometric series. A well studied example is the order $3$ mock theta
function
\be \label{fmock}
f(q) =  1+ \sum_{n=1}^\infty \frac{q^{n^2}}{(1+q)^2(1+q^2)^2 \cdots (1+q^n)^2} = \sum_{n=0}^\infty \frac{q^{n^2}}{(-q;q)_n^2} \;,   
\ee
where we have introduced the {\em $q$-Pochhammer} symbol
\be
(a;q)_n =\prod_{k=0}^{n-1} (1-aq^k) \,.
\ee

Interest in and applications of mock theta functions has burgeoned during the last decade following the work of
Zwegers \cite{zwegers}
who found an intrinsic definition of mock theta functions and their near modular behavior, and many applications
of his work can be found in combinatorics  \cite{BringmannOno2006,BringmannOno2010}, characters of infinite-dimensional Lie superalgebras \cite{Eguchi2008,Eguchi2009a}, topological field theory \cite{VafaWitten1994,Gottsche_Zagier,Troost:2010ud,Manschot2011}, the computation of quantum invariants
of three-dimensional manifolds \cite{Lawrence_Zagier}, and the counting of black hole states in string theory \cite{Dabholkar:2012nd}. Descriptions of this breakthrough and some of the history of mock theta functions can be found in
\cite{zagier_mock}, \cite{Folsom_what} and \cite{Ono_unearthing}.

Mock theta function are now understood as a special case of more general objects known as mock modular forms. 
A holomorphic function $h(\t)$ on $\HH$ is called a {\em (weakly holomorphic) mock modular form} of weight $k$ for a discrete group 
$\G$ (e.g. a congruence subgroup of $\SL_2(\ZZ)$) if it has at most exponential growth as $\t\to\alpha$ for any $\alpha\in \mathbb{Q}$, and if there exists a holomorphic modular form $f(\t)$ of weight $2-k$ on $\G$ such that the {\em completion} of $h$ given by
\be\label{def_mock}
\hat{h}(\t)=h(\t)+\left(4{i}\right)^{k-1}\int_{-\bar \t}^{\infty}(z+\t)^{-k}\overline{f(-\bar z)}{\rm d}z,
\ee
is a (non-holomorphic) modular form of weight $k$ for $\G$
for some multiplier system $\nu$ say. 
In this case the function $f$ is called the {\em shadow} of the mock modular form $h$.
Even though $h$ is not a modular form, it is common practice to call $\nu$  the multiplier system of $h$.  One can show that $\nu$ is the conjugate of the multiplier system of $f$.  In most of the examples in this paper we will deal with vector-valued mock modular forms so that the completion in fact transforms
as $\nu(\g)\hat{h}(\gamma \tau)(c\tau+d)^{-k} =  \hat{h}(\tau)$ in case the weight is $k$ for
all $\gamma= (\begin{smallmatrix} * & * \\ c & d \end{smallmatrix}) \in \Gamma$ where $\nu$ is a matrix-valued function on $\Gamma$.

The completion $\hat{h}(\t)$ satisfies interesting differential equations. For instance, 
completions of mock modular forms were identified as weak Maass forms (non-holomorphic modular forms which are eigenfunctions of the Laplace operator)  in \cite{BringmannOno2006} as a part of their solution to the longstanding Andrews--Dragonette conjecture. Note that we have the identity
\be
	2^{1-k}\pi\Im(\t)^k\frac{\partial\hat{h}(\t)}{\partial\bar{\t}}=-2\pi i\overline{f(\t)}
\ee
when $f$ is the shadow of $h$.

Thanks to Zweger's work we may define a {\em mock theta function} to be a $q$-series $h=\sum_n a_n q^n$ such that for some $\lambda \in \QQ$ the assignment $\t\mapsto q^{\lambda}h|_{q=e(\t)}$ defines a mock
modular form of weight $1/2$ whose shadow is a unary (i.e. attached to a quadratic form in one variable) theta series of weight ${3}/{2}$.

In this paper we add one more r\^ole for mock theta functions to the list mentioned earlier; namely we conjecture that specific sets of mock theta functions appear as McKay--Thompson series associated to (also conjectural) infinite-dimensional modules for a sequence of groups $G^{(\ll)}$ which we refer to as the {\em umbral groups} and label by the {\em lambent} integers $\ll \in \Lambda=\{2,3,4,5,7,13\}$, which are just those positive integers that are one greater than a divisor of $12$. 

Many of the mock theta functions that appear later in this paper appear either in Ramanujan's last letter to Hardy or in his lost notebook \cite{Ramanujan_lost}. These include an order $2$ mock theta function 
\be
\m(q) =  \sum_{n\geq 0}  \frac{(-1)^n\, q^{n^2}(q;q^2)_n}{(-q^2;q^2)_n^2}  \\\label{ll2mocktheta}
\ee
and an order $8$ mock theta function 
\be
 U_0(q) = \sum_{n\geq 0}  \frac{q^{n^2}(-q;q^2)_n}{(-q^4;q^4)_n},
\ee
both of which appear at lambency $2$ in connection with $G^{(2)}\simeq M_{24}$. The function $f(q)$ of \eqref{fmock} together with 
\begin{gather}
\begin{split}\label{eqn:forms:mock:ord3}
\phi(q) &=  1+\sum_{n=1}^\infty \frac{q^{n^2}}{(1+q^2)(1+q^4) \cdots (1+q^{2n})} \,  \\
\chi(q) &= 1+ \sum_{n=1}^\infty \frac{q^{n^2}}{(1-q+q^2)(1-q^2+q^4) \cdots (1-q^n+q^{2n})} \, \\
\omega(q) &=  \sum_{n=0}^\infty \frac{q^{2n(n+1)}}{(1-q)^2(1-q^3)^2 \cdots (1-q^{2n+1})^2} \, \\
\rho(q) &=  \sum_{n=0}^\infty \frac{q^{2n(n+1)}}{(1+q+q^2)(1+q^3+q^6) \cdots (1+q^{2n+1}+q^{4n+2})} \;
\end{split}
\end{gather}
constitute five order $3$ mock theta functions appearing at lambency $3$, and the four order $10$ mock theta functions 
\begin{gather}
\begin{split}\label{eqn:forms:mock:ord10}
\phi_{10}(q) &= \sum_{n=0}^\infty \frac{q^{n(n+1)/2}}{(q;q^2)_{n+1}} \; \\
\psi_{10}(q) &= \sum_{n=0}^\infty \frac{q^{(n+1)(n+2)/2}}{(q;q^2)_{n+1}} \; \\
X(q) &= \sum_{n=0}^\infty \frac{(-1)^n q^{n^2}}{(-q;q)_{2n}} \; \\
\chi_{10}(q) &= \sum_{n=0}^\infty \frac{(-1)^n q^{(n+1)^2}}{(-q;q)_{2n+1}} \;
\end{split} 
\end{gather}
appear at lambency $5$.

More mock theta functions were found later by others.  At lambency $4$ we will encounter  order $8$ mock theta functions discussed in \cite{Gordon_Mcintosh} with $q$-expansions
\begin{gather}
\begin{split}\label{eqn:forms:mock:ord8}
S_0(\t) &= \sum_{n\geq 0} \frac{q^{n^2} (-q;q^2)_n}{(-q^2;q^2)_n} \\
S_1(\t) &= \sum_{n\geq 0} \frac{q^{n(n+2)} (-q;q^2)_n}{(-q^2;q^2)_n} \\
T_0(\t) &= \sum_{n\geq 0} \frac{q^{(n+1)(n+2)} (-q^2;q^2)_n}{(-q;q^2)_n} \\
T_1(\t) &= \sum_{n\geq 0} \frac{q^{n(n+1)} (-q^2;q^2)_n}{(-q;q^2)_n} .
\end{split}
\end{gather}
It is curious to note that the order is divisible by the lambency in each example.

\subsection{Jacobi Forms}\label{sec:forms:jac}

We now discuss Jacobi forms following \cite{eichler_zagier}. We say a holomorphic function $\f: \HH \times \CC \to \CC$ is an {\em unrestricted Jacobi form} of weight $k$ and index $m$ for $\SL_2(\ZZ)$ if it transforms under the Jacobi group 
$\SL_2(\ZZ)\ltimes \ZZ^2$ as
\bea
\f(\t,z) &=& (c\t+d)^{{-k}} e(-m \tfrac{c z^2}{c\t+d})\, \f(\g(\t, z)) \label{modular}
\\
\f(\t,z) &=&e( m(\l^2 \t + 2\l z)) \, \f(\t, z+\l \t +\m) \label{elliptic}
\eea
where $\g \in \SL_2(\ZZ)$ and $ \l,\m\in \ZZ$.  In what follows we refer to the transformations  \eqref{modular} and \eqref{elliptic} as the {\em modular} and {\em elliptic} transformations, respectively.
The invariance of $\f(\tau,z)$ under $\tau \rightarrow \tau+1$ and $z \rightarrow z+1$ implies a Fourier expansion
\be\label{eqn:forms:jac:FouExp}
\f(\t,z) = \sum_{n,r \in \ZZ} c(n,r) q^n y^r 
\ee
and the elliptic transformation can be used to show that $c(n,r)$ depends only on the {\em discriminant} $r^2-4mn$ and $r ~{\rm mod}~ 2m$, and so we have $c(n,r)=C(r^2-4mn,{\til r})$
for some function $D\mapsto C(D,{\til r})$ where $\til r\in \{-m,\dots,m-1\}$, for example. An unrestricted Jacobi form is called a {\em weak Jacobi form}, a {\em (strong) Jacobi form}, or a {\em Jacobi cusp form} according as the Fourier coefficients satisfy $c(n,r)=0$ whenever $n< 0$, $C(D,{\til r})=0$ whenever $D > 0$, or $C(D,{\til r})=0$ whenever $D \ge 0$, respectively. In a slight departure from \cite{eichler_zagier} we denote the space of weak Jacobi forms of weight $k$ and index $m$ by $J_{k,m}$. 

In what follows we will need two further generalisations of the above definitions. The first is straightforward and replaces  $\SL_2(\ZZ)$  by a finite index subgroup $\Gamma \subset \SL_2(\ZZ)$ in the modular transformation law. The second is more subtle and leads to {\em meromorphic Jacobi forms} which obey the modular and elliptic transformation laws but are such that the functions $z\mapsto \phi(\t,z)$ are allowed to have poles lying at values of $z\in\CC$ that map to torsion points of the elliptic curve $\CC/(\ZZ \tau + \ZZ)$. Our treatment of meromorphic Jacobi forms follows \cite{zwegers} and \cite{Dabholkar:2012nd}. 

A property of (weak) Jacobi forms that will be important for us later is that they admit an expansion in terms of the {\em index $m$ theta functions},
\be\label{thetexp}
\theta^{(m)}_r(\tau,z) = \sum_{n \in \ZZ} q^{(2mn+r)^2/4m} y^{2mn+r},
\ee
given by
\be\label{eqn:forms:jac:thetaxpn}
\phi(\tau,z)= \sum_{r ({\rm mod}~2m)} \til h_r(\tau) \theta^{(m)}_r(\tau,z)
\ee
in case the index of $\phi$ is $m$, where the {\em theta-coefficients} $\til h_r(\tau)$ constitute the components of a vector-valued modular form of weight $k-1/2$ when $k$ is the weight of $\phi$.
Recall that a vector valued function $\til h=(\til h_{r})$ is called a {\em vector-valued modular form} of weight $k$ for $\G \subset \SL_2(\ZZ)$  if 
\be
\til h_{r}(\t) = \frac{1}{(c\t+d)^k} \sum_{s} \nu_{rs}(\g)\til h_{s}(\t)
\ee
for all $\g\in\G$ and $\t\in\HH$ for some matrix-valued function $\nu=(\nu_{rs})$ on $\G$ called the {\em multiplier system} for $\til h$.

The modular transformation law \eqref{modular} with $\gamma=-I_2$ implies that $\phi(\tau,-z)=(-1)^k \phi(\tau,z)$. Combining this with the identity
$\theta^{(m)}_{-r}(\tau,z)= \theta^{(m)}_r(\tau,-z)$ we see that $ \til h_r(\t) = (-1)^k \til h_{-r}(\t)$ and in particular we can recover a weak Jacobi form of weight $k$ and index $m$ from the $m-1$ theta-coefficients $\{\til h_1,\cdots,\til h_{m-1}\}$ in case $k$ is odd. 

In what follows we will encounter only weight $0$ and weight $1$ Jacobi forms; a typical such form will be denoted by $\phi(\tau,z)$ or $\psi(\tau,z)$ according as the weight is $0$ or $1$ and we will write the theta-expansion of a weight $1$ Jacobi form $\psi$ as
\be\label{eqn:forms:jac:wt1thetaxpn}
\psi(\tau,z)= \sum_{r=1}^{m-1}  \til h_r(\tau) \hat \theta^{(m)}_r(\tau,z)
\ee
where $\hat \theta^{(m)}_r(\tau,z)= \theta^{(m)}_{-r}(\tau,z)-\theta^{(m)}_{r}(\tau,z)$ (cf. (\ref{thetexp})) for $r \in \{1, 2, \cdots, m-1 \}$.

\subsection{Meromorphic Jacobi Forms}\label{sec:forms:mero}

We now explain a connection between the vector-valued mock modular forms we shall consider in this paper and meromorphic Jacobi forms. We specialize our discussion to weight $1$ meromorphic Jacobi forms of a particular form that arise in our pairing of Jacobi forms with groups $G^{(\ell)}$, and later in our computation of McKay--Thompson series; namely, we consider meromorphic weight $1$ and index $m$ Jacobi forms which can be written as
\be \label{psiphi}
\psi(\tau,z)=  \Psi_{1,1}(\tau,z) \phi(\tau,z) 
\ee
for some weight $0$ index $m-1$ (holomorphic) weak Jacobi form $\phi$ where $\Psi_{1,1}$ is the specific meromorphic Jacobi form of weight $1$ and index $1$ given by
\be\label{CoverA}
\Psi_{1,1}(\tau,z) = -i \,\frac{\theta_1(\tau,2z)\, \eta(\tau)^3}{(\theta_1(\tau,z))^2}= \frac{y+1}{y-1}- (y^2-y^{-2})q+ \cdots.
\ee
We note that $\psi(\tau,z)$ has a simple pole at $z=0$ with residue $  \phi(\tau,0)/\p i$.  Since $\phi(\tau,z)$ is a weak Jacobi form of weight $0$ the function $\t\mapsto\phi(\tau,0)$ is a modular form of weight $0$ (with no poles at any cusps) and is hence equal to a constant; we denote this constant by
\be\label{eqn:forms:mero:phi0}
\chi=\phi(\tau,0).
\ee

It was shown by Zwegers \cite{zwegers} that meromorphic Jacobi forms have a modified theta-expansion (cf. \eqref{eqn:forms:jac:thetaxpn}) in terms of vector-valued mock modular forms; in \cite{Dabholkar:2012nd} this expansion was recast as follows. Define the averaging operator 
\be
{\rm Av}^{(m)}\Big[F(y)\Big] = \sum_{k\in \ZZ} q^{m k^2} y^{2m k} F(q^k y )
\ee
which takes a function of $y=e(z)$ with polynomial growth and returns a function of $z$ which transforms like an index $m$ Jacobi form under the elliptic transformations \eqref{elliptic}. Now define the {\em polar part} of $\psi=\Psi_{1,1}\phi$ by
\be \label{polar}
\psi^{P}(\tau,z)=  \chi
	{\rm Av}^{(m)}\Big[ \frac{y+1}{y-1} \Big] 
\ee
where $\chi=\phi(\t,0)$ and define the {\em finite part} of $\psi$ by
\be \label{free}
\psi^{F}(\tau,z)= \psi(\tau,z)-\psi^{P}(\tau,z).
\ee
The term finite is appropriate because with the polar part subtracted the finite part  no longer has a pole at $z=0$. 

It follows from the analysis in \cite{Dabholkar:2012nd}  that  $\psi^{F}(\tau,z)$ is a weight $1$ index $m$ {\em mock Jacobi form},  meaning that it has a theta-expansion
\be \label{freedecom}
\psi^{F}(\tau,z)= \sum_{r=1}^{m-1} h_r(\tau) \hat \theta^{(m)}_r(\tau,z)
\ee
(cf. \eqref{eqn:forms:jac:wt1thetaxpn}) where the theta-coefficients $h_r$ comprise the components of a vector-valued mock modular form of weight $1/2$.

In the above we have again used the fact that $\psi$, $\psi^{P}$ and hence also $\psi^{F}$ pick up a minus sign under the transformation $z \to -z$. Moreover, the vector-valued mock modular forms obtained in this way always have shadow function given by the {\em unary theta series} 
 \be\label{unary_S_def}
S^{(m)}_{r}(\t) =\frac{1}{2\p i} \left.\frac{\pa}{\pa z} \th^{(m)}_{r}(\t,z) \right|_{z=0}=  \sum_{n\in\ZZ} (2m n+r)q^{{(2m n+r)^2}/{4m}}.
\ee
To see why this is so, and for later use, we introduce the functions
\be\label{gAPsum}
\m^{(m)}_{j} (\t,z) =(-1)^{1+2j}  \sum_{k\in \ZZ}  q^{m k^2} y^{2 m k} \frac{(yq^{k})^{-2j}+(yq^{k})^{-2j+1}+\dots+(yq^k)^{1+2j} }{ 1-yq^k}
\ee
for $m$ a positive integer and $2j\in\{0,1,\ldots,m-1\}$. Note that  $\m^{(m)}_{0}(\t,z)$ is proportional to the polar part of $\psi$, as identified in \eqref{polar},
\be
\m^{(m)}_{0} (\t,z) = {\rm Av}^{(m)}\Big[\frac{y+1}{y-1}\Big].
\ee
\begin{rmk}
The function $\m^{(2)}_0$ is closely related to the Appell-Lerch sum $\mu(t,z)$ which features prominently in \cite{zwegers}.
\end{rmk}

The $\m^{(m)}_0$ enjoy the following relation to the modular group $\SL_2(\ZZ)$. Define the {\em completion} of $\m^{(m)}_{0} (\t,z)$ by setting
\be\label{pole_completion}
\hat \m^{(m)} (\t,\bar \t,z) = \m^{(m)}_{0} (\t,z) + \frac{1}{\sqrt{2m}} \frac{1}{(4i)^{1/2}}  \sum_{r=-m}^{m-1} \th^{(m)}_{r}(\t,z)  \int^{i\inf}_{-\bar \t}  (z+\t)^{-1/2} \overline{S^{(m)}_{r}(-\bar z)} \, {\rm d}z. 
\ee
Then $\hat\m^{(m)}$ transforms like a Jacobi form of weight $1$ and index $m$ for $\SL_2(\ZZ)$ but is not holomorphic. Therefore, from the transformation of the polar part 
\be\label{eqn:forms:mero:psipolarmu}
\psi^{P}(\t,z) = \chi
	\m_0^{(m)}(\t,z)\,
\ee
we see that the shadow of $h=(h_r)$ is given by $\chi S^{(m)}= \big(\chi S^{(m)}_r\big)$, as we claimed. This means that the vector-valued mock modular forms arising in this way are closely related to mock theta functions: By the definition given in \S\ref{sec:forms:mock} we have $h_r(\tau)= q^{\lambda} M_r$ for some $\lambda \in \QQ$ with $M_r$ a mock theta function.

\subsection{Superconformal Algebra}\label{sec:forms:sca}

In \S\ref{sec:forms:mero} we saw how to associate a vector-valued mock modular form $(h_{r})$ to a weight $1$ meromorphic Jacobi form $\psi$ satisfying $\psi= \Psi_{1,1} \phi$ for some weak Jacobi form $\phi$ via the theta-expansion of the finite part of $\psi$. 
It will develop that the weight $0$ forms $\phi$ of relevance to us have a close relation to the representation theory of the $2$-dimensional $N=4$ superconformal algebra.
To see this recall (cf. \cite{Eguchi1987,Eguchi1988,MR2060475}) that this algebra contains subalgebras isomorphic to the affine Lie algebra $\hat{\mathfrak{sl}}_2$ and the Virasoro algebra, and in a unitary representation the former of these acts with level $m-1$, for some integer $m>1$, and the latter with central charge $c=6(m-1)$, and the unitary irreducible highest weight representations $V^{(m)}_{h,j}$ are labelled by the two {\em quantum numbers} $h$ and $j$ which are the eigenvalues of $L_0$ and $\frac{1}{2}J_0^3$, respectively, when acting on the highest weight state. (We adopt a normalisation of the {$\SU(2)$ current} $J^3$ such that the zero mode $J^3_0$ has integer eigenvalues.) In the Ramond sector of the
superconformal algebra there are two types of highest weight representations: the {\em massless} (or {\em BPS}, {\em supersymmetric}) ones with $h=\frac{m-1}{4}$ and $j\in\{ 0,\frac{1}{2},\cdots,\frac{m-1}{2}\}$, and the {\em massive} (or {\em non-BPS}, {\em non-supersymmetric}) ones with $h > \frac{m-1}{4}$ and $j\in\{ \frac{1}{2},1,\cdots, \frac{m-1}{2}\}$. Their {\em (Ramond) characters}, defined as 
\be
{\rm ch}^{(m)}_{h,j}(\t,z) = \tr_{V^{(m)}_{h,j}} \left( (-1)^{J_0^3}y^{J_0^3} q^{L_0-c/24}\right),
\ee
are given by 
\be \label{masslesschar}
	{\rm ch}^{(m)}_{\frac{m-1}{4},j} (\t,z)= 
	(\Psi_{1,1}(\tau,z))^{-1} \m^{(m)}_{j}  (\t,z)
\ee
and 
\be \label{massivechar}
	{\rm ch}^{(m)}_{h,j} (\t,z) =
	(-1)^{2j+1}
	(\Psi_{1,1}(\tau,z))^{-1} \,q^{h-\frac{m-1}{4}-\frac{j^2}{m}} \, \hat \th^{(m)}_{2j} (\t,z) 
\ee
in the massless and massive cases, respectively, \cite{Eguchi1988} where the function $\m^{(m)}_{j}(\t,z)$ is defined as in  \eq{gAPsum} and $\hat{\theta}^{(m)}_{r}(\t,z)$ is as in the sentence following (\ref{eqn:forms:jac:wt1thetaxpn}).

We can use the above results to derive a decomposition of an arbitrary weight $0$ index $m-1$ Jacobi form $\phi(\tau,z)$ into $N=4$ characters as follows. Set $\psi(\tau,z)=\Psi_{1,1}(\tau,z)\phi(\t,z)$ as in \S\ref{sec:forms:mero} and write 
\begin{gather}
h_r(\t)= \sum_{n}c_r(n-r^2/4m) q^{n-r^2/4m} 
\end{gather}
for the Fourier expansion of the theta-coefficient $h_r$ of the finite part $\psi^F$ of $\psi$ (cf. \eqref{freedecom}). Then use \eqref{free}, \eqref{freedecom} and (\ref{eqn:forms:mero:psipolarmu}) along with \eqref{masslesschar} and \eqref{massivechar}
to obtain
\begin{gather}\label{eqn:forms:sca:phisadecomp}
\begin{split}
\f
=& \chi \, {\rm ch}^{(m)}_{\frac{m-1}{4},0} 
+ \sum_{r=1}^{m-1} (-1)^{r+1} c_r(-\tfrac{r^2}{4m})  \left( {\rm ch}^{(m)}_{\frac{m-1}{4},\frac{r}{2}} +2 {\rm ch}^{(m)}_{\frac{m-1}{4},\frac{r-1}{2}} + {\rm ch}^{(m)}_{\frac{m-1}{4},\frac{r-2}{2}} \right)
\\
&+ \sum_{r=1}^{m-1}(-1)^{r+1} \sum_{n=1}^\infty c_r(n-\tfrac{r^2}{4m})\, {\rm ch}^{(m)}_{\frac{m-1}{4}+n,\frac{r}{2}}
\end{split}
\end{gather}
where $\chi=\phi(\t,0)$ is the constant such that $\chi S^{(m)}$ is the shadow of $h=(h_r)$ (cf. \S\ref{sec:forms:mero}). In deriving (\ref{eqn:forms:sca:phisadecomp}) we have used the relation 
\be
\m^{(m)}_{\frac{r-2}{2}}
+2\m^{(m)}_{\frac{r-1}{2}}
+\m^{(m)}_{\frac{r}{2}}
=(-1)^{r+1} \,q^{-\frac{r^2}{4m}} \hat \th^{(m)}_{r}
\ee
subject to the understanding that $\m^{(m)}_{-\frac{1}{2}} = {\rm ch}^{(m)}_{\frac{m-1}{4},-\frac{1}{2}}=0$.

One way a weak Jacobi form of weight $0$ and index $m-1$ having integer coefficients can arise in nature is as the elliptic genus of a $2$-dimensional $N=4$ superconformal field theory with central charge $c=6(m-1)$. There is a unique weak Jacobi form of weight $0$ and index $1$ up to scale and this (suitably scaled) turns out to be the elliptic genus of a superconformal field theory attached to a $K3$ surface. Then the above analysis at $m=2$ recovers the mock modular form $H^{(2)}$ (the vectors have $m-1=1$ components in this case) exhibiting the connection to the Mathieu group $M_{24}$ observed by Eguchi--Ooguri--Tachikawa in \cite{Eguchi2010}. 

In the next section we will construct a distinguished family of extremal weight $0$ weak Jacobi forms $\f^{(\ll)}(\tau,z)$ with corresponding vector-valued weight $1/2$ mock modular forms $H^{(\ll)}(\tau)$ according to the procedure (\ref{psiphi}-\ref{freedecom}) of \S\ref{sec:forms:mero}. In \S\ref{sec:grps} we will specify finite groups $G^{(\ll)}$ for which the forms $H^{(\ll)}$ will serve as generating functions for the graded dimensions of conjectural bi-graded infinite-dimensional $G^{(\ll)}$-modules.

\subsection{Extremal Jacobi Forms}\label{sec:forms:wtzero}

In this section we identify the significance of the divisors of $12$ from the point of view of the $N=4$ superconformal algebra. We introduce the notion of an {extremal} (weak) Jacobi form of weight $0$ and integral index and we prove that the space $J^{\rm ext}_{0,m-1}$ of extremal forms with index $m-1$ has dimension $1$ if $m-1$ divides $12$, and has dimension $0$ otherwise. 
Several of the extremal Jacobi forms we identify have appeared earlier in the literature in the context of studying decompositions of elliptic genera of Calabi--Yau manifolds \cite{Gri_EllGenCYMnflds,Gri_CplxVBsJacFrms,Eguchi2008} and in a recent study of the connection between black hole counting in superstring theory and mock modular forms \cite{Dabholkar:2012nd}. 

For $m$ a positive integer and $\phi$ a weak Jacobi form with weight $0$ and index $m-1$ say $\phi$ is {\em extremal} if it admits an expression
\begin{gather}\label{eqn:forms:wtzero:ext}
	\f
	=	
	a_{\frac{m-1}{4},0}{\rm ch}^{(m)}_{\frac{m-1}{4},0} 
	+  
	a_{\frac{m-1}{4},\frac{1}{2}}{\rm ch}^{(m)}_{\frac{m-1}{4},\frac{1}{2}} 
			+ \sum_{0<r<m}
			\sum_{\substack{n\in\ZZ\\ r^2-4mn< 0}}
			a_{\frac{m-1}{4}+n,\frac{r}{2}} {\rm ch}^{(m)}_{\frac{m-1}{4}+n,\frac{r}{2}}
\end{gather}
for some $a_{h,j}\in \CC$ (cf. (\ref{eqn:forms:sca:phisadecomp})) where the $N=4$ characters are as defined in (\ref{masslesschar}) and (\ref{massivechar}). Write $J^{\rm ext}_{0,m-1}$ for the subspace of $J_{0,m-1}$ consisting of extremal weak Jacobi forms. Note that the extremal condition restricts both the massless and massive $N=4$ representations that can appear, for generally there are non-zero massive characters ${\rm ch}^{(m)}_{\frac{m-1}{4}+n,\frac{r}{2}}$ with $r^2-4mn\geq 0$.

We observe here that the extremal condition has a very natural interpretation in terms of the mock modular forms of weight $1/2$ attached to weak Jacobi forms of weight $0$ via the procedure detailed in \S\ref{sec:forms:mero}. By comparing with (\ref{eqn:forms:sca:phisadecomp}) we find that the condition (\ref{eqn:forms:wtzero:ext}) on a weak Jacobi form $\phi$ of index $m-1$ is equivalent to requiring that the corresponding vector-valued mock modular form $(h_r)$ obtained from the theta-expansion (\ref{freedecom}) of the finite part of the weight $1$ Jacobi form $\psi=\Psi_{1,1}\f$ has a single polar term $q^{-\frac{1}{4m}}$ in the first component $h_1$ and has all other components vanishing as $\t\to i\inf$.

Our main result in this section is the following characterisation of extremal Jacobi forms.
\begin{thm}\label{thm:forms:wtzero:ext}
If $m$ is a positive integer then $\dim J^{\rm ext}_{0,m-1}=1$ in case $m-1$ divides $12$ and $\dim J^{\rm ext}_{0,m-1}=0$ otherwise.
\end{thm}

In preparation for the proof of Theorem \ref{thm:forms:wtzero:ext} we now summarise (aspects of) a useful construction given in \cite{Gri_EllGenCYMnflds}. The graded ring 
\be
J_{0,*}= \bigoplus_{m \geq 1} J_{0,m-1}
\ee
of (weak) Jacobi forms of weight $0$ and integral index 
has an ideal
\be
J_{0,*}(q)=\bigoplus_{m>1}J_{0,m-1}(q)= \left\{ \phi \in J_{0,*} \mid \phi(\tau,z)= \sum_{\substack{n, r \in \ZZ\\n >0}} c(n,r) q^n y^r \right\}
\ee
consisting of Jacobi forms that vanish in the limit as $\t\to i\inf$ (i.e. have vanishing coefficient of $q^0y^r$, for all $r$, in their Fourier development). This ideal is principal and generated by a weak Jacobi
form of weight $0$ and index $6$ given by
\be
\zeta(\tau,z)= \frac{\theta_1(\tau,z)^{12}}{\eta(\tau)^{12}}
\ee
(cf. \S\ref{sec:modforms} for $\theta_1$ and $\eta$). Gritsenko shows \cite{Gri_EllGenCYMnflds} that for any positive integer $m$ the quotient $J_{0,m-1}/J_{0,m-1}(q)$ is a vector space of dimension $m-1$ admitting a basis consisting of weight $0$ index $m-1$ weak Jacobi forms $\varphi^{(m)}_{n}$ (denoted $\psi^{(n)}_{0,m-1}$ in \cite{Gri_EllGenCYMnflds}) for $1\leq n\leq m-1$ such that the coefficient of $q^0y^k$ in $\varphi^{(m)}_{n}$ vanishes for $|k|>n$ but does not vanish for $|k|= n$. In fact Gritsenko works in the subring $J^{\ZZ}_{0,*}$ of Jacobi forms having integer Fourier coefficients and his $\varphi^{(m)}_{n}$ furnish a $\ZZ$-basis for the $\ZZ$-module $J^{\ZZ}_{0,*}/J^{\ZZ}_{0,*}(q)$. 

We show now that there are no non-zero extremal Jacobi forms in the ideal $J_{0,*}(q)$.
\begin{lem}\label{lem:forms:wtzero:constvan}
If $\phi$ is an extremal weak Jacobi form belonging to $J_{0,*}(q)$ then $\phi=0$.
\end{lem}
\begin{proof}
If $\phi$ belongs to $J_{0,*}(q)$ then the coefficients of $q^0y^k$ in the Fourier development of $\phi$ vanish for all $k$. This implies the vanishing of $a_{\frac{m-1}{4},0}$ and $a_{\frac{m-1}{4},\frac{1}{2}}$ in (\ref{eqn:forms:wtzero:ext}) where $m-1$ is the index of $\phi$, and this in turn implies that the meromorphic Jacobi form $\psi=\Psi_{1,1}\phi$ of weight $1$ and index $m$ coincides with its finite part $\psi=\psi^F=\sum_rh_r\hat\th^{(m)}_r$ (cf. (\ref{eqn:forms:mero:phi0}), (\ref{polar}), (\ref{freedecom})) and has theta-coefficients $h_r$ that remain bounded as $\t\to i\inf$. In particular, $\psi$ is a (strong) Jacobi form of weight $1$ and integral index but the space of such forms vanishes according to \cite{Sko_Thesis}, so $\phi$ must vanish also.
\end{proof}
As a consequence of Lemma \ref{lem:forms:wtzero:constvan} we obtain that there are no extremal Jacobi forms with vanishing massless contribution at spin $j=1/2$.
\begin{lem}\label{lem:forms:wtzero:nospinhalf}
If $\phi$ is an extremal weak Jacobi form of index $m-1$ such that $a_{\frac{m-1}{4},\frac{1}{2}}=0$ in (\ref{eqn:forms:wtzero:ext}) then $\phi=0$.
\end{lem}
\begin{proof}
If $\phi$ is of weight $0$ and index $m-1$ satisfying (\ref{eqn:forms:wtzero:ext}) then $\phi=a+b(y+y^{-1})+O(q)$ as $\t\to i\inf$ where $a=a_{\frac{m-1}{4},0}$ and $b=-a_{\frac{m-1}{4},\frac{1}{2}}$. Following \cite{Gri_EllGenCYMnflds} we set 
\be
\tilde{\phi}(\t,z)=\exp(-8\pi^2(m-1)G_2(\t)z^2)\phi(\t,z)
\ee
where $G_2(\t)=-\frac{1}{24}+\sum_{n>0}\sigma_1(n)q^n$ is the unique up to scale mock modular form of weight $2$ for $\SL_2(\ZZ)$ (with shadow a constant function) and consider the {\em Taylor expansion} $\tilde{\phi}(\t,z)=\sum_{n\geq 0} f_{n}(\t)z^{n}$. Then $f_0(\t)=\chi=a+2b$ and more generally $f_n(\t)$ is a modular form of weight $n$ for $\SL_2(\ZZ)$. In particular $f_2(\t)=0$. On the other hand the constant term of $f_2(\t)$ is $\frac{1}{3}\pi^2(m-1)\chi-4\pi^2b$ so $(m-1)\chi=12b$. By hypothesis $b=0$ so $\chi=0$ and $\phi$ belongs to $J_{0,m-1}(q)$ and thus vanishes according to Lemma \ref{lem:forms:wtzero:constvan}.
\end{proof}

Applying Lemma \ref{lem:forms:wtzero:nospinhalf} we obtain that the dimension of $J^{\rm ext}_{0,m-1}$ is at most $1$ for any $m$.
\begin{prop}\label{prop:forms:wtzero:dimleq1}
We have $\dim J^{\rm ext}_{0,m-1}\leq 1$ for any positive integer $m$.
\end{prop}
\begin{proof}
If $\phi'$ and $\phi''$ are two elements of $J^{\rm ext}_{0,m-1}$ then there exists $c\in \CC$ such that $\phi=\phi'-c\phi''$ is extremal and has vanishing $a_{\frac{m-1}{4},\frac{1}{2}}$ in (\ref{eqn:forms:wtzero:ext}). Then $\phi=0$ according to Lemma \ref{lem:forms:wtzero:nospinhalf} and thus $\phi'$ belongs to the linear span of $\phi''$. 
\end{proof}

We now present a result which ultimately shows that there are only finitely many $m$ for which $J^{\rm ext}_{0,m-1}\neq\{0\}$, and also illustrates the depth of the characterisation problem. In preparation for it let us write $S_2(m)$ for the space of cusp forms of weight $2$ for $\Gamma_0(m)$ and recall that $f\in S_2(m)$ is called a {\em newform} if it is a Hecke eigenform, satisfying $T(n)f=\lambda_f(n)f$ whenever $(n,m)=1$ for some $\lambda_f(n)\in\CC$, that is uniquely specified (up to scale) in $S_2(m)$ by its Hecke eigenvalues $\{\lambda_f(n)\mid (n,m)=1\}$ (cf. \cite[\S IX.7]{MR1193029}). Write $S_2^{\rm new}(m)$ for the subspace of $S_2(m)$ spanned by newforms. Given $f\in S_2(m)$ define $L(f,s)=\sum_{n>0}a_f(n)n^{-s}$ for $\Re(s)>1$ when $f(\t)=\sum_{n>0}a_f(n)q^n$. Then $L(f,s)$ is called the {\em Dirichlet $L$-function} attached to $f$ and admits an analytic continuation to $s\in\CC$ (cf. \cite[\S IX.4]{MR1193029} or \cite[\S3.6]{Shi_IntThyAutFns}.) To a newform $f$ in $S_2(m)$ (i.e. a Hecke eigenform in $S^{\rm new}_2(m)$) Eichler--Shimura theory attaches an Elliptic curve $E_f$ defined over $\QQ$ (cf. \cite[\S XI.11]{MR1193029}), and the Birch--Swinnerton-Dyer conjecture predicts that $E_f$ has rational points of infinite order whenever $L(f,s)$ vanishes at $s=1$ (cf. \cite{MR0179168,MR2238272}).
\begin{thm}\label{thm:forms:wtzero:ccv} 
If $m$ is a positive integer and $J^{\rm ext}_{0,m-1}\neq \{0\}$ then $L(f,1)=0$ for every $f\in S_2^{\rm new}(m)$. 
\end{thm}
\begin{proof}
Let $m$ be a positive integer. For $L=\sqrt{2m}\ZZ$ the spaces $H^+_{k,L}$ and $S_{2-k,L^-}$ of \cite[\S3]{BruFun_TwoGmtThtLfts} are naturally isomorphic to $\mathbb{J}_{k+1/2,m}$ and $S^{\rm skew}_{5/2-k,m}$, respectively, where $\mathbb{J}_{k,m}$ denotes the space of weak (pure) mock Jacobi forms of weight $k$ and index $m$ (cf. \cite[\S7.2]{Dabholkar:2012nd}) and $S^{\rm skew}_{k,m}$ denotes the space of cuspidal skew-holomorphic Jacobi forms of weight $k$ and index $m$ (cf. \cite{MR1074485}). Translating the construction (3.9) of \cite{BruFun_TwoGmtThtLfts} into this language, and taking $k=1/2$, we obtain a pairing 
\begin{gather}\label{eqn:forms:wtzero:ext:SJpairing}
\{\cdot\,,\cdot\}:S^{\rm skew}_{2,m}\times\mathbb{J}_{1,m}\to \CC.
\end{gather}

If $J^{\rm ext}_{0,m-1}\neq \{0\}$ then let $\phi\in J^{\rm ext}_{0,m-1}$ be a non-zero extremal Jacobi form and set $\psi=\Psi_{1,1}\phi$ as in \S\ref{sec:forms:mero}. Then the finite part $\psi^F$ of $\psi$ belongs to $\mathbb{J}_{1,m}$ and the shadow of $\psi^F$ (cf. \cite[\S7.2]{Dabholkar:2012nd}) is the cuspidal skew-holomorphic Jacobi form $\chi \sigma^{(m)}\in S^{\rm skew}_{2,m}$ where $\chi=\phi(\tau,0)$ and
\begin{gather}
	\sigma^{(m)}(\t,z)=\sum_{\text{$r$ mod $2m$}}\overline{S^{(m)}_{r}(\t)}{\theta}^{(m)}_r(\t,z).
\end{gather} 
Consider the linear functional $\lambda_{\phi}$ on $S^{\rm skew}_{2,m}$ defined by setting $\lambda_{\phi}(\varphi)=\{\varphi,\psi^F\}$. According to the definition of (\ref{eqn:forms:wtzero:ext:SJpairing}) (i.e. (3.9) of \cite{BruFun_TwoGmtThtLfts}) we have 
\begin{gather}\label{eqn:forms:wtzero:ext:lambdasigma}
	\lambda_{\phi}(\varphi)=C_m\langle\varphi,\chi\sigma^{(m)}\rangle
\end{gather}
for some (non-zero) constant $C_m$ (depending only on $m$) where $\langle\cdot\,,\cdot\rangle$ denotes the Petersson inner product on $S^{\rm skew}_{2,m}$ (cf. \cite{MR1074485}). On the other hand, inspection reveals that $h_{\pm 1}(\t)=\mp a_{\frac{m-1}{4},\frac{1}{2}}q^{-1/4m}(1+O(q))$ when $\psi^F=\sum_{\text{$r$ mod $2m$}}h_r{\th}^{(m)}_r$ so 
\begin{gather}\label{eqn:forms:wtzero:ext:lambdaC11}
	\lambda_{\phi}(\varphi)=-2a_{\frac{m-1}{4},\frac{1}{2}}C_{\varphi}(1,1)
\end{gather}
according to Proposition 3.5 of \cite{BruFun_TwoGmtThtLfts} where $\varphi(\t,z)=\sum C_{\varphi}(\Delta,r)\bar{q}^{\Delta/4m}q^{r^2/4m}y^r$ is the Fourier expansion of $\varphi$ (cf. \cite{MR1074485}). Applying Lemma \ref{lem:forms:wtzero:nospinhalf} we deduce from (\ref{eqn:forms:wtzero:ext:lambdaC11}) that $\lambda_{\phi}(\varphi)=0$ for $\varphi\in S^{\rm skew}_{2,m}$ if and only if $C_{\varphi}(1,1)=0$.

Now let $f$ be a newform in $S_2(m)$. If $\Lambda(f,s)=(2\pi)^{-s}m^{s/2}\G(s)L(f,s)$ then either $\Lambda(f,s)=\Lambda(f,2-s)$ or $\Lambda(f,s)=-\Lambda(f,2-s)$. In the latter case $L(f,1)$ necessarily vanishes. In the former case Theorem 1 of \cite{MR1116103} attaches a skew-holomorphic Jacobi form $\varphi_f\in P^{\rm skew}_{2,m}$ to $f$ having the same eigenvalues as $f$ under the Hecke operators $T(n)$ for $(n,m)=1$. (See \cite{MR1116103} for the action of Hecke operators on skew-holmorphic Jacobi forms.) Here $P^{\rm skew}_{2,m}$ denotes the subspace of $S^{\rm skew}_{2,m}$ spanned by cuspidal Hecke eigenforms which are not of the {trivial type}: say $\varphi\in S^{\rm skew}_{2,m}$ is of the {\em trivial type} if the Fourier coefficients $C_{\varphi}(\Delta,r)$ are non-vanishing only when $\Delta$ is a perfect square. Cusp forms of the trivial type are orthogonal to $P^{\rm skew}_{2,m}$ with respect to the Petersson inner product (cf. the proof of Theorem 3 in \cite{MR1116103}) and $\sigma^{(m)}$ is a cusp form of the trivial type, so we have $\lambda_{\phi}(\varphi_f)=0$ by virtue of (\ref{eqn:forms:wtzero:ext:lambdasigma}) and hence $C_{\varphi_f}(1,1)=0$ according to the previous paragraph. The proposition now follows by taking $(\Delta,r)=(1,1)$ in the formula at the bottom of p. 503 of \cite{MR1074485}.
\end{proof}

Applying Theorem \ref{thm:forms:wtzero:ccv} we see, for example, that there are no non-zero extremal Jacobi forms of index $10$, for it is known that $L(f,1)\neq 0$ for $f=\eta(\tau)^2\eta(11\tau)^2\in S_2(11)$. (Cf. \cite[p. 504]{MR1074485} where the corresponding Jacobi form $\varphi_f$ is computed explicitly.)

Duke showed \cite{MR1309975} that there is a constant $C$ such that at least $Cm/\log^2 m$ newforms $f\in S_2(m)$ satisfy $L(f,1)\neq 0$ when $m$ is a sufficiently large prime, and Ellenberg proved an effective version of this result \cite{MR2176151} that is valid also for composite $m$. We will apply the formulas of Ellenberg momentarily in order to deduce an upper bound on the $m$ for which $J^{\rm ext}_{0,m-1}\neq \{0\}$, but since these formulas give better results for values of $m$ with fewer divisors we first show that $J^{\rm ext}_{0,m-1}$ must vanish whenever $m$ is divisible by more than one prime, and whenever $m=p^{\nu}$ is a prime power with $\nu>2$.
\begin{prop}\label{prop:forms:wtzero:AL}
If $m$ is divisible by more than one prime then $J^{\rm ext}_{0,m-1}=\{0\}$.
\end{prop}
\begin{proof}
Suppose that $m=p^{\nu}d$ with $p$ prime, $\nu>0$, $d>1$ and $(p,d)=1$. With notation as in the proof of Theorem \ref{thm:forms:wtzero:ccv} we suppose $\phi\in J^{\rm ext}_{0,m-1}$ is non-zero and consider the linear functional $\lambda_{\phi}: S^{\rm skew}_{2,m}\to\CC$. The Atkin--Lehner involution $W_{p^{\infty}}$ (cf. \cite[\S4.4]{Dabholkar:2012nd}) acts on $S^{\rm skew}_{2,m}$ in such a way that 
\begin{gather}
	(W_{p^{\infty}}\varphi)(\t,z)=\sum_{\text{$r$ mod $2m$}}\overline{g_{r^*}(\t)}\th^{(m)}_r(\t,z)
\end{gather}
when $\varphi=\sum_{\text{$r$ mod $2m$}}\overline{g_r}\th^{(m)}_r$ is the theta-decomposition of $\varphi\in S^{\rm skew}_{2,m}$ and where the map $r\mapsto r^*$ is defined so that $r^*$ is the unique solution (mod $2m$) to $r^*\equiv -r\pmod{2p^{\nu}}$ and $r^*\equiv r\pmod{2d}$. 

By the definition of $\lambda_{\phi}$ we have $\lambda_{\phi}(\varphi)=\{\varphi,\psi^F\}=\overline{\chi}\sum_{r}\langle g_r,S^{(m)}_r\rangle$ with $\varphi$ as above since the shadow of $\psi^F$ is $\chi\sigma^{(m)}=\chi\sum_{r}\overline{S^{(m)}_r}\th^{(m)}_r$. Taking $\varphi=W_{p^{\infty}}\sigma^{(m)}$ we see that $\lambda_{\phi}(W_{p^{\infty}}\sigma^{(m)})=\overline{\chi}\sum_r\langle S^{(m)}_{r^*},S^{(m)}_r\rangle$ is not zero according to the computation 
\begin{gather}\label{eqn:forms:wtzero:ipsmrs}
\langle S^{(m)}_{r^*},S^{(m)}_r\rangle=\frac{1}{2^6\pi m^{3/2}}\left(\delta_{r^*,r\text{ (mod $2m$)}}-\delta_{r^*,-r\text{ (mod $2m$)}}\right)
\end{gather}
which can be obtained using the Rankin--Selberg formula. (Cf., e.g., (3.15) of \cite{Dabholkar:2012nd}. A very similar computation is carried out in Proposition 10.2 of \cite{Dabholkar:2012nd}.) On the other hand for $\varphi=W_{p^{\infty}}\sigma^{(m)}$ we have $C_{\varphi}(1,1)=0$ since $r^*\not\equiv\pm 1\pmod{2m}$ for $r\equiv \pm1 \pmod{2m}$, and $S^{(m)}_{r}=O(q^{2/4m})$ for $r\not\equiv \pm1 \pmod{2m}$. So $\lambda_{\phi}(W_{p^{\infty}}\sigma^{(m)})=0$ in light of (\ref{eqn:forms:wtzero:ext:lambdaC11}). This contradiction proves the claim.
\end{proof}

\begin{prop}\label{prop:forms:wtzero:Ud}
If $m$ is a prime power $m=p^{\nu}$ and $\nu>2$ then $J^{\rm ext}_{0,m-1}=\{0\}$.
\end{prop}
\begin{proof}
The proof is very similar to that of Proposition \ref{prop:forms:wtzero:AL}, computing $\lambda_{\phi}(\varphi)$ two ways using (\ref{eqn:forms:wtzero:ext:lambdasigma}) and (\ref{eqn:forms:wtzero:ext:lambdaC11}), but taking now $\varphi(\t,z)=\sigma^{(p)}(\t,p^{\mu}z)$ or $\varphi(\t,z)=\sigma^{(p^2)}(\t,p^{\mu}z)$ according as $\nu={2\mu+1}$ or $\nu=2\mu+2$ with $\mu>0$. The transformation $\varphi(\t,z)\mapsto \varphi(\t,tz)$ maps Jacobi forms of index $m$ to Jacobi forms of index $t^2m$, and is one of the Hecke-like operators of \cite[\S4.4]{Dabholkar:2012nd}. We leave the remaining details---the expression of $\varphi$ in terms of $S^{(m)}_r$ and $\th^{(m)}_r$, and an application of (\ref{eqn:forms:wtzero:ipsmrs})---to the reader.
\end{proof}

We require one more result in advance of the proof of Theorem \ref{thm:forms:wtzero:ext}.
\begin{lem}\label{lem:forms:wtzero:Vp}
If $m=p^2$ for some prime $p$ and there exists a newform $f\in S_2(p)$ with $L(f,1)\neq 0$ then $J^{\rm ext}_{0,m-1}=\{0\}$.
\end{lem}
\begin{proof}
Let $m$, $p$ and $f\in S_2(p)$ be as in the statement of the lemma. Then, as in the proof of Theorem \ref{thm:forms:wtzero:ccv}, there exists a uniquely determined $\varphi_f\in P^{\rm skew}_{2,p}$ with the same Hecke eigenvalues as $f$ according to the results of \cite{MR1074485,MR1116103}, and since $L(f,1)\neq 0$ we have  $C_{\varphi_f}(1,1)\neq 0$. Now set $\varphi_f'=\varphi_f|V_p$ where $V_p$ is the Hecke-like operator (cf. \cite[(4.37)]{Dabholkar:2012nd}) that maps Jacobi forms of index $m$ to forms of index $pm$, and whose action is given explicitly by
\begin{gather}\label{eqn:forms:wtzero:Vp}
	C_{\varphi|V_p}(\Delta,r)=\sum_{d|\left(\tfrac{\Delta-r^2}{4p^2},r,p\right)}dC_{\varphi}\left(\frac{\Delta}{d^2},\frac{r}{d}\right)
\end{gather}
in our special case that $\varphi=\varphi_f\in J^{\rm skew}_{2,p}$ when $\varphi=\sum C_{\varphi}(\Delta,r)\bar{q}^{\Delta/4p}q^{r^2/4p}y^r$ is the Fourier expansion of $\varphi$.

To prove that $J^{\rm ext}_{0,m-1}=\{0\}$ suppose that $\phi$ is a non-zero extremal form of index $m-1=p^2-1$ and consider the functional $\lambda_{\phi}:S^{\rm skew}_{2,p^2}\to\CC$ constructed in the proof of Theorem \ref{thm:forms:wtzero:ccv}. Since $C_{\varphi_f}(1,1)\neq 0$ we have $C_{\varphi_f'}(1,1)\neq 0$ according to (\ref{eqn:forms:wtzero:Vp}). So $\lambda_{\phi}(\varphi_f')\neq 0$ thanks to (\ref{eqn:forms:wtzero:ext:lambdaC11}) and Lemma \ref{lem:forms:wtzero:nospinhalf}. On the other hand the Hecke-like operator $V_p$ commutes with all Hecke operators $T(n)$ (with, in our case, $(n,p)=1$) and so restricts to a map $P^{\rm skew}_{2,p}\to P^{\rm skew}_{2,p^2}$. Thus $\varphi_f'$ is a Hecke-eigenform in $P^{\rm skew}_{2,p^2}$ and we have $\lambda_{\phi}(\varphi_f')=0$ by virtue of (\ref{eqn:forms:wtzero:ext:lambdasigma}) and the fact that $P^{\rm skew}_{2,p^2}$ is orthogonal to the cusp form $\sigma^{(p^2)}$ of the trivial type. This contradiction completes the proof.
\end{proof}

\begin{proof}[Proof of Theorem \ref{thm:forms:wtzero:ext}]
We have that $\dim J^{\rm ext}_{0,m-1}\leq 1$ for all $m$ according to Proposition \ref{prop:forms:wtzero:dimleq1}, and $\dim J^{\rm ext}_{0,m-1}=0$ unless $m=p$ or $m=p^2$ for some prime $p$ by virtue of Propositions \ref{prop:forms:wtzero:AL} and \ref{prop:forms:wtzero:Ud}. Applying Theorem 1 of \cite{MR2176151} with $\sigma=13/25\pi$, for example, shows that some $f\in S^{\rm new}_2(p)$ satisfies $L(f,1)\neq 0$ whenever $p$ is a prime greater than or equal to $3001$. The tables of arithmetic data on newforms posted at \cite{mfd} give the non-vanishing or otherwise for every newform of level up to $5134$, and inspecting the entries for the $430$ prime levels less than $3000$ we see that $S_2(p)$ has a newform $f$ with $L(f,1)\neq 0$ for every prime $p$ so long as $S_2(p)$ has a newform. We have $S^{\rm new}_2(p)=S_2(p)=\{0\}$ just when $p\in\{2,3,5,7,13\}$, so we conclude from Theorem \ref{thm:forms:wtzero:ccv} that $J^{\rm ext}_{0,m-1}=\{0\}$ for every prime $m=p$ such that $p-1$ does not divide $12$. Also, we conclude from Lemma \ref{lem:forms:wtzero:Vp} that $J^{\rm ext}_{0,m-1}=\{0\}$ whenever $m=p^2$ is the square of a prime and $p-1$ does not divide $12$. We inspect the tables of \cite{mfd} again to find newforms with non-vanishing critical central value at levels $7^2=49$ and $13^2=169$, and apply Theorem \ref{thm:forms:wtzero:ccv} to conclude $J^{\rm ext}_{0,48}=J^{\rm ext}_{0,168}=\{0\}$.

So we now require to determine the dimension of $J^{\rm ext}_{0,m-1}$ for $m\in\{2,3,4,5,7,9,13,25\}$, which are the primes and squares of primes $m$ for which $S_2(m)=\{0\}$. For this we utilise the basis $\{\varphi^{(m)}_n\}$ for $J_{0,*}$ determined by Gritsenko \cite{Gri_EllGenCYMnflds}, and discussed, briefly, after the statement of Theorem \ref{thm:forms:wtzero:ext} above. In more detail, the ring $J_{0,*}$ is finitely generated, by $\varphi^{(2)}_{1}$, $\varphi^{(3)}_{1}$ and $\varphi^{(4)}_{1}$, where
\begin{gather}
\begin{split}\label{phiform}
\varphi^{(2)}_{1}&=4\left(f_2^2+f_3^2+f_4^2\right),\\
\varphi^{(3)}_{1}&=2\left(f_2^2f_3^2+f_3^2f_4^2+f_4^2f_2^2\right),\\
\varphi^{(4)}_{1}&=4f_2^2f_3^2f_4^2,
\end{split}
\end{gather}
and $f_i(\tau,z)=\theta_i(\tau,z)/\theta_i(\tau,0)$ for $i\in\{2,3,4\}$ (cf. \S\ref{sec:JacTheta}). If we work over $\ZZ$ then we must include $\varphi^{(5)}_{1}$
\be
\varphi^{(5)}_{1}=\frac{1}{4}\left(\varphi^{(4)}_{1}\varphi^{(2)}_{1}-(\varphi^{(3)}_{1})^2\right)
\ee
as a generator also, so that $J^{\ZZ}_{0,*}=\ZZ[\varphi^{(2)}_{1},\varphi^{(3)}_{1},\varphi^{(4)}_{1},\varphi^{(5)}_{1}]$. Following \cite{Gri_EllGenCYMnflds} we define
\begin{gather}
\begin{split}
\varphi^{(7)}_{1} &= \varphi^{(3)}_{1} \varphi^{(5)}_{1} - (\varphi^{(4)}_{1})^2, \\
\varphi^{(9)}_{1} &= \varphi^{(3)}_{1} \varphi^{(7)}_{1} - (\varphi^{(5)}_{1})^2 , \\
\varphi^{(13)}_{1} &= \varphi^{(5)}_{1} \varphi^{(9)}_{1}- 2(\varphi^{(7)}_{1})^2,
\end{split}
\end{gather}
and define $\varphi^{(m)}_{1}$ for the remaining positive integers $m$ according to the following recursive procedure. For $(12,m-1)=1$ and $ m > 5$ we set
\be
\varphi^{(m)}_{1}=(12,m-5) \varphi^{(m-4)}_{1} \varphi^{(5)}_{1}+(12,m-3) \varphi^{(m-2)}_{1} \varphi^{(3)}_{1}-2(12,m-4) \varphi^{(m-3)}_{1} \varphi^{(4)}_{1}.
\ee
For $(12,m-1)=2$ and $m > 10$ we set
\be
\varphi^{(m)}_{1}= \frac{1}{2} \bigl((12,m-5) \varphi^{(m-4)}_{1} \varphi^{(5)}_{1}+(12,m-3) \varphi^{(m-2)}_{1} \varphi^{(3)}_{1}-2(12,m-4) \varphi^{(m-3)}_{1} \varphi^{(4)}_{1} \bigr).
\ee
For $(12,m-1)=3$ and $m > 9$ we set
\be
\varphi^{(m)}_{1}= \frac{2}{3} (12,m-4) \varphi^{(m-3)}_{1} \varphi^{(4)}_{1} + \frac{1}{3} (12,m-7) \varphi^{(m-6)}_{1} \varphi^{(7)}_{1}-(12,m-5) \varphi^{(m-4)}_{1} \varphi^{(5)}_{1}.
\ee
For $(12,m-1)=4$ and $m > 16$ we set
\be
\varphi^{(m)}_{1}= \frac{1}{4} \bigl( (12,m-13) \varphi^{(m-12)}_{1} \varphi^{(13)}_{1}+(12,m-5) \varphi^{(m-4)}_{1} \varphi^{(5)}_{1}-(12,m-9) \varphi^{(m-8)}_{1} \varphi^{(9)}_{1} \bigr).
\ee
For $(12,m-1)=6$ and $m> 18$ we set
\be
\varphi^{(m)}_{1}=  \frac{1}{3} (12,m-4) \varphi^{(m-3)}_{1} \varphi^{(4)}_{1} + \frac{1}{6} (12,m-7) \varphi^{(m-6)}_{1} \varphi^{(7)}_{1}-\frac{1}{2} (12,m-5) \varphi^{(m-4)}_{1} \varphi^{(5)}_{1}.
\ee
Finally, for $(12,m-1)=12$ and $m> 24$ we set\footnote{Our expression gives one possible correction for a small error in the last line on p. 10 of \cite{Gri_EllGenCYMnflds}.}
\be
\varphi^{(m)}_{1}= \frac{1}{6}(12,m-4) \varphi^{(m-3)}_{1} \varphi^{(4)}_{1}-\frac{1}{4}(12,m-5) \varphi^{(m-4)}_{1} \varphi^{(5)}_{1} + \frac{1}{12}(12,m-7) \varphi^{(m-6)}_{1}\varphi^{(7)}_{1}.
\ee
The $\varphi^{(m)}_{2}$ are defined by setting
\begin{gather}
	\begin{split}
\varphi^{(3)}_{2} &= (\varphi^{(2)}_{1})^2-24\, \varphi^{(3)}_{1},\\
\varphi^{(4)}_{2}  &= \varphi^{(2)}_{1} \varphi^{(3)}_{1} - 18 \,\varphi^{(4)}_{1},\\
\varphi^{(5)}_{2} &= \varphi^{(2)}_{1} \varphi^{(4)}_{1} - 16\, \varphi^{(5)}_{1},
	\end{split}
\end{gather}
and
\begin{gather}
\varphi^{(m)}_{2} = (12,m-4)\, \varphi^{(m-3)}_{1} \varphi^{(4)}_{1} - (12,m-5) \varphi^{(m-4)}_{1} \varphi^{(5)}_{1} - (12,m-1) \varphi^{(m)}_{1}
\end{gather}
for $m>5$, and the remaining $\varphi^{(m)}_n$ for $2\leq m\leq 25$ are given by
\begin{gather}\label{eqn:forms:wtzero:varphimn}
	\begin{split}
\varphi^{(m)}_{n} &= \varphi^{(m-3)}_{n-1} \varphi^{(4)}_{1},\\
\varphi^{(m)}_{m-2}&= (\varphi^{(2)}_{1})^{m-3} \varphi^{(3)}_{1}, \\
\varphi^{(m)}_{m-1} &= (\varphi^{(2)}_{1})^{m-1},  
	\end{split}
\end{gather}
where the first equation of (\ref{eqn:forms:wtzero:varphimn}) holds for $3\le n\le m-3$.

Next we calculate that $\varphi^{(m)}_1$ defines a non-zero element of $J^{\rm ext}_{0,m-1}$ when $m-1$ divides $12$, so $\dim J^{\rm ext}_{0,m-1}=1$ for $m\in\{2,3,4,5,7,13\}$, and it remains to show that $J^{\rm ext}_{0,m-1}=\{0\}$ for $m\in \{9,25\}$. To do this we first observe that a weight zero form with vanishing Fourier coefficients of $q^0 y^k$  for $|k|>1$ necessarily has the form 
\be\label{evil_candidates}
c_{0,1}\, \varphi^{(m)}_1+ \sum_{i=1}^{\lfloor \frac{m-1}{6}\rfloor} \sum_{j=1}^{m-6i -1} c_{i,j}\, \zeta^i \varphi_{j}^{(m-6i)}\;
\ee
for some $c_{i,j}\in\CC$. One has to show that the only such linear combination
satisfying the extremal condition (\ref{eqn:forms:wtzero:ext}) necessarily has $c_{i,j}=0$ for all $i,j$ (including $(i,j)=(0,1)$). 
An explicit inspection of the coefficients of terms $q^n y^r$ with $r^2-4mn \geq 0$ for $n=1,2,3,4$ is sufficient to establish that $\dim J^{\rm ext}_{0,m-1}=0$ when $m=9$. For $m=25$ we notice that there are at least as many possible polar terms in the above sum as the number of parameters $c_{i,j}$. Indeed, by explicitly solving the system of linear equations  we find that there is no accidental cancellation and the only solution to \eq{evil_candidates} that also satisfies the extremal condition (\ref{eqn:forms:wtzero:ext}) is zero. In this way we conclude that there is no solution to (\ref{eqn:forms:wtzero:ext}) in $J_{0,m-1}$ for $m\in\{9,25\}$. This completes the proof of the Theorem.
\end{proof}

The quantity $Z^{(2)}(\tau,z)= 2 \varphi^{(2)}_{1}(\tau,z)$ is equal to the elliptic genus of a(ny) $K3$ surface and the factor of two relating $Z^{(2)}$ to $\varphi^{(2)}_{1}$ is required in the $K3$/$M_{24}$ connection
in order for the mock modular form $H^{(2)}=(H^{(2)}_1)$ derived from $Z^{(2)}$ to have coefficients compatible with an interpretation as dimensions of representations of $M_{24}$ for which the corresponding McKay--Thompson series $H^{(2)}_g$ have integer Fourier coefficients. Inspired by this we define the {\em umbral Jacobi forms}
\begin{gather}\label{eqn:forms:wtzero:Zell}
	Z^{(\ll)}(\t,z)=2\varphi^{(\ll)}_{1}(\t,z)
\end{gather}
for $\ll\in\LL=\{2,3,4,5,7,13\}$. We also set $\chi^{(\ll)}$ to be the constant $Z^{(\ll)}(\t,0)$ and find that $\chi^{(\ll)}=24/(\ll-1)$ for all $\ll\in\LL$.
\begin{rmk}
The identity $(m-1)\chi=12b$ for $\phi=a+b(y+y^{-1})+O(q)$ established in Lemma \ref{lem:forms:wtzero:nospinhalf} shows that there is no 
such Jacobi form with $b=1$ and $a$ an integer unless $m-1$ divides $12$. In particular, there is no extremal Jacobi form $\phi\in J^{\rm ext}_{0,m-1}$ such that $a_{\frac{m-1}{4},\frac{1}{2}}=1$ and $a_{\frac{m-1}{4},0}$ is an integer unless $m-1$ divides $12$. Inspecting the $Z^{(\ll)}$ we conclude that the divisors of $12$ are exactly the values of $m-1$ for which such a Jacobi form exists.
\end{rmk}
Following the discussion in \S\S\ref{sec:forms:mero},\ref{sec:forms:sca}, each of the umbral Jacobi forms $Z^{(\ll)}$ leads to an ($\ll-1)$-vector-valued mock modular form $H^{(\ll)}=\big(H^{(\ll)}_r\big)$ through the decomposition of $Z^{(\ll)}$ into $N=4$ characters, or equivalently through the decomposition of $\psi^{(\ll)}=\Psi_{1,1}Z^{(\ll)}$ into its polar and finite parts and the theta-expansion of the finite part, 
\begin{gather}
\psi^{(\ll),F}(\tau,z)= \sum_{r=1}^{\ll-1} H^{(\ll)}_r(\tau) \hat \theta^{(\ll)}_r(\tau,z).
\end{gather} 
As we have explained above, the extremal condition translates into a natural condition on the vector-valued mock modular forms $H^{(\ll)}$: The requirement that $Z^{(\ll)}$ takes the form \eq{eqn:forms:wtzero:ext} is equivalent to requiring that the only polar term in $H^{(\ll)}=\big(H^{(\ll)}_r\big)$ is $-2q^{-\frac{1}{4\ll}}$ in the component  $H^{(\ll)}_1$ and all the other components $H^{(\ll)}_r$ for $r\neq 1$ vanish as $\t\to i\inf$, so that
\be\label{extremal_MMF}
H^{(\ll)}_r(\t) = -2\delta_{r,1}q^{-\frac{1}{4\ll}}+O(q^{\frac{1}{4\ll}})
\ee
for $0<r<\ll$.

Some low order terms in the Fourier expansions of the component functions $H^{(\ll)}_r$ obtained from the extremal forms $Z^{(\ll)}$ for $\ll\in\{2,3,4\}$ are given as follows.
\begin{gather}
H^{(2)}_1(\tau) = 2 q^{-1/8} \left( -1 + 45 q + 231 q^2 +770 q^3 + 2277 q^4+ 5796 q^5+ \cdots \right)\\
\begin{split}
H^{(3)}_1(\tau) &= 2q^{-1/12} \left( -1 + 16 q + 55 q^2 + 144 q^3 + 330 q^4 + 704 q^5 +  \cdots \right) \\
H^{(3)}_2(\tau) &= 2q^{2/3} \left( 10 + 44 q + 110 q^2 + 280 q^3 + 572 q^4 + 1200 q^5   + \cdots \right)
\end{split}\\
\begin{split}
H^{(4)}_1(\tau) &= 2q^{-1/16} \left( -1 +7q+21q^2+43q^3+94q^4+168q^5+ \cdots \right)\\
H^{(4)}_2(\tau) &= 2q^{3/4} \left( 8+24q+56q^2+112q^3+216q^4+392q^5 + \cdots \right)\\ 
H^{(4)}_3(\tau) &= 2q^{7/16} \left( 3+14q+28q^2+69q^3+119q^4+239q^5 + \cdots \right)
\end{split}
\end{gather}
As a prelude to the next section we note that the coefficients $16$, $55$ and $144$ appearing in $H^{(3)}_1$  are dimensions of irreducible representations of the Mathieu group $M_{12}$ and that 
$\chi^{(3)}=12$ is the dimension of the defining permutation representation of $M_{12}$ just as $\chi^{(2)}=24$ is the dimension of the defining permutation representation of $M_{24}$. We further note that the low-lying coefficients appearing in $H^{(3)}_2$ are dimensions of faithful irreducible representations of a group $2.M_{12}$. Here the notation $2.G$ denotes a group with a $\ZZ/2 \ZZ$ normal subgroup such that $2.G/(\ZZ/2 \ZZ)=G$. Note also that $2.M_{12}$ has a faithful (and irreducible) $12$-dimensional signed permutation representation (cf. \S\ref{sec:grps:sgnprm}). The pattern that the coefficients of $H^{(\ll)}_r$ for $r$ odd are dimensions of representations of a group $ \bar G^{(\ll)}$ with a permutation representation of dimension $\chi^{(\ll)}$ and the coefficients of $H^{(\ll)}_r$ for $r$ even are dimensions of faithful representations of a group $G^{(\ll)}=2. \bar G^{(\ll)}$ with a signed permutation representation of degree $\chi^{(\ll)}$ persists for all $\ll\in \LL$; detailed descriptions of the (unsigned) permutation and signed permutation representations of $\bar G^{(\ll)}$ and $G^{(\ll)}$ are given in \S\ref{sec:grps:pcnst} and \S\ref{sec:grps:spcnst}.

The leading terms in the $q$-expansions of the mock modular forms $H^{(\ll)}_1$ that correspond via the procedure described earlier to the Jacobi forms
$Z^{(\ll)}$ for $\ll\in\{5,7,13\}$ are 
\begin{gather}
\begin{split}
H^{(5)}_1(\tau) &= 2q^{-1/20} \left( -1 + 4q +9q^2+20q^3+35q^4+60q^5+ \cdots \right), \\
H^{(7)}_1 (\tau)&= 2q^{-1/28} \left( -1+2q+3q^2+5q^3+10q^4+15q^5 + 21q^6+\cdots \right),    \\
H^{(13)}_1(\tau) &= 2 q^{-1/52} \left(-1+q+q^2+q^4+q^5+2q^6+3q^7 + \cdots \right).  \\
\end{split}
\end{gather}
To avoid clutter we refrain from describing low order terms in the $q$-expansions of $H^{(\ll)}_r$ for $\ll\in\{5,7,13\}$ and  $r>1$ here, but these can be read off
from the 1A entries in Tables \ref{tab:coeffs:5_2}-\ref{tab:coeffs:5_4}, \ref{tab:coeffs:7_2}-\ref{tab:coeffs:7_6}, and \ref{tab:coeffs:13_2}-\ref{tab:coeffs:13_12}.

We conclude this section with a comparison of our condition (\ref{eqn:forms:wtzero:ext}) with other notions of extremal in the literature. A notion of {\em extremal holomorphic conformal field theory} was given in \cite{Witten2007} following earlier related work on vertex operator superalgebras in \cite{Hoehn2007}. Extremal CFT's have central charge $c=24k$ with $k$  a positive integer and
are defined in such a way that their partition function satisfies
\be \label{extrpart:C}
Z =q^{-k}\prod_{n>1}^\infty \frac{1}{1-q^n}+O(q)\prod_{n>0}^\infty \frac{1}{1-q^n}
\ee
as $\t\to i\inf$ where the term $O(q)$ is presumed to be a series in $q$ with integer coefficients, so that the second term $O(q)\prod_{n>0}(1-q^n)^{-1}$ in (\ref{extrpart:C}) represents an integer combination of irreducible characters of the Virasoro algebra. The first term in \eqref{extrpart:C}  is the {vacuum character} of the Virasoro algebra, generated by the vacuum state, and thus any Virasoro primaries above the vacuum are only allowed to contribute positive powers of $q$ to the partition function. At this time the only known example of an extremal CFT is that determined by the monster vertex operator algebra $V^{\natural}$ whose partition function has $k=1$ in (\ref{extrpart:C}) so this notion is, at the very least, good at singling out extraordinary structure. 

If $\phi$ is a weak Jacobi form of weight $0$ and index $m-1$ that is extremal in our sense then we have
\be\label{eqn:forms:wtzero:spinzerospinhalf}
\f=a_{\frac{m-1}{4},0}{\rm ch}^{(m)}_{\frac{m-1}{4},0}+ a_{\frac{m-1}{4},\frac{1}{2}}{\rm ch}^{(m)}_{\frac{m-1}{4},\frac{1}{2}}
	+ \sum_{0<r<m}O(q) \frac{\hat \th^{(m)}_r}{\Psi_{1,1}}
\ee
as $\t\to i\inf$ for some $a_{\frac{m-1}{4},0}$ and $a_{\frac{m-1}{4},\frac{1}{2}}$, and the third term in (\ref{eqn:forms:wtzero:spinzerospinhalf}) is a natural counterpart to the second term in (\ref{extrpart:C}) since the massive $N=4$ characters (in the Ramond sector with $(-1)^F$ insertion) are all of the form $q^{\alpha}\hat\th^{(m)}_r/\Psi_{1,1}$ for some $\alpha$ and some $r$. It is harder to argue that the massless $N=4$ contributions in (\ref{eqn:forms:wtzero:spinzerospinhalf}) are in direct analogy with the first term in (\ref{extrpart:C}) since the vacuum state in a superconformal field theory with $N=4$ supersymmetry will generally (i.e. when the index is greater than $1$) give rise to non-zero massless $N=4$ character contributions with spin greater than $1/2$. We remark that the stronger condition  
\be\label{eqn:forms:wtzero:spinzeroplusbigoh}
\f=a_{\frac{m-1}{4},0}{\rm ch}^{(m)}_{\frac{m-1}{4},0}+ \sum_{0<r<m}O(q) \frac{\hat \th^{(m)}_r}{\Psi_{1,1}}
\ee
has no solutions according to Lemma \ref{lem:forms:wtzero:nospinhalf}. According to Theorem \ref{thm:forms:wtzero:ext} the six Jacobi forms $Z^{(\ll)}$ of umbral moonshine are the unique solutions to (\ref{eqn:forms:wtzero:spinzerospinhalf}) having $a_{\frac{m-1}{4},\frac{1}{2}}=-2$.

A notion of {\em extremal conformal field theory with $N=(2,2)$ superconformal symmetry} was introduced in \cite{Gaberdiel:2008xb} and the Ramond sector partition function of such an object defines a weak Jacobi form of weight $0$ and some index but will generally not coincide with a Jacobi form that is extremal in our sense since, as has been mentioned, our functions are free from contributions arising from states in the Ramond sector that are related by spectral flow to the vacuum state in the Neveu--Schwarz sector, except in the case of index $1$.

Despite the absence of a notion of extremal conformal field theory underlying our extremal Jacobi forms it is interesting to reflect on the fact that the one known example of an extremal CFT is that (at $k=1$) determined by the moonshine vertex operator algebra of \cite{FLM} with partition function $Z(\tau)=J(\t)$. This function encodes dimensions of irreducible representations of the monster and, as we have seen to some extent above and will see in more detail below, the quantities which play the same role with respect to the groups $G^{(\ll)}$ of umbral moonshine are precisely the mock modular forms $H^{(\ll)}(\t)=\big(H^{(\ll)}_r(\tau)\big)$ which arise naturally from the extremal Jacobi forms $Z^{(\ll)}(\t,z)$ according to the procedure of \S\ref{sec:forms:mero}. One of the main motivations for the notion of extremal CFT introduced in \cite{Witten2007} was a possible connection to pure (chiral) gravity theory in $3$-dimensional Anti-de Sitter space via the AdS/CFT correspondence and it is interesting to compare this with the important r\^ole that Rademacher sum constructions play in monstrous moonshine \cite{DunFre_RSMG}, umbral moonshine at $\ll=2$ \cite{Cheng2011}, and in umbral moonshine more generally (cf. \S\ref{sec:intro} and \S\ref{sec:conj:moon}). We refer to \cite{MalWit_QGPtnFn3D,LiSonStr_ChGrav3D,Maloney:2009ck,Carlip:2008eq,Carlip:2008jk,Giribet:2008bw,Grumiller:2008pr,Strominger:2008dp,Carlip:2008qh,Li:2008yz,Myung:2008dm,Sachs:2008gt,Gibbons:2008vi,Skenderis:2009nt} for further discussions on the AdS/CFT correspondence and extremal CFT's.

\subsection{Siegel Modular Forms}\label{sec:forms:siegel}

Siegel modular forms are automorphic forms which generalise modular forms by replacing the modular group $\SL_2(\ZZ)$ by the {\em genus $n$ Siegel
modular group} $\Sp_{2n}(\ZZ)$ and the upper half-plane $\HH$ by the {\em genus $n$ Siegel upper half-plane} $\HH_n$. Here we restrict ourselves to the
case $n=2$. Define 
\be
J= \begin{pmatrix} 0 & -I_2 \\ I_2 & 0 \end{pmatrix}
\ee
with $I_2$ the unit $2\times 2$ matrix and let $\Sp_4(\ZZ)$ be the group of $4\times 4$ matrices $\gamma\in M_4(\ZZ)$ satisfying 
$\g J \g^t = J$. If we write $\g$ in terms of $2\times 2$ matrices with integer entries $A,B,C,D$,
\be\label{eqn:forms:siegel:gABCD}
\g= \begin{pmatrix} A & B \\ C & D \end{pmatrix}
\ee
then the condition $\g J \g^t=J$ becomes
\be
A B^t=BA^t, \qquad CD^t=D C^t, \qquad AD^t-B C^t=1.
\ee
Just as $\SL_2(\ZZ)$ has a natural action on $\HH$, the (genus $2$) Siegel modular group $\Sp_4(\ZZ)$ has a natural action on the (genus $2$) Siegel upper half-plane, $\HH_2$,
defined as the set of $2\times 2$ complex, symmetric matrices
\be
\Omega= \begin{pmatrix} \tau & z \\ z & \sigma \end{pmatrix}
\ee
obeying
\be
{\rm Im}(\tau)>0, \qquad {\rm Im}(\sigma)>0, \qquad {\rm det}({\rm Im}(\Omega))>0.
\ee
The action of $\g \in \Sp_4(\ZZ)$ is given by
\be
\g\Omega = (A \Omega+B)(C \Omega+D)^{-1}
\ee
when $\g$ is given by (\ref{eqn:forms:siegel:gABCD}). A {\em Siegel modular form} of weight $k$ for $\G\subset\Sp_4(\ZZ)$ is a holomorphic function $F: \HH_2 \rightarrow \CC$ obeying the transformation law
\be
F((A \Omega+B)(C \Omega+D)^{-1})= {\rm det}(C \Omega+D)^k F(\Omega)
\ee
for $\g \in \G$. We can write a {\em Fourier-Jacobi decomposition} of $F(\Omega)$ in terms of $p=e(\sigma)$ as
\be
F(\Omega)= \sum_{m=0}^\infty \varphi_m(\tau,z) p^m
\ee
and the transformation law for $F(\Omega)$ can be used to show that the Fourier-Jacobi coefficient $\varphi_m(\tau,z)$ is a Jacobi form of weight $k$ and index $m$. (This is one of the main motivations for the notion of Jacobi form; cf. \cite{feingold_frenkel} for an early analysis with applications to affine Lie algebras.) We can thus write
\be
\varphi_m(\tau,z)= \sum_{\substack{n,r\in\ZZ\\4mn-r^2 \ge 0}} c(m,n,r) q^n y^r
\ee
and then
\be
F(\Omega)= \sum_{m,n,r} c(m,n,r) p^m q^n y^r .
\ee

A special class of Siegel modular forms, called {\em Spezialschar} by Maass, arise by taking the Jacobi-Fourier coefficients $\varphi_m$ to be given by $\varphi_m=\varphi_1|V_m$ where $\varphi_1 \in J_{k,1}$ and $V_m$ is the {\em Hecke-like operator} defined so that if $\varphi_1=\sum_{n,r}c(n,r) q^n y^r$ then 
\be
\varphi_1|V_m= \sum_{n,r} \left( \sum_{j|(n,r,m)} j^{k-1}c\left(\frac{nm}{j^2},\frac{r}{j}\right) \right) q^n y^r.
\ee
It develops that the function
\be
F(\Omega)= \sum_{m=0}^\infty (\varphi_1|V_m)(\tau,z)p^m
\ee
is a weight $k$ Siegel modular form known as the {\em Saito--Kurokawa} or {\em additive lift} of the weight $k$ and index $1$ Jacobi form $\varphi_1$.
An important example arises by taking the additive lift of the Jacobi form
\be
\varphi_{10,1}(\tau,z)= - \eta(\tau)^{18} \theta_1(\tau,z)^2
\ee
(cf. \S\ref{sec:modforms}) which produces the {\em Igusa cusp form}
\be
\Phi_{10}(\Omega)= \sum_{m=1}^\infty (\varphi_{10,1}|V_m)(\tau,z) p^m
\ee

The Igusa cusp form also admits a product representation obtained from the {\em Borcherds} or {\em exponential lift} of the umbral Jacobi form
$Z^{(2)}=2\varphi_{2,1}$ (cf. \eqref{phiform}) which is
\be
\Phi_{10}(\Omega)= pqy \prod_{\substack{m,n,r\in\ZZ\\(m,n,r)>0}}(1-p^mq^n y^r )^{c^{(2)}(4mn-r^2)}
\ee
where the coefficients $c^{(2)}$ are defined by the Fourier expansion of $Z^{(2)}$
\be
Z^{(2)}(\tau,z)= \sum_{n,r} c^{(2)}(4n-r^2) q^n y^r
\ee
and where the condition $(m,n,r)>0$ is that either $m>0$ or $m=0$ and $n>0$ or $m=n=0$ and $r<0$.

According to Theorem 2.1 in \cite{GriNik_AutFrmLorKMAlgs_II} the umbral Jacobi form $Z^{(\ll)}$ defines a Siegel modular form $\Phi^{(\ll)}$ on the {\em paramodular group} $\G_{\ll-1}^+<\Sp_4(\QQ)$ via the Borcherds lift for each $\ll\in\LL$. (We refer the reader to \cite{GriNik_AutFrmLorKMAlgs_II} for details.) For small values of $\ll$ these Siegel forms $\Phi^{(\ll)}$ have appeared in the literature (in particular in the work \cite{GriNik_AutFrmLorKMAlgs_I,GriNik_AutFrmLorKMAlgs_II} of Gritsenko--Nikulin). Taking $\ll\in\LL$ and defining $c^{(\ll)}(n,r)$ so that 
\be
Z^{(\ll)}(\tau,z)= \sum_{n,r}c^{(\ll)}(n,r) q^n y^r
\ee
we have the exponential lift $\Phi^{(\ll)}$ of weight $k=c^{(\ll)}(0,0)/2$ for $\G_{\ll-1}^+$ given by 
\be \label{borchlift}
\Phi^{(\ll)}(\Omega) = p^{A(\ll)} q^{B(\ll)} y^{C(\ll)} \prod_{(m,n,r)>0}(1-p^{m}q^n y^r )^{c^{(\ll)}(mn,r)}
\ee
where
\be
A(\ell)=\frac{1}{24} \sum_r c^{(\ll)}(0,r), \quad B(\ll)=\frac{1}{2} \sum_{r>0} r c^{(\ll)}(0,r), \quad C(\ll)= \frac{1}{4} \sum_r r^2 c^{(\ll)}(0,r).
\ee
Note that for $\ll\in\{2,3,4,5\}$ we have $\Phi^{(\ll)}=(\Delta_k)^2$ in the notation of \cite{GriNik_AutFrmLorKMAlgs_II} where $k$ is given by $k=(7-\ell)/(\ell-1)$. These four functions $\Delta_k$ appear as denominator functions for Lorentzian Kac--Moody Lie (super)algebras in \cite[\S5.1]{GriNik_AutFrmLorKMAlgs_II} (see also \cite{GrNik_SieAutFrmCorrLorKMAlgs}) and in connection with mirror symmetry for $K3$ surfaces in \cite[\S5.2]{GriNik_AutFrmLorKMAlgs_II}. We refer the reader to \S\ref{sec:conj:geomphys} for more discussion on this.

\section{Finite Groups}\label{sec:grps}

In this section we introduce the {\em umbral groups} $G^{(\ll)}$ for $\ll\in{\LL}=\{2,3,4,5,7,13\}$. It will develop that the representation theory of $G^{(\ll)}$ is intimately related to the vector-valued mock modular form $H^{(\ll)}$ of \S\ref{sec:forms:wtzero}.

We specify the abstract isomorphism types of the groups explicitly in \S\ref{sec:grps:spec}. Each group admits a quotient $\bar{G}^{(\ll)}$ which is naturally realised as a permutation group on $24/(\ll-1)$ points. We construct these permutations explicitly in \S\ref{sec:grps:pcnst}. In order to construct the $G^{(\ll)}$ we use signed permutations---a notion we discuss in \S\ref{sec:grps:sgnprm}---and explicit generators for the $G^{(\ll)}$ are finally obtained in \S\ref{sec:grps:spcnst}. The signed permutation constructions are then used (again in \S\ref{sec:grps:spcnst}) to define characters for $G^{(\ll)}$---the {\em twisted Euler characters}---which turn out to encode the automorphy of the mock modular forms $H^{(\ll)}_g$ that are be described in detail in \S\ref{sec:mckay}. In \S\S\ref{sec:grps:dynI}-\ref{sec:grps:dynII} we describe curious connections between the $G^{(\ll)}$ and certain Dynkin diagrams.

\subsection{Specification}\label{sec:grps:spec}

For $\ll\in\LL$ let abstract groups $G^{(\ll)}$ and $\bar{G}^{(\ll)}$ be as specified in Table \ref{tab:Gms}. We will now explain the notation used therein (moving from right to left).

We write $n$ as a shorthand for $\ZZ/n\ZZ$ (in the second and third rows of Table \ref{tab:Gms}) and $\Sym_n$ denotes the symmetric group on $n$ points. We write $\Alt_n$ for the alternating group on $n$ points, which is the subgroup of $\Sym_n$ consisting of all even permutations. 

We say that $G$ is a {\em double cover} of a group $H$ and write $G\simeq 2.H$ (cf. \cite[\S5.2]{ATLAS}) in case $G$ has a subgroup $Z$ of order $2$ that is normal (and therefore central, being of order two) with the property that $G/Z$ is isomorphic to $H$. We say that $G$ is a {\em non-trivial} double cover of $H$ if it is a double cover that is not isomorphic to the direct product $2\times H$. (We don't usually write $G\simeq 2.H$ unless $G$ is a non-trivial cover of $H$.) The cyclic group of order $4$ is a non-trivial double cover of the group of order $2$. 

For $n$ a positive integer and $q$ a prime power we write $\GL_n(q)$ for the {general linear group} of invertible $n\times n$ matrices with coefficients in the finite field $\FF_q$ with $q$ elements, and we write $\SL_n(q)$ for the subgroup consisting of matrices with determinant $1$. We write $\PGL_n(q)$ for the quotient of $\GL_n(q)$ by its centre but we adopt the ATLAS convention (cf. \cite[\S2.1]{ATLAS}) of writing $L_n(q)$ as a shorthand for $\PSL_n(q)$---being the quotient of $\SL_n(q)$ by its centre---since this group is typically simple. In the case that $q=3$ and $n$ is even the centre of $\SL_n(q)$ has order $2$ and $\SL_n(q)\simeq 2.L_n(q)$ is a non-trivial double cover of $L_n(q)$.

In case $q=5$ the centre of $\GL_n(q)$ is a cyclic group of order $4$ and so $\GL_n(5)$ has a unique central subgroup of order $2$. We write $\GL_n(5)/2$ for the quotient of $\GL_n(5)$ by this subgroup. In case $n=2$ the exceptional isomorphism $\PGL_2(5)\simeq \Sym_5$ tells us then that $\GL_2(5)/2$ is a double cover of the symmetric group $\Sym_5$. For $n\geq 4$ there are two isomorphism classes of non-trivial double covers of $\Sym_n$ with the property that the central subgroup of order $2$ is contained in the commutator subgroup of the double cover---these are the so-called {\em Schur double covers} (cf. \cite[\S4.1, \S6.7]{ATLAS}). The group $\GL_2(5)/2$ is curious in that its only subgroup of index $2$ is a direct product $2\times \Alt_5$ and its commutator subgroup is a copy of $\Alt_5$, so it is a non-trivial double cover of $\Sym_5$ that is not isomorphic to either of the Schur double covers.

The symbols $\AGL_n(q)$ denote the {\em affine linear group} generated by the natural action of $\GL_n(q)$ on $\FF_q^n$ and the translations $x\mapsto x+v$ for $v\in \FF_q^n$. The group $\AGL_n(q)$ may be realised as a subgroup of $\GL_{n+1}(q)$ so in particular $\AGL_3(2)$ embeds in $\GL_4(2)$. The group $\GL_4(2)$ is unusual amongst the $\GL_n(2)$ in that it admits a non-trivial double cover $2.\GL_4(2)$, which is a manifestation of the exceptional isomorphism $\GL_4(2)\simeq \Alt_8$. There are two conjugacy classes of subgroups of $2.\Alt_8$ of order $2688$, which is twice the order of $\AGL_3(2)$. The groups in both classes are isomorphic to a particular non-trivial double cover of $\AGL_3(2)$ and this is the group we denote $2.\AGL_3(2)$ in Table \ref{tab:Gms}.

The symbols $M_{24}$ and $M_{12}$ denote the sporadic simple Mathieu groups attached to the binary and ternary Golay codes, respectively. The group $M_{24}$ is the automorphism group of the binary Golay code, which is the unique self-dual linear binary code of length $24$ with minimum weight $8$ (cf. \cite{Ple_UnqGolCdes}). The ternary Golay code is the unique self-dual linear ternary code of length $12$ with minimum weight $6$ (cf. \cite{Ple_UnqGolCdes}) and its automorphism group is a non-trivial double cover of $M_{12}$. There is a unique such group up to isomorphism (cf. \cite{ATLAS}) which we denote by $2.M_{12}$ in Table \ref{tab:Gms}.

\begin{table}
\begin{center}
\caption{The Umbral Groups}\label{tab:Gms}
\smallskip
\begin{tabular}{c||cccccccc}
$\ll$&	2&	3&	4&	5&	7&	13\\
	\hline
$G^{(\ll)}$&			$M_{24}$&	$2.M_{12}$&	$2.\AGL_3(2)$&	$\GL_2(5)/2$&	$\SL_2(3)$&	$4$&\\
$\bar{G}^{(\ll)}$&		$M_{24}$&	$M_{12}$&	$\AGL_3(2)$&	$\PGL_2(5)$&	$L_2(3)$&		$2$&\\
\end{tabular}
\end{center}
\end{table}

Observe that for $\ll>2$ the group $G^{(\ll)}$ has a unique central subgroup of order $2$. Then $\bar{G}^{(\ll)}$ may be described for $\ll>2$ by $\bar{G}^{(\ll)}=G^{(\ll)}/2$ and $G^{(\ll)}$ is a non-trivial double cover of $\bar{G}^{(\ll)}$ for each $\ll>2$. It will develop in \S\ref{sec:grps:sgnprm} that $\bar{G}^{(\ll)}$ is a permutation group of degree $n=24/(\ll-1)$ for each $\ll\in\LL$. Since $G^{(2)}\simeq M_{24}$ has trivial centre and is already a permutation group of degree $24=24/(2-1)$ we set $\bar{G}^{(2)}=G^{(2)}$. 

Generalising the notation used above for double covers we write $G\simeq N.H$ to indicate that $G$ fits into a short exact sequence
\begin{gather}\label{eqn:grps_NG}
	1 \to N \to G \to H \to 1
\end{gather}
for some groups $N$ and $H$. We write $G\simeq N\tc H$ for a group of the form $N.H$ for which the sequence (\ref{eqn:grps_NG}) splits (i.e. in case $G$ is a semi-direct product $N\rtimes  H$). Then we have $\AGL_3(2)\simeq 2^3\tc L_2(7)$ according to the exceptional isomorphism $L_3(2)\simeq L_2(7)$ where $2^3$ denotes an elementary abelian group of order $8$.

\subsection{Signed Permutations}\label{sec:grps:sgnprm}

The group of {\em signed permutations of degree $n$} or {\em hyperoctahedral group of degree $n$}, denoted here by $\HO_n$, is a semi-direct product $2^n\tc \Sym_n$ where $2^n$ denotes an elementary abelian group of order $2^n$. We may realise it explicitly as the subgroup of invertible linear transformations of an $n$-dimensional vector space generated by sign changes and permutations in a chosen basis. (This recovers $2^n\tc \Sym_n$ so long as the vector space is defined over a field of characteristic other than $2$. In case the characteristic is $2$ the sign changes are trivial and we recover $\Sym_n$.) When we write $\HO_n$ we usually have in mind the data of fixed subgroups $N\simeq 2^n$ and $H\simeq \Sym_n$ such that $N$ is normal and $\HO_n/N\simeq H$, such as are given by sign changes and coordinate permutations, respectively, when we realise $\HO_n$ explicitly as described above. In what follows we assume such data to be chosen for each $n$ and write $\HO_n\to \Sym_n$ for the composition $\HO_n\to \HO_n/N\simeq \Sym_n$.

We will show in \S\ref{sec:grps:spcnst} that each $G^{(\ll)}$ for $\ll\in{\LL}$ admits a realisation as a subgroup of $\HO_n$ for $n=24/(\ll-1)$ such that the image of $G^{(\ll)}$ under the map $\HO_n\to \Sym_n$ is $\bar{G}^{(\ll)}$.

In \S\ref{sec:grps:spcnst} we use the following modification of the usual cycle notation for permutations to denote elements of $\HO_n$. Suppose $\O$ is a set with $n$ elements and $V$ is the vector space generated over a field $\kk$ by the symbols $\{e_{x}\}_{x\in \O}$. Whereas the juxtaposition $xy$ occurring in a cycle $(...xy...)$ indicates a coordinate permutation mapping $e_x$ to $e_y$, we write $x\bar{y}$ to indicate that a sign change is applied so that $e_x$ is mapped to $-e_{y}$. Then, for example, if $ \O=\{\infty,0,1,2,3,4\}$ we write $\sigma=(\bar{\infty}2\bar{3}4)(\bar{0})$ for the element of $\GL(V)$ determined by
\begin{gather}
	\sigma:
	e_{\infty}\mapsto e_2,\quad
	e_{2}\mapsto -e_{3},\quad
	e_3\mapsto e_4,\quad
	e_4\mapsto -e_{\infty},\quad
	e_0\mapsto -e_0.
\end{gather}
(We think of the bar over the $3$ in $(...2\bar{3}4...)$ as ``occurring between'' the $2$ and $3$.) We call the symbol $1$ a {\em fixed point} of $\sigma$ and we call $0$ an {\em anti-fixed point}. 

We define the {\em signed permutation character} of $\HO_n$ by setting
\begin{gather}\label{eqn:grps:spchar}
	\begin{split}
		\chi:\HO_n&\to\ZZ\\
			g&\mapsto \chi_g=h_{g,+}-h_{g,-}
	\end{split}
\end{gather}
where $h_{g,+}$ is the number of fixed points of $g$ and $h_{g,-}$ is the number of anti-fixed points. Then $\chi$ is just the character of the ordinary representation of $\HO_n$ furnished by $V$ that obtains when the field $\kk$ is taken to be $\CC$. Writing $\bar{\chi}$ for the composition of $\HO_n\to \Sym_n$ with the usual permutation character of $\Sym_n$ we have
\begin{gather}\label{eqn:grps:pchar}
	\begin{split}
		\bar{\chi}:\HO_n&\to\ZZ\\
			g&\mapsto \bar{\chi}_g=h_{g,+}+h_{g,-}
	\end{split}
\end{gather}
and we call $\bar{\chi}$ the {\em unsigned permutation character} of $\HO_n$. Then the signed and unsigned permutation characters $\chi$ and $\bar{\chi}$ together encode the number of fixed and anti-fixed points for each $g\in \HO_n$.

Let us take $\kk=\CC$ in the realisation of $\HO_n$ as a subgroup of $\GL(V)$. Then to each element $g\in \HO_n$ we assign a {\em signed permutation Frame shape} $\Pi_g$ which encodes the eigenvalues of the corresponding (necessarily diagonalisable) linear transformation of $V$ in the following way. An expression
\begin{gather}\label{eqn:grps:sgnprm_frmshp}
	\Pi_g=\prod_{k\geq 1}k^{m_g(k)}
\end{gather}
with $m_g(k)\in \ZZ$ and $m_g(k)=0$ for all but finitely many $k$ indicates that $g$ defines a linear transformation with $m_g(k)$ eigenvalues equal to $\ex(j/k)$ for each $0\leq j<k$ in the case that all the $m_g(k)$ are non-negative. If some $m_g(k)$ are negative then we find the number of eigenvalues equal to $\xi$ say by looking at how many copies of $\xi$ are present in $\Pi_g^+=\prod_{k\geq 1,\,m_g(k)>0}k^{m_g(k)}$ and subtracting the number of copies indicated by $\Pi_g^-=\prod_{k\geq 1,\,m_g(k)<0}k^{-m_g(k)}$. Observe that the signed permutation character of $\HO_n$ is recovered from the signed permutation Frame shapes via $\chi_g=m_g(1)$. Also, the Frame shape is invariant under conjugacy.

The map $\HO_n\to \Sym_n$ (which may be realised by ``ignoring'' sign changes) furnishes a (non-faithful) permutation representation of $\HO_n$ on $n$ points; we call it the {\em unsigned permutation representation} of $\HO_n$. We write $g\mapsto \bar{\Pi}_g$ for the corresponding cycle shapes and call them the {\em unsigned permutation Frame shapes} attached to $\HO_n$. Observe that the unsigned permutation character of $\HO_n$ is recovered from the  unsigned permutation Frame shapes via $\bar{\chi}_g=\bar{m}_g(1)$ when $\bar{\Pi}_g=\prod_{k\geq 1}k^{\bar{m}_g(k)}$. 

Observe that we obtain a faithful permutation representation $\HO_n\to \Sym_{2n}$ by regarding $g\in \HO_n$ as a permutation of the $2n$ points $\{\pm e_x\}_{x\in \Omega}$. We denote the corresponding cycle shapes $\tilde{\Pi}_g$ and call them the {\em total permutation Frame shapes} attached to $\HO_n$. Then the total permutation Frame shapes can be recovered from the signed and unsigned permutation Frame
shapes for if we define a formal product on Frame shapes by setting 
\begin{gather}\label{eqn:grps:prdfrms}
	\Pi\,\Pi'=\prod_{k\geq 1}k^{m(k)+m'(k)}.
\end{gather}
in case $\Pi=\prod_{k\geq 1}k^{m(k)}$ and $\Pi'=\prod_{k\geq 1}k^{m'(k)}$ then we have $\tilde{\Pi}_g=\Pi_g\bar{\Pi}_g$ for all $g\in \HO_n$.

Suppose that $G$ is a subgroup of $\HO_n$ with the property that the intersection $G\cap N$ has order $2$. (This will be the case for $G=G^{(\ll)}$ when $\ll\in\LL$ and $\ll\neq 2$.) Let $z$ be the unique and necessarily central involution in $G\cap N$ and write $\bar{G}$ for the image of $G$ under the map $\HO_n\to \Sym_n$. Then the conjugacy class of $zg$ depends only on the conjugacy class of $g$ and we obtain an involutory map $[g]\mapsto[zg]$ on the conjugacy classes of $G$. We call $[g]$ and $[zg]$ {\em paired} conjugacy classes and we say that $[g]$ is {\em self-paired} if $[g]=[zg]$. Observe that $z$ must act as $-1$ times the identity in any faithful irreducible representation of $G$ so if $g\mapsto \chi(g)$ is the character of such a representation then $\chi(zg)=-\chi(g)$. In particular, $\chi(g)=0$ if $g$ is self-paired.

\subsection{Realisation Part I}\label{sec:grps:pcnst}

In this section we construct the groups $\bar{G}^{(\ll)}$ as subgroups of $\Sym_n$ where $n=24/(\ll-1)$. This construction will be nested in the sense that each group $\bar{G}^{(\ll)}$ will be realised as a subgroup of some $\bar{G}^{(\ll')}$ for $\ll'-1$ a divisor of $\ll-1$.

For $\ll=2$ let $\O^{(2)}$ be a set with $24$ elements and let $\mc{G}\subset \mc{P}(\O^{(2)})$ be a copy of the Golay code on $\O^{(2)}$. Then the subgroup of the symmetric group $\Sym_{\O^{(2)}}$ of permutations of the set $\O^{(2)}$ that preserves $\mc{G}$ is isomorphic to $M_{24}$. So we may set
$$
	\bar{G}^{(2)}=\left\{\s\in \Sym_{\O^{(2)}}\mid\s(C)\in \mc{G},\,\forall C\in\mc{G}\right\}.
$$

For $\ll=3$ the largest proper divisor of $\ll-1=2$ is $1=2-1$. Choose a partition of $\O^{(2)}$ into $2/1=2$ subsets of $24/2=12$ elements such that each $12$-element set belongs to $\mc{G}$ and denote it $P^{(3)}=\{\O^{(3)},\O^{(3)'}\}$. Equivalently, take $\O^{(3)}$ to be the symmetric difference of the sets of fixed-points of elements $\s,\s'\in \bar{G}^{(2)}$ of cycle shape $1^82^8$ such that $\s\s'$ has order $6$. (That is, $\O^{(3)}$ consists of the points that are fixed by $\s$ or $\s'$ but not both. In turns out that if the product $\s\s'$ has order $6$ then its cycle shape is $1^22^23^26^2$ and so the fixed-point sets of $\s$ and $\s'$ have intersection of size $2$ and $\O^{(3)}$ and $\O^{(3)'}=\O^{(2)}\setminus\O^{(3)}$ both have $12$ elements.) Then the group of permutations of $\O^{(3)}$ induced by the subgroup of $\bar{G}^{(2)}$ that fixes this partition is a copy of $\bar{G}^{(3)}\simeq M_{12}$. Explicitly, and to explain what we mean here by fix, if $\bar{G}^{(2)}_{P^{(3)}}$ denotes the subgroup of $\bar{G}^{(2)}$ that fixes the partition $P^{(3)}$ in the sense that
\begin{gather}\label{eqn:grps:cnst:barG2P3}
	\bar{G}^{(2)}_{P^{(3)}}=\left\{\s\in \bar{G}^{(2)}\mid\s(C)=C,\,\forall C\in P^{(3)}\right\},
\end{gather}
and if $\varrho:\bar{G}^{(2)}_{P^{(3)}}\to\Sym_{\O^{(3)}}$ denotes the natural map obtained by restricting elements of $\bar{G}^{(2)}_{P^{(3)}}$ to $\O^{(3)}$,
\begin{gather}\label{eqn:grps:cnst:rstrcn}
	\begin{split}
	\varrho:\bar{G}^{(2)}_{P^{(3)}}&\to\Sym_{\O^{(3)}}\\
			\s&\mapsto\s|_{\O^{(3)}},
	\end{split}
\end{gather}
then $\bar{G}^{(3)}$ is the image of $\varrho$. The kernel of this map is the subgroup 
\begin{gather}\label{eqn:grps:cnst:barG2P3O3}
	\bar{G}^{(2)}_{P^{(3)},\O^{(3)}}=\left\{\s\in \bar{G}^{(2)}_{P^{(3)}}\mid \s(x)=x,\,\forall x\in \O^{(3)}\right\}
\end{gather}
comprised of permutations in $\bar{G}^{(2)}_{P^{(3)}}$ that fix every element of $\O^{(3)}$ so we have
\begin{gather*}
	\bar{G}^{(3)}=\varrho\left(\bar{G}^{(2)}_{P^{(3)}}\right)\simeq\bar{G}^{(2)}_{P^{(3)}}/\bar{G}^{(2)}_{P^{(3)},\O^{(3)}}\simeq M_{12}.
\end{gather*}

For $\ll=4$ the largest proper divisor of $\ll-1=3$ is $1=2-1$. Choose a partition of $\O^{(2)}$ into $3/1=3$ subsets of $24/3=8$ elements that belong to $\mc{G}$ and denote it $P^{(4)}=\{\O^{(4)},\O^{(4)'},\O^{(4)''}\}$. Equivalently, choose $\O^{(4)}$ and $\O^{(4)'}$ to be the fixed-point sets of respective elements $\s$ and $\s'$ of order $2$ in $\bar{G}^{(2)}$ having cycle shapes $1^82^8$ and disjoint fixed-point sets and the property that $\s\s'$ has order $3$. Then the group of permutations of $\O^{(4)}$ induced by the subgroup of $\bar{G}^{(2)}$ that fixes this partition is a copy of the group $\bar{G}^{(4)}\simeq \AGL_3(2)$ and we have
\begin{gather*}
	\bar{G}^{(4)}=\varrho\left(\bar{G}^{(2)}_{P^{(4)}}\right)\simeq\bar{G}^{(2)}_{P^{(4)}}/\bar{G}^{(2)}_{P^{(4)},\O^{(4)}}\simeq \AGL_3(2)
\end{gather*}
where $\bar{G}^{(2)}_{P^{(4)}}$, the map $\varrho$ and the subgroup $\bar{G}^{(2)}_{P^{(4)},\O^{(4)}}$ are defined according to the natural analogues of (\ref{eqn:grps:cnst:barG2P3}), (\ref{eqn:grps:cnst:rstrcn}) and (\ref{eqn:grps:cnst:barG2P3O3}), respectively.

For $\ll=5$ the largest proper divisor of $\ll-1=4$ is $2=3-1$. Choose a partition $P^{(5)}=\{\O^{(5)},\O^{(5)'}\}$ of $\O^{(3)}$ into $4/2=2$ subsets of $12/2=6$ elements such that neither of the sets in $P^{(5)}$ is contained in a set of size $8$ in $\mc{G}$. Equivalently, take $\O^{(5)}$ to be the symmetric difference of the sets of fixed-points of elements $\s,\s'\in \bar{G}^{(3)}$ of cycle shape $1^42^4$ such that $\s\s'$ has order $6$. (This condition forces the fixed-point sets of $\s$ and $\s'$ to have a single point of intersection so that $\O^{(5)}$ and $\O^{(5)'}=\O^{(3)}\setminus\O^{(5)}$ both have $6$ elements.) Then the group of permutations of $\O^{(5)}$ induced by the subgroup 
of $\bar{G}^{(3)}$ that fixes this partition is a copy of $\bar{G}^{(5)}\simeq \PGL_2(5)$.
\begin{gather*}
	\bar{G}^{(5)}=\varrho\left(\bar{G}^{(3)}_{P^{(5)}}\right)\simeq\bar{G}^{(3)}_{P^{(5)}}/\bar{G}^{(3)}_{P^{(5)},\O^{(5)}}\simeq \PGL_2(5)
\end{gather*}

For $\ll=7$ the largest proper divisor of $\ll-1=6$ is $3=4-1$. Choose a partition $P^{(7)}=\{\O^{(7)},\O^{(7)'}\}$ of $\O^{(4)}$ into $6/3=2$ subsets of $8/2=4$ elements such that $\O^{(7)}$ (and therefore also $\O^{(7)'}=\O^{(4)}\setminus\O^{(7)}$) is the set of fixed points of an element of order $2$ in $\bar{G}^{(4)}$ having cycle shape $1^42^2$. (There are three conjugacy classes of elements of order $2$ in $\bar{G}^{(4)}$ but for just one of these classes do the elements have the cycle shape $1^42^2$.) Then the group of even permutations of $\O^{(7)}$ induced by the subgroup of $\bar{G}^{(4)}$ that fixes this partition is a copy of the group $\bar{G}^{(7)}\simeq L_2(3)$, which is of course isomorphic to the alternating group on $4$ points.
\begin{gather*}
	\bar{G}^{(7)}=\varrho\left(\bar{G}^{(4)}_{P^{(7)}}\right)\cap \Alt_{\O^{(7)}}\simeq L_2(3)
\end{gather*}
(The quotient $\bar{G}^{(4)}_{P^{(7)}}/\bar{G}^{(4)}_{P^{(7)},\O^{(7)}}$ is isomorphic to the full symmetric group on $4$ points so the restriction to even permutations is not redundant.)

The next largest divisor of $7-1=6$ is $2=3-1$. An alternative construction of $\bar{G}^{(7)}$ is obtained by choosing a partition $P^{(7)}=\{\O^{(7)},\O^{(7)'},\O^{(7)''}\}$ of $\O^{(3)}$ into $6/2=3$ subsets of $12/3=4$ elements such that $\O^{(7)}$ and $\O^{(7)'}$ are the fixed-point sets of respective elements $\s$ and $\s'$ of order $2$ in $\bar{G}^{(3)}$ having cycle shapes $1^42^4$ and disjoint fixed-point sets and the property that $\s\s'$ has order $3$. Then the group of permutations of $\O^{(7)}$ induced by the subgroup of $\bar{G}^{(3)}$ that fixes this partition is again a copy of the group $\bar{G}^{(7)}\simeq L_2(3)$.
\begin{gather*}
	\bar{G}^{(7)}=\varrho\left(\bar{G}^{(3)}_{P^{(7)}}\right)\simeq\bar{G}^{(3)}_{P^{(7)}}/\bar{G}^{(3)}_{P^{(7)},\O^{(7)}}\simeq L_2(3)
\end{gather*}

For $\ll=13$ the largest proper divisor of $\ll-1=12$ is $6=7-1$. Choose a partition $P^{(13)}=\{\O^{(13)},\O^{(13)'}\}$ of $\O^{(7)}$ into $12/6=2$ subsets of $4/2=2$ elements such that $\O^{(13)}$ is the orbit of an element of order $2$ in $\bar{G}^{(7)}$. (That is, choose any partition of $\O^{(7)}$ into subsets of size $2$. The elements of order $2$ in $\bar{G}^{(7)}$ all have cycle shape $2^2$.) Then the group of permutations of $\O^{(13)}$ induced by the subgroup of $\bar{G}^{(7)}$ that fixes this partition is a copy of the group $\bar{G}^{(13)}\simeq\Sym_2$.
\begin{gather*}
	\bar{G}^{(13)}=\varrho\left(\bar{G}^{(7)}_{P^{(13)}}\right)\simeq\bar{G}^{(7)}_{P^{(13)}}/\bar{G}^{(7)}_{P^{(13)},\O^{(13)}}\simeq \Sym_2
\end{gather*}

The next largest divisor of $13-1=12$ is $4=5-1$. An alternative construction of $\bar{G}^{(13)}$ is obtained by choosing a partition $P^{(13)}=\{\O^{(13)},\O^{(13)'},\O^{(13)''}\}$ of $\O^{(5)}$ into $12/4=3$ subsets of $6/3=2$ elements such that $\O^{(13)}$ and $\O^{(13)'}$ are the fixed-point sets of respective elements $\s$ and $\s'$ of order $2$ in $\bar{G}^{(3)}$ having cycle shapes $1^22^2$ and disjoint fixed-point sets and the property that $\s\s'$ has order $3$. Then the group of permutations of $\O^{(13)}$ induced by the subgroup of $\bar{G}^{(5)}$ that fixes this partition is a copy of the group $\bar{G}^{(13)}\simeq\Sym_2$.
\begin{gather*}
	\bar{G}^{(13)}=\varrho\left(\bar{G}^{(5)}_{P^{(13)}}\right)\simeq \bar{G}^{(5)}_{P^{(13)}}/\bar{G}^{(5)}_{P^{(13)},\O^{(13)}}\simeq\Sym_2
\end{gather*}

\begin{rmk}
The group $\bar{G}^{(\ll')}_{P^{(\ll)},\O^{(\ll)}}$ is trivial in every instance except for when $(\ll',\ll)$ is $(2,4)$ or $(4,7)$ so apart from these cases we have $\bar{G}^{(\ll)}\simeq \bar{G}^{(\ll')}_{P^{(\ll)}}$. The group $\bar{G}^{(2)}_{P^{(4)},\O^{(4)}}$ has order $8$ and $\bar{G}^{(4)}_{P^{(7)},\O^{(7)}}$ has order $4$.
\end{rmk}

\begin{rmk}\label{rmk:grps:cnst:steiner}
A {\em Steiner system} with parameters $(t,k,n)$ is the data of an $n$-element set $\O$ together with a collection of subsets of $\O$ of size $k$, called the {\em blocks} of the system, with the property that any $t$-element subset of $\O$ is contained in a unique block. It is well-known that $M_{24}\simeq \bar{G}^{(2)}$ may be realised as the automorphism group of a Steiner system with parameters $(5,8,24)$ and $M_{12}\simeq \bar{G}^{(3)}$ may be described as the automorphism group of a Steiner system with parameters $(5,6,12)$. Indeed, for $\bar{G}^{(2)}\simeq M_{24}$ we may take the blocks to be the $\binom{24}{5}/\binom{8}{5}=759$ sets of size $8$ in a copy $\mc{G}$ of the Golay code on the $24$ element set $\O=\O^{(2)}$ and for $\bar{G}^{(3)}\simeq M_{12}$ regarded as the subgroup of $\bar{G}^{(2)}$ preserving a set $\O^{(3)}$ of size $12$ in $\mc{G}$ we may take the blocks to be the $\binom{12}{5}/\binom{6}{5}=132$ subsets of size $6$ that arise as an intersection $C\cap \O^{(3)}$ for $C\in\mc{G}$ having size $8$. The group $\bar{G}^{(4)}\simeq \AGL_3(2)$ also admits such a description: it is the automorphism group of the unique Steiner system with parameters $(3,4,8)$. The blocks of such a system constitute the $\binom{8}{3}/\binom{4}{3}=14$ codewords of weight $4$ in the unique doubly even linear binary code of length $8$, the length $8$ {\em Hamming code}.
\end{rmk}

\begin{rmk}\label{rmk:cnst:shuffle}
Following \cite[Chp. 11]{sphere_packing} let $n$ be a divisor of $12$, take $n$ cards labelled by $0$ through $n-1$ and consider the group $M_n<\Sym_n$ generated by the {\em reverse shuffle} $r_n:t\mapsto n-1-t$ and the {\em Mongean shuffle} $s_n:t\mapsto \min\{2t,2n-1-2t\}$. Then the notation is consistent with that used above, for the group $M_{12}$ just defined is indeed isomorphic to the Mathieu group of permutations on $12$ points. So we have $M_{12}\simeq \bar{G}^{(3)}$ according to Table \ref{tab:Gms} and it turns out that in fact $M_n\simeq \bar{G}^{(\ll)}$ whenever $n=24/(\ll-1)$, and so this shuffle construction recovers all of the $\bar{G}^{(\ll)}$ except for $G^{(2)}\simeq M_{24}$  and $G^{(4)}\simeq \AGL_3(2)$ (i.e. all $\bar{G}^{(\ll)}$ for $\ll$ odd).
\end{rmk}

\subsection{Realisation Part II}\label{sec:grps:spcnst}

Having constructed the $\bar{G}^{(\ll)}$ as permutation groups we now describe the umbral groups $G^{(\ll)}$ as groups of signed permutations. We retain the notation of \S\ref{sec:grps:pcnst}.

To construct $G^{(3)}$ choose an element $\t$ of order $11$ in $\bar{G}^{(3)}$ and (re)label the elements of $\O^{(3)}$ so that
$$\O^{(3)}=\{\infty,0,1,2,3,4,5,6,7,8,9,X\}$$
and the action of $\t$ on $\O^{(3)}$ is given by $x\mapsto x+1$ modulo $11$ (where $X=10$). Then for any set $C\subset\O^{(3)}$ of size $4$ there is a unique element $\s\in \bar{G}^{(3)}$ with cycle shape $1^42^4$ such that the fixed-point set of $\s$ is precisely $C$. In case $C=\{0,1,4,9\}$ for example $\s=(\infty 6)(2X)(35)(78)$ and $\bar{G}^{(3)}$ is generated by $\s$ and $\t$. Let $\hat{\O}^{(3)}=\{e_x\mid x\in\O^{(3)}\}$ be a basis for a $12$-dimensional vector space over $\CC$ say, let $\hat{\t}$ be the element of the hyperoctahedral group $\HO_{\hat{\O}^{(3)}}\simeq \HO_{12}$ given by $\hat{\t}:e_x\mapsto e_{x+1}$, so that $\hat{\t}=(0123456789X)$ in (signed) cycle notation (cf. \S\ref{sec:grps:sgnprm}), and set $\hat{\s}=(\infty 6)(\bar{2}\bar{X})(35)(\bar{7}\bar{8})$. Then for $G^{(3)}$ the group generated by $\hat{\s}$ and $\hat{\t}$ we have that $G^{(3)}\simeq 2.M_{12}$ and the image of $G^{(3)}$ under the map $\HO_{\hat{\O}^{(3)}}\to S_{\O^{(3)}}$ is $\bar{G}^{(3)}$.

To construct $G^{(4)}$ choose an element $\t$ of order $7$ in $\bar{G}^{(4)}$ and (re)label the elements of $\O^{(4)}$ so that
$$\O^{(4)}=\{\infty,0,1,2,3,4,5,6\}$$
and the action of $\t$ on $\O^{(4)}$ is given by $x\mapsto x+1$ modulo $7$. There is a unique conjugacy class of elements of $\bar{G}^{(4)}$ having cycle shape $1^42^2$ and the fixed-point sets of these elements are the $14$ subsets of $\O^{(4)}$ of size $4$ comprising the blocks of a $(3,4,8)$ Steiner system preserved by $\bar{G}^{(4)}$ (cf. Remark \ref{rmk:grps:cnst:steiner}). The particular block $C=\{0,1,3,6\}$ is the unique one with the property that $\s$ and $\t$ generate $\bar{G}^{(4)}$ in case $\s$ is any (of the $3$) involution(s) with $C$ as its fixed-point set. Take $\s=(\infty 5)(24)$ and let $\hat{\O}^{(4)}=\{e_x\mid x\in\O^{(4)}\}$ be a basis for an $8$-dimensional vector space over $\CC$ say, let $\hat{\t}$ be the element of the hyperoctahedral group $\HO_{\hat{\O}^{(4)}}\simeq \HO_{8}$ given by $\hat{\t}:e_x\mapsto e_{x+1}$, so that $\hat{\t}=(0123456)$, and set $\hat{\s}=(\infty 5)(\bar{0})(\bar{1})(24)$. Then for $G^{(4)}$ the group generated by $\hat{\s}$ and $\hat{\t}$ we have that $G^{(4)}\simeq 2.\AGL_3(2)$ and the image of $G^{(4)}$ under the map $\HO_{\hat{\O}^{(4)}}\to S_{\O^{(4)}}$ is $\bar{G}^{(4)}$.

For the remaining $\ll\geq 5$ we can construct $G^{(\ll)}$ from $G^{(3)}$ by proceeding in analogy with the constructions of $\bar{G}^{(\ll)}$ from $\bar{G}^{(3)}$ (and its subgroups) given earlier in this section. Concretely, for $\ll=5$ we may set 
\begin{gather*}
	G^{(5)}=\hat{\varrho}\left(G^{(3)}_{\hat{P}^{(5)}}\right)\simeq G^{(3)}_{\hat{P}^{(5)}}/G^{(3)}_{\hat{P}^{(5)},\hat{\O}^{(5)}}\simeq \GL_2(5)/2
\end{gather*}
where $\hat{\O}^{(5)}=\{e_x\mid x\in\O^{(5)}\}$ is a basis for $24/(5-1)=6$-dimensional vector space over $\CC$, we define $\hat{C}=\{e_x\mid x\in C\}$ for $C\subset\O^{(5)}$ and set $\hat{P}^{(5)}=\{\hat{\O}^{(5)},\hat{\O}^{(5)'}\}$ for $P^{(5)}=\{\O^{(5)},\O^{(5)'}\}$ as in the construction of $\bar{G}^{(5)}$ given above, the groups $G^{(3)}_{\hat{P}^{(5)}}$ and $G^{(3)}_{\hat{P}^{(5)},\hat{\O}^{(5)}}$ are defined by 
\begin{gather}
	{G}^{(3)}_{\hat{P}^{(5)}}=\left\{\s\in {G}^{(3)}\mid\s(\Span\hat{C})\subset\Span\hat{C},\,\forall \hat{C}\in \hat{P}^{(5)}\right\},\label{eqn:grps:cnst:G3P5}\\
	{G}^{(3)}_{\hat{P}^{(5)},\hat{\O}^{(5)}}=\left\{\s\in {G}^{(3)}_{\hat{P}^{(5)}}\mid \s(e_x)=e_x,\,\forall x\in {\O}^{(5)}\right\},\label{eqn:grps:cnst:G3P5O5}
\end{gather}
and $\hat{\varrho}$ is the map $G^{(3)}_{\hat{P}^{(5)}}\to\HO_{\hat{\O}^{(5)}}$ obtained by restriction.
\begin{gather}\label{eqn:grps:cnst:rstrcnhat}
	\begin{split}
	\hat{\varrho}:{G}^{(3)}_{\hat{P}^{(5)}}&\to\HO_{\hat{\O}^{(5)}}\\
			\s&\mapsto\s|_{\Span\hat{\O}^{(5)}}
	\end{split}
\end{gather}
In direct analogy with this we have that
\begin{gather*}
	\begin{split}
	G^{(7)}&=\hat{\varrho}\left(G^{(3)}_{\hat{P}^{(7)}}\right)\simeq G^{(3)}_{\hat{P}^{(7)}}/G^{(3)}_{\hat{P}^{(7)},\hat{\O}^{(7)}}\simeq \SL_2(3),\\
	G^{(13)}&=\hat{\varrho}\left(G^{(5)}_{\hat{P}^{(13)}}\right)\simeq G^{(5)}_{\hat{P}^{(13)}}/G^{(5)}_{\hat{P}^{(13)},\hat{\O}^{(13)}}\simeq 4,\\
	\end{split}
\end{gather*}
when $G^{(\ll')}_{\hat{P}^{(\ll)}}$ and $G^{(\ll')}_{\hat{P}^{(\ll)},\hat{\O}^{(\ll)}}$ are defined as in (\ref{eqn:grps:cnst:G3P5}) and (\ref{eqn:grps:cnst:G3P5O5}), respectively, for $P^{(\ll)}=\{\O^{(\ll)},\cdots\}$ as specified in the construction of $\bar{G}^{(\ll')}$ given above, and $\hat{\varrho}$ as in (\ref{eqn:grps:cnst:rstrcnhat}).

We conclude this section with explicit signed permutation presentations of the $G^{(\ll)}$ as subgroups of the $\HO_{\hat{\O}^{(\ll)}}$ for $\ll\geq 3$. The presentations for $\ll=3$ and $\ll=4$ were obtained above, and the remaining ones can be found in a similar manner. In each case we label the index set $\O^{(\ll)}$ so that $\O^{(\ll)}=\{\infty,0,\cdots,n-1\}$ where $n=24/(\ll-1)-1=(25-\ll)/(\ll-1)$ and seek a presentation for which the coordinate permutation $e_x\mapsto e_{x+1}$ (with indices read modulo $n$) is an element of order $n$ in $G^{(\ll)}$ (although we must take this element to be trivial in case $\ll=13$).
\begin{gather}
	G^{(3)}=\left\langle (\infty 6)(\bar{2}\bar{X})(35)(\bar{7}\bar{8}),(0123456789X)\right\rangle\\
	G^{(4)}=\left\langle (\infty 5)(\bar{0})(\bar{1})(24),(0123456)\right\rangle\\
	G^{(5)}=\left\langle (\infty \bar{0})(3\bar{1})(2\bar{4}),(01234)\right\rangle\\
	G^{(7)}=\left\langle (\infty \bar{0})(\bar{1}2),(012)\right\rangle\\
	G^{(13)}=\left\langle (\infty \bar{0})\right\rangle
\end{gather}

Equipped now with explicit realisations $G^{(\ll)}<\HO_n$ of the umbral groups $G^{(\ll)}$ as signed permutation groups we define symbols $\Pi^{(\ll)}_g$, $\chi^{(\ll)}_g$,$\bar{\Pi}^{(\ll)}_g$ and $\bar{\chi}^{(\ll)}_g$ as follows for $\ll\in\LL$ and $g\in G^{(\ll)}$. 
We write $\Pi^{(\ll)}_g$ for the signed permutation Frame shape attached to $g\in G^{(\ll)}$ (when regarded as an element of $\HO_n$) as defined in \S\ref{sec:grps:sgnprm} and we write $\bar{\Pi}^{(\ll)}_g$ for the cycle shape attached to the image of $g\in G^{(\ll)}$ under the composition $G^{(\ll)}\to \HO_n\to\Sym_n$. 
We define $g\mapsto \chi^{(\ll)}_g$ by restricting the signed permutation character (cf. (\ref{eqn:grps:spchar})) to $G^{(\ll)}$ and we define $g\mapsto \bar{\chi}^{(\ll)}_g$ by restricting the unsigned permutation character (cf. (\ref{eqn:grps:pchar})) to $G^{(\ll)}$. 
We call $\chi^{(\ll)}_g$ the {\em signed twisted Euler character} attached to $g\in G^{(\ll)}$ and we call $\bar{\chi}^{(\ll)}_g$ the {\em unsigned twisted Euler character}  attached to $g\in G^{(\ll)}$. 

As is explained in \S\ref{sec:grps:sgnprm} the data $g\mapsto \Pi^{(\ll)}_g$ is sufficient to determine the cycle shapes $\bar{\Pi}^{(\ll)}_g$ and the twisted Euler characters $\chi^{(\ll)}_g$ and $\bar{\chi}^{(\ll)}_g$. We will attach vector-valued mock modular forms $H^{(\ll)}_g$ to each $g\in G^{(\ll)}$ in \S\ref{sec:mckay} and it will develop that the shadows and multiplier systems of these functions are encoded in the Frame shapes $\Pi^{(\ll)}_g$. We list the Frame shapes $\Pi^{(\ll)}_g$ and $\bar{\Pi}^{(\ll)}_g$ and the twisted Euler characters $\chi^{(\ll)}_g$ and $\bar{\chi}^{(\ll)}_g$ for all $g\in G^{(\ll)}$ and $\ll\in \LL$ in the tables of \S\ref{sec:chars:eul}. 

We conclude this section with an extraordinary property relating the Frame shapes $\Pi^{(\ll)}_g$ and $\bar{\Pi}^{(\ll)}_g$ attached to the umbral groups $G^{(\ll)}$ at $\ll=2$ and $\ll=4$ which can be verified explicitly using the tables of \S\ref{sec:chars:eul}.
\begin{prop}\label{prop:grps:sqellfrms}
Let $g\in G^{(4)}$ and suppose that the Frame shape $\Pi^{(4)}_g=\prod_k k^{m_g(k)}$ is a cycle shape, so that $m_g(k)\geq 0$ for all $k$. Then there exists $g'\in G^{(2)}$ with $o(g')=2 o(g)$ such that 
\begin{gather}
	\bar{\Pi}^{(2)}_{g'}=\prod_kk^{\bar{m}_g(k)}(2 k)^{m_g(k)}
\end{gather}
when $\bar{\Pi}^{(4)}_g=\prod_k k^{\bar{m}_g(k)}$ except in case $g$ is of class 4B.
\end{prop}
As will be explained in \S\ref{sec:mckay} this result implies direct relationships between the functions $H^{(2)}_{g'}$ and $H^{(4)}_g$ when $g$ and $g'$ are as in the statement of the proposition.

\subsection{Dynkin Diagrams Part I}\label{sec:grps:dynI}

The {\em McKay correspondence} \cite{McKay_Corr} relates finite subgroups of $\SU(2)$ to the affine Dynkin diagrams of ADE type by associating irreducible representations of the finite groups to nodes of the corresponding diagrams and by now is well understood in terms of resolutions of {\em simple singularities} $\CC^2/G$ for $G<\SU(2)$ \cite{Slo_SmpSngSmpAlgGps,GonVer_GeomCnstMcKCorr}. A more mysterious Dynkin diagram correspondence also due to McKay is his {\em monstrous $E_8$ observation} \cite{McKay_Corr} (see also \cite[\S14]{Con_CnstM}) which associates nine of the conjugacy classes of the monster group to the nodes of the affine $E_8$ Dynkin diagram and extends to similar correspondences relating certain subgroups of the Monster to other affine Dynkin diagrams \cite{GlaNor_McKE8M}. We find analogues of both McKay's Dynkin diagram observations manifesting in the groups $G^{(\ll)}$, as we shall now explain.

For $\ll\in \LL=\{2,3,4,5,7,13\}$ the number $(25-\ll)/(\ll-1)$ is an odd integer $p$ such that $p+1$ divides $24$ and is a prime in case $\ll$ is not $13$. Recall the construction of $\bar{G}^{(\ll)}$ as permutations of $\Omega^{(\ll)}$ from \S\ref{sec:grps:pcnst}. By inspection, with the assistance of \cite{GAP4}, we obtain the following lemma.
\begin{lem}\label{lem:grps:dyn:transl2p}
Let $\ll\in\{2,3,4,5,7\}$ and set $p=(25-\ll)/(\ll-1)$. Then the group $\bar{G}^{(\ll)}$ has a unique conjugacy class of subgroups isomorphic to $L_2(p)$ that act transitively on the degree $p+1=24/(\ll-1)$ permutation representation of $\bar{G}^{(\ll)}$ on $\Omega^{(\ll)}$. 
\end{lem}
Armed with Lemma \ref{lem:grps:dyn:transl2p} we pick a transitive subgroup of $\bar{G}^{(\ll)}$ isomorphic to $L_2(p)$ for each $\ll$ in $\{2,3,4,5,7\}$---these being the cases that $(25-\ll)/(\ll-1)$ is prime---and denote it $\bar{L}^{(\ll)}$. 
For future reference we note that there are two conjugacy classes of subgroups of $\bar{G}^{(\ll)}$ isomorphic to $L_2(p)$ in case $\ll\in\{3,4\}$ but in each case only one of these classes contains subgroups acting transitively.
\begin{lem}\label{lem:grps:dyn:nontransl2p}
For $\ll\in\{3,4\}$ and $p=(25-\ll)/(\ll-1)$ there is a unique conjugacy class of subgroups of $\bar{G}^{(\ll)}$ isomorphic to $L_2(p)$ that does not act transitively in the degree $p+1=24/(\ll-1)$ permutation representation of $\bar{G}^{(\ll)}$ on the set $\Omega^{(\ll)}$.
\end{lem}

It is a result due to Galois (proven in a letter to Chevalier written on the eve of his deadly duel \cite[Chp. 10]{sphere_packing}) that the group $L_2(p)$ has no transitive permutation representation of degree less than $p+1$ if $p>11$. However for the remaining primes $p$ not exceeding $11$ there are transitive permutation representations on exactly $p$ points, and in fact we have the following statement.
\begin{lem}\label{lem:grps:dyn:ptrans}
Let $p\in \{2,3,5,7,11\}$ and set $\bar{L}=L_2(p)$. Then there is a subgroup $\bar{D}<\bar{L}$ of index $p$ with the property that $\bar{L}=\lab \sigma\rab\bar{D}$ for $\sigma$ an element of order $p$ in $\bar{L}$, so that every element $g\in \bar{L}$ admits a unique expression $g=sd$ where $s\in \lab\sigma\rab$ and $d\in \bar{D}$. 
\end{lem}
\begin{rmk}
See \cite[Chp. 10]{sphere_packing} for applications of the result of Lemma \ref{lem:grps:dyn:ptrans} to exceptional isomorphisms among finite simple groups of small order, and see \cite{Kos_GphTrnIcoGal,Kos_StrcTrncIcos} for an application of the case that $p=11$ to the truncated icosahedron (which describes the structure of buckminsterfullerenes, a.k.a. buckyballs).
\end{rmk}
According to Lemma \ref{lem:grps:dyn:ptrans} we obtain a permutation representation of degree $p$ for $\bar{L}$ by taking the natural action of $\bar{L}$ on cosets of $\bar{D}$ and this action is transitive since $\sigma$ induces a $p$-cycle. Using this result we choose a copy of $\bar{D}$ in $\bar{L}=\bar{L}^{(\ll)}$ for each $\ll\in\{3,4,5,7\}$ and denote it $\bar{D}^{(\ll)}$. (Such subgroups are uniquely determined up to isomorphism, but there are as many conjugacy classes of subgroups of $\bar{L}$ isomorphic to $\bar{D}$ as there are conjugacy classes of elements of order $p$ in $\bar{L}$, which is to say there are $2$ classes in case $p\in\{11,7\}$ and a unique class in case $p\in\{5,3\}$.) We describe the groups $\bar{D}^{(\ll)}$ in Table \ref{tab:dyntab} and observe that each one is isomorphic to a finite group $\bar{D}_0<SO_3(\RR)$. As such there is a corresponding finite group $D_0<\SU(2)$ that maps onto $\bar{D}_0$ via the $2$-fold covering $\SU(2)\to \SO_3(\RR)$, and a corresponding Dynkin diagram $\Delta^{(\ll)}$ according to the McKay correspondence. We list the Dynkin diagrams $\Delta^{(\ll)}$ also in Table \ref{tab:dyntab}. 

\begin{table}
\begin{center}
\caption{The Umbral Groups and Dynkin Data}\label{tab:dyntab}
\smallskip
\begin{tabular}{l||c|cccc|c}
$\ll$&			2&	3&	4&	5&	7&	13\\
$p$&	23&	11&	7&	5&	3&	1\\
	\hline
$G^{(\ll)}$&	$M_{24}$&	$2.M_{12}$&	$2.\AGL_3(2)$&$\GL_2(5)/2$&$\SL_2(3)$&	4\\
${L}^{(\ll)}$&	&	$\SL_2({11})$&	$\SL_2(7)$&	&	$\SL_2(3)$&	
	\\
${D}^{(\ll)}$&	&	$2.\Alt_5$&		$2.\Sym_4$&		&		$Q_8$\\
	\hline
$\bar{G}^{(\ll)}$&		$M_{24}$&	$M_{12}$&	$\AGL_3(2)$&	$\PGL_2(5)$&	$L_2(3)$&	2\\
$\bar{L}^{(\ll)}$&	$L_2({23})$&	$L_2({11})$&	$L_2(7)$&	$L_2(5)$&	$L_2(3)$		\\
$\bar{D}^{(\ll)}$&		&	$\Alt_5$&		$\Sym_4$&		$\Alt_4$&		$2^2$&	\\
	\hline
$\Delta^{(\ll)}$& 	
	&$\hat{E}_8$&	$\hat{E}_7$&	$\hat{E}_6$&	$\hat{D}_4$\\	
\end{tabular}
\end{center}
\end{table}

Traditionally finite subgroups of $\SO_3(\RR)$ are called {\em ternary}, owing to the fact that their elements are described using $3\times 3$ matrices, and finite subgroups of $\SU(2)$ are called {\em binary}. The map $\SU(2)\to \SO_3(\RR)$ determines a correspondence between binary and ternary groups whereby the ternary polyhedral groups (of orientation preserving symmetries of regular polyhedra) correspond to binary groups of the form $2.\Alt_5$, $2.\Sym_4$ and $2.\Alt_4$ (depending upon the polyhedron), a dihedral subgroup $\Dih_{n}<\SO_3(\RR)$ corresponds to a generalised quaternion group of order $4n$ in $\SU(2)$, and a binary cyclic group corresponds to a ternary cyclic group of the same order. Thus, given a finite subgroup $\bar{D}_0<\SO_3(\RR)$, we may speak of the {\em associated binary group} $D_0<\SU(2)$. 

Next we observe that each group $G^{(\ll)}$ contains a subgroup $D^{(\ll)}$ isomorphic to the binary group associated to $\bar{D}^{(\ll)}$ (when $\bar{D}^{(\ll)}$ is regarded as a subgroup of $\SO_3(\RR)$) for all of the above cases except when $\ll=p=5$. More particularly, for $\ll\in\{3,7\}$ there is a unique conjugacy class of subgroups isomorphic to the binary group attached to $\bar{D}^{(\ll)}$ while for $\ll=4$ there are two such conjugacy classes and for $\ll=5$ there are none.  For each $\ll\in\{3,4,7\}$ we pick a subgroup isomorphic to the binary group attached to $\bar{D}^{(\ll)}$ and denote it $D^{(\ll)}$ and we display the (isomorphism types of the) groups $D^{(\ll)}$ in Table \ref{tab:dyntab}. We write $Q_8$ there for the quaternion group of order $8$. 

To see how the $D^{(\ll)}$ arise in $G^{(\ll)}$ (and fail to do so in the case that $\ll=5$) recall from Lemma \ref{lem:grps:dyn:transl2p} that $\bar{G}^{(\ll)}$ contains a unique transitive subgroup isomorphic to $L_2(p)$ up to conjugacy for $\ll\in\{3,4,5,7\}$. The preimage of such a subgroup under the natural map $G^{(\ll)}\to \bar{G}^{(\ll)}$ (cf. \S\ref{sec:grps:spcnst}) is a double cover of $L_2(p)$ that is in fact isomorphic to $\SL_2(p)$ (cf. \S\ref{sec:grps:spec}) except when $\ll=5$. In case $\ll=5$ we have $G^{(5)}\simeq \GL_2(5)/2$ which has the same order as $\SL_2(5)$ but is not isomorphic to it. 
We see then that for $\ll\in\{3,4,7\}$ we may find a copy of $\SL_2(p)$ in the preimage of $\bar{L}^{(\ll)}$ under the map $G^{(\ll)}\to \bar{G}^{(\ll)}$; we do so and denote it $L^{(\ll)}$. Then we may take $D^{(\ll)}$ to be the preimage of $\bar{D}^{(\ll)}$ under the map $L^{(\ll)}\to \bar{L}^{(\ll)}$. The fact that there is no $\SL_2(5)$ in $G^{(5)}$ explains why there is no group $D^{(5)}$. 

Observe that the rank of $\Delta^{(\ll)}$ is given by $11-\ll$ for each $\ll\in \{3,4,5,7\}$. This may be taken as a hint to the following uniform construction of the $\Delta^{(\ll)}$. Starting with the (finite type) $E_8$ Dynkin diagram, being star shaped with three {\em branches}, construct a sequence of diagrams iteratively by removing the end node from a branch of maximal length at each iteration. In this way we obtain the sequence ${E}_8$, ${E}_7$, $E_6$, $D_5$, $D_4$, $A_3$, $A_2$, $A_1$, and it is striking to observe that our list $\D^{(\ll)}$, obtained by applying the McKay correspondence to distinguished subgroups of the $G^{(\ll)}$, is a subsequence of the one obtained from this by affinisation.

We summarise the main observations of this section as follows.
\begin{quote}
{\em For $\ll\in\{3,4,5,7\}$ the group $\bar{G}^{(\ll)}$ admits a distinguished isomorphism class of subgroups $\bar{D}^{(\ll)}$. This connects $\bar{G}^{(\ll)}$ to a Dynkin diagram of rank $11-\ll$, for the group $\bar{D}^{(\ll)}$ is the image in $\SO_3(\RR)$ of a finite subgroup $D^{(\ll)}$ in $\SU(2)$ which corresponds to the affinisation $\Delta^{(\ll)}$ of a Dynkin diagram of rank $11-\ll$ according to McKay's correspondence. The Dynkin diagrams arising belong to a naturally defined sequence.}
\end{quote}

\subsection{Dynkin Diagrams Part II}\label{sec:grps:dynII}

Recall from Lemma \ref{lem:grps:dyn:nontransl2p} that the cases $\ll\in\{3,4\}$ are distinguished in that for such $\ll$ the group $\bar{G}^{(\ll)}$ has a unique conjugacy class of subgroups isomorphic to $L_2(p)$ and not acting transitively in the degree $p+1$ permutation representation (cf. \S\ref{sec:grps:pcnst}). Since such an $L_2(p)$ subgroup acts non-trivially on $\Omega^{(\ll)}$ it follows from Lemma \ref{lem:grps:dyn:ptrans} that it must have one fixed point and act transitively on the remaining $p$ points of $\Omega^{(\ll)}$, and thus we have witnesses within $G^{(\ll)}$ to the exceptional degree $p$ permutation representations of $L_2(p)$ in case $\ll\in\{3,4\}$ and $p\in\{11,7\}$. For these two special cases of lambency $3$ and $4$ we find direct analogues of McKay's monstrous $E_8$ observation \cite{McKay_Corr} attaching certain conjugacy classes of $G^{(\ll)}$ to the nodes of $\Delta^{(\ll)}$. At lambency $5$ and $7$ (where $p$ is $5$ and $3$, respectively) we find similar analogues where the diagram $\Delta^{(\ll)}$ is replaced by one obtained via folding with respect to a diagram automorphism of order $3$.

In case $\ll=3$ let $T$ denote the conjugacy class of elements $g$ of order $2$ in $G^{(3)}\simeq 2.M_{12}$ such that $g$ has $4$ fixed points in both the signed and unsigned permutation representations of $G^{(3)}$. This is the conjugacy class labelled $2B$ in Table \ref{tab:chars:irr:3} and we have $\chi^{(3)}_{g}=\bar{\chi}^{(3)}_{g}=4$ in the notation of Table \ref{tab:chars:eul:3}. Then $T^2=\{gh\mid g,h\in T\}$ is a union of conjugacy classes of $G^{(3)}$ and in fact there are exactly nine of the twenty-six conjugacy classes of $G^{(3)}$ that appear. In (\ref{eqn:grps:dyn:e8g3}) we use these classes (and the notation of Table \ref{tab:chars:irr:3}) to label the affine $E_8$ Dynkin diagram.
\begin{equation}\label{eqn:grps:dyn:e8g3}
     \xy
     (0,0)*+{1A}="1L";
     (14,0)*+{2B}="2L";
     (28,0)*+{3A}="3L";
     (42,0)*+{4C}="4L";
     (56,0)*+{5A}="5L";
     (70,0)*+{6C}="6L";
     (70,12)*+{3B}="3S";
     (84,0)*+{4B}="4S";
     (98,0)*+{2C}="2S";
     {\ar@{-} "1L";"2L"};
     {\ar@{-} "2L";"3L"};
     {\ar@{-} "3L";"4L"};
     {\ar@{-} "4L";"5L"};
     {\ar@{-} "5L";"6L"};
     {\ar@{-} "6L";"3S"};
     {\ar@{-} "6L";"4S"};
     {\ar@{-} "4S";"2S"};
     \endxy
\end{equation}
Observe that the labelling of (\ref{eqn:grps:dyn:e8g3}) recovers the {highest root} labelling when the classes are replaced with their orders. In (\ref{eqn:grps:dyn:e8g3frm}) we replace the labels of (\ref{eqn:grps:dyn:e8g3}) with the cycle shapes attached to these classes via the total permutation action (cf. \S\ref{sec:grps:sgnprm}) of $G^{(3)}$ on the $24=2.24/(\ll-1)$ elements $\{\pm e_{i}\mid i\in \Omega^{(3)}\}$ appearing in the signed permutation representation of $G^{(3)}$ (cf. \S\ref{sec:grps:spcnst}). 
\begin{equation}\label{eqn:grps:dyn:e8g3frm}
     \xy
     (0,0)*+{1^{24}}="1L";
     (16,0)*+{1^82^8}="2L";
     (32,0)*+{1^63^6}="3L";
     (48,0)*+{1^42^24^4}="4L";
     (64,0)*+{1^45^4}="5L";
     (80,0)*+{1^22^23^26^2}="6L";
     (80,12)*+{3^8}="3S";
     (96,0)*+{2^44^4}="4S";
     (112,0)*+{2^{12}}="2S";
     {\ar@{-} "1L";"2L"};
     {\ar@{-} "2L";"3L"};
     {\ar@{-} "3L";"4L"};
     {\ar@{-} "4L";"5L"};
     {\ar@{-} "5L";"6L"};
     {\ar@{-} "6L";"3S"};
     {\ar@{-} "6L";"4S"};
     {\ar@{-} "4S";"2S"};
     \endxy
\end{equation}
Explicitly, the cycle shape attached to $g\in G^{(3)}$ is the total permutation Frame shape $\tilde{\Pi}^{(3)}_g=\Pi^{(3)}_{g}\bar{\Pi}^{(3)}_{g}$ realised as the product (cf. \S\ref{sec:grps:sgnprm}) of the signed and unsigned permutation Frame shapes attached to $G^{(3)}$ (cf. Table \ref{tab:chars:eul:3}).
\begin{rmk}
The conjugacy class labelled $2C$ in Table \ref{tab:chars:irr:3} consists of elements of the form $gz$ where $g$ belongs to the class $2B$ (denoted $T$ above) and $z$ is the central involution of $G^{(3)}$, so we obtain exactly the nine conjugacy classes of (\ref{eqn:grps:dyn:e8g3}) by considering products $gh$ where $g$ and $h$ are involutions in the class $2C$.
\end{rmk}

\begin{rmk}
In \cite{CumDun_E8M24} an analogue of McKay's monstrous $E_8$ observation is found for $M_{24}$. Namely, it is observed that there are exactly nine conjugacy classes of elements of $M_{24}$ having representatives of the form $gh$ where $g$ and $h$ belong to the class labelled 2A in Table \ref{tab:chars:irr:2}---this is the class with cycle shape $1^82^8$ in the defining degree $24$ permutation representation---and the nodes of the affine $E_8$ Dynkin diagram can be labelled by these nine classes in such a way that their orders recover the highest root labelling. This condition leaves some ambiguity as to the correct placement of the two classes of order $2$, and similarly for the pairs of orders $3$ and $4$, but an analogue of the procedure described in \cite{Dun_ArthGrpsAffE8Dyn} is shown to determine a particular choice that recovers the original correspondence for the monster group via the modular functions of monstrous moonshine. Observe that if we express the $G^{(2)}\simeq M_{24}$ labelling of the affine $E_8$ diagram given in \cite{CumDun_E8M24} using cycle shapes to name the conjugacy classes then we recover exactly the labelling (\ref{eqn:grps:dyn:e8g3frm}) obtained here from $G^{(3)}\simeq 2.M_{12}$.
\end{rmk}
In case $\ll=4$ let $T$ denote the conjugacy class of elements $g$ of order $2$ in $G^{(4)}$ such that $g$ has $4$ fixed points in the unsigned permutation representation of $G^{(4)}$ but has $2$ fixed points and $2$ anti-fixed points in the signed permutation representation. This is the conjugacy class labelled $2C$ in Table \ref{tab:chars:irr:4} and we have $\chi^{(4)}_{g}=0$ and $\bar{\chi}^{(4)}_{g}=4$ in the notation of Table \ref{tab:chars:eul:4}. Then $T^2$ is the union of eight of the sixteen conjugacy classes of $G^{(4)}$ and we use these classes (and the notation of Table \ref{tab:chars:irr:4}) to label the nodes of the affine $E_7$ Dynkin diagram in (\ref{eqn:grps:dyn:e7g4}).
\begin{equation}\label{eqn:grps:dyn:e7g4}
     \xy
     (0,0)*+{1A}="1L";
     (14,0)*+{2C}="2L";
     (28,0)*+{3A}="3L";
     (42,0)*+{4C}="CC";
     (56,0)*+{6A}="3R";
     (70,0)*+{4A}="2R";
     (42,12)*+{2B}="1U";
     (84,0)*+{2A}="1R";
     {\ar@{-} "1L";"2L"};
     {\ar@{-} "2L";"3L"};
     {\ar@{-} "3L";"CC"};
     {\ar@{-} "CC";"3R"};
     {\ar@{-} "3R";"2R"};
     {\ar@{-} "CC";"1U"};
     {\ar@{-} "2R";"1R"};
     \endxy
\end{equation}
In (\ref{eqn:grps:dyn:e7g4frm8}) we replace the labels of (\ref{eqn:grps:dyn:e7g4}) with the cycle shapes $\bar{\Pi}^{(4)}_{g}$ (cf. Table \ref{tab:chars:eul:4}) attached to these classes via the degree $8$ permutation action of $G^{(4)}$ on $\Omega^{(4)}$. The orders of these permutations are the orders of the images of the corresponding elements of $G^{(4)}$ under the map $G^{(4)}\to \bar{G}^{(4)}$ so the labelling (\ref{eqn:grps:dyn:e7g4frm8}) demonstrates that we recover the highest root labelling of the affine $E_7$ Dynkin diagram when we replace the conjugacy classes of (\ref{eqn:grps:dyn:e7g4}) with the orders of their images in $\bar{G}^{(4)}$.
\begin{equation}\label{eqn:grps:dyn:e7g4frm8}
     \xy
     (0,0)*+{1^{8}}="1L";
     (16,0)*+{1^42^2}="2L";
     (32,0)*+{1^23^2}="3L";
     (48,0)*+{1^22^14^1}="CC";
     (64,0)*+{1^23^2}="3R";
     (80,0)*+{2^4}="2R";
     (48,12)*+{2^4}="1U";
     (96,0)*+{1^8}="1R";
     {\ar@{-} "1L";"2L"};
     {\ar@{-} "2L";"3L"};
     {\ar@{-} "3L";"CC"};
     {\ar@{-} "CC";"3R"};
     {\ar@{-} "3R";"2R"};
     {\ar@{-} "CC";"1U"};
     {\ar@{-} "2R";"1R"};
     \endxy
\end{equation}
In (\ref{eqn:grps:dyn:e7g4frm24}) we replace the labels of (\ref{eqn:grps:dyn:e7g4}) with the cycle shapes of degree $24$ given by the products $\Pi^{(4)}_{g}\bar{\Pi}^{(4)}_{g}\bar{\Pi}^{(4)}_{g}$ (cf. Table \ref{tab:chars:eul:4}) for $g$ an element of $G^{(4)}$ arising in (\ref{eqn:grps:dyn:e7g4}). Observe that all the cycle shapes in (\ref{eqn:grps:dyn:e7g4frm24}) are {\em balanced}, meaning that they are invariant (and well-defined) under the operation $\prod_{k\geq 1}k^{m(k)}\mapsto \prod_{k\geq 1}(N/k)^{m(k)}$ for some integer $N>1$, and constitute a subset of the cycle shapes attached to $G^{(3)}$ in (\ref{eqn:grps:dyn:e8g3frm}). Observe also that the order two symmetry of the affine $E_7$ diagram that identifies the two long branches is realised by squaring: the cycle shapes obtained by squaring permutations represented by the cycle shapes on the right-hand branch of (\ref{eqn:grps:dyn:e7g4frm24}) are just those that appear on the left-hand branch.
\begin{equation}\label{eqn:grps:dyn:e7g4frm24}
     \xy
     (0,0)*+{1^{24}}="1L";
     (16,0)*+{1^82^8}="2L";
     (32,0)*+{1^63^6}="3L";
     (48,0)*+{1^42^24^4}="CC";
     (64,0)*+{1^22^23^26^2}="3R";
     (80,0)*+{2^44^4}="2R";
     (48,12)*+{2^{12}}="1U";
     (96,0)*+{1^82^8}="1R";
     {\ar@{-} "1L";"2L"};
     {\ar@{-} "2L";"3L"};
     {\ar@{-} "3L";"CC"};
     {\ar@{-} "CC";"3R"};
     {\ar@{-} "3R";"2R"};
     {\ar@{-} "CC";"1U"};
     {\ar@{-} "2R";"1R"};
     \endxy
\end{equation}
\begin{rmk}
It is interesting to note that while the conjugacy class $2C$ is stable under multiplication by the central involution there are just eight conjugacy classes of $G^{(4)}$ that are contained in $T^2$ when $T$ is the class labelled $4A$ in Table \ref{tab:chars:irr:4}. In fact the eight conjugacy classes appearing are $8A$ together with all those of (\ref{eqn:grps:dyn:e7g4}) except for $4C$. Thus we obtain analogues of the labelings (\ref{eqn:grps:dyn:e7g4}), (\ref{eqn:grps:dyn:e7g4frm8}) and (\ref{eqn:grps:dyn:e7g4frm24}) where $4C$, $1^22^14^1$ and $1^42^24^4$, respectively, are replaced by $8A$, $4^2$ and $4^28^2$, respectively. Under this correspondence the highest root labelling is again obtained by taking orders in $\bar{G}^{(4)}$, but while the eta product attached to $4^28^2$ is multiplicative it only has weight $2$ (i.e. less than $4$) and so does not appear in \cite{CumDun_E8M24}.
\end{rmk}
\begin{rmk}
The correspondence of \cite{CumDun_E8M24} may be viewed as attaching the nine multiplicative eta products of weight at least $4$ to the nodes of the affine $E_8$ Dynkin diagram when we regard a cycle shape $\prod_{k\geq 1} k^{m(k)}$ as a shorthand for the {\em eta product} function $\prod_{k\geq 1}\eta(k\tau)^{m(k)}$ (cf. \S\ref{sec:mdlrfrms:dedeta}). Observe that the cycle shapes appearing in (\ref{eqn:grps:dyn:e7g4frm24}) are just those whose corresponding eta products are multiplicative, have weight at least $4$, and have level dividing $24$. (The eta products defined by $1^45^4$ and $3^8$ have level $5$ and $9$, respectively.)
\end{rmk}

For $\ll=5$ the Dynkin diagram $\Delta^{(5)}=\hat{E}_6$ admits a $\Sym_3$ group of automorphisms. {Folding} by either of the three-fold symmetries we obtain the affine Dynkin diagram of type $G_2$. Let $T$ be either of the two conjugacy class of elements $g$ of order $4$ in $G^{(5)}$ such that $g^2$ is central. Then we have $\chi^{(5)}_{g}=\bar{\chi}^{(5)}_{g}=0$ in the notation of Table \ref{tab:chars:eul:5} and $T$ is either the class labelled $4A$ in Table \ref{tab:chars:irr:5} or the class labelled $4B$, and there are exactly three of the fourteen conjugacy classes of $G^{(5)}$ that are subsets of $T^2$. In (\ref{eqn:grps:dyn:g2g5}) we use these classes (and the notation of Table \ref{tab:chars:irr:5}) to label the affine $G_2$ Dynkin diagram.
\begin{equation}\label{eqn:grps:dyn:g2g5}
     \xy
     (0,0)*+{2A}="1L";
     (16,0)*+{2C}="2L";
     (32,0)*+{6A}="3L";
     {\ar@{-} "1L";"2L"};
     {\ar@{-} "2L";"3L"};
     {\ar@{-} @/_1pc/ "2L";"3L"};
     {\ar@{-} @/^1pc/ "2L";"3L"};
     \endxy
\end{equation}
In (\ref{eqn:grps:dyn:g2g5frm6}) we replace the labels of (\ref{eqn:grps:dyn:g2g5}) with the cycle shapes $\bar{\Pi}^{(5)}_{g}$ (cf. Table \ref{tab:chars:eul:5}) attached to these classes via the degree $6$ permutation action of $G^{(5)}$ on $\Omega^{(5)}$. The orders of these permutations are the orders of the images of the corresponding elements of $G^{(5)}$ under the map $G^{(5)}\to \bar{G}^{(5)}$ so the labelling (\ref{eqn:grps:dyn:g2g5frm6}) demonstrates that we recover the highest root labelling of the affine $G_2$ Dynkin diagram---which is the labelling induced form the highest root labelling of $\hat{E}_6$---when we replace the conjugacy classes of (\ref{eqn:grps:dyn:g2g5}) with the orders of their images in $\bar{G}^{(5)}$.
\begin{equation}\label{eqn:grps:dyn:g2g5frm6}
     \xy
     (0,0)*+{1^6}="1L";
     (16,0)*+{1^22^2}="2L";
     (32,0)*+{3^2}="3L";
     {\ar@{-} "1L";"2L"};
     {\ar@{-} "2L";"3L"};
     {\ar@{-} @/_1pc/ "2L";"3L"};
     {\ar@{-} @/^1pc/ "2L";"3L"};
     \endxy
\end{equation}

For $\ll=7$ the Dynkin diagram $\Delta^{(7)}=\hat{D}_4$ admits a $\Sym_4$ group of automorphisms. Folding by any three-fold symmetry we again obtain the affine Dynkin diagram of type $G_2$. Let $T$ be the unique conjugacy class of elements of order $4$ in $G^{(7)}$. Then we have $\chi^{(7)}_{g}=\bar{\chi}^{(7)}_{g}=0$ in the notation of Table \ref{tab:chars:eul:7} and $T$ is the class labelled $4A$ in Table \ref{tab:chars:irr:7}, and there are exactly three of the seven conjugacy classes of $G^{(7)}$ that have a representative of the form $gh$ for some $g,h\in T$ (and $T^2$ is the union of these three conjugacy classes). In (\ref{eqn:grps:dyn:g2g7}) we use these classes (and the notation of Table \ref{tab:chars:irr:7}) to label the affine $G_2$ Dynkin diagram.
\begin{equation}\label{eqn:grps:dyn:g2g7}
     \xy
     (0,0)*+{1A}="1L";
     (16,0)*+{4A}="2L";
     (32,0)*+{2A}="3L";
     {\ar@{-} "1L";"2L"};
     {\ar@{-} "2L";"3L"};
     {\ar@{-} @/_1pc/ "2L";"3L"};
     {\ar@{-} @/^1pc/ "2L";"3L"};
     \endxy
\end{equation}
In (\ref{eqn:grps:dyn:g2g7frm4}) we replace the labels of (\ref{eqn:grps:dyn:g2g7}) with the cycle shapes $\bar{\Pi}^{(7)}_{g}$ (cf. Table \ref{tab:chars:eul:7}) attached to these classes via the degree $4$ permutation action of $G^{(7)}$ on $\Omega^{(7)}$. The orders of these permutations are the orders of the images of the corresponding elements of $G^{(7)}$ under the map $G^{(7)}\to \bar{G}^{(7)}$ so the labelling (\ref{eqn:grps:dyn:g2g7frm4}) demonstrates that we recover the labelling of the affine $G_2$ Dynkin diagram that is induced by the highest root labelling of $\hat{D}_4$ when we replace the conjugacy classes of (\ref{eqn:grps:dyn:g2g7}) with the orders of their images in $\bar{G}^{(7)}$.
\begin{equation}\label{eqn:grps:dyn:g2g7frm4}
     \xy
     (0,0)*+{1^4}="1L";
     (16,0)*+{2^2}="2L";
     (32,0)*+{1^4}="3L";
     {\ar@{-} "1L";"2L"};
     {\ar@{-} "2L";"3L"};
     {\ar@{-} @/_1pc/ "2L";"3L"};
     {\ar@{-} @/^1pc/ "2L";"3L"};
     \endxy
\end{equation}

We summarise the results of this section with the following statement.
\begin{quote}
{\em For $\ll\in\{3,4,5,7\}$ there is an analogue of McKay's monstrous $E_8$ observation that relates $G^{(\ll)}$ to (a folding of) the affinisation of a Dynkin diagram of rank $11-\ll$, and the diagrams arising are precisely those that appear in \S\ref{sec:grps:dynI}.}
\end{quote}

We conclude by mentioning that McKay's monstrous observation has partially been explained, using vertex operator algebra theory, by the work of Sakuma \cite{MR2347298} and Lam--Yamada--Yamauchi \cite{MR2309178,MR2160172}. Important related work appears in \cite{MR2874931,MR2890302}.

\section{McKay--Thompson Series}\label{sec:mckay}

In \S\ref{sec:forms:wtzero} we made the observation that the first few positive degree Fourier coefficients of the vector-valued mock modular forms $H^{(\ll)}=\big(H^{(\ll)}_r\big)$ coincide (up to a factor of $2$) with dimensions of irreducible representations of the group $G^{(\ll)}$ described in the previous section. This observation suggests the possibility that 
\begin{gather}
H^{(\ll)}_r(\t)= \sum_{k\in\ZZ}c^{(\ll)}_r(k-{r^2}/{4\ll})q^{k-{r^2}/{4\ll}} 
=-2\delta_{r,1}q^{-1/4\ll}+\sum_{\substack{k\in\ZZ\\ r^2-4k\ll <0}}\dim K^{(\ll)}_{r,k-{r^2}/{4\ll}}q^{k-{r^2}/{4\ll}}  
\end{gather}
for some $\ZZ\times \QQ$-graded infinite-dimensional $G^{(\ll)}$-module $K^{(\ll)}=\bigoplus_{r,d}K^{(\ll)}_{r,d}$. To further test this possibility we would like to see if there are similar vector-valued mock modular forms $H^{(\ll)}_g=\big(H^{(\ll)}_{g,r}\big)$ whose positive degree Fourier coefficients recover characters of representations of $G^{(\ll)}$. In other words, for an element $g$ of the group  $G^{(\ll)}$ we would like to see if we can find a mock modular form $H^{(\ll)}_g$ compatible with the hypothesis that 
\be\label{moonshine_conjecture_1}
H^{(\ll)}_{g,r}(\t) = \sum_{k}c^{(\ll)}_{g,r}(k-r^2/4\ll)q^{k-r^2/4\ll}=-2\delta_{r,1}q^{-1/4\ll}+\sum_{\substack{k\in\ZZ\\ r^2-4k\ll<0}}\tr_{K^{(\ll)}_{r,k-r^2/4\ll}}(g)q^{k-r^2/4\ll}
\ee
for a hypothetical bi-graded $G^{(\ll)}$-module $K^{(\ll)}$.
We refer to such a generating function as a {\em McKay--Thompson series}. Notice that we recover the generating functions of the $\dim {K^{(\ll)}_{r,d}}$ when $g$ is the identity element. Moreover,  the McKay--Thompson series attached to $g\in G^{(\ll)}$ is invariant under conjugation, since the trace is such, so $H^{(\ll)}_g$ depends only on the conjugacy class $[g]$ of $g$. 

Having such McKay--Thompson series for each conjugacy class of $G^{(\ll)}$ not only provides strong evidence for the existence of the $G^{(\ll)}$-module $K^{(\ll)}$ but in fact it uniquely specifies it up to $G^{(\ll)}$-module isomorphism. This is because, since there are as many irreducible representations as conjugacy classes of a finite group, given the characters $\tr_{K^{(\ll)}_{r,d}}(g)$ for all $[g]\subset G^{(\ll)}$ we can simply invert the character table to have a unique decomposition of the conjectural module $K^{(\ll)}_{r,d}$ into irreducible representations. What is not clear, a priori, is that we will end up with a decomposition into non-negative integer multiplies of irreducible representations of $G^{(\ll)}$, but remarkably it appears that this property holds for all $\ll\in\LL$ and all bi-degrees $(r,d)$. For evidence in support of this see the explicit decompositions for small degrees tabulated in \S\ref{sec:decompositions}.

In this section we construct a set of vector-valued mock modular forms $H^{(\ll)}_{g}=\big(H^{(\ll)}_{g,r}\big)$ for each lambency $\ll$ and we formulate a precise conjecture implying (\ref{moonshine_conjecture_1}) in \S\ref{sec:conj:mod}.

\subsection{Forms of Higher Level}
\label{Forms of Higher Levels}
In \S \ref{sec:forms} we have discussed the relation between certain vector-valued mock modular forms, meromorphic Jacobi forms of weight 1 and Jacobi forms of weight $0$ under the group $\SL_2(\ZZ)$. 
In order to investigate the McKay--Thompson series of the groups $G^{(\ll)}$ we need to generalise the discussion in \S\ref{sec:forms:wtzero} to modular forms of higher level and consider forms transforming under $\G_0(N)$ (cf. \eq{congruence1}) with $N>1$.

As in \S \ref{sec:forms} we would like to consider $(\ll-1)$-vector-valued mock modular forms $\big(H^{(\ll)}_{g,r}\big)$ for a group $\G_0(N_g)<\SL_2(\ZZ)$ with shadow given by the unary theta series $S^{(\ll)} $  (cf. \eq{unary_S_def}). The levels $N_g$ will be specified explicitly for all $g$ in \S\ref{sec:mckay:aut}.
A first difference between the case of $N_g=1$ and $N_g>1$ is the following. Since the components of $S^{(\ll)}=\big(S^{(\ll)}_{r}\big)$ have more than one orbit under $\G_0(N_g)$ when $N_g$ is even it is natural to consider mock modular forms with shadows given by $\big( \chi^{(\ll)}_{g,r} S^{(\ll)}_{r}\big)$ where the $\chi^{(\ll)}_{g,r}$ are not necessarily equal for different values of $r$. 
It will develop that in the cases of interest to us $\chi^{(\ll)}_{g,r}$ depends only on $r$ modulo $2$. In fact, we will find that 
\begin{gather}\label{shadow_multiplicity}
\chi^{(\ll)}_{g,r} =  
		\begin{cases}
	\chi^{(\ll)}_g, &\text{ for $r\equiv 0 \pmod{2}$;}\\
	\bar\chi^{(\ll)}_{g}, &\text{ for $r\equiv 1 \pmod{2}$;}
		\end{cases}
\end{gather}
where $\chi^{(\ll)}_g$ and $\bar\chi^{(\ll)}_{g}$ are as defined in \S\ref{sec:grps:spcnst} and as tabulated in \S\ref{sec:chars:eul}.

Group theoretically  the multiplicities $ \bar \chi^{(\ll)}_g$ and $ \chi^{(\ll)}_g$ appearing in the shadow of $H^{(\ll)}_g$ are determined by the number of fixed points and anti-fixed points in the signed permutation representation of the group $G^{(\ll)}$, as explained in \S\ref{sec:grps:sgnprm}. 
For instance, we have $ \bar \chi^{(\ll)}_g =   \chi^{(\ll)}_g =\chi^{(\ll)}$ for the identity element $g=e$. 
From this interpretation we can deduce that  the vanishing of $\bar\chi^{(\ll)}_g$ implies the vanishing of $\chi^{(\ll)}_{g}$, while it is possible to have  $\chi^{(\ll)}_{g}=0$ and $\bar \chi^{(\ll)}_{g}\neq 0$.
It will also turn out that 
\be
 \bar \chi^{(\ll)}_g  =   \chi^{(\ll)}_g  \quad{\rm unless}\quad 2|N_g.
\ee
For later use we define the combinations
\be
\chi^{(\ll)}_{g,+} = \tfrac{1}{2}\big( \bar \chi^{(\ll)}_g+  \chi^{(\ll)}_g \big),\quad
\chi^{(\ll)}_{g,-} = \tfrac{1}{2}\big(\bar \chi^{(\ll)}_g- \chi^{(\ll)}_g \big),
\ee
which enumerate the number of fixed and anti-fixed points, respecitlvely, in the signed permutation representation of $G^{(\ll)}$ (cf. (\ref{eqn:grps:spchar}-\ref{eqn:grps:pchar})). 
Of course in the cases where $ \bar \chi^{(\ll)}_g=  \chi^{(\ll)}_g=0$ the function $H^{(\ll)}_{g}=\big(H^{(\ll)}_{g,r}\big)$ has vanishing shadow and is a vector-valued modular form in the usual sense.

Interestingly, just as in the $\SL_2(\ZZ)$ case that was considered in \S\ref{sec:forms:mero}, the higher level vector-valued mock modular forms with shadows as described above are again closely related to the finite parts of meromorphic Jacobi forms of weight $1$. 
More explicitly, from the transformation \eq{pole_completion} and the simple fact $\th^{(\ll)}_{r}(\t,z+1/2) = (-1)^r \th^{(\ll)}_{r}(\t,z)$ it is not difficult to check that the function
\be\label{weight_one_higher_level}
\psi^{(\ll)}_{g}(\t,z) = 
\chi^{(\ll)}_{g,+}\m^{(\ll)}_{0}(\t,z) - \chi^{(\ll)}_{g,-} \m^{(\ll)}_{0}(\t,z+1/2)
+
\sum_{r=1}^{\ll-1} H^{(\ll)}_{g,r}(\t) \hat\th^{(\ll)}_{r} (\t,z) 
\ee
transforms as a Jacobi form of weight 1 and index $\ll$ under the congruence subgroup $\G_0(N_g)$ when $H^{(\ll)}_{g}$ is a mock modular form of weight $1/2$ for $\G_0(N_g)$ with shadow $S^{(\ll)}_g=\big(\chi^{(\ll)}_{g,r}S^{(\ll)}_{r}\big)$.
Observe that, for those $g$ with both $\chi^{(\ll)}_{g,+}$ and  $\chi^{(\ll)}_{g,-}$ being non-zero, the weak Jacobi form $\psi^{(\ll)}_{g}(\t,z) $ has a pole not only at $z=0$ but also at $z=1/2$. From 
\be
\m^{(\ll)}_{0} (\t,z) = {\rm Av}^{(\ll)}\left[\frac{y+1}{y-1}\right],\quad -\m^{(\ll)}_{0} (\t,z+1/2) = {\rm Av}^{(\ll)}\left[\frac{1-y}{1+y}\right],
\ee
we see that the last two terms of the decomposition given in \eq{weight_one_higher_level} have the interpretation as the polar parts at the poles $z=0$ and $z=1/2$, respectively. 
In other words, we again have a decomposition $\psi^{(\ll)}_{g}=\psi^{(\ll),P}_g+\psi^{(\ll),F}_g$ into polar and finite parts given by 
\begin{gather}
	\begin{split}
\psi^{(\ll),P}_{g}(\t,z) &=\chi^{(\ll)}_{g,+}\m^{(\ll)}_{0}(\t,z) - \chi^{(\ll)}_{g,-} \m^{(\ll)}_{0}(\t,z+1/2), \\ 
\psi^{(\ll),F}_{g}(\t,z) &= \sum_{r=1}^{\ll-1} H^{(\ll)}_{g,r}(\t) \hat\th^{(\ll)}_{r} (\t,z), 
	\end{split}
\end{gather}
and the components $H^{(\ll)}_{g,r}$ of the mock modular form $H^{(\ll)}_{g}$ may again be interpreted as the theta-coefficients of a meromorphic Jacobi form; namely, $\psi^{(\ll)}_{g}(\t,z)$. 
 
Moreover, analogous to the  $\SL_2(\ZZ)$ case, the mock modular form $H^{(\ll)}_{g}$ also enjoys a close relationship with a weight $0$ index $\ll-1$ weak Jacobi form $Z^{(\ll)}_{g}$ which admits a decomposition into characters of the  $N=4$ superconformal algebra at level $\ll-1$. 
To see this, observe that 
\be
\psi_{g}^{(\ll)}(\t,z)   =\frac{\chi^{(\ll)}_{g,+}}{\bar\chi^{(\ll)}_{g}}\,\Psi_{1,1}(\t,z)\, Z^{(\ll)}_g(\t,z) - \frac{\chi^{(\ll)}_{g,-}}{\bar\chi^{(\ll)}_{g}}\,\Psi_{1,1}(\t,z+1/2)\, Z^{(\ll)}_g(\t,z+1/2) 
\ee
when $\bar\chi^{(\ll)}_{g} \neq 0$  and
\be
\psi_{g}^{(\ll)}(\t,z)   =\Psi_{1,1}(\t,z)\, Z^{(\ll)}_g(\t,z) 
\ee
when $\bar\chi^{(\ll)}_{g} = 0$, where $\Psi_{1,1}(\t,z)$ is as in (\ref{CoverA}) and $Z^{(\ll)}_{g}$ is the weak Jacobi form of weight $0$ and index $\ll-1$ given by 
\be
Z_g^{(\ll)}(\t,z) =\frac{1}{\Psi_{1,1}(\t,z)} \left( 
	\bar \chi^{(\ll)}_{g} \m^{(\ll)}_0(\t,z) 
	+
	\sum_{0< r< \ll }\left(1+\frac{\bar \chi^{(\ll)}_{g}-\chi^{(\ll)}_{g,r}}{\chi^{(\ll)}_{g}}\right)H^{(\ll)}_{g,r} (\t)\hat \th^{(\ll)}_{r} (\t,z)
\right)
\ee
for $ \chi^{(\ll)}_{g} \neq0$ and
\be
Z_g^{(\ll)}(\t,z) = \frac{1}{\Psi_{1,1}(\t,z)}
	\left( 
	\bar \chi^{(\ll)}_{g} \m^{(\ll)}_0(\t,z) 
	+
	\sum_{0<r<\ll}H^{(\ll)}_{g,r} (\t) \hat \th^{(\ll)}_{r}(\t,z)
	\right)
\ee
otherwise. 

Anticipating their relation to the conjugacy classes $[g]$ of $G^{(\ll)}$, another comment on the weight 0 forms $Z^{(\ll)}_g$ is in order here. As discussed in 
\S\ref{sec:grps:sgnprm}, if $\ll\in\LL$ and $\ll>2$ then $G^{(\ll)}$ has a unique central element $z$ of order $2$ and for any $g\in G^{(\ll)}$ the conjugacy casses $[g]$ and $[zg]$ are said to be paired, and a class $[g]$ is said to be self-paired if $[g]=[zg]$. We have 
\be\label{conjugate_pairs}
\bar\chi^{(\ll)}_{g} =\bar\chi^{(\ll)}_{zg},
\quad \chi^{(\ll)}_{g} =-\chi^{(\ll)}_{zg},
\ee
and as a consequence $Z_g^{(\ll)}(\t,z)=Z_{zg}^{(\ll)}(\t,z)$. Therefore, while the vector-valued mock modular forms $H^{(\ll)}_{g} $ and the weight $1$ meromorphic Jacobi forms $\psi^{(\ll)}_g(\t,z)$ are generally distinct for different conjugacy classes $[g]$ of the group $G^{(\ll)}$,  the weight $0$ index $\ll-1$ weak Jacobi forms $Z^{(\ll)}_{g}$ cannot distinguish between two  classes that are paired in the above sense and are more naturally associated to the group $\bar G^{(\ll)}$ which is the quotient of $G^{(\ll)}$ by the subgroup $\langle z\rangle$.

Equipped with the $H^{(\ll)}$ defined in terms of decompositions of Jacobi forms as discussed in \S \ref{sec:forms:wtzero} it will develop that a convenient way to specify the vector-valued mock modular forms $H^{(\ll)}_{g}=\big(H^{(\ll)}_{g,r}\big)$ corresponding to non-identity conjugacy classes of $G^{(\ll)}$ will be to specify certain sets of weight $2$ modular forms. 
To see how a weight $2$ modular form is naturally associated with a vector-valued modular form with the properties described above, observe that we can eliminate the presence of the polar part in the weight 1 Jacobi form $\psi^{(\ll)}_{g}(\t,z)$ (cf. \eq{weight_one_higher_level}) by taking a linear combination with $\psi^{(\ll)}(\t,z)$. To be more precise, note that
\be
 \hat  \psi^{(\ll)}_{g}(\t,z)=\sum_{r=1}^{\ll-1} \hat H^{(\ll)}_{g,r}(\t) \hat \th^{(\ll)}_{r}(\t,z) =\psi^{(\ll)}_{g}(\t,z) - \frac{\chi^{(\ll)}_{g,+}}{\chi^{(\ll)}}  \psi^{(\ll)}(\t,z) + \frac{\chi^{(\ll)}_{g,-}}{\chi^{(\ll)}}  \psi^{(\ll)}(\t,z+1/2) 
\ee
is a weight 1 index $\ll$ Jacobi form for $\G_0(N_g)$ with no poles in $z$, and the functions
\be\label{weigh2_1}
 \hat H^{(\ll)}_{g,r}(\t) = H^{(\ll)}_{g,r}(\t) - \frac{\chi^{(\ll)}_{g,r} }{ \chi^{(\ll)}} H^{(\ll)}_{r}(\t)
\ee
are the components of a vector-valued modular form $\hat H^{(\ll)}_{g}$ for $\G_0(N_g)$ in the usual sense (i.e. a mock modular form with vanishing shadow). From this we readily conclude that the function  
\be\label{weigh2_2}
 \wttwo^{(\ll)}_g(\t)=\sum_{r=1}^{\ll-1} \hat H^{(\ll)}_{g,r}(\t)S^{(\ll)}_{r}(\t)=-\frac{1}{4\p i} \frac{\pa}{\pa z} \hat  \psi^{(\ll)}_{g}(\t,z)\big\lvert_{z=0}
\ee
is a weight 2 modular form for the group $\G_0(N_g)$.

For general values of $\ll$, specifying the weight 2 form $\wttwo^{(\ll)}_g(\t)$ is not sufficient to specify the whole vector-valued modular form $\hat H^{(\ll)}_{g}$ since we have evidently collapsed information in taking the particular combination (\ref{weigh2_2}). However, for $\ll\in\{2,3\}$ we are in the privileged situation that specifying the weight $2$ form $\wttwo^{(\ll)}_g(\t)$ completely specifies all the components of $H^{(\ll)}_{g}$, as will be explained in more detail in the following section. For $\ll>3$ we need more weight $2$ forms, and we will also consider 
\be\label{weigh2_2_2}
 \wttwo^{(\ll),2}_g(\t)=\sum_{r=1}^{\ll-1} (-1)^{r+1} \hat H^{(\ll)}_{g,r}(\t)S^{(\ll)}_{\ll-r}(\t).
\ee

In the next section we give our concrete proposals for the  McKay--Thompson series $H^{(\ll)}_{g}$ for all but a few of the conjugacy classes $[g]$ arising. We give closed expressions for all the $H^{(\ll)}_{g}$ in case $\ll\in\{2,3,4,5\}$ and we partially determine the $H^{(\ll)}_{g}$ for $\ll=\{7,13\}$. Although we do not offer analytic expressions for all the $H^{(\ll)}_{g}$ with $\ll=7$ or $\ll=13$ we have predictions for the low degree terms in the Fourier developments of all the McKay--Thompson series at all lambencies $\ll\in\LL$ and these are detailed in the tables of \S\ref{sec:coeffs}.

\subsection{Lambency Two}\label{subsec:m2}

\begin{table}[h!] \centering    
 \begin{tabular}{cCc}
 \toprule
$[g]$  & N_g&$\wttwo^{(2)}_g(\t) $\\\midrule
$1A$ &  1&$0 $\\
$2A$ & 2&$-16\Lambda_2$\\
$2B $ & 4&$24\Lambda_2-8\Lambda_4=-2\eta(\tau)^8/\eta(2\tau)^4$\\
$3A$& 3&$-6\Lambda_3$\\
$3B$& 9&$-2\eta(\tau)^6/\eta(3\tau)^2$\\
$4A$& 8&$-4\Lambda_2+6\Lambda_4-2\Lambda_8=-2\eta(2\tau)^8/\eta(4\tau)^4$\\
$4B$& 4&$4(\Lambda_2-\Lambda_4)$\\
$4C$&16&$ -2\eta(\tau)^4\eta(2\tau)^2/\eta(4\tau)^2$\\
$5A$&5&$ -2\Lambda_5$\\
$6A$&6&$ 2(\Lambda_2+\Lambda_3-\Lambda_6)$\\
$6B$&36&$ -2\eta(\tau)^2\eta(2\tau)^2\eta(3\tau)^2/\eta(6\tau)^2$\\
$7AB$&7&$ -\Lambda_7$\\
$8A$&8&$ \Lambda_4-\Lambda_8$\\
$10A$&20&$ -2\eta(\tau)^3\eta(2\tau)\eta(5\tau)/\eta(10\tau)$\\
$11A$&11&$ \tfrac{2}{5}(-\Lambda_{11}+11f_{11})$\\
$12A$&24&$ -2\eta(\tau)^3\eta(4\tau)^2\eta(6\tau)^3/\eta(2\tau)\eta(3\tau)\eta(12\tau)^2$\\
$12B$&144&$ -2\eta(\tau)^4\eta(4\tau)\eta(6\tau)/\eta(2\tau)\eta(12\tau)$\\
$14AB$&14&$ \tfrac{1}{3}(\Lambda_2+\Lambda_7-\Lambda_{14}+14f_{14})$\\
$15AB$&15&$ \tfrac{1}{4}(\Lambda_3+\Lambda_5-\Lambda_{15}+15f_{15})$\\
$21AB$&63&$ \tfrac{1}{3}(-7\eta(\tau)^3\eta(7\tau)^3/\eta(3\tau)\eta(21\tau)+\eta(\tau)^6/\eta(3\tau)^2)$\\
$23AB$&23&$  \tfrac{1}{11}(-\Lambda_{23}+23f_{23,a}+69f_{23,b})$\\
 \bottomrule
  \end{tabular}\caption{\label{h_g} {The list of weight $2$ modular forms $\wttwo^{(2)}_g(\t) $ for $\Gamma_0(N_g)$. }}
  \end{table}

When $\ll=2$ the vector-valued mock modular form $H^{(2)}_g$ has only one component $H^{(2)}_{g,1}$ and our conjecture relating the $H^{(2)}_{g}$ and the group $G^{(2)}$ is nothing but the conjecture relating (scalar-valued) mock modular forms and the largest Mathieu group $M_{24}$ that has been investigated  recently \cite{Eguchi2010,Cheng2010_1,Cheng2011,Gaberdiel2010,Gaberdiel2010a,Eguchi2010a,Eguchi2011,Govindarajan2010a}. See also \cite{MR2985326,Volpato:2012qe} for reviews and \cite{Gaberdiel2011,Taormina:2011rr} for related discussions. Explicit expressions for the McKay--Thompson series arising from the $M_{24}$-module that is conjectured to underlie this connection have been proposed in \cite{Cheng2010_1,Gaberdiel2010,Gaberdiel2010a,Eguchi2010a}. As mentioned before, one convenient way to express them is via a set of weight $2$ modular forms $\wttwo^{(2)}_g(\t)$. In this case (\ref{weigh2_1}-\ref{weigh2_2}) simply reduces to
\be\label{h_g_explicit}
H^{(2)}_{g,1}(\t) =  \frac{\chi^{(2)}_g}{24} H^{(2)}_1(\t) + \frac{ \wttwo^{(2)}_g(\t) }{\h(\t)^3}\;,
\ee
where we have used $\chi^{(2)}=24$ and $S^{(2)}_1(\t) = \h(\t)^3$.
For later use and for the sake of completeness we collect the explicit expressions for $\wttwo^{(2)}_g(\t) $ for all conjugacy classes $[g]\subset M_{24}$ in Table \ref{h_g}. They are given in terms of eta quotients and standard generators of the weight $2$ modular forms of level $N$. Among the latter are those denoted here by $\L_N(\t)$ and $f_N(\t)$ and given explicitly in Appendix \ref{sec:modforms}. 

\subsection{Lambency Three}\label{subsec:m3}

\begin{table}\begin{center}
\begin{tabular}{ccc}\toprule
$[g]$&$N_g$& $\wttwo^{(3)}_g(\t)$\\\midrule   
1A&1&0\\  [5pt]
2A&4&0\\   [5pt]
4A&16&$-2 {\h(\t)^4\h(2\t)^2}/{\h(4\t)^2}$\\   [5pt]
2B&2&$-16\L_2$\\   [5pt]
2C&4&$16(\L_2- \L_4/3)$\\   [5pt]
3A&3&$-6\L_3$\\ [5pt]
6A&12&$ -9\L_2-2\L_3+3\L_4+3\L_6-\L_{12}$\\   [5pt]
3B&9&$ 8\L_3 - 2 \L_9+2\,{\h^6(\t)}/{\h^2(3\t)}$\\   [5pt]
6B&36&
$ -2{\eta(\tau)^5\eta(3\tau)}/{\eta(2\tau)\eta(6\tau)}$
\\   [5pt]
4B&8&$-2 {\h(2\t)^8}/{\h(4\t)^4} $\\   [5pt]
4C&4&$-8\L_4/3$\\   [5pt]
5A&5&$-2\L_5$\\   [5pt]
10A&20&$\sum_{d|20} c_{10A}(d) \L_d+\tfrac{20}{3}f_{20}$\\   [5pt]
12A&144&$-2 {\h(\t)\h(2\t)^5\h(3\t) }/{\h(4\t)^2 \h(6\t)} 
$\\   [5pt]
6C&6&$ 2(\L_2+\L_3-\L_6)$\\ [5pt]
6D&12&$ -5\L_2-2\L_3+\tfrac{5}{3}\L_4+3\L_6-\L_{12}$\\   [5pt]
8AB&32&$ -2 {\h(2\t)^4 \h(4\t)^2}/{\h(8\t)^2}$\\   [5pt]
8CD&8&$-2\L_2+\tfrac{5}{3}\L_4-\L_8$\\   [5pt]
20AB&80&$-2 {\h(2\t)^7 \h(5\t)}/{\h(\t) \h(4\t)^2 \h(10\t)}$\\   [5pt]
11AB&11&$-\frac{2}{5}\L_{11}-\frac{33}{5}f_{11}$\\   [5pt]
22AB&44& $\sum_{d|44} c_{22AB}(d) \L_d(\t) -\tfrac{11}{5}\sum_{d|4} c'_{22AB}(d)f_{11}(d\t)+\tfrac{22}{3}f_{44}(\t)$\\   [5pt]
\midrule
\multicolumn{3}{c}{$ c_{10A}(d)= -5 , \tfrac{5}{3}  ,-\tfrac{2}{3} ,1, -\tfrac{1}{3}$ for $d=2,4,5,10,20$}\\
\multicolumn{3}{c}{ $ c_{22AB}(d)= -\tfrac{11}{5} , \tfrac{11}{15} ,-\tfrac{2}{15} ,\tfrac{1}{5} ,-\tfrac{1}{15}$ for $d=2,4,11,22,44$}\\
\multicolumn{3}{c}{ $c'_{22AB}(d)=1,4,8$ for $d=1,2,4$}\\
\bottomrule
\end{tabular}
\end{center}\caption{\label{m3table} {The list of weight $2$ modular forms $\wttwo^{(3)}_g(\t) $ for $\Gamma_0(N_g)$. }}\end{table}

We would like to specify the two components $H^{(3)}_{g,1}(\t)$ and $H^{(3)}_{g,2}(\t)$ of the vector-valued mock modular form $\big(H^{(3)}_{g,r}(\t)\big)$ which we propose to be the McKay--Thompson series arising from the $G^{(3)}$-module $K^{(3)}$. 
As mentioned earlier, to specify $H^{(3)}_{g,1}(\t)$ and $H^{(3)}_{g,2}(\t)$ it is sufficient to specify the weight $2$ forms defined in \eq{weigh2_2}. 
To see this, recall that to any given conjugacy class $[g]$ we can associate a conjugacy class $[zg]$ such that \eq{conjugate_pairs} holds. From this we obtain
\begin{gather}
	\begin{split}
   H^{(3)}_{g,1}(\t) &= \frac{\bar\chi^{(3)}_{g} }{ \chi^{(3)}} H^{(3)}_{1}(\t) + \frac{1}{2} \frac{\h(4\t)^2}{\h(2\t)^5} \big(\wttwo^{(3)}_g(\t)+\wttwo^{(3)}_{zg}(\t)\big), \\\label{relate_weight2_mock}
      H^{(3)}_{g,2}(\t) &= \frac{\chi^{(3)}_{g} }{ \chi^{(3)}} H^{(3)}_{2}(\t) + \frac{1}{4} \frac{\h(2\t)}{\h(\t)^2\h(4\t)^2} \big(\wttwo^{(3)}_g(\t)-\wttwo^{(3)}_{zg}(\t)\big),
	\end{split}
\end{gather}
where we have used the eta quotient expressions 
\be
S^{(3)}_1(\t) = \frac{\h(2\t)^5}{\h(4\t)^2},\quad S^{(3)}_2(\t) =2 \frac{\h(\t)^2\h(4\t)^2}{\h(2\t)},
\ee
for the components of the unary theta series at $\ll=3$.

The explicit expressions for $\wttwo^{(3)}_g(\t)$ are listed in Table \ref{m3table}. 

We also note that for the classes $3B$ and $6B$ we also have the following  alternative expressions for the McKay--Thompson
 series in terms of eta quotients:  
\be
H_{3B,1}^{(3)}(\t) 
=H_{6B,1}^{(3)}(\t) 
 =-2 \frac{\h(\t)\h(6\t)^5}{\h(3\t)^3 \h(12\t)^2},\quad 
 H_{3B,2}^{(3)}(\t) 
 =-H_{6B,2}^{(3)}(\t) 
 = -4 \frac{\h(\t)\h(12\t)^2}{\h(3\t) \h(6\t)}.
\ee
Coincidences with Ramanujan's mock theta functions will be discussed (\ref{subsec:MTmocktheta}).

\subsection{Lambency Four}\label{subsec:mis4}

The McKay--Thompson series for lambency $4$ display an interesting relation with those for lambency $2$. To see this, notice the following relation among the theta functions
\be
S^{(4)}_{\;1}(2\t)-S^{(4)}_{\;3}(2\t) = S^{(2)}_{\;1}(\t).
\ee
Given this relation it is natural to consider the function
\be
H^{(4)} _{g,\ast}(\tau) :=H^{(4)}_{g,1}(\tau) - H^{(4)}_{g,3}(\tau)
\ee
for each conjugacy classe $[g]$ of $G^{(4)}$. 
Note that
$H^{(4)}_{g,1}( \tau)$ and $H^{(4)}_{g,3}( \tau)$ can be reconstructed from $H^{(4)}_{g,\ast}( \tau)$ since they  are $q$-series of the form $q^{-1/16}$ times a series of even or odd powers of $q^{1/2}$, respectively. Explicitly, we have
\begin{gather}
	\begin{split}
H^{(4)}_{g,1}( \tau)&=\tfrac{1}{2}\big(H^{(4)}_{g,\ast}( \tau)+e(\tfrac{1}{16}) H^{(4)}_{g,\ast}( \tau+1) \big),\\ \label{relation_odd_mis4}
H^{(4)}_{g,3}( \tau)&=\tfrac{1}{2}\big(-H^{(4)}_{g,\ast}( \tau)+e(\tfrac{1}{16}) H^{(4)}_{g,\ast}( \tau+1) \big).
	\end{split}
\end{gather}

In order to obtain an expression for $ H^{(4)} _{g,\ast}(\tau)$ we rely on the following two observations. 
First, recall in \S\ref{sec:grps:spcnst} we have observed a relation between the Frame shapes of $g \in G^{(4)}$ and $g' \in G^{(2)}$.
It turns out that for a pair of group elements  $g \in G^{(4)}$ and $g' \in G^{(2)}$ related in the way described in Proposition \ref{prop:grps:sqellfrms} their McKay--Thompson series $H^{(4)}_{g}(2\t)$ and $H^{(2)}_{g'}(\t)$ are also related in a simple way. As examples of this we have 
\begin{gather}
	\begin{split}
H^{(4)}_{\;1A,\ast}( \tau)&=H^{(4)}_{\;2A,\ast}( \tau) = H^{(2)}_{{2A}}(\t/2) \\
H^{(4)}_{\;2B,\ast}( \tau)&= H^{(2)}_{{4A}}(\t/2)
\\
H^{(4)}_{\;2C,\ast}( \tau)&= H^{(2)}_{{4B}}(\t/2)\\
H^{(4)}_{\;3A,\ast}( \tau) &=H^{(4)}_{\;6A,\ast}( \tau)=H^{(2)}_{{6A}}(\t/2) \\
H^{(4)}_{\;4C,\ast}( \tau) &=H^{(2)}_{{8A}}(\t/2) \\
H^{(4)}_{\;6BC,\ast}( \tau) &=H^{(2)}_{{12A}}(\t/2) 
\\
H^{(4)}_{\;7AB,\ast}( \tau) &=H^{(4)}_{\;14AB,\ast}( \tau) =H^{(2)}_{{14AB}}(\t/2).
	\end{split}
\end{gather}
This leaves us just the classes $4A$, $4B$, and $8A$ that are not related to any element of $G^{(2)}\simeq M_{24}$ in the way described in Proposition \ref{prop:grps:sqellfrms}. All of these classes have $\chi^{(4)}_{g}=\bar\chi^{(4)}_{g}=0$. 
A second observation is that all the $H^{(3)}_{g,r}(\t)$ for those classes $[g]\subset G^{(3)}$ with $\chi^{(3)}_{g}=\bar\chi^{(3)}_{g}=0$ have an expression in terms of eta quotients according to Table \ref{m3table}, and so do the $H^{(2)}_g(\t)$ for those classes $[g]\subset G^{(2)}$ with $\chi^{(2)}_g =0$.  This is consistent with our expectation that the shadow of the vector-valued mock modular form $\big(H^{(\ll)}_{g,r}(\t)\big)$ is given by $\big(\chi^{(\ll)}_{g,r} S^{(\ll)}_{r}(\t)\big)$, and hence $\big(H^{(\ll)}_{g,r}(\t)\big)$ is nothing but a vector-valued modular form in the usual sense when $\chi^{(\ll)}_{g}=\bar\chi^{(\ll)}_{g}=0$. 
Given this it is natural to ask whether we can find similar eta quotient expressions for $H^{(4)}_{g,\ast}(\t)$ with $\chi^{(\ll)}_{g}=\bar\chi^{(\ll)}_{g}=0$ when $\ll>3$. We find
\begin{gather}
	\begin{split}
H^{(4)}_{\;4A,\ast}( \tau)&= -2 \frac{\eta(\t/2)\eta( \tau)^2}{\eta(2 \tau)^2}, \\
H^{(4)}_{\;4B,\ast}( \tau)&= -2 \frac{\eta(\t/2)\eta(\t)^4}{\eta(\t)^2\eta(4 \tau)^2}, \\
H^{(4)}_{\;8A,\ast}( \tau)&= -2 \frac{ \eta( \tau)^3}{\eta(\tau/2) \eta(4 \tau)}.
	\end{split}
\end{gather}

Together with \eq{relation_odd_mis4} and Table \ref{h_g}, the above equations completely specify $H^{(4)}_{g,1}(\t)$ and $H^{(4)}_{g,3}(\t)$ for all conjugacy classes $[g]$ of $G^{(4)}$. 
We are left to determine the second components $H^{(4)}_{g,2}(\t)$.
To start with, we have $H^{(4)}_{2}(\t)=H^{(4)}_{\;1A,2}(\t)=-H^{(4)}_{\;2A,2}(\t)$ given in terms of the decomposition of the weight $0$ Jacobi form as explained in \S\ref{sec:forms:wtzero}. To determine the rest, notice that \eq{conjugate_pairs} and \eq{weigh2_2}  implies that 
\be
S^{(4)}_{\;2}(\t) \left( H^{(4)}_{g,2}(\t) - \frac{\chi^{(4)}_{g,2} }{\chi^{(4)}} H^{(4)}_{2}(\t)\right) 
\ee
should be a weight $2$ modular form. Employing this consideration we arrive at  
\begin{gather}
	\begin{split}
{ H}^{(4)}_{\;3A,2}(\t) &=-{ H}^{(4)}_{\;6A,2}(\t)  \\
& = \frac{1}{4}H^{(4)}_{2}(\t) +\frac{1}{2\h(2\t)^3}\Big(-3\L_2(\t) -4\L_3(\t) + \L_6(\t)\Big),\\
{ H}^{(4)}_{\;7AB,2}(\t) &=-{ H}^{(4)}_{\;14AB,2}(\t)\\
& =  \frac{1}{8}H^{(4)}_{2}(\t) +\frac{1}{12\,\h(2\t)^3}\Big(-{7} \L_2(\t)  -4 \L_7(\t) + \L_{14}(\t)  +{28} f_{14}(\t)\Big),
	\end{split}
\end{gather}
where we have used $S^{(4)}_{\;2}(\t) = 2\h(2\t)^3$ and the hypothesis that  
\be
{H}^{(4)}_{g,2}(\t) =0 
\ee
for all classes not in the set $\{1A,2A,3A,6A,7AB,14AB\}$. This finishes our proposal for the McKay--Thompson series $\big(H^{(4)}_{g,r}\big)$ at $\ll=4$. Coincidences between some of the $H^{(4)}_{g,r}$ and Ramanujan's mock theta functions will be discussed in \S\ref{subsec:MTmocktheta}.

\subsection{Lambency Five}\label{subsec:mis5}

\begin{table}[h!]
\centering
\begin{tabular}{cccC}\toprule
$[g]$&$N_g$& $\wttwo^{(5)}_g(\t)$&\wttwo^{(5),2}_g(\t)\\\midrule   
1A&1&0&0\\    [5pt]
2A&4&0&0\\    [5pt]
2B&4&$16(\L_2- \L_4/3)$& -\frac{8}{3} {\h(\t)^8}/{\h(\t/2)^4}\\     [5pt]
2C&2&$-16\L_2$&e(\tfrac{1}{4})\,\wttwo^{(5),2}_{2B}(\t+1)\\    [5pt]
3A&9&$ -2{\h^6(\t)}/{\h^2(3\t)}$&-2 \frac{\h(\t)^8 \h(3\t/2)^2 \h(6\t)^2}{\h(\t/2)^2 \h(2\t)^2 \h(3\t)^4}   
\\      [5pt]
6A&36&$ -2 \,{\h(\t)^2\h(2\t)^2\h(3\t)^2}/{\h(6\t)^2}$     &2 \,{\h(\t/2)^2 \h(\t)^2 \h(3\t)^2}/{\h(3\t/2)^2}   \\   [5pt]
4AB&16&$-2{\h(2\t)^{14}}/{\h(\t)^4\h(4\t)^6}$ &-16 {\h(\t)^{2}\h(4\t)^{6}}/{\h(2\t)^4}\\[5pt]
4CD&8&$ -8/3 \L_4(\t)$&-\frac{32}{3} \,{\h(2\t)^2\h(4\t)^2}/{\h(\t)^2}\\[5pt]
5A&5&$-2\L_5$&e(\tfrac{1}{4})\,\wttwo^{(5),2}_{10A}(\t+1)\\   [5pt]
10A&20&$\sum_{d|20} c_{10A}(d) \L_d-\tfrac{40}{3}f_{20}$&\sum_{d|20} c_{10A,2}(d) \L_d(\t/4)+\tfrac{10}{3}f_{20}(\t/4) \\   [5pt]
12AB&144& $-2 {\h(\t)^2\h(2\t)^2\h(6\t)^4}/{\h(3\t)^2\h(12\t)^2}$&-4 {\h(\t)^2\h(2\t)^2\h(12\t)^2}/{\h(6\t)^2}\\
\midrule
\multicolumn{3}{c}{$ c_{10A}(d)= -5 , \tfrac{5}{3}  ,-\tfrac{2}{3} ,1, -\tfrac{1}{3}$ for $d=2,4,5,10,20$}\\[5pt]
\multicolumn{3}{c}{ $ c_{10A,2}(d)= -\tfrac{5}{4} , \tfrac{5}{24}  ,-\tfrac{1}{3} ,\tfrac{1}{4} , -\tfrac{1}{24}$ for $d=2,4,5,10,20$}\\
\bottomrule
\end{tabular}\caption{\label{m5table} {The list of weight $2$ modular forms $\wttwo^{(5)}_g(\t) $ on $\Gamma_0(N_g)$ and the other set of weight $2$ modular forms $\wttwo^{(5),2}_g(\t) $ . }}
\end{table}

As mentioned before, for $\ll=5$ it is no longer sufficient to specify the weight $2$ forms $ \wttwo^{(\ll)}_g(\t)$ as defined in \eq{weigh2_2}. 
In this case we also need to specify the weight $2$ forms $ \wttwo^{(\ll),2}_g(\t)$ defined in \eq{weigh2_2_2}.

With the data of $ \wttwo^{(5)}_g(\t)$ and $ \wttwo^{(5),2}_g(\t)$ we can determine the $H^{(5)}_{g,r}$ using the relations
\be\label{lis5_1}
H^{(5)}_{g,r} (\t) = \hat H^{(5)}_{g,r} (\t) +\frac{\chi^{(5)}_{g,r}}{6} H^{(5)}_r(\t)
\ee
and
\begin{gather}
\begin{split}
\hat H^{(5)}_{g,1}(\t) S^{(5)}_1(\t)+\hat H^{(5)}_{g,3}(\t) S^{(5)}_3(\t) &= \tfrac{1}{2} \Big( \wttwo^{(5)}_g(\t) + \wttwo^{(5)}_{zg}(\t)\Big) \\
\hat H^{(5)}_{g,1}(\t) S^{(5)}_4(\t)+\hat H^{(5)}_{g,3}(\t) S^{(5)}_2(\t) &= \tfrac{1}{2} \Big( \wttwo^{(5),2}_g(\t) + \wttwo^{(5),2}_{zg}(\t)\Big)\\
\hat H^{(5)}_{g,2}(\t) S^{(5)}_2(\t)+\hat H^{(5)}_{g,4}(\t) S^{(5)}_4(\t) &= \tfrac{1}{2} \Big( \wttwo^{(5)}_g(\t) - \wttwo^{(5)}_{zg}(\t)\Big)\\ \label{lis5_2}
\hat H^{(5)}_{g,2}(\t) S^{(5)}_3(\t)+\hat H^{(5)}_{g,4}(\t) S^{(5)}_1(\t) &= \tfrac{1}{2} \Big( -\wttwo^{(5),2}_g(\t) + \wttwo^{(5),2}_{zg}(\t)\Big)
\end{split}
\end{gather}
where $g$ and $zg$ again form a pair satisfying \eq{conjugate_pairs}.

Moreover, notice that in $ \wttwo^{(5),2}_g(\t)$ the part of the sum involving even $r$  has an expansion in $q^{1/4+k}$ for non-negative integers $k$, while the sum over odd $r$ has an expansion in $q^{3/4+k}$. As a result, it is sufficient to specify $ \wttwo^{(\ll),2}_g(\t)$ for one of the paired classes $g$ and $zg$ as they are related to each other by
\be\label{lis5_3}
\wttwo^{(5),2}_{zg}(\t) = e(\tfrac{1}{4})\,\wttwo^{(\ll),2}_{g}(\t+1)\;.
\ee

Using (\ref{lis5_1}-\ref{lis5_3}), the data recorded in Table \ref{m5table} are sufficient to explicitly specify all the mock modular forms $H^{(5)}_{g}(\t)$ for all $[g]\subset G^{(5)}$ at lambency $5$.

\subsection{Lambencies Seven and Thirteen}\label{subsec:remaining}

For $\ll=7$, apart from the $1A$ and $2A$ classes whose McKay-Thomson series have been given in terms of the weight $0$ Jacobi form $Z^{(7)}$ in \S\ref{sec:forms:wtzero}, there are five more classes $3AB$, $4A$, and $6AB$ whose  McKay-Thomson series $H^{(7)}_{g}$ we would like to identify. 
We specify the weight $2$ forms associated to these classes. Notice that these expressions are not sufficient to completely determine all the components $H^{(7)}_{g,r}$. Nevertheless, they provide strong evidence for the mock modular properties of the proposed McKay--Thompson
 series. 

\begin{gather}
	\begin{split}
\wttwo^{(7)}_{4A}(\t) &= -2\frac{\h(\t)^4 \h(2\t)^2}{\h(4\t)^2}, \\
\wttwo^{(7),2}_{4A}(\t) &= 4 \frac{\h(\t)^6 \h(4\t)^2}{\h(2\t)^4}, \\
\wttwo^{(7)}_{3AB}(\t) &= -6 \L_3(\t), \\
\wttwo^{(7)}_{6AB}(\t) &= -9\L_2(\t) -2\L_3(\t) +3\L_4(\t) +3\L_6(\t) -\L_{12}(\t), \\
\wttwo^{(7),2}_{6AB}(4\t) &= -\frac{9}{4} \L_2(\t) -\L_3(\t) +\frac{3}{8}\L_4(\t) +\frac{3}{4}\L_6(\t) -\frac{1}{8}\L_{12}(\t).
	\end{split}
\end{gather}

Similarly, for lambency $13$ we have 
\begin{gather}
	\begin{split}
\wttwo^{(13)}_{4AB}(\t) & = -2 \frac{\h(2\t)^{14}}{\h(\t)^4\h(4\t)^6},\\
\wttwo^{(13),2}_{4AB}(\t) & =-16 \frac{\h(\t)^2 \h(4\t)^6}{\h(2\t)^4}.
	\end{split}
\end{gather}

\subsection{Mock Theta Functions}\label{subsec:MTmocktheta}

As mentioned in \S\ref{sec:forms:mock} we observe that in many cases across different lambencies the components $H^{(\ll)}_{g,r}$ of the McKay--Thompson series coincide with known mock theta functions. In particular, we often encounter Ramanujan's mock theta functions, and always in such a way that the order is divisible by the lambency. In this section we will list these conjectural identities between $H^{(\ll)}_{g,r}$ and mock theta functions identified previously in the literature.

For $\ll=2$ two of the functions $H^{(2)}_{g}(\t)$ are related to  Ramanujan's mock theta functions of orders $2$ and $8$ (cf. (\ref{ll2mocktheta})) through
\begin{gather}
	\begin{split}
H^{(2)}_{4B}(\t) & = -2 q^{-1/8} \m(q),
\\
H^{(2)}_{8A}(\t) & = -2 q^{-1/8} U_0(q).
	\end{split}
\end{gather}

For $\ll=3$ we encounter the following  order $3$ mock theta functions of  Ramanujan: 
\begin{gather}
	\begin{split}
H^{(3)}_{2B,1}(\t) &=H^{(3)}_{2C,1}(\t)=H^{(3)}_{4C,1}(\t)=-2q^{-1/12} f(q^2), \\
H^{(3)}_{6C,1}(\t) &= H^{(3)}_{6D,1}(\t) = -2 q^{-1/12} \chi(q^2),
 \\
H^{(3)}_{8C,1}(\t) &= H^{(3)}_{8D,1}(\t) = -2 q^{-1/12} \phi(-q^2), \\
H^{(3)}_{2B,2}(\t) &= - H^{(3)}_{2C,2}(\t) = -4 q^{2/3} \omega(-q), 
\\\label{L3mocktheta}
H^{(3)}_{6C,2}(\t) &= - H^{(3)}_{6D,2}(\t) = 2 q^{2/3} \rho(-q).
	\end{split}
\end{gather}
(Cf. (\ref{eqn:forms:mock:ord3}).) The description of the shadow of $H^{(3)}_{g}$ is consistent with the fact that, among the seven order $3$ mock theta functions of Ramanujan, $f(\t)$, $\f(\t)$, $\psi(\t)$ and $\chi(\t)$ form a group with the same shadow ($S^{(3)}_{1}(\t)$) while the other three $\o(\t)$, $\n(\t)$ and $\r(\t)$ form another group with another shadow ($S^{(3)}_{\;2}(\t)$). Moreover, various relations among these order $3$ mock theta functions can be obtained as consequences of the above identification.

For $\ll=4$ we encounter the following order  $8$  mock theta functions
\begin{gather}
	\begin{split}
{H}^{(4)}_{2C,1}(\t)&=q^{-\frac{1}{16}} \big(-2 S_0(q) + 4T_0(q)\big), \\
{H}^{(4)}_{2C,3}(\t)&=q^{\frac{7}{16}} \big(2 S_1(q) - 4T_1(q)\big), \\ 
{H}^{(4)}_{4C,1}(\t)&= -2\, q^{-\frac{1}{16}} S_0(q), \\
{H}^{(4)}_{4C,3}(\t)&=2 \,q^{\frac{7}{16}} \, S_1(q) .
	\end{split}
\end{gather}
(Cf. (\ref{eqn:forms:mock:ord8}).) Comparing with \eq{ll2mocktheta} we see that our proposal implies the identities 
\begin{gather}
	\begin{split}
\m(q) &= S_0(q^2) - 2T_0(q^2) + q \big(S_1(q^2) - 2T_1(q^2) \big)\\
	& = U_0(q)-2U_1(q), \\ 
U_0(q) &= S_0(q^2)+ q \,S_1(q^2),\\
 U_1(\t)&= T_0(q^2)+ q \,T_1(q^2), 
	\end{split}
\end{gather}
between different mock theta functions of Ramanujan. See, for instance, \cite{GorMcI_SvyMckTht} for a collection of such identities.

For $\ll=5$ we encounter four of  Ramanujan's order $10$ mock theta functions:
\begin{gather}
	\begin{split}
&H^{(5)}_{2BC,1}(\t) =H^{(5)}_{4CD,1}(\t) = -2 q^{{-\frac{1}{20}}} \,X(q^2),
\\ 
&H^{(5)}_{2BC,3}(\t) =H^{(5)}_{4CD,3}(\t) = -2 q^{{-\frac{9}{20}}} \,\chi_{10}(q^2), 
\\ 
&H^{(5)}_{2C,2}(\t) =-H^{(5)}_{2B,2}(\t)  = 2q^{-\frac{1}{5}} \,\psi_{10}(-q),
\\ 
&H^{(5)}_{2C,4}(\t) =-H^{(5)}_{2B,4}(\t)  = -2q^{\frac{1}{5}} \,\f_{10}(-q). 
	\end{split}
\end{gather}
(Cf. (\ref{eqn:forms:mock:ord10}).)

\subsection{Automorphy}\label{sec:mckay:aut}

In this section we discuss the automorphy of the proposed McKay--Thompson series $H^{(\ll)}_{g} = \big( H^{(\ll)}_{g,r}\big) $. As mentioned in \S\ref{Forms of Higher Levels}, the function $H^{(\ll)}_g$ is a vector-valued mock modular form with shadow $\big(\chi_{g,r}S^{(\ll)}_r\big)$ for some $\G_g \subset SL_2(\ZZ)$ with a certain (matrix-valued) multiplier $\n_g$. In this subsection we will specify the group $\G_g$ and the multiplier $\n_g$. The multipliers we specify here can be verified explicitly using the data given in \S\S\ref{subsec:m2}-\ref{subsec:remaining} (except for the few conjugacy classes at $\ll=7$ and $\ll=13$ for which the McKay--Thompson series $H^{(\ll)}_g$ have not been completely determined). 

We find that the automorphy of the vector-valued function $H_g^{(\ll)}$ is governed, in way that we shall describe presently, by the signed permutation representation of $G^{(\ll)}$ arising from the construction (as a subgroup of $\HO_m$ for $m=24/(\ll-1)$) given in \S\ref{sec:grps:spcnst}. We use this representation to define the symbols $n_g|h_g$ which appear in the second row of each table \ref{tab:coeffs:2_1}-\ref{tab:coeffs:13_12}, and also in the twisted Euler character tables \ref{tab:chars:eul:2}-\ref{tab:chars:eul:13}, and we will explain below how to use these symbols to determine the multiplier system for each $H^{(\ll)}_{g}$. It will develop also that the twisted Euler character $g\mapsto \chi^{(\ll)}_g$ (cf. Tables \ref{tab:chars:eul:2}-\ref{tab:chars:eul:13}) attached to the signed permutation representation of $G^{(\ll)}$ determines the shadow of $H^{(\ll)}_g$.

Recall from \S\ref{sec:grps:spcnst} that the signed permutation representation of $G^{(\ll)}$ naturally induces a permutation representation of the same degree, and this permutation representation factors through the quotient $\bar{G}^{(\ll)}=G^{(\ll)}/2$ of $G^{(\ll)}$ by its unique central subgroup of order $2$ (except in case $\ll=2$ when the signed and unsigned permutation representations coincide, and we have $G^{(2)}=\bar{G}^{(2)}$). Let $g\mapsto \bar{g}$ denote the natural map $G^{(\ll)}\to \bar{G}^{(\ll)}$ and observe that (for $\ll>2$) if the unique central involution of $G^{(\ll)}$ belongs to the cyclic subgroup $\langle g\rangle$ generated by $g$ then $o(g)=2o(\bar{g})$ and otherwise $o(g)=o(\bar{g})$. We say that $g\in G^{(\ll)}$ is {\em split} over $\bar{G}^{(\ll)}$ in case $o(g)=o(\bar{g})$ and we call $g$ {\em non-split} otherwise. (By this definition every element of $G^{(2)}=\bar{G}^{(2)}$ is split.) 

Recall from \S\ref{sec:grps:spcnst} that $g\mapsto \Pi^{(\ll)}_{g}$ denotes the map attaching signed permutation Frame shapes to elements of $G^{(\ll)}$ and $g\mapsto \bar{\Pi}_g^{(\ll)}$ denotes the Frame shapes (actually cycle shapes) arising from the (unsigned) permutation representation (on $m=24/(\ll-1)$ points). Taking a formal product of Frame shapes $\tilde{\Pi}^{(\ll)}_g=\Pi^{(\ll)}_{g}\bar{\Pi}^{(\ll)}_{g}$ (defined so that $j^{m_1}j^{m_2}=j^{m_1+m_2}$, \&c.) we obtain the Frame shape of $g\in G^{(\ll)}$ regarded as a permutation of the $2m$ points $\{\pm e_i\}$ for $i\in \O^{(\ll)}$ (cf. \S\ref{sec:grps:spcnst}). In particular, $\tilde{\Pi}_g$ is a cycle shape and none of the exponents appearing in $\tilde{\Pi}_g$ are negative. Given $g\in G^{(\ll)}$ and $\tilde{\Pi}_g=j_1^{m_1}\cdots j_k^{m_l}$ with $j_1<\cdots< j_k$ and $m_i>0$ define $N_g=j_1j_k$. That is, set $N_g$ to be the product of the shortest and longest cycle lengths appearing in a cycle decomposition for $g$ regarded as a permutation on the $2m$ points $\{\pm e_i\}$. Now define the symbols $n_g|h_g$ by setting $n_g=o(\bar{g})$ for all $g\in G^{(\ll)}$ and all $\ll\in\LL$, and by setting $h_g=N_g/n_g$ for all $g\in G^{(\ll)}$ and $\ll\in\LL$ except when $\ll=4$ and $g$ is non-split in which case set $h_g=N_g/2n_g$. The symbols $n_g|h_g$ are specified in Tables \ref{tab:chars:eul:2}-\ref{tab:chars:eul:13}, and also in Tables \ref{tab:coeffs:2_1}-\ref{tab:coeffs:13_12}, in the rows labelled $\G_g$. We omit the $|h_g$ when $h_g=1$, so $n_g$ is a shorthand for $n_g|1$ in these tables. 

The significance of the value $n_g$ is that it is the minimal positive integer $n$ for which $H^{(\ll)}_g$ is a mock modular form (of weight $1/2$) on $\Gamma_0(n)$, and the significance of $h_g$ is that, as we shall see momentarily, it is the minimal positive integer $h$ for which the multiplier for $H^{(\ll)}_g$ coincides with the conjugate multiplier for the (vector-valued) cusp form $S^{(\ll)}$ when restricted to $\Gamma_0(nh)$ (for $n=n_g$). Since $n_gh_g\leq N_g$ for all $g$ the multiplier for $H^{(\ll)}_g$ coincides with the conjugate multiplier for $S^{(\ll)}$ when regarded as a mock modular form on $\Gamma_0(N_g)$ for all $g$. It is very curious that this coincidence of multipliers extends to the larger group $\Gamma_0(N_g/2)$ in the case that $\ll=4$ and $g$ is non-split (i.e. $o(g)=2o(\bar{g})$).

\begin{table}[h]
\begin{center}
\caption{Admissible $v^{(\ll)}$}\label{tab:vells}
\smallskip
\begin{tabular}{c||cccccc}
$\ll$&	2&	3&	4&	5&	7&	13\\
	\hline
$v^{(\ll)}$&1&5&3&7&1&7
\end{tabular}
\end{center}
\end{table}
Given a pair of positive integers $(n,h)$ we define a matrix-valued function $\rho^{(\ll)}_{n|h}$ on $\G_0(n)$ as follows. For each $\ll\in \LL$ let the integer $v^{(\ll)}$ be as specified in Table \ref{tab:vells}. When $h$ divides $n$ we set 
\be
\rho^{(\ll)}_{n|h}(\g)=\ex\left(-v^{(\ll)}\frac{cd}{nh}\right){\rm I}_{\ll-1}
\ee
where ${\rm I}_{\ll-1}$ denotes the $(\ll-1)\times (\ll-1)$ identity matrix. When $h$ does not divide $n$ and $n$ is even we set
\be
\rho^{(\ll)}_{n|h}(\g)=\ex\left(-v^{(\ll)}\frac{cd}{nh}\frac{(n,h)}{n}\right){\rm J}_{\ll-1}^{c(d+1)/n}{\rm K}_{\ll-1}^{c/n}
\ee
where $(n,h)$ denotes the greatest common divisor of $n$ and $h$, and ${\rm J}_{\ll-1}$ and ${\rm K}_{\ll-1}$ are the $(\ll-1)\times (\ll-1)$ matrices given by
\be		
	{\rm J}_{\ll-1}
	=\begin{pmatrix}
		1&0&0&\cdots&0\\
		0&-1&0&\cdots&0\\
		0&0&1&\cdots&0\\
		\vdots&\vdots&\vdots&\ddots&\vdots\\
		0&0&0&\cdots&(-1)^{\ll}
	\end{pmatrix},\quad
	{\rm K}_{\ll-1}
	=\begin{pmatrix}
		0&\cdots&0&0&1\\
		0&\cdots&0&1&0\\
		0&\cdots&1&0&0\\
		\vdots&\iddots&\vdots&\vdots&\vdots\\
		1&\cdots&0&0&0
	\end{pmatrix},
\ee
and when $h$ does not divide $n$ and $n$ is odd we set
\be
\rho^{(\ll)}_{n|h}(\g)=\ex\left(-v^{(\ll)}\frac{cd}{nh}\frac{n}{(n,h)}\right){\rm J}_{\ll-1}^{c(d+1)/n}{\rm K}_{\ll-1}^{c/n}.
\ee
Now the multiplier system $\nu_g^{(\ll)}$ for the umbral mock modular form $H^{(\ll)}_g$ is the matrix-valued function on $\G_0(n)=\G_0(n_g)$ given simply by 
\be\label{eqn:mtseries:aut:nug}
\nu_g^{(\ll)}=\nu^{(\ll)}_{n|h}=\rho^{(\ll)}_{n|h}\overline{\sigma^{(\ll)}}
\ee
where $n=n_g$ and $h=h_g$ and $\sigma^{(\ll)}=(\sigma^{(\ll)}_{ij})$ denotes the (matrix-valued) multiplier system for the (vector-valued) theta series $S^{(\ll)}$ (cf. (\ref{unary_S_def})) and satisfies
\be\label{eqn:mtseries:aut:sigma}
	\sigma^{(\ll)}(\g)S^{(\ll)}(\g \tau){\jac}(\g,\tau)^{3/4}=S^{(\ll)}(\t)
\ee
for $\g\in \G_0(1)=\SL_2(\ZZ)$ where $\jac(\g,\t)=(c\t+d)^{-2}$ in case $\g$ has lower row $(c,d)$.

Recall from \S\ref{sec:grps:spcnst} that the character of $G^{(\ll)}$ attached to its signed permutation representation is denoted $g\mapsto \chi^{(\ll)}_{g}$ in Tables \ref{tab:chars:eul:2}-\ref{tab:chars:eul:13} and that of the (unsigned) permutation representation is denoted $g\mapsto \bar{\chi}^{(\ll)}_g$. Define $\chi^{(\ll)}_{g,r}$ for $0<r<\ll$ by setting $\chi^{(\ll)}_{g,r}=\chi^{(\ll)}_{g}$ in case $r$ is even and $\chi^{(\ll)}_{g,r}=\bar{\chi}^{(\ll)}_{g}$ otherwise. Then the shadow of $H^{(\ll)}_g$ is the function $S^{(\ll)}_g=\big(S^{(\ll)}_{g,r}\big)$ with components related to those of $S^{(\ll)}$ by $S^{(\ll)}_{g,r}=\chi^{(\ll)}_{g,r}S^{(\ll)}_r$. In particular, $H^{(\ll)}_g$ is a vector-valued modular form of weight $1/2$ for $\Gamma_0(n_g)$ when $\chi^{(\ll)}_g=\bar{\chi}^{(\ll)}_g=0$.

To summarise, we claim that our proposed McKay--Thompson series $H^{(\ll)}_g$ are such that if we define $\hat H^{(\ll)}_{g}=\big(\hat H^{(\ll)}_{g,r}\big)$ by setting 
\begin{gather}
\hat H^{(\ll)}_{g,r}(\t) = H_{g,r}(\t) -\frac{\chi^{(\ll)}_{g,r}}{\sqrt{2\ll}} \frac{1}{(4i)^{1/2}} \int_{-\bar\t}^{i\infty} (z+\t)^{-1/2} S^{(\ll)}_r(z){\rm d}z
\end{gather}
for $0<r<\ll$ then $\hat H^{(\ll)}_g$ is invariant for the weight $1/2$ action of $\G_0(n)$ on $(\ll-1)$-vector-valued functions on $\HH$ given by
\begin{gather}\label{eqn:mtseries:aut:wthalfaction}
\left(\left.\hat H^{(\ll)}_g\right|_{1/2,n|h}\g\right)(\t)= 
	 \n_{n|h} (\g) \hat H^{(\ll)}_g(\g\t)
	\jac(\g,\t)^{1/4}
\end{gather}
where $n=n_g$ and $h=h_g$ and $\n_{n|h}$ is defined by (\ref{eqn:mtseries:aut:nug}). This statement completely describes the (conjectured) automorphy of the mock modular forms $H^{(\ll)}_{g}$ (cf. \S\ref{sec:conj:aut}).

\section{Conjectures}\label{sec:conj}

We have described the umbral forms $Z^{(\ll)}$, $H^{(\ll)}$ and $\Phi^{(\ll)}$ in \S\ref{sec:forms} and the umbral groups $G^{(\ll)}$ in \S\ref{sec:grps} and we have introduced families $\{H^{(\ll)}_g\mid g\in G^{(\ll)}\}$ of vector-valued mock modular forms in \S\ref{sec:mckay}. The discussions of those sections clearly demonstrate the distinguished nature of these objects, and we have mentioned some coincidences relating the groups $G^{(\ll)}$ and the forms $H^{(\ll)}_g$ directly. In this section we present, in a more systematic fashion, evidence that the relationship between the $G^{(\ll)}$ and the $H^{(\ll)}$ is more than coincidental. Our observations lead naturally to conjectures that we hope will serve as first steps in revealing the structural nature of the mechanism underlying umbral moonshine.

\subsection{Modules}\label{sec:conj:mod}

As was mentioned in \S\ref{sec:mckay}, after comparison of the character tables (cf. \S\ref{sec:chars:irr}) of the umbral groups $G^{(\ll)}$ with the Fourier coefficient tables (cf. \S\ref{sec:coeffs}) for the forms $H^{(\ll)}_g$ it becomes apparent that the low degree Fourier coefficients of $H^{(\ll)}=H^{(\ll)}_e$ may be interpreted as degrees of representations of $G^{(\ll)}$ in such a way that the corresponding coefficients of $H^{(\ll)}_g$ are recovered by substituting character values at $g$; we have tabulated evidence for this in the form of explicit combinations of irreducible representations in \S\ref{sec:decompositions}. This observation suggests the existence of bi-graded $G^{(\ll)}$-modules $K^{(\ll)}=\bigoplus K^{(\ll)}_{r,d}$ whose bi-graded dimensions are recovered via Fourier coefficients from the vector-valued mock modular forms $H^{(\ll)}=\big(H^{(\ll)}_r\big)$. 
\begin{conj}\label{conj:conj:mod:Kell}
We conjecture that for $\ll\in\LL=\{2,3,4,5,7,13\}$ there exist naturally defined $\ZZ\times\QQ$-graded ${G}^{(\ll)}$-modules 
\begin{gather}
	K^{(\ll)}=\bigoplus_{\substack{r\in\ZZ\\0<r<\ll}}K^{(\ll)}_r=\bigoplus_{\substack{r,k\in\ZZ\\0<r<\ll}}K^{(\ll)}_{r,{k-r^2/4\ll}}
\end{gather}
such that the graded dimension of $K^{(\ll)}$ is related to the vector-valued mock modular form $H^{(\ll)}=\big(H^{(\ll)}_r\big)$ by
\begin{gather}
	H^{(\ll)}_r(\tau)=-2\delta_{r,1}q^{-1/4\ll}+\sum_{\substack{k\in\ZZ\\r^2-4k\ll<0}}\dim\left({K^{(\ll)}_{r,k-r^2/4\ll}}\right)q^{k-r^2/4\ll}
\end{gather}
for $q=\ex(\t)$ and such that the vector-valued mock modular forms $H^{(\ll)}_g=\big(H^{(\ll)}_{g,r}\big)$ described in \S\ref{sec:mckay} (and partially in the tables of \S\ref{sec:coeffs}) are recovered from $K^{(\ll)}$ via graded trace functions according to
\begin{gather}\label{eqn:conj:mod:str}
	H^{(\ll)}_{g,r}(\tau)=-2\delta_{r,1}q^{-1/4\ll}+\sum_{\substack{k\in\ZZ\\r^2-4k\ll<0}}\tr_{K^{(\ll)}_{r,k-r^2/4\ll}}(g)q^{k-r^2/4\ll}.
\end{gather}
\end{conj}
\begin{rmk}
Recall that a {\em superspace} is a $\ZZ/2\ZZ$-graded vector space and if $V=V_{\bar{0}}\oplus V_{\bar{1}}$ is such an object and $T:V\to V$ is a linear operator preserving the grading then the {\em supertrace} of $T$ is given by $\str_VT=\tr_{V_{\bar{0}}}T-\tr_{V_{\bar{1}}}T$ where $\tr_WT$ denotes the usual trace of $T$ on $W$. Since the coefficient of $q^{-1/4\ll}$ in $H^{(\ll)}_{g,1}$ is $c^{(\ll)}_{g,1}(-1/4\ll)=-2$ for all $g\in G^{(\ll)}$ for all $\ll$ it is natural to expect that this term may be interpreted as the supertrace of $g\in G^{(\ll)}$ on a trivial $G^{(\ll)}$-supermodule $K^{(\ll)}_{1,-1/4\ll}$ with vanishing even part and $2$-dimensional odd part. Thus Conjecture \ref{conj:conj:mod:Kell} implies the existence of a $G^{(\ll)}$-supermodule  $K^{(\ll)}=\bigoplus K^{(\ll)}_{r,d}$ such that $K^{(\ll)}_{r,d}$ is purely even or purely odd according as $d$ is positive or negative, and such that $H^{(\ll)}_{g,r}=\sum_k \str_{K^{(\ll)}_{r,k-r^2/4\ll}}(g)q^{k-r^2/r\ll}$ for each $0<r<\ll$ and $g\in G^{(\ll)}$.
\end{rmk}

\begin{rmk}
For $\ll\in\LL$ and $0<r,s<\ll$ we have that $r^2\equiv s^2\pmod{4\ll}$ implies $r=s$ so the first index $r$ in $K^{(\ll)}_{r,d}$ can be deduced from $d$ since it is the unique $0<r<\ll$ such that $d+r^2/4\ll$ is an integer. Thus we may dispense with the first of the two gradings on the conjectural $G^{(\ll)}$-modules $K^{(\ll)}$ and regard them as $\QQ$-graded $K^{(\ll)}=\bigoplus_d K^{(\ll)}_{d}$ by the rationals of the form $d=n-r^2/4\ll$ for $n\in\ZZ$ and $0<r<\ll$ without introducing any ambiguity.
\end{rmk}

\subsection{Moonshine}\label{sec:conj:moon}

In the case of monstrous moonshine the McKay--Thompson series $T_g$ for $g$ in the monster group have the astonishing property that they all serve as generators for the function fields of their invariance groups (cf. \cite{conway_norton}). In other words, if $\G_g$ is the subgroup of $\PSL_2(\RR)$ consisting of the isometries $\g:\HH\to\HH$ such that $T_g(\g\t)=T_g(\t)$ for all $\t\in \HH$ then $T_g$ induces an isomorphism from the compactification of $\Gamma_g\backslash\HH $ to the Riemann sphere (being the one point compactification of $\CC$). In particular, $\Gamma_g$ is a genus zero subgroup of $\PSL_2(\RR)$, and so this property is commonly referred to as the {\em genus zero property} of monstrous moonshine. 

By now there are many methods extant for constructing graded vector spaces with algebraic structure whose graded dimensions are modular functions---suitable classes of vertex algebras, for example, serve this purpose (cf. \cite{Zhu_ModInv,Dong2000})---and so one can expect to obtain analogues of the McKay--Thompson series $T_g$ by equipping such an algebraic structure with the action of a group. But there is no guarantee that such a procedure will result in functions that have the genus zero property of monstrous moonshine, and so it is this genus zero property that distinguishes the observations of McKay, Thompson, Conway and Norton regarding the Monster group from any number of more generic connections between finite groups and modular functions. In what follows we will propose a conjecture that may be regarded as the natural analogue of the Conway--Norton conjecture for umbral moonshine.

Suppose that $T:\HH\to \CC$ is a holomorphic function with invariance group $\G<\PSL_2(\RR)$ and a simple pole in $q=e(\t)$ as $\t\to i\infty$. Suppose also that $\G$ is commensurable with $\PSL_2(\ZZ)$ and that the {\em translation subgroup} $\G_{\infty}$, consisting of the elements of $\G$ with upper-triangular preimages in $\SL_2(\RR)$, is generated by $\t\mapsto \t+1$. It is shown in \cite{DunFre_RSMG} that such a function $T$ is a generator for the field of $\G$-invariant functions on $\HH$ if and only if $T(\t)$ coincides with the weight $0$ {\em Rademacher sum}
\begin{gather}\label{eqn:conj:monradsum}
	R_{\G}(\t)={\rm Reg}\left(\sum_{\g\in \G_{\infty}\backslash\G} \left.q^{-1}\right|_0{\g}\right)
\end{gather}
attached to $\G$. The sum here is over representatives for the cosets of $\G_{\infty}$ in $\G$ and $f\mapsto f|_0\g$ denotes the weight $0$ action of $\g$ on holomorphic functions (viz., $(f|_0\g)(\t)=f(\g\t)$). We write ${\rm Reg}(\cdot)$ to indicate a regularisation procedure first realised by Rademacher (for the case that $\G=\PSL_2(\ZZ)$) in \cite{Rad_FuncEqnModInv}. According then to the result of \cite{DunFre_RSMG} the genus zero property of monstrous moonshine may be reformulated in the following way: 
\begin{quote}
For each element $g$ in the monster group the McKay--Thompson series $T_g$ satisfies $T_g=R_{\G_g}$ when $\G_g$ is the invariance group of $T_g$.
\end{quote}

This may be compared with the article \cite{Cheng2011} which considers the mock modular forms $H^{(2)}_g$ attached to the largest Mathieu group $G^{(2)}\simeq M_{24}$ via the observation of Eguchi--Ooguri--Tachikawa and applies some of the philosophy of \cite{DunFre_RSMG} to the problem of finding a uniform construction for these functions. The solution developed in \cite{Cheng2011} is that the function $H^{(2)}_g$ coincides with the weight $1/2$ Rademacher sum
\begin{gather}\label{eqn:conj:m24radsum}
	R^{(2)}_{n|h}(\t)={\rm Reg}\left(\sum_{\g\in \G_{\infty}\backslash\G_0(n)} \left.-2q^{-1/8}\right|_{1/2,n|h}{\g}\right)
\end{gather}
where $n=n_g$ and $h=h_g$ are integers determined by the defining permutation representation of $M_{24}$ (cf. Table \ref{tab:chars:eul:2}) and $f\mapsto f|_{1/2,n|h}\g$ denotes a certain weight $1/2$ action of $\G_0(n)$ determined by $n$ and $h$. (We refer to \cite{Cheng2011} for full details.) By comparison with the previous paragraph we thus arrive at a direct analogue of the genus zero property of monstrous moonshine that holds for all the umbral forms with $\ll=2$:
\begin{quote}
For each element $g$ in the largest Mathieu group the McKay--Thompson series $H^{(2)}_g$ satisfies $H^{(2)}_g=R^{(2)}_{n|h}$ when $n=n_g$ and $h=h_g$.
\end{quote}
We conjecture that this genus zero property extends to all the functions of umbral moonshine.
\begin{conj}[umbral moonshine]\label{conj:conj:moon}
We conjecture that for each $\ll\in \LL$ and each $g\in G^{(\ll)}$ we have $H^{(\ll)}_g=R^{(\ll)}_{n|h}$ where $n=n_g$ and $h=h_g$ are as specified in \S\ref{sec:chars:eul} and 
\begin{gather}
	R^{(\ll)}_{n|h}
	={\rm Reg}\left(\sum_{\g\in \G_{\infty}\backslash\G_0(n)} 
	\left.\begin{pmatrix}
		-2q^{-1/4\ll}\\
		0\\
		\vdots\\
		0
	\end{pmatrix}
	\right|_{1/2,n|h}{\g}\right)
\end{gather}
is a vector-valued generalisation of the Rademacher sum $R^{(2)}_{n|h}$ adapted to the weight $1/2$ action of $\G_0(n)$ on $(\ll-1)$-vector-valued functions on $\HH$ that is defined in (\ref{eqn:mtseries:aut:wthalfaction}).
\end{conj}

\subsection{Modularity}\label{sec:conj:aut}

A beautiful feature of the umbral moonshine conjecture is that it implies the precise nature of the modularity of the Thompson series $H^{(\ll)}_g$. We record these implications explicitly as conjectures in this short section.

\begin{conj}\label{conj:conj:aut:shad}
We conjecture that the graded supertrace functions (\ref{eqn:conj:mod:str}) for fixed $\ll\in\LL$ and $g\in G^{(\ll)}$ and varying $0<r<\ll$ define the components of a vector-valued mock modular form $H^{(\ll)}_{g}$ of weight $1/2$ on $\Gamma_0(n_g)$ with shadow function $S^{(\ll)}_g=(S^{(\ll)}_{g,r})=(\chi^{(\ll)}_{g,r}S^{(\ll)}_r)$ where $S^{(\ll)}=(S^{(\ll)}_{r})$ is the vector-valued theta series described in \S\ref{sec:forms:mero}, the $\chi^{(\ll)}_{g,r}$ are determined from the twisted Euler characters of $G^{(\ll)}$ (cf. \S\ref{sec:chars:eul}) by $\chi^{(\ll)}_{g,r}=\bar{\chi}^{(\ll)}_g$ for $r$ odd and $\chi^{(\ll)}_{g,r}=\chi^{(\ll)}_g$ for $r$ even, and $n_g$ denotes the order of the image of $g\in G^{(\ll)}$ in the factor group $\bar{G}^{(\ll)}$ (cf. \S\ref{sec:grps:spec},\ref{sec:grps:spcnst}).
\end{conj}

Recall from \S\ref{sec:mckay:aut} that for $\ll\in \LL$ and a pair $(n,h)$ of positive integers we have the matrix-valued function $\nu_{n|h}^{(\ll)}$ on $\G_0(n)$ defined in (\ref{eqn:mtseries:aut:nug}). Recall also that we have attached a pair $(n_g,h_g)$ to each $g\in G^{(\ll)}$ for each $\ll\in \LL$ in \S\ref{sec:mckay:aut}.
\begin{conj}\label{conj:conj:aut:mult}
We conjecture that the multiplier system of $H^{(\ll)}_g$ is given by $\nu_{n|h}^{(\ll)}$ when $n=n_g$ and $h=h_g$ for all $g\in G^{(\ll)}$ for all $\ll\in\LL$.
\end{conj}

\subsection{Discriminants}\label{sec:conj:disc}

One of the most striking features of umbral moonshine is the apparently intimate relation between the number fields on which the irreducible representations of $G^{(\ll)}$ are defined and the discriminants of the vector-valued mock modular form $H^{(\ll)}$. We discuss the evidence for this relation and formulate conjectures about it in this section.

First we observe that the discriminants of the components $H^{(\ll)}_r$ of the mock modular form $H^{(\ll)}=H^{(\ll)}_e$ determine some important properties of the representations of $G^{(\ll)}$, where we say that an integer $D$ is a {\em discriminant of $H^{(\ll)}$} if there exists a term $q^d=q^{-\frac{D}{4\ll}}$ with non-vanishing Fourier coefficient in at least one of the components. The following result can be verified explicitly using the tables in \S\S\ref{sec:chars},\ref{sec:coeffs}.
\begin{prop}\label{discri1}
Let $\ll\in \L$. If $n>1$ is an integer satisfying 
\begin{enumerate}
\item{there exists an element of $G^{(\ll)}$ of order $n$}, and
\item{there exists an integer $\l$ that is co-prime to $n$ such that $D = -n \l^2$ is a discriminant of $H^{(\ll)}$,}
\end{enumerate}
then there exists at least one pair of irreducible representations $\varrho$ and $\varrho^*$ of $G^{(\ll)}$ and at least one element $g \in G^{(\ll)}$ such that $\tr_{\varrho}(g)$ is not rational but
\be\label{n_type}
{\tr}_{\varrho} (g), {\tr}_{ \varrho^*} (g) \in \QQ(\sqrt{-n})
\ee
and $n$ divides $o(g)$.
\end{prop}
The finite list of integers $n$ satisfying the two conditions of Proposition \ref{discri1} is given in Table \ref{list_discriminant}.
 
From now on we say that an irreducible representation $\varrho$ of the umbral group $G^{(\ll)}$ is of {\it type $n$} if $n$ is an integer satisfying the two conditions of Proposition \ref{discri1} and the character values of $\varrho$ generate the field $\QQ(\sqrt{-n})$. Evidently, irreducible representations of {type $n$} come in pairs $(\varrho,\varrho^*)$ with ${\tr}_{\varrho^*} (g)$ the complex conjugate of ${\tr}_{ \varrho} (g)$ for all $g\in G^{(\ll)}$. The list of all irreducible representations of type $n$ is given in Table \ref{list_discriminant}. (See \S\ref{sec:chars:irr} for the character tables of the $G^{(\ll)}$ and our notation for irreducible representations.) 
 
Recall that the {\em Frobenius--Schur indicator} of an irreducible ordinary representation of a finite group is $1$, $-1$ or $0$ according as the representation admits an invariant symmetric bilinear form, an invariant skew-symmetric bilinear form, or no invariant bilinear form, respectively. The representations admitting no invariant bilinear form are precisely those whose character values are not all real. We can now state the next observation.
\begin{prop}\label{FS_indicator}
For each $\ll\in\LL$ an irreducible representation $\varrho$ of $G^{(\ll)}$ has Frobenius--Schur indicator $0$ if and only if it is of type $n$ for some $n$. 
\end{prop}

The {\em Schur index} of an irreducible representation $\varrho$ of a finite group $G$ is the smallest positive integer $s$ such that there exists a degree $s$ extension $k$ of the field generated by the character values ${\tr}_{\varrho}( g)$ for $g\in G$ such that $\varrho$ can be realised over $k$. Inspired by Proposition \ref{FS_indicator} we make the following conjecture.
\begin{conj}\label{conj:conj:disc:sch}
If $\varrho$ is an irreducible representation of $G^{(\ll)}$ of type $n$ then the Schur index of $\varrho$ is equal to $1$.
\end{conj}
In other words, we conjecture that the irreducible $G^{(\ll)}$-representations of type $n$ can be realised over $\QQ(\sqrt{-n})$. For $\ll=2$ this speculation is in fact a theorem, since it is known \cite{Ben_SchInd} that the Schur indices for $M_{24}$ are always $1$. For $\ll=3$ it is also known \cite{Ben_SchInd} that the Schur indices for $\bar G^{(3)}= M_{12}$ are also always $1$. Moreover,  the representations of $G^{(3)}\simeq 2.\bar{G}^{(3)}$ with characters $\chi_{16}$ and $\chi_{17}$ in the notation of Table \ref{tab:chars:irr:3} have been constructed explicitly over $\QQ(\sqrt{-2})$ in \cite{Mar_M12}. Finally, Proposition \ref{FS_indicator} constitutes a non-trivial consistency check for Conjecture \ref{conj:conj:disc:sch} since the Schur index is at least $2$ for a representation with Frobenius--Schur indicator equal to $-1$.

\vspace{18pt}
\begin{table}[h!] \centering  \begin{tabular}{CCC}
\toprule
\ll & n &(\varrho,\varrho^*)\\\midrule
2& 7,15,23 & (\chi_{3},\chi_{4}),(\chi_{5},\chi_{6}),(\chi_{10},\chi_{11}),(\chi_{12},\chi_{13}),(\chi_{15},\chi_{16})\\
3& 5,8,11,20 & (\chi_{4},\chi_{5}),(\chi_{16},\chi_{17}),(\chi_{20},\chi_{21}),(\chi_{22},\chi_{23}),(\chi_{25},\chi_{26})\\
4& 3,7&(\chi_{2},\chi_{3}),(\chi_{13},\chi_{14}),(\chi_{15},\chi_{16})\\ 
5& 4&(\chi_{8},\chi_{9}),(\chi_{10},\chi_{11}),(\chi_{12},\chi_{13})\\
7&3&(\chi_{2},\chi_{3}),(\chi_{6},\chi_{7})\\
13&4&(\chi_3,\chi_4)\\
\bottomrule
\end{tabular}
\caption{\label{list_discriminant} 
	The irreducible representations of type $n$.}
\end{table}

Armed with the preceding discussion we are now ready to state our main observation for the discriminant property of umbral moonshine. For the purpose of stating this we temporarily write $K^{(\ll)}_{r,d}$ for the ordinary representation of $G^{(\ll)}$ with character $g\mapsto c^{(\ll)}_{g,r}(d)$ where the coefficients $c^{(\ll)}_{g,r}(d)$ are assumed to be those given in \S\ref{sec:coeffs}.
\begin{prop}
Let $n$ be one of the integers in Table \ref{list_discriminant} and let $\l_{n}$ be the smallest positive integer such that $D = -n \l_{n}^2$ is a discriminant of $H^{(\ll)}$. Then $K^{(\ll)}_{r,-D/4\ll} = \varrho_{n} \oplus \varrho_{n}^*$ where $\varrho_{n}$ and $ \varrho_{n}^*$ are dual irreducible representations of type $n$. Conversely, if $\varrho$ is an irreducible representation of type $n$ and $-D$ is the smallest positive integer such that $K^{(\ll)}_{r,-D/4\ll}$ has $\varrho$ as an irreducible constituent then there exists an integer $\l$ such that $D = - n \l^2$. 
\end{prop} 
 
Extending this we make the following conjecture.
\begin{conj}\label{conj:conj:disc:dualpair}
If $D$ is a discriminant of $H^{(\ll)}$ which satisfies $D = -n \l^2$  for some integer $\l$ then the representation $K^{(\ll)}_{r,-D/4\ll}$ has at least one dual pair of irreducible representations of type $n$ arising as irreducible constituents. 
\end{conj}

We conclude this section with conjectures arising from the observation (cf. \S\ref{sec:decompositions}) that the conjectural $G^{(\ll)}$-module $K^{(\ll)}_{r,d}$ is typically isomorphic to several copies of a single representation. Say a $G$-module $V$ is a {\em doublet} if it is isomorphic to the direct sum of two copies of a single representation of $G$.
\begin{conj}\label{conj:conj:disc:doub}
For $\ll\in\LL=\{2,3,4,5,7,13\}$ the representation $K^{(\ll)}_{r,-D/4\ll}$ is a doublet if and only if $D \neq -n\l^2$ for any integer $\l$ for any $n$ satisfying the conditions of Proposition \ref{discri1}. 
\end{conj}

To see some evidence for Conjecture \ref{conj:conj:disc:doub} one can inspect the proposed decompositions of the representations $K^{(\ll)}_{r,d}$ in the tables in \S\ref{sec:decompositions} for the following discriminants: 
\begin{itemize}
\item{$-D=7,15,23,63,135,175,207$ for $\ll=2$,}
\item{$-D=8,11,20,32,44,80$ for $\ll=3$,}
\item{$-D=7,12,28,63,108$ for $\ll=4$,}
\item{$-D=4,16,64,144,196$ for $\ll=5$,}
\item{$-D=3\l^2$, $\l = 1,\dots,9$, $\l\neq 7$ for $\ll=7$,}
\item{$-D=4 \l^2$, $\l = 1,\dots,11$ for $\ll=13$.}
\end{itemize}

\subsection{Geometry and Physics}\label{sec:conj:geomphys}

Beyond the conjectures already mentioned above several interesting and important questions remain regarding the structural nature of umbral moonshine. In the case that $\ll=2$ there are strong indications that a deep relationship to $K3$ surfaces---extending in some way the relation \cite{Mukai,Kondo} between finite groups of $K3$ surface symplectomorphisms and subgroups of $M_{23}$---is responsible for the relationship between $G^{(2)}\simeq M_{24}$ and $H^{(2)}$. One such indication is the fact that the Jacobi form $Z^{(2)}$, from which $H^{(2)}$ may be obtained by decomposing with respect to superconformal characters, coincides with the elliptic genus of a(ny) complex $K3$ surface. It is natural then to ask if there are analogous geometric interpretations for the remaining extremal Jacobi forms $Z^{(\ll)}$ for $\ll\in \LL$, and a positive answer to this question will be a first step in determining the geometric significance of the umbral groups $G^{(\ll)}$ and the attached mock modular forms $H^{(\ll)}_g$. 

In a series of papers \cite{GriNik_K3SrfsLorKMAlgsMrrSym,GriNik_ArthMrrSymCYMfds,GriNik_AutFrmLorKMAlgs_I,GriNik_AutFrmLorKMAlgs_II,GrNik_SieAutFrmCorrLorKMAlgs} Gritsenko--Nikulin develop applications of Siegel modular forms to mirror symmetry for $K3$ surfaces (cf. \cite{Dol_MrrSymK3Srfs}). Many of the Siegel forms arising are realised as additive or exponential lifts of (weak) Jacobi forms, and amongst many examples the exponential lifts $\Phi^{(\ll)}$ of the particular forms $Z^{(\ll)}$ for $\ll\in\{2,3,4,5\}\subset\LL$ appear in connection with explicitly defined families of lattice polarised $K3$ surfaces in \cite{GriNik_AutFrmLorKMAlgs_II}. It is natural to ask if there is an analogous relationship between the umbral $Z^{(\ll)}$ and suitably defined mirror families of $K3$ surfaces for the remaining $\ll$ in $\LL$. 

As mentioned in the introduction, monstrous moonshine involves aspects of conformal field theory and string theory and there are several hints that umbral
moonshine will also play a role in string theory. The most obvious hint at $\ll=2$ is through the fact mentioned above that the weak Jacobi form
$Z^{(2)}$ coincides with the elliptic genus of a $K3$ surface and the fact that $K3$ surfaces play a prominent role in the study of superstring compactification.
The Siegel form $\Phi^{(2)}=(\Delta_5)^2$ which is the multiplicative lift of the weak Jacobi form $Z^{(2)}$ also plays a distinguished role
in type II string theory on $K3 \times E$ where $E$ is an elliptic curve and $1/\Phi^{(2)}$ occurs as the generating function which counts the number of $1/4$ BPS dyon states \cite{DVV}. BPS states in a supersymmetric theory are states that are annihilated by some of the supercharges of the theory; the $1/4$ BPS
states of interest here are annihilated by $4$ of the $16$ supercharges of a theory with $N=4$ supersymmetry in four spacetime dimensions and are termed
{\em dyons} because they necessarily carry both electric and magnetic charges.
The relation between the zeros of $\Phi^{(2)}$ and the wall-crossing phenomenon in this theory has been studied in \cite{SenJHEP0705:0392007,Cheng2007a,Cheng2008a,Dabholkar2008} and this relation leads to a connection between the counting of BPS black hole states in string theory and mock modular forms \cite{Dabholkar:2012nd}.
For a pedagogical review of this material see \cite{Dabholkar2007,Sen2007c}. 

There are several hints that umbral moonshine for $\ll>2$ will also play a role in string theory.   The first of these occurs in the study of
$N=2$ dual pairs of string theories \cite{Kachru:1995wm,Ferrara:1995yx}. An $N=2$ dual pair consists of a compactification of the heterotic string on  a product of
a $K3$ surface with an elliptic curve $E$
with a specific choice of $E_8 \times E_8$ gauge bundle on $K3 \times E$ and a conjectured dual description
of this model in terms of IIA string theory  on a  Calabi--Yau threefold $X$ which admits a $K3$ fibration. The low-energy description of such theories
involves a vector-multiplet moduli space which is a special K\"ahler manifold  of the form 
\be
{\cal M}_{vm}^{s+2,2} = \frac{SU(1,1)}{U(1)} \times {\cal N}^{s+2,2}
\ee
where ${\cal N}^{s+2,2}$ is the quotient $\Gamma\backslash {\cal H}^{s+1,1}$ of the generalised upper half-plane
\be
{\cal H}^{s+1,1}= O(s+2,2;\RR)/(O(s+2) \times O(2))
\ee
by an arithmetic subgroup $\Gamma$ in $SO(s+2,2; \RR)$, with $\Gamma$ depending on the specific model in question. One-loop string computations in the heterotic string lead to
automorphic forms on ${\cal N}^{s+2,2}$ via a generalised theta lift constructed in \cite{Harvey:1995fq,Bor_AutFmsSngGrs}, see \cite{Kon_PdtFmlsMdlrFmsAftBor} for a review.
In models with $s=1$ this leads to automorphic forms on $\Gamma \backslash O(3,2;\RR)/(O(3) \times O(2))$, that is to Siegel modular forms on the genus $2$ Siegel upper half-space $\HH_2\cong {\cal H}^{2,1}$. 

A particular example of such an $N=2$ dual pair was studied in \cite{LopesCardoso:1996zj,LopesCardoso:1996nc} with $s=1$ and involving a dual description
in terms of Type II string theory on a Calabi--Yau $3$-fold with Hodge numbers $(h^{1,1},h^{2,1})=(4,148)$. In the heterotic description the arithmetic subgroup
which appears is the paramodular group $\Gamma_2$, the Siegel modular form which occurs is the exponential lift of $Z^{(3)}$ described in \S\ref{sec:forms:siegel} and the Calabi--Yau
$3$-fold is also elliptically fibered and the elliptic fiber is of $E_7$ type. We thus see many of the ingredients of $\ll=3$ umbral moonshine appearing
in the context of a specific dual pair of string theories with $N=2$ spacetime supersymmetry. It will be very interesting to investigate whether other aspects
of $\ll=3$ umbral moonshine can be realized in these models and whether $N=2$ dual pairs exist which exhibit elements of umbral moonshine for
$\ll>3$.

A second way that elements of umbral moonshine for $\ll>2$ are likely to appear in string theory involves a generalisation of the generating function discussed
above that counts $1/4$ BPS states. Type II string theory on $K3 \times E$ preserves $N=4$ spacetime supersymmetry. There exist orbifold versions of this model
that also preserve $N=4$ spacetime supersymmetry that are known in the physics literature as CHL models. To construct such a model one chooses a $K3$
surface with a $\ZZ/m \ZZ$ hyper-K\"ahler automorphism and constructs the orbifold theory $(K3 \times E)/(\ZZ/m \ZZ)$ where $\ZZ/m \ZZ$ is a freely acting symmetry realised as the product of a hyper-K\"ahler automorphism of the $K3$ surface and an order $m$ translation along $E$. For $m=2,3,4$ it is possible to find $K3$
surfaces with $(\ZZ/m \ZZ) \times (\ZZ/m \ZZ)$ symmetry, and in this case one can construct a CHL model that utilises the first $\ZZ/m \ZZ$ factor in the orbifold
construction and has the  second $\ZZ/m \ZZ$ factor acting as  a symmetry which preserves the holomorphic $3$-form of the Calabi--Yau space
$(K3 \times E)/(\ZZ/m \ZZ)$. It was proposed in  \cite{Govindarajan2010} that the generating function which counts $1/4$ BPS states weighted by an element
of $\ZZ/m \ZZ$ is the reciprocal of the Siegel form $\Phi^{(m+1)}$ for $m \in \{2,3,4\}$.  In both this construction and in the study of $N=2$ dual pairs the appearance of
some of the umbral Siegel forms $\Phi^{(\ll)}$ can be anticipated, although many details remain to be worked out. The action of the umbral groups $G^{(\ll)}$
remains more elusive, and deeper insight into possible connections between string theory and umbral moonshine will undoubtedly require progress in understanding
the actions of these groups in terms of their action on BPS states.

We plan to elaborate further on the topics mentioned above in future work.

\section*{Acknowledgements}

We thank Chris Cummins, Atish Dabholkar, Noam Elkies, Igor Frenkel, George Glauberman, Ian Grojnowski, Kobi Kremnizer, Anthony Licata, Greg Moore, Sameer Murthy, Ken Ono, Cumrun Vafa, Erik Verlinde, Edward Witten, Shing-Tung Yau and Don Zagier for helpful comments and discussions. We thank Noam Elkies and Ken Ono in particular for communication regarding the non-vanishing of critical central values for modular $L$-functions.

We are especially grateful to Don Zagier for showing us his computation of the first few Fourier coefficients of the functions (that we have here denoted) $H^{(\ll)}$ for $\ll\in\{2,3,4,5,7\}$ at the ``Workshop on Mathieu Moonshine'' held in July 2011 at the Swiss Federal Institute of Technology (ETH), Z\"urich. This event was an important catalyst for precipitating the idea that there could be a family of groups attached to a distinguished family of vector-valued modular forms, in a way that generalises the connection between $M_{24}$ and the (scalar-valued) mock modular form $H^{(2)}$. We are also grateful to Kazuhiro Hikami for his talk at the ETH meeting on character decompositions of elliptic genera, and for kindly sharing his slides from that presentation. The analysis described therein proved useful for us in developing a computational approach to approximate Fourier coefficients of the $H^{(\ll)}$ and this furnished an important source of experimental evidence to guide our umbral explorations. We warmly thank the organisers of the ETH meeting for making these exchanges possible and for providing an excellent workshop experience more generally. 

The research of MC is partially supported by the NSF grant FRG 0854971. The work of JH was supported by the NSF grant 0855039, the hospitality of the Aspen Center for Physics, the LPTHE of the Universit\'e Pierre et Marie Curie, and the Isaac Newton Institute for Mathematical Sciences. We also thank George Armhold for instruction on the efficient implementation of Pari in OS X, and Jia-Chen Fu for the provision of computational resources at a crucial stage in our investigations.

We are grateful to George Glauberman, Sander Mack-Crane, and the referee, for pointing out errors in an earlier draft. We are grateful to the referee also for suggesting many valuable improvements to the text.

\appendix

\section{Modular Forms}\label{sec:modforms}

\subsection{Dedekind Eta Function}\label{sec:mdlrfrms:dedeta}
The {\em Dedekind eta function}, denoted $\eta(\t)$, is a holomorphic function on the upper half-plane defined by the infinite product 
$$\eta(\t)=q^{1/24}\prod_{n\geq 1}(1-q^n)$$
where $q=\ex(\t)=e^{\tpi \t}$. It is a modular form of weight $1/2$ for the modular group $\SL_2(\ZZ)$ with multiplier $\e:\SL_2(\ZZ)\to\CC^*$, which means that 
$$\e(\g)\eta(\g\t){\rm jac}(\g,\t)^{1/4}=\eta(\t)$$
for all $\g = \big(\begin{smallmatrix} a&b\\ c&d \end{smallmatrix}\big) \in\SL_2(\ZZ)$, where ${\rm jac}(\g,\t)=(c\t+d)^{-2}$. The {\em multiplier system} $\e$ may be described explicitly as 
\be\label{Dedmult}
\e\bem a&b\\ c&d\eem 
=
\begin{cases}
	\ex(-b/24),&c=0,\,d=1\\
	\ex(-(a+d)/24c+s(d,c)/2+1/8),&c>0
\end{cases}
\ee
where $s(d,c)=\sum_{m=1}^{c-1}(d/c)((md/c))$ and $((x))$ is $0$ for $x\in\ZZ$ and $x-\lfloor x\rfloor-1/2$ otherwise. We can deduce the values $\e(a,b,c,d)$ for $c<0$, or for $c=0$ and $d=-1$, by observing that $\e(-\g)=\e(\g)\ex(1/4)$ for $\g\in\SL_2(\ZZ)$.

Let $T$ denote the element of $\SL_2(\ZZ)$ such that $\tr(T)=2$ and $T\t=\t+1$ for $\t\in \HH$. Observe that
$$
\e(T^m\g)=\e(\g T^m)=\ex(-m/24)\e(\g)
$$
for $m\in\ZZ$.

\subsection{Jacobi Theta Functions}\label{sec:JacTheta}
We define the {\em Jacboi theta functions} $\th_i(\t,z)$ as follows for $q=e(\t)$ and $y=e(z)$.
$$	\th_1(\t,z)
	= -i q^{1/8} y^{1/2} \prod_{n=1}^\inf (1-q^n) (1-y q^n) (1-y^{-1} q^{n-1})$$
$$	\th_2(\t,z)
	=  q^{1/8} y^{1/2} \prod_{n=1}^\inf (1-q^n) (1+y q^n) (1+y^{-1} q^{n-1})$$ 
$$	\th_3(\t,z)
	=  \prod_{n=1}^\inf (1-q^n) (1+y \,q^{n-1/2}) (1+y^{-1} q^{n-1/2})$$
$$	\th_4(\t,z) 
	=  \prod_{n=1}^\inf (1-q^n) (1-y \,q^{n-1/2}) (1-y^{-1} q^{n-1/2})$$
Note that there are competing conventions for $\th_1(\t,z)$ in the literature and our normalisation may differ from another by a factor of $-1$ (or possibly $\pm i$).

\subsection{Higher Level Modular Forms}

The congruence subgroups of the modular group $\SL_2(\ZZ)$ that are most relevant for this paper are 
\bea\label{congruence1}
\G_0(N) &=& \bigg\{\bigg[\begin{array}{cc} a&b\\c&d\end{array}\bigg]\in \SL_2(\ZZ)  , c= 0 \text{ mod }N\;
\bigg\}.
\eea
For $N>1$ a (non-zero) modular form of weight $2$ for $\G_0(N)$ is given by
\bea\label{Eisenstein_form}
\L_N(\t)&=&N\, q\pa_q\log\left(\frac{\eta(N\tau)}{\eta(\tau)}\right)\\\notag&=&\frac{N(N-1)}{24}\left(1+\frac{24}{N-1}\sum_{k>0}\s(k) (q^k -N q^{Nk})\right)
\eea
where $\s(k)$ is the divisor function $\s(k)=\sum_{d\lvert k}d$.

Observe that a modular form on $\Gamma_0(N)$ is a modular form on $\Gamma_0(M)$ whenever $N|M$, and for some small $N$ the space of forms of weight $2$ is spanned by the $\L_d(\t)$ for $d$ a divisor of $N$. By contrast, in the case that $N=11$ we have the {\em newform} 
\be\label{11newforms}
f_{11}(\t) = \eta^{2}(\t)\eta^{2}(11\t)
\ee
which is a cusp form of weight $2$ for $\Gamma_0(11)$ that is not a multiple of $\L_{11}(\t)$. We meet the newforms
\bea
f_{14}(\t) &=& \eta(\t)\eta(2\t)\eta(7\t)\eta(14\t,)\\
f_{15}(\t) &=& \eta(\t)\eta(3\t)\eta(5\t)\eta(15\t),\\
f_{20}(\t) &=& \eta(2\t)^2 \eta(10\t)^2,
\eea
at $N=14$, $N=15$ and $N=20$, respectively, and together with $f_N$ the functions $\L_d(\t)$ for $d|N$ span the space of weight $2$ forms on $\G_0(N)$ for $N=11,14,15$.

For $N=23$ there is a two dimensional space of newforms. We may use the basis
\begin{gather}
\begin{split}
f_{23,a}(\t)&= \frac{\eta(\tau)^3\eta(23\tau)^3}{\eta(2\tau)\eta(46\tau)}
	+3\h(\t)^2\h(23\t)^2
	+4\eta(\tau)\eta(2\tau)\eta(23\tau)\eta(46\tau)
	+4\eta(2\tau)^2\eta(46\tau)^2\\
f_{23,b}(\t)&= \h(\t)^2\h(23\t)^2\label{phi232}
\end{split}
\end{gather}
satisfying $f_{23,a}=q+O(q^3)$ and $f_{23,b}=q^2+O(q^3)$. (The normalised Hecke-eigenforms of weight $2$ for $\G_0(23)$ are $f_{23,a}-\frac{1}{2}(1\pm \sqrt{5})f_{23,b}$.) 

For $N=44$ there is a unique newform up to a multiplicative constant. It satisfies 
$$
f_{44}(\t) = q+q^3-3 q^5+2 q^7-2 q^9-q^{11}-4 q^{13}-3 q^{15}+6 q^{17}+8 q^{19}+2 q^{21}-3 q^{23}+4 q^{25}-5 q^{27}+O(q^{28}). 
$$
See \cite{modi} and Chapter 4.D of \cite{from_number} for more details.
A discussion of the ring of weak Jacobi forms of higher level can be found in \cite{aoki}. 

\section{Characters}\label{sec:chars}

In \S\ref{sec:chars:irr} we give character tables (with power maps and Frobenius--Schur indicators) for each group $G^{(\ll)}$. These were computed with the aid of the computer algebra package GAP4 \cite{GAP4} using the explicit presentations for the $G^{(\ll)}$ that appear in \S\ref{sec:grps:spcnst}. We use the abbreviations $a_n=\sqrt{-n}$ and $b_n=(-1+\sqrt{-n})/2$ in these tables. 

The tables in \S\ref{sec:chars:eul} furnish the Frame shapes $\Pi_g^{(\ll)}$ and character values $\chi_g^{(\ll)}$ attached to the signed permutation representations (cf. \S\ref{sec:grps:sgnprm}) of the groups $G^{(\ll)}$ given in \S\ref{sec:grps:spcnst}. These Frame shapes and character values can easily be computed by hand; we detail them here explicitly since they can be used to define symbols $n_g|h_g$ which encode the automorphy of the vector-valued mock modular forms $H^{(\ll)}_g$ according to the prescription of \S\ref{sec:mckay:aut}. The symbols $n_g|h_g$ are given in the rows labelled $\G_g$ in \S\ref{sec:chars:eul}.

\begin{sidewaystable}
\subsection{Irreducible Characters}\label{sec:chars:irr}
\begin{center}
\caption{Character table of $G^{(2)}\simeq M_{24}$}\label{tab:chars:irr:2}
\smallskip
\begin{small}
\begin{tabular}{c@{ }|c|@{ }r@{ }r@{ }r@{ }r@{ }r@{ }r@{ }r@{ }r@{ }r@{ }r@{ }r@{ }r@{ }r@{ }r@{ }r@{ }r@{ }r@{ }r@{ }r@{ }r@{ }r@{ }r@{ }r@{ }r@{ }r@{ }r} \toprule
$[g]$	&{FS}&1A	&2A	&2B	&3A	&3B	&4A	&4B	&4C	&5A	&6A	&6B	&7A	&7B	&8A	&10A	&11A	&12A	&12B	&14A	&14B	&15A	&15B	&21A	&21B	&23A	&23B	\\
	\midrule
$[g^{2}]$	&&1A	&1A	&1A	&3A	&3B	&2A	&2A	&2B	&5A	&3A	&3B	&7A	&7B	&4B	&5A	&11A	&6A	&6B	&7A	&7B	&15A	&15B	&21A	&21B	&23A	&23B	\\
$[g^{3}]$	&&1A	&2A	&2B	&1A	&1A	&4A	&4B	&4C	&5A	&2A	&2B	&7B	&7A	&8A	&10A	&11A	&4A	&4C	&14B	&14A	&5A	&5A	&7B	&7A	&23A	&23B	\\
$[g^{5}]$	&&1A	&2A	&2B	&3A	&3B	&4A	&4B	&4C	&1A	&6A	&6B	&7B	&7A	&8A	&2B	&11A	&12A	&12B	&14B	&14A	&3A	&3A	&21B	&21A	&23B	&23A	\\
$[g^{7}]$	&&1A	&2A	&2B	&3A	&3B	&4A	&4B	&4C	&5A	&6A	&6B	&1A	&1A	&8A	&10A	&11A	&12A	&12B	&2A	&2A	&15B	&15A	&3B	&3B	&23B	&23A	\\
$[g^{11}]$	&&1A	&2A	&2B	&3A	&3B	&4A	&4B	&4C	&5A	&6A	&6B	&7A	&7B	&8A	&10A	&1A	&12A	&12B	&14A	&14B	&15B	&15A	&21A	&21B	&23B	&23A	\\
$[g^{23}]$	&&1A	&2A	&2B	&3A	&3B	&4A	&4B	&4C	&5A	&6A	&6B	&7A	&7B	&8A	&10A	&11A	&12A	&12B	&14A	&14B	&15A	&15B	&21A	&21B	&1A	&1A	\\
	\midrule
$\chi_{1}$	&$+$&$1$	&$1$	&$1$	&$1$	&$1$	&$1$	&$1$	&$1$	&$1$	&$1$	&$1$	&$1$	&$1$	&$1$	&$1$	&$1$	&$1$	&$1$	&$1$	&$1$	&$1$	&$1$	&$1$	&$1$	&$1$	&$1$	\\
$\chi_{2}$	&$+$&$23$	&$7$	&$-1$	&$5$	&$-1$	&$-1$	&$3$	&$-1$	&$3$	&$1$	&$-1$	&$2$	&$2$	&$1$	&$-1$	&$1$	&$-1$	&$-1$	&$0$	&$0$	&$0$	&$0$	&$-1$	&$-1$	&$0$	&$0$	\\
$\chi_{3}$	&$\circ$&$45$	&$-3$	&$5$	&$0$	&$3$	&$-3$	&$1$	&$1$	&$0$	&$0$	&$-1$	&$b_7$	&$\overline{b_7}$	&$-1$	&$0$	&$1$	&$0$	&$1$	&$-b_7$	&$-\overline{b_7}$	&$0$	&$0$	&$b_7$	&$\overline{b_7}$	&$-1$	&$-1$	\\
$\chi_{4}$	&$\circ$&$45$	&$-3$	&$5$	&$0$	&$3$	&$-3$	&$1$	&$1$	&$0$	&$0$	&$-1$	&$\overline{b_7}$	&$b_7$	&$-1$	&$0$	&$1$	&$0$	&$1$	&$-\overline{b_7}$	&$-b_7$	&$0$	&$0$	&$\overline{b_7}$	&$b_7$	&$-1$	&$-1$	\\
$\chi_{5}$	&$\circ$&$231$	&$7$	&$-9$	&$-3$	&$0$	&$-1$	&$-1$	&$3$	&$1$	&$1$	&$0$	&$0$	&$0$	&$-1$	&$1$	&$0$	&$-1$	&$0$	&$0$	&$0$	&$b_{15}$	&$\overline{b_{15}}$	&$0$	&$0$	&$1$	&$1$	\\
$\chi_{6}$	&$\circ$&$231$	&$7$	&$-9$	&$-3$	&$0$	&$-1$	&$-1$	&$3$	&$1$	&$1$	&$0$	&$0$	&$0$	&$-1$	&$1$	&$0$	&$-1$	&$0$	&$0$	&$0$	&$\overline{b_{15}}$	&$b_{15}$	&$0$	&$0$	&$1$	&$1$	\\
$\chi_{7}$	&$+$&$252$	&$28$	&$12$	&$9$	&$0$	&$4$	&$4$	&$0$	&$2$	&$1$	&$0$	&$0$	&$0$	&$0$	&$2$	&$-1$	&$1$	&$0$	&$0$	&$0$	&$-1$	&$-1$	&$0$	&$0$	&$-1$	&$-1$	\\
$\chi_{8}$	&$+$&$253$	&$13$	&$-11$	&$10$	&$1$	&$-3$	&$1$	&$1$	&$3$	&$-2$	&$1$	&$1$	&$1$	&$-1$	&$-1$	&$0$	&$0$	&$1$	&$-1$	&$-1$	&$0$	&$0$	&$1$	&$1$	&$0$	&$0$	\\
$\chi_{9}$	&$+$&$483$	&$35$	&$3$	&$6$	&$0$	&$3$	&$3$	&$3$	&$-2$	&$2$	&$0$	&$0$	&$0$	&$-1$	&$-2$	&$-1$	&$0$	&$0$	&$0$	&$0$	&$1$	&$1$	&$0$	&$0$	&$0$	&$0$	\\
$\chi_{10}$&$\circ$&$770$	&$-14$	&$10$	&$5$	&$-7$	&$2$	&$-2$	&$-2$	&$0$	&$1$	&$1$	&$0$	&$0$	&$0$	&$0$	&$0$	&$-1$	&$1$	&$0$	&$0$	&$0$	&$0$	&$0$	&$0$	&$b_{23}$	&$\overline{b_{23}}$	\\
$\chi_{11}$&$\circ$&$770$	&$-14$	&$10$	&$5$	&$-7$	&$2$	&$-2$	&$-2$	&$0$	&$1$	&$1$	&$0$	&$0$	&$0$	&$0$	&$0$	&$-1$	&$1$	&$0$	&$0$	&$0$	&$0$	&$0$	&$0$	&$\overline{b_{23}}$	&$b_{23}$	\\
$\chi_{12}$&$\circ$&$990$	&$-18$	&$-10$	&$0$	&$3$	&$6$	&$2$	&$-2$	&$0$	&$0$	&$-1$	&$b_7$	&$\overline{b_7}$	&$0$	&$0$	&$0$	&$0$	&$1$	&$b_7$	&$\overline{b_7}$	&$0$	&$0$	&$b_7$	&$\overline{b_7}$	&$1$	&$1$	\\
$\chi_{13}$&$\circ$&$990$	&$-18$	&$-10$	&$0$	&$3$	&$6$	&$2$	&$-2$	&$0$	&$0$	&$-1$	&$\overline{b_7}$	&$b_7$	&$0$	&$0$	&$0$	&$0$	&$1$	&$\overline{b_7}$	&$b_7$	&$0$	&$0$	&$\overline{b_7} $	&$b_7$	&$1$	&$1$	\\
$\chi_{14}$&$+$&$1035$	&$27$	&$35$	&$0$	&$6$	&$3$	&$-1$	&$3$	&$0$	&$0$	&$2$	&$-1$	&$-1$	&$1$	&$0$	&$1$	&$0$	&$0$	&$-1$	&$-1$	&$0$	&$0$	&$-1$	&$-1$	&$0$	&$0$	\\
$\chi_{15}$&$\circ$&$1035$	&$-21$	&$-5$	&$0$	&$-3$	&$3$	&$3$	&$-1$	&$0$	&$0$	&$1$	&$2b_7$	&$2\overline{b_7}$	&$-1$	&$0$	&$1$	&$0$	&$-1$	&$0$	&$0$	&$0$	&$0$	&$-b_7$	&$-\overline{b_7} $	&$0$	&$0$	\\
$\chi_{16}$&$\circ$&$1035$	&$-21$	&$-5$	&$0$	&$-3$	&$3$	&$3$	&$-1$	&$0$	&$0$	&$1$	&$2\overline{b_7}$	&$2b_7$	&$-1$	&$0$	&$1$	&$0$	&$-1$	&$0$	&$0$	&$0$	&$0$	&$-\overline{b_7}$	&$-b_7 $	&$0$	&$0$	\\
$\chi_{17}$&$+$&$1265$	&$49$	&$-15$	&$5$	&$8$	&$-7$	&$1$	&$-3$	&$0$	&$1$	&$0$	&$-2$	&$-2$	&$1$	&$0$	&$0$	&$-1$	&$0$	&$0$	&$0$	&$0$	&$0$	&$1$	&$1$	&$0$	&$0$	\\
$\chi_{18}$&$+$&$1771$	&$-21$	&$11$	&$16$	&$7$	&$3$	&$-5$	&$-1$	&$1$	&$0$	&$-1$	&$0$	&$0$	&$-1$	&$1$	&$0$	&$0$	&$-1$	&$0$	&$0$	&$1$	&$1$	&$0$	&$0$	&$0$	&$0$	\\
$\chi_{19}$&$+$&$2024$	&$8$	&$24$	&$-1$	&$8$	&$8$	&$0$	&$0$	&$-1$	&$-1$	&$0$	&$1$	&$1$	&$0$	&$-1$	&$0$	&$-1$	&$0$	&$1$	&$1$	&$-1$	&$-1$	&$1$	&$1$	&$0$	&$0$	\\
$\chi_{20}$&$+$&$2277$	&$21$	&$-19$	&$0$	&$6$	&$-3$	&$1$	&$-3$	&$-3$	&$0$	&$2$	&$2$	&$2$	&$-1$	&$1$	&$0$	&$0$	&$0$	&$0$	&$0$	&$0$	&$0$	&$-1$	&$-1$	&$0$	&$0$	\\
$\chi_{21}$&$+$&$3312$	&$48$	&$16$	&$0$	&$-6$	&$0$	&$0$	&$0$	&$-3$	&$0$	&$-2$	&$1$	&$1$	&$0$	&$1$	&$1$	&$0$	&$0$	&$-1$	&$-1$	&$0$	&$0$	&$1$	&$1$	&$0$	&$0$	\\
$\chi_{22}$&$+$&$3520$	&$64$	&$0$	&$10$	&$-8$	&$0$	&$0$	&$0$	&$0$	&$-2$	&$0$	&$-1$	&$-1$	&$0$	&$0$	&$0$	&$0$	&$0$	&$1$	&$1$	&$0$	&$0$	&$-1$	&$-1$	&$1$	&$1$	\\
$\chi_{23}$&$+$&$5313$	&$49$	&$9$	&$-15$	&$0$	&$1$	&$-3$	&$-3$	&$3$	&$1$	&$0$	&$0$	&$0$	&$-1$	&$-1$	&$0$	&$1$	&$0$	&$0$	&$0$	&$0$	&$0$	&$0$	&$0$	&$0$	&$0$	\\
$\chi_{24}$&$+$&$5544$	&$-56$	&$24$	&$9$	&$0$	&$-8$	&$0$	&$0$	&$-1$	&$1$	&$0$	&$0$	&$0$	&$0$	&$-1$	&$0$	&$1$	&$0$	&$0$	&$0$	&$-1$	&$-1$	&$0$	&$0$	&$1$	&$1$	\\
$\chi_{25}$&$+$&$5796$	&$-28$	&$36$	&$-9$	&$0$	&$-4$	&$4$	&$0$	&$1$	&$-1$	&$0$	&$0$	&$0$	&$0$	&$1$	&$-1$	&$-1$	&$0$	&$0$	&$0$	&$1$	&$1$	&$0$	&$0$	&$0$	&$0$	\\
$\chi_{26}$&$+$&$10395$	&$-21$	&$-45$	&$0$	&$0$	&$3$	&$-1$	&$3$	&$0$	&$0$	&$0$	&$0$	&$0$	&$1$	&$0$	&$0$	&$0$	&$0$	&$0$	&$0$	&$0$	&$0$	&$0$	&$0$	&$-1$	&$-1$	\\\bottomrule
\end{tabular}
\end{small}
\end{center}
\end{sidewaystable}

\begin{sidewaystable}

\begin{center}
\caption{Character table of $G^{(3)}\simeq 2.M_{12}$}\label{tab:chars:irr:3}

\smallskip

\begin{tabular}{c@{ }|c|@{\;}r@{ }r@{ }r@{ }r@{ }r@{ }r@{ }r@{ }r@{ }r@{ }r@{ }r@{ }r@{ }r@{ }r@{ }r@{ }r@{ }r@{ }r@{ }r@{ }r@{ }r@{ }r@{ }r@{ }r@{ }r@{ }r}\toprule
$[g]$&FS&   	1A&   	2A&   	4A&   	2B&   	2C&   	3A&   	6A&   	3B&   	6B&   	4B&   	4C&   	5A&   	10A&   	12A&   	6C&   	6D&   	8A&   	8B&   	 8C&   8D&   	20A&   	20B&   	11A&   	22A&   	11B&   	22B\\ 
	\midrule
$[g^2]$&&	1A&			1A&			2A&			1A&			1A&			3A&			3A&			3B&			3B&			2B&			2B&			5A&			5A&			6B&			3A&			3A&			4B&			4B&			4C&			4C&			10A&		10A&		11B&		11B&		11A&		11A\\		
$[g^3]$& &  1A&   		2A&   		4A&   		2B&   		2C&   		1A&   		2A&   		1A&   		2A&   		4B& 	  		4C&   		5A&   		10A&   		4A&   		2B&   		2C&   		8A&   		8B&   	 	8C&  		8D&   		20A&   		20B&   		11A&   		22A&   		11B&   		22B\\
$[g^5]$& &  1A&   		2A&   		4A&   		2B&   		2C&   		3A&   		6A&   		3B&   		6B&   		4B& 	  		4C&   		1A&   		2A&   		12A&   		6C&   		6D&   		8B&   		8A&   	 	8D&  		8C&   		4A&   		4A&   		11A&   		22A&   		11B&   		22B\\
$[g^{11}]$&&   1A&   		2A&   		4A&   		2B&   		2C&   		3A&   		6A&   		3B&   		6B&   		4B& 	  		4C&   		5A&   		10A&   		12A&   		6C&   		6D&   		8A&   		8B&   	 	8C&  		8D&   		20B&   		20A&   		1A&   		2A&   		1A&   		2A\\
	\midrule
$\chi_{1}$&$+$&   $1$&   $1$&   $1$&   $1$&   $1$&   $1$&   $1$&   $1$&   $1$&   $
1$&   $1$&   $1$&   $1$&   $1$&   $1$&   $1$&   $1$&   $1$&   $1$&   $1$&   $
1$&   $1$&   $1$&   $1$&   $1$&   $1$\\
$\chi_{2}$&$+$&   $11$&   $11$&   $-1$&   $3$&   $3$&   $2$&   $2$&   $-1$&   $
-1$&   $-1$&   $3$&   $1$&   $1$&   $-1$&   $0$&   $0$&   $-1$&   $-1$&   $
1$&   $1$&   $-1$&   $-1$&   $0$&   $0$&   $0$&   $0$\\
$\chi_{3}$&$+$&   $11$&   $11$&   $-1$&   $3$&   $3$&   $2$&   $2$&   $-1$&   $
-1$&   $3$&   $-1$&   $1$&   $1$&   $-1$&   $0$&   $0$&   $1$&   $1$&   $
-1$&   $-1$&   $-1$&   $-1$&   $0$&   $0$&   $0$&   $0$\\
$\chi_{4}$&$\circ$&   $16$&   $16$&   $4$&   $0$&   $0$&   $-2$&   $-2$&   $1$&   $
1$&   $0$&   $0$&   $1$&   $1$&   $1$&   $0$&   $0$&   $0$&   $0$&   $0$&   $
0$&   $-1$&   $-1$&   $b_{11}$&   $
b_{11}$&   $\overline{b_{11}}$&   $\overline{b_{11}}$\\
$\chi_{5}$&$\circ$&   $16$&   $16$&   $4$&   $0$&   $0$&   $-2$&   $-2$&   $1$&   $
1$&   $0$&   $0$&   $1$&   $1$&   $1$&   $0$&   $0$&   $0$&   $0$&   $0$&   $
0$&   $-1$&   $-1$&   $\overline{b_{11}}$&   $
\overline{b_{11}}$&   $b_{11}$&   $b_{11}$\\
$\chi_{6}$&$+$&   $45$&   $45$&   $5$&   $-3$&   $-3$&   $0$&   $0$&   $3$&   $
3$&   $1$&   $1$&   $0$&   $0$&   $-1$&   $0$&   $0$&   $-1$&   $-1$&   $
-1$&   $-1$&   $0$&   $0$&   $1$&   $1$&   $1$&   $1$\\
$\chi_{7}$&$+$&   $54$&   $54$&   $6$&   $6$&   $6$&   $0$&   $0$&   $0$&   $
0$&   $2$&   $2$&   $-1$&   $-1$&   $0$&   $0$&   $0$&   $0$&   $0$&   $
0$&   $0$&   $1$&   $1$&   $-1$&   $-1$&   $-1$&   $-1$\\
$\chi_{8}$&$+$&   $55$&   $55$&   $-5$&   $7$&   $7$&   $1$&   $1$&   $1$&   $
1$&   $-1$&   $-1$&   $0$&   $0$&   $1$&   $1$&   $1$&   $-1$&   $-1$&   $
-1$&   $-1$&   $0$&   $0$&   $0$&   $0$&   $0$&   $0$\\
$\chi_{9}$&$+$&   $55$&   $55$&   $-5$&   $-1$&   $-1$&   $1$&   $1$&   $1$&   $
1$&   $3$&   $-1$&   $0$&   $0$&   $1$&   $-1$&   $-1$&   $-1$&   $-1$&   $
1$&   $1$&   $0$&   $0$&   $0$&   $0$&   $0$&   $0$\\
$\chi_{10}$&$+$&   $55$&   $55$&   $-5$&   $-1$&   $-1$&   $1$&   $1$&   $1$&   $
1$&   $-1$&   $3$&   $0$&   $0$&   $1$&   $-1$&   $-1$&   $1$&   $1$&   $
-1$&   $-1$&   $0$&   $0$&   $0$&   $0$&   $0$&   $0$\\
$\chi_{11}$&$+$&   $66$&   $66$&   $6$&   $2$&   $2$&   $3$&   $3$&   $0$&   $
0$&   $-2$&   $-2$&   $1$&   $1$&   $0$&   $-1$&   $-1$&   $0$&   $0$&   $
0$&   $0$&   $1$&   $1$&   $0$&   $0$&   $0$&   $0$\\
$\chi_{12}$&$+$&   $99$&   $99$&   $-1$&   $3$&   $3$&   $0$&   $0$&   $3$&   $
3$&   $-1$&   $-1$&   $-1$&   $-1$&   $-1$&   $0$&   $0$&   $1$&   $1$&   $
1$&   $1$&   $-1$&   $-1$&   $0$&   $0$&   $0$&   $0$\\
$\chi_{13}$&$+$&   $120$&   $120$&   $0$&   $-8$&   $-8$&   $3$&   $3$&   $0$&   $
0$&   $0$&   $0$&   $0$&   $0$&   $0$&   $1$&   $1$&   $0$&   $0$&   $0$&   $
0$&   $0$&   $0$&   $-1$&   $-1$&   $-1$&   $-1$\\
$\chi_{14}$&$+$&   $144$&   $144$&   $4$&   $0$&   $0$&   $0$&   $0$&   $-3$&   $
-3$&   $0$&   $0$&   $-1$&   $-1$&   $1$&   $0$&   $0$&   $0$&   $0$&   $
0$&   $0$&   $-1$&   $-1$&   $1$&   $1$&   $1$&   $1$\\
$\chi_{15}$&$+$&   $176$&   $176$&   $-4$&   $0$&   $0$&   $-4$&   $-4$&   $
-1$&   $-1$&   $0$&   $0$&   $1$&   $1$&   $-1$&   $0$&   $0$&   $0$&   $
0$&   $0$&   $0$&   $1$&   $1$&   $0$&   $0$&   $0$&   $0$\\
$\chi_{16}$&$\circ$&   $10$&   $-10$&   $0$&   $-2$&   $2$&   $1$&   $-1$&   $-2$&   $
2$&   $0$&   $0$&   $0$&   $0$&   $0$&   $1$&   $-1$&   $a_{2}$&   $
\overline{a_{2}}$&   $a_{2}$&   $\overline{a_{2}}$&   $0$&   $0$&   $-1$&   $
1$&   $-1$&   $1$\\
$\chi_{17}$&$\circ$&   $10$&   $-10$&   $0$&   $-2$&   $2$&   $1$&   $-1$&   $-2$&   $
2$&   $0$&   $0$&   $0$&   $0$&   $0$&   $1$&   $-1$&   $\overline{a_{2}}$&   $
a_{2}$&   $\overline{a_{2}}$&   $a_{2}$&   $0$&   $0$&   $-1$&   $
1$&   $-1$&   $1$\\
$\chi_{18}$&$+$&   $12$&   $-12$&   $0$&   $4$&   $-4$&   $3$&   $-3$&   $0$&   $
0$&   $0$&   $0$&   $2$&   $-2$&   $0$&   $1$&   $-1$&   $0$&   $0$&   $
0$&   $0$&   $0$&   $0$&   $1$&   $-1$&   $1$&   $-1$\\
$\chi_{19}$&$-$&   $32$&   $-32$&   $0$&   $0$&   $0$&   $-4$&   $4$&   $2$&   $
-2$&   $0$&   $0$&   $2$&   $-2$&   $0$&   $0$&   $0$&   $0$&   $0$&   $
0$&   $0$&   $0$&   $0$&   $-1$&   $1$&   $-1$&   $1$\\
$\chi_{20}$&$\circ$&   $44$&   $-44$&   $0$&   $4$&   $-4$&   $-1$&   $1$&   $2$&   $
-2$&   $0$&   $0$&   $-1$&   $1$&   $0$&   $1$&   $-1$&   $0$&   $0$&   $
0$&   $0$&   $a_{5}$&   $\overline{a_{5}}$&   $0$&   $0$&   $0$&   $0$\\
$\chi_{21}$&$\circ$&   $44$&   $-44$&   $0$&   $4$&   $-4$&   $-1$&   $1$&   $2$&   $
-2$&   $0$&   $0$&   $-1$&   $1$&   $0$&   $1$&   $-1$&   $0$&   $0$&   $
0$&   $0$&   $\overline{a_{5}}$&   $a_{5}$&   $0$&   $0$&   $0$&   $0$\\
$\chi_{22}$&$\circ$&   $110$&   $-110$&   $0$&   $-6$&   $6$&   $2$&   $-2$&   $
2$&   $-2$&   $0$&   $0$&   $0$&   $0$&   $0$&   $0$&   $0$&   $
a_{2}$&   $\overline{a_{2}}$&   $\overline{a_{2}}$&   $a_{2}$&   $0$&   $
0$&   $0$&   $0$&   $0$&   $0$\\
$\chi_{23}$&$\circ$&   $110$&   $-110$&   $0$&   $-6$&   $6$&   $2$&   $-2$&   $
2$&   $-2$&   $0$&   $0$&   $0$&   $0$&   $0$&   $0$&   $0$&   $
\overline{a_{2}}$&   $a_{2}$&   $a_{2}$&   $\overline{a_{2}}$&   $0$&   $
0$&   $0$&   $0$&   $0$&   $0$\\
$\chi_{24}$&$+$&   $120$&   $-120$&   $0$&   $8$&   $-8$&   $3$&   $-3$&   $
0$&   $0$&   $0$&   $0$&   $0$&   $0$&   $0$&   $-1$&   $1$&   $0$&   $0$&   $
0$&   $0$&   $0$&   $0$&   $-1$&   $1$&   $-1$&   $1$\\
$\chi_{25}$&$\circ$&   $160$&   $-160$&   $0$&   $0$&   $0$&   $-2$&   $2$&   $
-2$&   $2$&   $0$&   $0$&   $0$&   $0$&   $0$&   $0$&   $0$&   $0$&   $0$&   $
0$&   $0$&   $0$&   $0$&   $-b_{11}$&   $
b_{11}$&   $-\overline{b_{11}}$&   $\overline{b_{11}}$\\
$\chi_{26}$&$\circ$&   $160$&   $-160$&   $0$&   $0$&   $0$&   $-2$&   $2$&   $
-2$&   $2$&   $0$&   $0$&   $0$&   $0$&   $0$&   $0$&   $0$&   $0$&   $0$&   $
0$&   $0$&   $0$&   $0$&   $-\overline{b_{11}}$&   $
\overline{b_{11}}$&   $-b_{11}$&   $b_{11}$\\ \bottomrule
\end{tabular}
\end{center}

\end{sidewaystable}

\begin{table}
\begin{center}
\caption{Character table of ${G}^{(4)}\simeq 2.\AGL_3(2)$}\label{tab:chars:irr:4}

\smallskip
\begin{small}
\begin{tabular}{c|c|rrrrrrrrrrrrrrrrrrrrrrrrrrr} \toprule
$[g]$&FS&	   1A&   2A&   2B&   4A&   4B&   2C&   3A&   6A&   6B&   6C&   8A&   4C&   7A&   14A&   7B&   14B
	\\ \midrule
$[g^2]$&&   	1A&   1A&   	1A&   	2A&			2B&			1A&   	3A&   	3A&   		3A&   	3A&   	4A&   	2C&   	7A&		7A&   7B&   7B\\ 
$[g^3]$&&   	1A&   2A&   	2B&   	4A&			4B&			2C&   	1A&   	2A&   		2B&   	2B&   	8A&   	4C&   	7B&   	14B&   7A&   14A\\ 
$[g^7]$&&   	1A&   2A&   	2B&   	4A& 			4B&  		2C&   	3A&   	6A&   		6B&   	6C&   	8A&   	4C&   	1A&   	2A&   1A&   2A\\ 
	\midrule	
$\chi_{1}$&$+$&   $1$&   $1$&   $1$&   $1$&   $1$&   $1$&   $1$&   $1$&   $1$&   $
1$&   $1$&   $1$&   $1$&   $1$&   $1$&   $1$\\
$\chi_{2}$&$\circ$&   $3$&   $3$&   $3$&   $-1$&   $-1$&   $-1$&   $0$&   $0$&   $
0$&   $0$&   $1$&   $1$&   $b_{7}$&   $b_{7}$&   $\overline{b_{7}}$&   $\overline{b_{7}}$\\
$\chi_{3}$&$\circ$&   $3$&   $3$&   $3$&   $-1$&   $-1$&   $-1$&   $0$&   $0$&   $
0$&   $0$&   $1$&   $1$&   $\overline{b_{7}}$&   $\overline{b_{7}}$&   $b_{7}$&   $b_{7}$\\
$\chi_{4}$&$+$&   $6$&   $6$&   $6$&   $2$&   $2$&   $2$&   $0$&   $0$&   $0$&   $
0$&   $0$&   $0$&   $-1$&   $-1$&   $-1$&   $-1$\\
$\chi_{5}$&$+$&   $7$&   $7$&   $7$&   $-1$&   $-1$&   $-1$&   $1$&   $1$&   $
1$&   $1$&   $-1$&   $-1$&   $0$&   $0$&   $0$&   $0$\\
$\chi_{6}$&$+$&   $8$&   $8$&   $8$&   $0$&   $0$&   $0$&   $-1$&   $-1$&   $
-1$&   $-1$&   $0$&   $0$&   $1$&   $1$&   $1$&   $1$\\
$\chi_{7}$&$+$&   $7$&   $7$&   $-1$&   $3$&   $-1$&   $-1$&   $1$&   $1$&   $
-1$&   $-1$&   $1$&   $-1$&   $0$&   $0$&   $0$&   $0$\\
$\chi_{8}$&$+$&   $7$&   $7$&   $-1$&   $-1$&   $-1$&   $3$&   $1$&   $1$&   $
-1$&   $-1$&   $-1$&   $1$&   $0$&   $0$&   $0$&   $0$\\
$\chi_{9}$&$+$&   $14$&   $14$&   $-2$&   $2$&   $-2$&   $2$&   $-1$&   $-1$&   $
1$&   $1$&   $0$&   $0$&   $0$&   $0$&   $0$&   $0$\\
$\chi_{10}$&$+$&   $21$&   $21$&   $-3$&   $1$&   $1$&   $-3$&   $0$&   $0$&   $
0$&   $0$&   $-1$&   $1$&   $0$&   $0$&   $0$&   $0$\\
$\chi_{11}$&$+$&   $21$&   $21$&   $-3$&   $-3$&   $1$&   $1$&   $0$&   $0$&   $
0$&   $0$&   $1$&   $-1$&   $0$&   $0$&   $0$&   $0$\\
$\chi_{12}$&$+$&   $8$&   $-8$&   $0$&   $0$&   $0$&   $0$&   $2$&   $-2$&   $
0$&   $0$&   $0$&   $0$&   $1$&   $-1$&   $1$&   $-1$\\
$\chi_{13}$&$\circ$&   $8$&   $-8$&   $0$&   $0$&   $0$&   $0$&   $-1$&   $1$&   $a_{3}$&   $\overline{a_{3}}$&   $0$&   $0$&   $1$&   $-1$&   $1$&   $-1$\\
$\chi_{14}$&$\circ$&   $8$&   $-8$&   $0$&   $0$&   $0$&   $0$&   $-1$&   $1$&   $\overline{a_{3}}$&   $a_{3}$&   $0$&   $0$&   $1$&   $-1$&   $1$&   $-1$\\
$\chi_{15}$&$\circ$&   $24$&   $-24$&   $0$&   $0$&   $0$&   $0$&   $0$&   $0$&   $0$&   $0$&   $0$&   $0$&   $\overline{b_{7}}$&   $-\overline{b_{7}}$&   $b_{7}$&   $-b_{7}$\\
$\chi_{16}$&$\circ$&   $24$&   $-24$&   $0$&   $0$&   $0$&   $0$&   $0$&   $0$&   $0$&   $0$&   $0$&   $0$&   $b_{7}$&   $-b_{7}$&   $\overline{b_{7}}$&   $-\overline{b_{7}}$\\\bottomrule
\end{tabular}
\end{small}
\end{center}

\begin{center}
\caption{Character table of ${G}^{(5)}\simeq \GL_2(5)/2$}\label{tab:chars:irr:5}

\smallskip
\begin{small}
\begin{tabular}{c|c|rrrrrrrrrrrrrrrrrrrrrrrrrrr}\toprule
$[g]$&FS&   1A&   2A&   2B&   2C&   3A&   6A&   5A&   10A&   4A&   4B&   4C&   4D&   12A&   12B\\
	\midrule
$[g^2]$&&	1A&	1A&	1A&	1A&	3A&	3A&	5A&	5A&		2A&	2A&	2C&	2C&	6A&	6A\\
$[g^3]$&&	1A&	2A&	2B&	2C&	1A&	2A&	5A&	10A&	4B&	4A&	4D&	4C&	4B&	4A\\
$[g^5]$&&	1A&	2A&	2B&	2C&	3A&	6A&	1A&	2A&		4A&	4B&	4C&	4D&	12A&	12B\\
	\midrule
${\chi}_{1}$&$+$&   $1$&   $1$&   $1$&   $1$&   $1$&   $1$&   $1$&   $1$&   $1$&   $1$&   $1$&   $1$&   $1$&   $1$\\
${\chi}_{2}$&$+$&   $1$&   $1$&   $1$&   $1$&   $1$&   $1$&   $1$&   $1$&   $
-1$&   $-1$&   $-1$&   $-1$&   $-1$&   $-1$\\
${\chi}_{3}$&$+$&   $4$&   $4$&   $0$&   $0$&   $1$&   $1$&   $-1$&   $-1$&   $
2$&   $2$&   $0$&   $0$&   $-1$&   $-1$\\
${\chi}_{4}$&$+$&   $4$&   $4$&   $0$&   $0$&   $1$&   $1$&   $-1$&   $-1$&   $
-2$&   $-2$&   $0$&   $0$&   $1$&   $1$\\
${\chi}_{5}$&$+$&   $5$&   $5$&   $1$&   $1$&   $-1$&   $-1$&   $0$&   $0$&   $
1$&   $1$&   $-1$&   $-1$&   $1$&   $1$\\
${\chi}_{6}$&$+$&   $5$&   $5$&   $1$&   $1$&   $-1$&   $-1$&   $0$&   $0$&   $-1$&   $-1$&   $1$&   $1$&   $-1$&   $-1$\\
${\chi}_{7}$&$+$&   $6$&   $6$&   $-2$&   $-2$&   $0$&   $0$&   $1$&   $1$&   $
0$&   $0$&   $0$&   $0$&   $0$&   $0$\\
${\chi}_{8}$&$\circ$&   $1$&   $-1$&   $1$&   $-1$&   $1$&   $-1$&   $1$&   $-1$&   $
a_1$&   $-a_1$&   $a_1$&   $-a_1$&   $a_1$&   $-a_1$\\
${\chi}_{9}$&$\circ$&   $1$&   $-1$&   $1$&   $-1$&   $1$&   $-1$&   $1$&   $-1$&   $
-a_1$&   $a_1$&   $-a_1$&   $a_1$&   $-a_1$&   $a_1$\\
${\chi}_{10}$&$\circ$&   $4$&   $-4$&   $0$&   $0$&   $1$&   $-1$&   $-1$&   $1$&   $
2a_1$&   $-2a_1$&   $0$&   $0$&   $-a_1$&   $a_1$\\
${\chi}_{11}$&$\circ$&   $4$&   $-4$&   $0$&   $0$&   $1$&   $-1$&   $-1$&   $1$&   $
-2a_1$&   $2a_1$&   $0$&   $0$&   $a_1$&   $-a_1$\\
${\chi}_{12}$&$\circ$&   $5$&   $-5$&   $1$&   $-1$&   $-1$&   $1$&   $0$&   $0$&   $
a_1$&   $-a_1$&   $-a_1$&   $a_1$&   $a_1$&   $-a_1$\\
${\chi}_{13}$&$\circ$&   $5$&   $-5$&   $1$&   $-1$&   $-1$&   $1$&   $0$&   $0$&   $
-a_1$&   $a_1$&   $a_1$&   $-a_1$&   $-a_1$&   $a_1$\\
${\chi}_{14}$&$+$&   $6$&   $-6$&   $-2$&   $2$&   $0$&   $0$&   $1$&   $-1$&   $
0$&   $0$&   $0$&   $0$&   $0$&   $0$\\\bottomrule
\end{tabular}
\end{small}
\end{center}
\end{table}

\begin{table}
\begin{center}
\caption{Character table of ${G}^{(7)}\simeq\SL_2(3)$}\label{tab:chars:irr:7}
\smallskip
\begin{tabular}{c|c|rrrrrrr}\toprule
$[g]$	&FS&   1A&   2A&   4A&   3A&   6A&   3B&   6B\\
	\midrule
$[g^2]$ 	&&1A	&	1A&		2A&		3B&		3A&		3A&		3B\\
$[g^3]$	&&1A	&	2A&		4A&		1A&		2A&		1A&		2A\\
	\midrule
$\chi_1$&$+$&   $1$&   $1$&   $1$&   $1$&   $1$&   $1$&   $1$\\
$\chi_2$&$\circ$&   $1$&   $1$&   $1$&   ${b_3}$&   $\overline{b_3}$&   $\overline{b_3}$&   ${b_3}$\\
$\chi_3$&$\circ$&   $1$&   $1$&   $1$&   $\overline{b_3}$&   ${b_3}$&   ${b_3}$&   $\overline{b_3}$\\
$\chi_4$&$+$&   $3$&   $3$&   $-1$&   $0$&   $0$&   $0$&   $0$\\
$\chi_5$&$-$&   $2$&   $-2$&   $0$&   $-1$&   $1$&   $-1$&   $1$\\
$\chi_6$&$\circ$&   $2$&   $-2$&   $0$&   $-\overline{b_3}$&   ${b_3}$&   $-{b_3}$&   $\overline{b_3}$\\
$\chi_7$&$\circ$&   $2$&   $-2$&   $0$&   $-{b_3}$&   $\overline{b_3}$&   $-\overline{b_3}$&   ${b_3}$\\\bottomrule
\end{tabular}
\end{center}
\end{table}

\begin{table}
\begin{center}
\caption{Character table of ${G}^{(13)}\simeq 4$}\label{tab:chars:irr:13}
\smallskip
\begin{tabular}{c|c|rrrr}\toprule
$[g]$&FS&   1A&   2A&   4A&   4B\\
	\midrule
$[g^2]$&&	1A&	1A&	2A&	2A\\
	\midrule
$\chi_1$&$+$&   $1$&   $1$&   $1$&   $1$\\
$\chi_2$&$+$&   $1$&   $1$&   $-1$&   $-1$\\
$\chi_3$&$\circ$&   $1$&   $-1$&   $a_1$&   $\overline{a_1}$\\
$\chi_4$&$\circ$&   $1$&   $-1$&   $\overline{a_1}$&   $a_1$\\\bottomrule
\end{tabular}
\end{center}
\end{table}

\clearpage

\begin{sidewaystable}
\subsection{Euler Characters}\label{sec:chars:eul}

The tables in this section describe the Frame shapes $\Pi^{(\ll)}_{g}$ and twisted Euler characters $\chi^{(\ll)}_{g}$ attached to each group $G^{(\ll)}$ via the signed permutation representations given in \S\ref{sec:grps:spcnst}. The rows labelled $\bar{\Pi}^{(\ll)}_{g}$ and $\bar{\chi}^{(\ll)}_g$ describe the corresponding data for the (unsigned) permutation representations. According to the discussion of \S\ref{sec:mckay:aut} the Frame shapes $\Pi^{(\ll)}_{g}$ and $\bar{\Pi}^{(\ll)}_{g}$ (or even just the $\Pi^{(\ll)}_g$) can be used to define symbols $n_g|h_g$ which encode the automorphy of the vector-valued mock modular form $H^{(\ll)}_{g}$; these symbols are given in the rows labelled $\G_g$. We write $n_g$ here as a shorthand for $n_g|1$.

\begin{center}
\caption{Twisted Euler characters and Frame shapes at $\ll=2$}\label{tab:chars:eul:2}
\smallskip
\begin{tabular}{c@{ }|@{ }c@{ }c@{ }c@{ }c@{ }c@{ }c@{ }c@{ }c@{ }c@{ }c@{ }c@{ }c@{ }c@{ }c@{ }c}
\toprule
$[g]$	&1A	&2A	&2B	&3A	&3B	&4A	&4B	&4C	&5A	&6A	&6B	\\
	\midrule
$\G_g$&$1$&$2$&${2|2}$&$3$&$3|3$&$4|2$&$4$&${4|4}$&$5$&$6$&$6|6$&\\
	\midrule
$\chi^{(2)}_{g}$&   $24$&   $8$&   		$0$&   		$6$&   		$0$&   		$0$&   		$4$&  		$0$&   		$4$&   		$2$&   		$0$&  \\
	\midrule
$\Pi^{(2)}_{g}$&$1^{24}$&	$1^{8}2^8$&	$2^{12}$&		$1^63^6$&	$3^8$&	$2^44^4$&	$1^42^24^4$&	$4^6$&		$1^45^4$&		$1^22^23^26^2$&	$6^4$\\\midrule 
\midrule
$[g]$	& 7AB	&8A	&10A	&11A&12A	&12B	&14AB	&15AB	&21AB	&23AB	\\
	\midrule
$\G_g$&$7$&$8$&$10|2$&$11$&$12|2$&$12|12$&$14$&$15$&$21|3$&$23$\\
	\midrule
$\chi^{(2)}_{g}$&   		$3$&   		$2$&   		$0$&   		$2$&   		$0$&   		$0$& 		$1$&   		$1$&   		$0$&   		$1$\\
	\midrule
$\Pi^{(2)}_{g}$&	$1^37^3$&	$1^22^14^18^2$&	$2^210^2$&		$1^211^2$&$2^14^16^112^1$&$12^2$&	$1^12^17^114^1$&	$1^13^15^115^1$&$3^121^1$&	$1^123^1$\\\bottomrule
\end{tabular}
\smallskip
\end{center}

We have $\chi^{(2)}_g=\chi_1(g)+\chi_2(g)$ in the notation of Table \ref{tab:chars:irr:2}.
\end{sidewaystable}

\begin{sidewaystable}
\begin{center}
\caption{Twisted Euler characters and Frame shapes at $\ll=3$}\label{tab:chars:eul:3}
\smallskip
\begin{tabular}{c@{ }|@{\;}c@{\,}c@{\,}c@{\,}c@{\,}c@{\,}c@{\,}c@{\,}c@{\,}c@{\,}c@{\,}c@{\,}c@{\,}c@{\,}c@{\,}c@{\,}c@{\,}c@{\,}c@{\,}c@{\,}c@{\,}c@{\,}c@{\,}c@{\,}c@{\,}c@{\,}c}\toprule
$[g]$&   	1A&   		2A&   		4A&   		2B&   		2C&   		3A&   		6A&   		3B&   		6B&   		4B& 	  		4C&   		5A&   		10A&   		12A&   		6C&   		6D&   		8AB&   	 	8CD&   		20AB&   		11AB&   		22AB\\ 
	\midrule
$\G_g$&$1$&$1|4$&${2|8}$&$2$&$2|2$&$3$&$3|4$&${3|3}$&${3|12}$&$4|2$&$4$&$5$&$5|4$&$6|24$&$6$&$6|2$&$8|4$&$8$&${10|8}$&$11$&$11|4$\\
	\midrule
$\bar{\chi}^{(3)}_{g}$&   $12$&   $12$&   		$0$&   		$4$&   		$4$&   		$3$&   		$3$&  		$0$&   		$0$&   		$0$&   		$4$&   		$2$&   		$2$&   		$0$&   		$1$&   		$1$&   		$0$&   		$2$&   		$0$&   		$1$&   		$1$\\
$\chi^{(3)}_{g}$&   $12$&   $-12$&   		$0$&   		$4$&   		$-4$&   		$3$&   		$-3$&   		$0$&   		$0$&   		$0$&  		$0$&   		$2$&   		$-2$&   		$0$&   		$1$&   		$-1$&   		$0$&   		$0$&   		$0$&   		$1$&   		$-1$\\
	\midrule
$\bar{\Pi}^{(3)}_{g}$&$1^{12}$&	$1^{12}$&	$2^6$&		$1^42^4$&	$1^42^4$&	$1^33^3$&	$1^33^3$&	$3^4$&		$3^4$&		$2^24^2$&	$1^44^2$&	$1^25^2$&	$1^25^2$&	$6^2$&		$1^12^13^16^1$&$1^12^13^16^1$&$4^18^1$&	$1^22^18^1$&	$2^110^1$&$1^111^1$&	$1^111^1$\\
$\Pi^{(3)}_{g}$&$1^{12}$&	$\tfrac{2^{12}}{1^{12}}$&$\tfrac{4^6}{2^6}$&	$1^42^4$&	$\frac{2^8}{1^4}$&	$1^33^3$&	$\tfrac{2^36^3}{1^33^3}$&$3^4$&	$\tfrac{6^4}{3^4}$&	$2^24^2$&	$2^24^2$&	$1^25^2$&$\tfrac{2^210^2}{1^25^2}$&	$\tfrac{12^2}{6^2}$&		$1^12^13^16^1$&$\frac{2^26^2}{1^13^1}$&$4^18^1$&$4^18^1$&$\frac{4^120^1}{2^110^1}$&$1^111^1$&$\tfrac{2^122^1}{1^111^1}$\smallskip\\\bottomrule
\end{tabular}
\smallskip
\end{center}

We have $\chi^{(3)}_{g}=\chi_{18}(g)$ and $\bar{\chi}^{(3)}_{g}=\chi_{1}(g)+\chi_2(g)$ in the notation of Table \ref{tab:chars:irr:3}.

\begin{center}
\caption{Twisted Euler characters and Frame shapes at $\ll=4$}\label{tab:chars:eul:4}
\smallskip
\begin{tabular}{c@{\, }|@{\;}c@{\, }c@{\, }c@{\, }c@{\, }c@{\, }c@{\, }c@{\, }c@{\, }c@{\, }c@{\, }c@{\, }c@{\, }c}\toprule
$[g]$&   		1A&   2A&   	2B&   	4A&			4B&			2C&   	3A&   	6A&   		6BC&   	8A&   	4C&   	7AB&   	14AB\\ 
	\midrule
$\G_g$&$1$& $1|2$&	$2|2$&	$2|4$&			${4|{4}}$&	$2$& 	$3$& 	$3|2$&		$6|2$&	${4|{8}}$&		$4$&  	$7$&	$7|2$\\	
	\midrule
$\bar{\chi}^{(4)}_g$&   $8$&$8$&	$0$& 	$0$& 		$0$&		$4$&  	$2$& 	$2$&  		$0$& 	$0$& 	$2$& 	$1$& 	$1$\\
$\chi^{(4)}_g$&   $8$&$-8$&	$0$&	$0$& 		$0$&		$0$&  	$2$& 	$-2$& 		$0$& 	$0$& 	$0$& 	$1$& 	$-1$\\
	\midrule
$\bar{\Pi}^{(4)}_g$&	$1^8$&$1^8$&$2^4$&	$2^4$&		$4^2$&		$1^42^2$&$1^23^2$&$1^23^2$&	$2^16^1$&$4^2$&	$1^22^14^1$&$1^17^1$&$1^17^1$\\
$\Pi^{(4)}_g$&	$1^8$&$\tfrac{2^8}{1^8}$&$2^4$&$\tfrac{4^4}{2^4}$&	$4^2$&		$2^4$&	$1^23^2$&$\tfrac{2^26^2}{1^23^2}$&	$2^16^1$&$\tfrac{8^2}{4^2}$&$4^2$	&$1^17^1$&$\frac{2^114^1}{1^17^1}$\smallskip\\\bottomrule
\end{tabular}
\smallskip
\end{center}

\setstretch{1.4}

We have $\chi^{(4)}_{g}=\chi_{12}$ and $\bar{\chi}^{(4)}_{g}=\chi_1(g)+\chi_8(g)$ in the notation of Table \ref{tab:chars:irr:4}.

\end{sidewaystable}

\begin{table}
\begin{center}
\caption{Twisted Euler characters and Frame shapes at $\ll=5$}\label{tab:chars:eul:5}
\smallskip
\begin{tabular}{c@{\, }|@{\;}c@{\, }c@{\, }c@{\, }c@{\, }c@{\, }c@{\, }c@{\, }c@{\, }c@{\, }c@{\, }c@{\, }c@{\, }c}\toprule
$[g]$&   		1A&		2A&   	2B&   	2C&			3A&			6A&   	5A&   	10A&   		4AB&   	4CD&	12AB\\ 
	\midrule
$\G_g$&		$1$&	$1|4$&	$2|2$&	$2$&		$3|3$&		$3|12$&	$5$&	$5|4$&		$2|8$&		$4$&	$6|24$	\\	
	\midrule
$\bar{\chi}^{(5)}_{g}$&   $6$&	$6$&	$2$& 	$2$& 		$0$&		$0$&  	$1$& 	$1$&  		$0$& 	$2$& 	$0$ 	\\
$\chi^{(5)}_{g}$&   $6$&	$-6$&	$-2$&	$2$& 		$0$&		$0$&  	$1$& 	$-1$& 		$0$&	$0$& 	$0$ 	\\
	\midrule
$\bar{\Pi}^{(5)}_{g}$&	$1^6$&	$1^6$&	$1^22^2$&$1^22^2$&	$3^2$&		$3^2$&	$1^15^1$&$1^15^1$&	$2^3$&	$1^24^1$&$6^1$\\
$\Pi^{(5)}_{g}$&	$1^6$&	$\frac{2^6}{1^6}$&$\frac{2^4}{1^2}$&$1^22^2$&$3^2$&$\frac{6^2}{3^2}$&$1^15^1$&$\frac{2^110^1}{1^15^1}$&$\frac{4^3}{2^3}$&$2^14^1$&$\frac{12^1}{6^1}$\smallskip\\\bottomrule
\end{tabular}
\smallskip
\end{center}

We have $\chi^{(5)}_{g}=\chi_{14}(g)$ and $\bar{\chi}^{(5)}_{g}=\chi_1(g)+\chi_6(g)$ in the notation of Table \ref{tab:chars:irr:5}.
\end{table}

\begin{table}
\begin{center}
\caption{\label{tab:FrmG7fp}
Twisted Euler characters and Frame shapes at $\ll=7$}\label{tab:chars:eul:7}\smallskip
\begin{tabular}{c|ccccc}\toprule
$[g]$&   1A&   2A&   4A&   3AB&   6AB\\ 
	\midrule
$\G_g$&		$1$&	$1|4$&	$2|8$&	$3$&	$3|4$\\
	\midrule
$\bar{\chi}^{(7)}_{g}$&	4&	4&	0&	1&	1\\
$\chi^{(7)}_{g}$&	4&	-4&	0&	1&	-1\\
	\midrule
$\bar{\Pi}^{(7)}_{g}$&	$1^4$&	$1^4$&	$2^2$&	$1^13^1$&	$1^13^1$\\
$\Pi^{(7)}_{g}$&	$1^4$&	$\tfrac{2^4}{1^4}$&	$\tfrac{4^2}{2^2}$&	$1^13^1$&	$\tfrac{2^16^1}{1^13^1}$\smallskip\\
\bottomrule
\end{tabular}
\end{center}

We have $\chi^{(7)}_{g}=\chi_6(g)+\chi_7(g)$ in the notation of Table \ref{tab:chars:irr:7}.
\end{table}

\begin{table}
\begin{center}
\caption{Twisted Euler characters and Frame shapes at $\ll=13$}\label{tab:chars:eul:13}\smallskip
\begin{tabular}{r|rrr}\toprule
	$[g]$&	1A&	2A&	4AB\\
		\midrule
$\G_g$&		$1$&$1|4$&${2|8}$\\	
		\midrule
	$\bar{\chi}^{(13)}_{g}$	&2&2&0\\
	$\chi^{(13)}_{g}$	&2&-2&0\\
		\midrule
	$\bar{\Pi}^{(13)}_{g}$	&$1^2$&$1^2$&$2^1$\\
	$\Pi^{(13)}_{g}$&$1^2$&$\tfrac{2^2}{1^2}$&$\frac{4^1}{2^1}$\smallskip\\\bottomrule
\end{tabular}
\end{center}

We have $\chi^{(13)}_g=\chi_3(g)+\chi_4(g)$ in the notation of Table \ref{tab:chars:irr:13}.
\end{table}

\clearpage

\section{Coefficients}\label{sec:coeffs}

In this section we furnish tables of Fourier coefficients of small degree for the vector-valued mock modular forms $H^{(\ll)}_{g}$ that we attach to the conjugacy classes of the groups $G^{(\ll)}$ for $\ll\in \LL$. For each $\ll$ and $0<r<\ll$ we give a table that displays the coefficients of $H^{(\ll)}_{g,r}$ for ($g$ ranging over a set of representatives for) each conjugacy class $[g]$ in $G^{(\ll)}$. The first row of each table labels the conjugacy classes, and the first column labels exponents of $q$ (or rather $q^{1/4\ll}$), so that for the table captioned $H^{(\ll)}_{g,r}$ (for some $\ll\in\LL$ and $0<r<\ll$) the entry in the row labelled $d$ and the column labelled $nZ$ is the coefficient of $q^{d/4\ll}$ in the Fourier expansion of $H^{(\ll)}_{g,r}$ for $[g]= nZ$. Occasionally the functions $H^{(\ll)}_{g}$ and $H^{(\ll)}_{g'}$ coincide for non-conjugate $g$ and $g'$ and when this happens we condense information into a single column, writing $7AB$ in Table \ref{tab:coeffs:2_1}, for example, to indicate that the entries in that column are Fourier coefficients for both $H^{(2)}_{7A}$ and $H^{(2)}_{7B}$. 

\clearpage

\begin{sidewaystable}
\subsection{Lambency Two}
\begin{small}
\centering
\caption{McKay--Thompson series $H^{(2)}_{g,1}$}\label{tab:coeffs:2_1}\smallskip
\begin{tabular}{c@{ }|@{\;}r@{ }r@{ }r@{ }r@{ }r@{ }r@{ }r@{ }r@{ }r@{ }r@{ }r@{ }r@{ }r@{ }r@{ }r@{ }r@{ }r@{ }r@{ }r@{ }r@{ }r@{ }r@{ }r@{ }r@{ }r@{ }r}\toprule
$[g]$	&1A	&2A	&2B	&3A	&3B	&4A	&4B	&4C	&5A	&6A	&6B	&7AB	&8A	&10A	&11A	&12A	&12B	&14AB	&15AB	&21AB	&23AB	\\
\midrule
$\G_g$&$1$	&$2$	&$2|2$	&$3$	&$3|3$	&$4|2$	&$4$	&$4|4$	&$5$	&$6$	&$6|6$	&$7$	&$8$	&$10|2$	&$11$	&$12|2$	&$12|12$	&$14$	&$15$	&$21|3$	&$23$	\\
\midrule
-1	&-2	&-2	&-2	&-2	&-2	&-2	&-2	&-2	&-2	&-2	&-2	&-2	&-2	&-2	&-2	&-2	&-2	&-2	&-2	&-2	&-2	\\
7	&90	&-6	&10	&0	&6	&-6	&2	&2	&0	&0	&-2	&-1	&-2	&0	&2	&0	&2	&1	&0	&-1	&-2	\\
15	&462	&14	&-18	&-6	&0	&-2	&-2	&6	&2	&2	&0	&0	&-2	&2	&0	&-2	&0	&0	&-1	&0	&2	\\
23	&1540	&-28	&20	&10	&-14	&4	&-4	&-4	&0	&2	&2	&0	&0	&0	&0	&-2	&2	&0	&0	&0	&-1	\\
31	&4554	&42	&-38	&0	&12	&-6	&2	&-6	&-6	&0	&4	&4	&-2	&2	&0	&0	&0	&0	&0	&-2	&0	\\
39	&11592	&-56	&72	&-18	&0	&-8	&8	&0	&2	&-2	&0	&0	&0	&2	&-2	&-2	&0	&0	&2	&0	&0	\\
47	&27830	&86	&-90	&20	&-16	&6	&-2	&6	&0	&-4	&0	&-2	&2	&0	&0	&0	&0	&2	&0	&-2	&0	\\
55	&61686	&-138	&118	&0	&30	&6	&-10	&-2	&6	&0	&-2	&2	&-2	&-2	&-2	&0	&-2	&2	&0	&2	&0	\\
63	&131100	&188	&-180	&-30	&0	&-4	&4	&-12	&0	&2	&0	&-3	&0	&0	&2	&2	&0	&-1	&0	&0	&0	\\
71	&265650	&-238	&258	&42	&-42	&-14	&10	&10	&-10	&2	&6	&0	&-2	&-2	&0	&-2	&-2	&0	&2	&0	&0	\\
79	&521136	&336	&-352	&0	&42	&0	&-8	&16	&6	&0	&2	&0	&-4	&-2	&0	&0	&-2	&0	&0	&0	&2	\\
87	&988770	&-478	&450	&-60	&0	&18	&-14	&-6	&0	&-4	&0	&6	&2	&0	&2	&0	&0	&-2	&0	&0	&0	\\
95	&1830248	&616	&-600	&62	&-70	&-8	&8	&-16	&8	&-2	&-6	&0	&0	&0	&2	&-2	&2	&0	&2	&0	&0	\\
103	&3303630	&-786	&830	&0	&84	&-18	&22	&6	&0	&0	&-4	&-6	&2	&0	&0	&0	&0	&-2	&0	&0	&2	\\
111	&5844762	&1050	&-1062	&-90	&0	&10	&-6	&18	&-18	&6	&0	&0	&2	&-2	&0	&-2	&0	&0	&0	&0	&2	\\
119	&10139734	&-1386	&1334	&118	&-110	&22	&-26	&-10	&4	&6	&2	&-4	&-2	&4	&0	&-2	&2	&0	&-2	&2	&0	\\
127	&17301060	&1764	&-1740	&0	&126	&-12	&12	&-28	&0	&0	&6	&0	&0	&0	&-4	&0	&2	&0	&0	&0	&0	\\
135	&29051484	&-2212	&2268	&-156	&0	&-36	&28	&12	&14	&-4	&0	&0	&-4	&-2	&0	&0	&0	&0	&-1	&0	&0	\\
143	&48106430	&2814	&-2850	&170	&-166	&14	&-18	&38	&0	&-6	&-6	&8	&-2	&0	&-2	&2	&2	&0	&0	&2	&-2	\\
151	&78599556	&-3612	&3540	&0	&210	&36	&-36	&-20	&-24	&0	&-6	&0	&0	&0	&2	&0	&-2	&0	&0	&0	&0	\\
159	&126894174	&4510	&-4482	&-228	&0	&-18	&14	&-42	&14	&4	&0	&-6	&-2	&-2	&0	&0	&0	&2	&2	&0	&0	\\
167	&202537080	&-5544	&5640	&270	&-282	&-40	&48	&16	&0	&6	&6	&4	&4	&0	&-2	&2	&-2	&0	&0	&-2	&0	\\
175	&319927608	&6936	&-6968	&0	&300	&24	&-16	&48	&18	&0	&4	&-7	&4	&2	&0	&0	&0	&-1	&0	&-1	&0	\\
183	&500376870	&-8666	&8550	&-360	&0	&54	&-58	&-18	&0	&-8	&0	&0	&-2	&0	&4	&0	&0	&0	&0	&0	&2	\\
191	&775492564	&10612	&-10556	&400	&-392	&-28	&28	&-60	&-36	&-8	&-8	&0	&0	&4	&0	&-4	&0	&0	&0	&0	&0	\\
199	&1191453912	&-12936	&13064	&0	&462	&-72	&64	&32	&12	&0	&-10	&12	&-4	&4	&0	&0	&2	&0	&0	&0	&0	\\
207	&1815754710	&15862	&-15930	&-510	&0	&22	&-34	&78	&0	&10	&0	&0	&-6	&0	&0	&-2	&0	&0	&0	&0	&-1	\\
215	&2745870180	&-19420	&19268	&600	&-600	&84	&-76	&-36	&30	&8	&8	&-10	&4	&-2	&-2	&0	&0	&-2	&0	&2	&0	\\
223	&4122417420	&23532	&-23460	&0	&660	&-36	&36	&-84	&0	&0	&12	&2	&0	&0	&0	&0	&0	&-2	&0	&2	&0	\\
231	&6146311620	&-28348	&28548	&-762	&0	&-92	&100	&36	&-50	&-10	&0	&-6	&4	&-2	&-2	&-2	&0	&2	&-2	&0	&0	\\
239	&9104078592	&34272	&-34352	&828	&-840	&48	&-40	&96	&22	&-12	&-8	&0	&4	&-2	&4	&0	&0	&0	&-2	&0	&0	\\
247	&13401053820	&-41412	&41180	&0	&966	&108	&-116	&-44	&0	&0	&-10	&0	&-4	&0	&0	&0	&-2	&0	&0	&0	&-2	\\\bottomrule
\end{tabular}
\end{small}
\end{sidewaystable}

\clearpage

\begin{sidewaystable}
\subsection{Lambency Three}
\vspace{-16pt}
\centering
\caption{McKay--Thompson series $H^{(3)}_{g,1}$}\smallskip
\begin{tabular}{c@{ }|@{\;}r@{ }r@{ }r@{ }r@{ }r@{ }r@{ }r@{ }r@{ }r@{ }r@{ }r@{ }r@{ }r@{ }r@{ }r@{ }r@{ }r@{ }r@{ }r@{ }r@{ }r@{ }r@{ }r@{ }r@{ }r@{ }r}\toprule
$[g]$&1A	&2A	&4A	&2B	&2C	&3A	&6A	&3B	&6B	&4B	&4C	&5A	&10A	&12A	&6C	&6D	&8AB	&8CD	&20AB	&11AB	&22AB	\\
	\midrule
$\G_g$&$1$&$1|4$&${2|8}$&$2$&$2|2$&$3$&$3|4$&${3|3}$&${3|12}$&$4|2$&$4$&$5$&$5|4$&$6|24$&$6$&$6|2$&$8|4$&$8$&${10|8}$&$11$&$11|4$\\
	\midrule
-1	&-2	&-2	&-2	&-2	&-2	&-2	&-2	&-2	&-2	&-2	&-2	&-2	&-2	&-2	&-2	&-2	&-2	&-2	&-2	&-2	&-2	\\
11	&32	&32	&8	&0	&0	&-4	&-4	&2	&2	&0	&0	&2	&2	&2	&0	&0	&0	&0	&-2	&-1	&-1	\\
23	&110	&110	&-10	&-2	&-2	&2	&2	&2	&2	&6	&-2	&0	&0	&2	&-2	&-2	&-2	&2	&0	&0	&0	\\
35	&288	&288	&8	&0	&0	&0	&0	&-6	&-6	&0	&0	&-2	&-2	&2	&0	&0	&0	&0	&-2	&2	&2	\\
47	&660	&660	&-20	&4	&4	&-6	&-6	&6	&6	&-4	&4	&0	&0	&-2	&-2	&-2	&4	&0	&0	&0	&0	\\
59	&1408	&1408	&32	&0	&0	&4	&4	&4	&4	&0	&0	&-2	&-2	&-4	&0	&0	&0	&0	&2	&0	&0	\\
71	&2794	&2794	&-30	&-6	&-6	&4	&4	&-8	&-8	&2	&-6	&4	&4	&0	&0	&0	&2	&-2	&0	&0	&0	\\
83	&5280	&5280	&40	&0	&0	&-12	&-12	&6	&6	&0	&0	&0	&0	&-2	&0	&0	&0	&0	&0	&0	&0	\\
95	&9638	&9638	&-58	&6	&6	&8	&8	&2	&2	&-10	&6	&-2	&-2	&2	&0	&0	&-2	&-2	&2	&2	&2	\\
107	&16960	&16960	&80	&0	&0	&4	&4	&-14	&-14	&0	&0	&0	&0	&2	&0	&0	&0	&0	&0	&-2	&-2	\\
119	&29018	&29018	&-102	&-6	&-6	&-16	&-16	&8	&8	&10	&-6	&-2	&-2	&0	&0	&0	&2	&2	&-2	&0	&0	\\
131	&48576	&48576	&112	&0	&0	&12	&12	&6	&6	&0	&0	&6	&6	&-2	&0	&0	&0	&0	&2	&0	&0	\\
143	&79530	&79530	&-150	&10	&10	&6	&6	&-24	&-24	&-6	&10	&0	&0	&0	&-2	&-2	&2	&2	&0	&0	&0	\\
155	&127776	&127776	&200	&0	&0	&-24	&-24	&18	&18	&0	&0	&-4	&-4	&2	&0	&0	&0	&0	&0	&0	&0	\\
167	&202050	&202050	&-230	&-14	&-14	&18	&18	&12	&12	&10	&-14	&0	&0	&4	&-2	&-2	&2	&-2	&0	&2	&2	\\
179	&314688	&314688	&272	&0	&0	&12	&12	&-30	&-30	&0	&0	&-2	&-2	&2	&0	&0	&0	&0	&2	&0	&0	\\
191	&483516	&483516	&-348	&12	&12	&-36	&-36	&24	&24	&-12	&12	&6	&6	&0	&0	&0	&-4	&0	&2	&0	&0	\\
203	&733920	&733920	&440	&0	&0	&24	&24	&12	&12	&0	&0	&0	&0	&-4	&0	&0	&0	&0	&0	&0	&0	\\
215	&1101364	&1101364	&-508	&-12	&-12	&16	&16	&-44	&-44	&20	&-12	&-6	&-6	&-4	&0	&0	&-4	&4	&2	&0	&0	\\
227	&1635680	&1635680	&600	&0	&0	&-52	&-52	&32	&32	&0	&0	&0	&0	&0	&0	&0	&0	&0	&0	&2	&2	\\
239	&2406116	&2406116	&-740	&20	&20	&38	&38	&20	&20	&-20	&20	&-4	&-4	&4	&2	&2	&4	&0	&0	&-2	&-2	\\
251	&3507680	&3507680	&888	&0	&0	&20	&20	&-64	&-64	&0	&0	&10	&10	&0	&0	&0	&0	&0	&-2	&0	&0	\\
263	&5071000	&5071000	&-1040	&-24	&-24	&-68	&-68	&46	&46	&16	&-24	&0	&0	&-2	&0	&0	&0	&-4	&0	&0	&0	\\
275	&7274464	&7274464	&1208	&0	&0	&52	&52	&28	&28	&0	&0	&-6	&-6	&-4	&0	&0	&0	&0	&-2	&-1	&-1	\\
287	&10359030	&10359030	&-1450	&22	&22	&30	&30	&-84	&-84	&-26	&22	&0	&0	&-4	&-2	&-2	&-2	&-2	&0	&0	&0	\\
299	&14650176	&14650176	&1744	&0	&0	&-96	&-96	&60	&60	&0	&0	&-4	&-4	&4	&0	&0	&0	&0	&4	&2	&2	\\
311	&20585334	&20585334	&-2018	&-26	&-26	&66	&66	&36	&36	&30	&-26	&14	&14	&4	&-2	&-2	&-2	&2	&2	&0	&0	\\
323	&28747840	&28747840	&2320	&0	&0	&40	&40	&-116	&-116	&0	&0	&0	&0	&4	&0	&0	&0	&0	&0	&0	&0	\\
335	&39914402	&39914402	&-2750	&34	&34	&-130	&-130	&86	&86	&-30	&34	&-8	&-8	&-2	&-2	&-2	&2	&2	&0	&0	&0	\\
347	&55114400	&55114400	&3240	&0	&0	&92	&92	&50	&50	&0	&0	&0	&0	&-6	&0	&0	&0	&0	&0	&0	&0	\\
359	&75704904	&75704904	&-3712	&-40	&-40	&54	&54	&-156	&-156	&32	&-40	&-6	&-6	&-4	&2	&2	&0	&-4	&-2	&0	&0	\\\bottomrule
\end{tabular}
\end{sidewaystable}

\begin{sidewaystable}
\centering
\caption{McKay--Thompson series $H^{(3)}_{g,2}$}\smallskip
\begin{tabular}{c@{ }|@{\;}r@{ }r@{ }r@{ }r@{ }r@{ }r@{ }r@{ }r@{ }r@{ }r@{ }r@{ }r@{ }r@{ }r@{ }r@{ }r@{ }r@{ }r@{ }r@{ }r@{ }r@{ }r@{ }r@{ }r@{ }r@{ }r}\toprule
$[g]$	&1A	&2A	&4A	&2B	&2C	&3A	&6A	&3B	&6B	&4B	&4C	&5A	&10A	&12A	&6C	&6D	&8AB	&8CD	&20AB	&11AB	&22AB	\\
	\midrule
$\G_g$&$1$&$1|4$&${2|8}$&$2$&$2|2$&$3$&$3|4$&${3|3}$&${3|12}$&$4|2$&$4$&$5$&$5|4$&$6|24$&$6$&$6|2$&$8|4$&$8$&${10|8}$&$11$&$11|4$\\
	\midrule
8	&20	&-20	&0	&-4	&4	&2	&-2	&-4	&4	&0	&0	&0	&0	&0	&2	&-2	&0	&0	&0	&-2	&2	\\
20	&88	&-88	&0	&8	&-8	&-2	&2	&4	&-4	&0	&0	&-2	&2	&0	&2	&-2	&0	&0	&0	&0	&0	\\
32	&220	&-220	&0	&-12	&12	&4	&-4	&4	&-4	&0	&0	&0	&0	&0	&0	&0	&0	&0	&0	&0	&0	\\
44	&560	&-560	&0	&16	&-16	&2	&-2	&-4	&4	&0	&0	&0	&0	&0	&-2	&2	&0	&0	&0	&-1	&1	\\
56	&1144	&-1144	&0	&-24	&24	&-8	&8	&4	&-4	&0	&0	&4	&-4	&0	&0	&0	&0	&0	&0	&0	&0	\\
68	&2400	&-2400	&0	&32	&-32	&6	&-6	&0	&0	&0	&0	&0	&0	&0	&2	&-2	&0	&0	&0	&2	&-2	\\
80	&4488	&-4488	&0	&-40	&40	&6	&-6	&-12	&12	&0	&0	&-2	&2	&0	&2	&-2	&0	&0	&0	&0	&0	\\
92	&8360	&-8360	&0	&56	&-56	&-10	&10	&8	&-8	&0	&0	&0	&0	&0	&2	&-2	&0	&0	&0	&0	&0	\\
104	&14696	&-14696	&0	&-72	&72	&8	&-8	&8	&-8	&0	&0	&-4	&4	&0	&0	&0	&0	&0	&0	&0	&0	\\
116	&25544	&-25544	&0	&88	&-88	&2	&-2	&-16	&16	&0	&0	&4	&-4	&0	&-2	&2	&0	&0	&0	&2	&-2	\\
128	&42660	&-42660	&0	&-116	&116	&-18	&18	&12	&-12	&0	&0	&0	&0	&0	&-2	&2	&0	&0	&0	&2	&-2	\\
140	&70576	&-70576	&0	&144	&-144	&16	&-16	&4	&-4	&0	&0	&-4	&4	&0	&0	&0	&0	&0	&0	&0	&0	\\
152	&113520	&-113520	&0	&-176	&176	&12	&-12	&-24	&24	&0	&0	&0	&0	&0	&4	&-4	&0	&0	&0	&0	&0	\\
164	&180640	&-180640	&0	&224	&-224	&-26	&26	&16	&-16	&0	&0	&0	&0	&0	&2	&-2	&0	&0	&0	&-2	&2	\\
176	&281808	&-281808	&0	&-272	&272	&18	&-18	&12	&-12	&0	&0	&8	&-8	&0	&-2	&2	&0	&0	&0	&-1	&1	\\
188	&435160	&-435160	&0	&328	&-328	&10	&-10	&-32	&32	&0	&0	&0	&0	&0	&-2	&2	&0	&0	&0	&0	&0	\\
200	&661476	&-661476	&0	&-404	&404	&-42	&42	&24	&-24	&0	&0	&-4	&4	&0	&-2	&2	&0	&0	&0	&2	&-2	\\
212	&996600	&-996600	&0	&488	&-488	&30	&-30	&12	&-12	&0	&0	&0	&0	&0	&2	&-2	&0	&0	&0	&0	&0	\\
224	&1482536	&-1482536	&0	&-584	&584	&20	&-20	&-52	&52	&0	&0	&-4	&4	&0	&4	&-4	&0	&0	&0	&0	&0	\\
236	&2187328	&-2187328	&0	&704	&-704	&-50	&50	&40	&-40	&0	&0	&8	&-8	&0	&2	&-2	&0	&0	&0	&0	&0	\\
248	&3193960	&-3193960	&0	&-840	&840	&40	&-40	&28	&-28	&0	&0	&0	&0	&0	&0	&0	&0	&0	&0	&0	&0	\\
260	&4629152	&-4629152	&0	&992	&-992	&20	&-20	&-64	&64	&0	&0	&-8	&8	&0	&-4	&4	&0	&0	&0	&0	&0	\\
272	&6650400	&-6650400	&0	&-1184	&1184	&-78	&78	&48	&-48	&0	&0	&0	&0	&0	&-2	&2	&0	&0	&0	&-2	&2	\\
284	&9490536	&-9490536	&0	&1400	&-1400	&54	&-54	&24	&-24	&0	&0	&-4	&4	&0	&2	&-2	&0	&0	&0	&0	&0	\\
296	&13441032	&-13441032	&0	&-1640	&1640	&36	&-36	&-96	&96	&0	&0	&12	&-12	&0	&4	&-4	&0	&0	&0	&0	&0	\\
308	&18920240	&-18920240	&0	&1936	&-1936	&-100	&100	&68	&-68	&0	&0	&0	&0	&0	&4	&-4	&0	&0	&0	&-2	&2	\\
320	&26457464	&-26457464	&0	&-2264	&2264	&74	&-74	&44	&-44	&0	&0	&-6	&6	&0	&-2	&2	&0	&0	&0	&0	&0	\\
332	&36792560	&-36792560	&0	&2640	&-2640	&38	&-38	&-124	&124	&0	&0	&0	&0	&0	&-6	&6	&0	&0	&0	&2	&-2	\\
344	&50865232	&-50865232	&0	&-3088	&3088	&-140	&140	&88	&-88	&0	&0	&-8	&8	&0	&-4	&4	&0	&0	&0	&0	&0	\\
356	&69966336	&-69966336	&0	&3584	&-3584	&102	&-102	&48	&-48	&0	&0	&16	&-16	&0	&2	&-2	&0	&0	&0	&0	&0		\\\bottomrule
\end{tabular}
\end{sidewaystable}

\begin{table}
\vspace{-16pt}
\subsection{Lambency Four}
\vspace{-16pt}
\centering
\caption{McKay--Thompson series $H^{(4)}_{g,1}$}\smallskip
\begin{tabular}{c@{ }|@{\;}r@{ }r@{ }r@{ }r@{ }r@{ }r@{ }r@{ }r@{ }r@{ }r@{ }r@{ }r@{ }r}\toprule
$[g]$	&1A	&2A	&2B	&4A	&4B	&2C	&3A	&6A	&6BC	&8A	&4C	&7AB	&14AB	\\
\midrule
$\G_g$&$1$& $1|2$&	$2|2$&	$2|4$&			${4|{4}}$&	$2$& 	$3$& 	$3|2$&		$6|2$&	${4|{8}}$&		$4$&  	$7$&	$7|2$\\	
	\midrule
-1	&-2		&-2		&-2	&-2	&-2	&-2	&-2	&-2	&-2	&-2&-2	&-2	&-2\\
15	&14	&14	&-2	&6	&-2	&-2	&2	&2	&-2	&2	&-2	&0	&0	\\
31	&42	&42	&-6	&-6	&2	&2	&0	&0	&0	&2	&-2	&0	&0	\\
47	&86	&86	&6	&6	&-2	&-2	&-4	&-4	&0	&-2	&2	&2	&2	\\
63	&188	&188	&-4	&-12	&-4	&4	&2	&2	&2	&0	&0	&-1	&-1	\\
79	&336	&336	&0	&16	&0	&-8	&0	&0	&0	&0	&-4	&0	&0	\\
95	&616	&616	&-8	&-16	&0	&8	&-2	&-2	&-2	&4	&0	&0	&0	\\
111	&1050	&1050	&10	&18	&2	&-6	&6	&6	&-2	&-2	&2	&0	&0	\\
127	&1764	&1764	&-12	&-28	&-4	&12	&0	&0	&0	&-4	&0	&0	&0	\\
143	&2814	&2814	&14	&38	&-2	&-18	&-6	&-6	&2	&2	&-2	&0	&0	\\
159	&4510	&4510	&-18	&-42	&6	&14	&4	&4	&0	&2	&-2	&2	&2	\\
175	&6936	&6936	&24	&48	&0	&-16	&0	&0	&0	&-4	&4	&-1	&-1	\\
191	&10612	&10612	&-28	&-60	&-4	&28	&-8	&-8	&-4	&-4	&0	&0	&0	\\
207	&15862	&15862	&22	&78	&-2	&-34	&10	&10	&-2	&2	&-6	&0	&0	\\
223	&23532	&23532	&-36	&-84	&4	&36	&0	&0	&0	&4	&0	&-2	&-2	\\
239	&34272	&34272	&48	&96	&0	&-40	&-12	&-12	&0	&0	&4	&0	&0	\\
255	&49618	&49618	&-46	&-126	&-6	&50	&10	&10	&2	&-6	&2	&2	&2	\\
271	&70758	&70758	&54	&150	&-2	&-66	&0	&0	&0	&6	&-6	&2	&2	\\
287	&100310	&100310	&-74	&-170	&6	&70	&-10	&-10	&-2	&6	&-2	&0	&0	\\
303	&140616	&140616	&88	&192	&0	&-72	&18	&18	&-2	&-4	&8	&0	&0	\\
319	&195888	&195888	&-96	&-232	&-8	&96	&0	&0	&0	&-4	&0	&0	&0	\\
335	&270296	&270296	&104	&272	&0	&-120	&-22	&-22	&2	&4	&-8	&-2	&-2	\\
351	&371070	&371070	&-130	&-306	&6	&126	&18	&18	&2	&6	&-2	&0	&0	\\
367	&505260	&505260	&156	&348	&4	&-140	&0	&0	&0	&-4	&8	&0	&0	\\
383	&684518	&684518	&-170	&-410	&-10	&174	&-22	&-22	&-2	&-10	&2	&2	&2	\\
399	&921142	&921142	&182	&486	&-2	&-202	&28	&28	&-4	&6	&-10	&-2	&-2	\\
415	&1233708	&1233708	&-228	&-540	&12	&220	&0	&0	&0	&8	&-4	&0	&0	\\
431	&1642592	&1642592	&272	&608	&0	&-248	&-34	&-34	&2	&-8	&12	&0	&0	\\
447	&2177684	&2177684	&-284	&-708	&-12	&292	&32	&32	&4	&-8	&4	&-2	&-2	\\
463	&2871918	&2871918	&318	&814	&-2	&-346	&0	&0	&0	&6	&-14	&0	&0	\\
479	&3772468	&3772468	&-380	&-908	&12	&380	&-38	&-38	&-2	&12	&0	&0	&0	\\
495	&4932580	&4932580	&436	&1020	&4	&-412	&46	&46	&-2	&-8	&12	&2	&2	\\
511	&6425466	&6425466	&-486	&-1174	&-14	&490	&0	&0	&0	&-14	&2	&-2	&-2	\\
527	&8335418	&8335418	&538	&1338	&-6	&-566	&-52	&-52	&4	&10	&-14	&0	&0	\\
543	&10776290	&10776290	&-622	&-1494	&18	&610	&50	&50	&2	&14	&-6	&0	&0	\\
559	&13879290	&13879290	&714	&1666	&2	&-678	&0	&0	&0	&-10	&18	&-2	&-2	\\
575	&17818766	&17818766	&-786	&-1898	&-18	&790	&-58	&-58	&-6	&-14	&2	&0	&0	\\
591	&22798188	&22798188	&860	&2148	&-4	&-900	&72	&72	&-4	&8	&-20	&0	&0		\\\bottomrule
\end{tabular}
\end{table}

\begin{table}
\centering
\caption{McKay--Thompson series $H^{(4)}_{g,2}$}\smallskip
\begin{tabular}{c@{ }|@{\;}r@{ }r@{ }r@{ }r@{ }r@{ }r@{ }r@{ }r@{ }r@{ }r@{ }r@{ }r@{ }r}\toprule
$[g]$	&1A	&2A	&2B	&4A	&4B	&2C	&3A	&6A	&6BC	&8A	&4C	&7AB	&14AB	\\
\midrule
$\G_g$&$1$& $1|2$&	$2|2$&	$2|4$&			${4|{4}}$&	$2$& 	$3$& 	$3|2$&		$6|2$&	${4|{8}}$&		$4$&  	$7$&	$7|2$\\	
\midrule
12	&16	&-16	&0	&0	&0	&0	&-2	&2	&0	&0	&0	&2	&-2	\\
28	&48	&-48	&0	&0	&0	&0	&0	&0	&0	&0	&0	&-1	&1	\\
44	&112	&-112	&0	&0	&0	&0	&4	&-4	&0	&0	&0	&0	&0	\\
60	&224	&-224	&0	&0	&0	&0	&-4	&4	&0	&0	&0	&0	&0	\\
76	&432	&-432	&0	&0	&0	&0	&0	&0	&0	&0	&0	&-2	&2	\\
92	&784	&-784	&0	&0	&0	&0	&4	&-4	&0	&0	&0	&0	&0	\\
108	&1344	&-1344	&0	&0	&0	&0	&-6	&6	&0	&0	&0	&0	&0	\\
124	&2256	&-2256	&0	&0	&0	&0	&0	&0	&0	&0	&0	&2	&-2	\\
140	&3680	&-3680	&0	&0	&0	&0	&8	&-8	&0	&0	&0	&-2	&2	\\
156	&5824	&-5824	&0	&0	&0	&0	&-8	&8	&0	&0	&0	&0	&0	\\
172	&9072	&-9072	&0	&0	&0	&0	&0	&0	&0	&0	&0	&0	&0	\\
188	&13872	&-13872	&0	&0	&0	&0	&12	&-12	&0	&0	&0	&-2	&2	\\
204	&20832	&-20832	&0	&0	&0	&0	&-12	&12	&0	&0	&0	&0	&0	\\
220	&30912	&-30912	&0	&0	&0	&0	&0	&0	&0	&0	&0	&0	&0	\\
236	&45264	&-45264	&0	&0	&0	&0	&12	&-12	&0	&0	&0	&2	&-2	\\
252	&65456	&-65456	&0	&0	&0	&0	&-16	&16	&0	&0	&0	&-1	&1	\\
268	&93744	&-93744	&0	&0	&0	&0	&0	&0	&0	&0	&0	&0	&0	\\
284	&132944	&-132944	&0	&0	&0	&0	&20	&-20	&0	&0	&0	&0	&0	\\
300	&186800	&-186800	&0	&0	&0	&0	&-22	&22	&0	&0	&0	&-2	&2	\\
316	&260400	&-260400	&0	&0	&0	&0	&0	&0	&0	&0	&0	&0	&0	\\
332	&360208	&-360208	&0	&0	&0	&0	&28	&-28	&0	&0	&0	&2	&-2	\\
348	&494624	&-494624	&0	&0	&0	&0	&-28	&28	&0	&0	&0	&4	&-4	\\
364	&674784	&-674784	&0	&0	&0	&0	&0	&0	&0	&0	&0	&-2	&2	\\
380	&914816	&-914816	&0	&0	&0	&0	&32	&-32	&0	&0	&0	&0	&0	\\
396	&1232784	&-1232784	&0	&0	&0	&0	&-36	&36	&0	&0	&0	&0	&0	\\
412	&1652208	&-1652208	&0	&0	&0	&0	&0	&0	&0	&0	&0	&-2	&2	\\
428	&2202704	&-2202704	&0	&0	&0	&0	&44	&-44	&0	&0	&0	&0	&0	\\
444	&2921856	&-2921856	&0	&0	&0	&0	&-48	&48	&0	&0	&0	&0	&0	\\
460	&3857760	&-3857760	&0	&0	&0	&0	&0	&0	&0	&0	&0	&4	&-4	\\
476	&5070560	&-5070560	&0	&0	&0	&0	&56	&-56	&0	&0	&0	&-2	&2	\\
492	&6636000	&-6636000	&0	&0	&0	&0	&-60	&60	&0	&0	&0	&0	&0	\\
508	&8649648	&-8649648	&0	&0	&0	&0	&0	&0	&0	&0	&0	&0	&0	\\
524	&11230448	&-11230448	&0	&0	&0	&0	&68	&-68	&0	&0	&0	&-2	&2	\\
540	&14526848	&-14526848	&0	&0	&0	&0	&-76	&76	&0	&0	&0	&0	&0	\\
556	&18724176	&-18724176	&0	&0	&0	&0	&0	&0	&0	&0	&0	&2	&-2	\\
572	&24051808	&-24051808	&0	&0	&0	&0	&88	&-88	&0	&0	&0	&4	&-4	\\
588	&30793712	&-30793712	&0	&0	&0	&0	&-94	&94	&0	&0	&0	&-2	&2	\\
604	&39301584	&-39301584	&0	&0	&0	&0	&0	&0	&0	&0	&0	&0	&0		\\\bottomrule
\end{tabular}
\end{table}

\begin{table}
\centering
\caption{McKay--Thompson series $H^{(4)}_{g,3}$}\smallskip
\begin{tabular}{c@{ }|@{\;}r@{ }r@{ }r@{ }r@{ }r@{ }r@{ }r@{ }r@{ }r@{ }r@{ }r@{ }r@{ }r}\toprule
$[g]$	&1A	&2A	&2B	&4A	&4B	&2C	&3A	&6A	&6BC	&8A	&4C	&7AB	&14AB	\\
\midrule
$\G_g$&$1$& $1|2$&	$2|2$&	$2|4$&			${4|{4}}$&	$2$& 	$3$& 	$3|2$&		$6|2$&	${4|{8}}$&		$4$&  	$7$&	$7|2$\\	
\midrule
7	&6	&6	&6	&-2	&-2	&-2	&0	&0	&0	&2	&2	&-1	&-1	\\
23	&28	&28	&-4	&4	&-4	&4	&-2	&-2	&2	&0	&0	&0	&0	\\
39	&56	&56	&8	&0	&0	&-8	&2	&2	&2	&-4	&0	&0	&0	\\
55	&138	&138	&-6	&2	&2	&10	&0	&0	&0	&-2	&2	&-2	&-2	\\
71	&238	&238	&14	&-10	&-2	&-10	&-2	&-2	&2	&2	&2	&0	&0	\\
87	&478	&478	&-18	&6	&-2	&14	&4	&4	&0	&2	&-2	&2	&2	\\
103	&786	&786	&18	&-6	&2	&-22	&0	&0	&0	&-2	&-2	&2	&2	\\
119	&1386	&1386	&-22	&10	&2	&26	&-6	&-6	&2	&2	&2	&0	&0	\\
135	&2212	&2212	&36	&-12	&-4	&-28	&4	&4	&0	&4	&4	&0	&0	\\
151	&3612	&3612	&-36	&20	&-4	&36	&0	&0	&0	&0	&0	&0	&0	\\
167	&5544	&5544	&40	&-16	&0	&-48	&-6	&-6	&-2	&-4	&-4	&0	&0	\\
183	&8666	&8666	&-54	&18	&2	&58	&8	&8	&0	&-2	&2	&0	&0	\\
199	&12936	&12936	&72	&-32	&0	&-64	&0	&0	&0	&4	&4	&0	&0	\\
215	&19420	&19420	&-84	&36	&-4	&76	&-8	&-8	&0	&0	&-4	&2	&2	\\
231	&28348	&28348	&92	&-36	&4	&-100	&10	&10	&2	&-4	&-4	&-2	&-2	\\
247	&41412	&41412	&-108	&44	&4	&116	&0	&0	&0	&0	&4	&0	&0	\\
263	&59178	&59178	&138	&-62	&-6	&-126	&-12	&-12	&0	&6	&6	&0	&0	\\
279	&84530	&84530	&-158	&66	&-6	&154	&14	&14	&-2	&2	&-2	&-2	&-2	\\
295	&118692	&118692	&180	&-68	&4	&-188	&0	&0	&0	&-8	&-4	&0	&0	\\
311	&166320	&166320	&-208	&88	&8	&216	&-18	&-18	&2	&-4	&4	&0	&0	\\
327	&230092	&230092	&252	&-108	&-4	&-244	&16	&16	&0	&8	&4	&2	&2	\\
343	&317274	&317274	&-294	&122	&-6	&282	&0	&0	&0	&2	&-6	&-1	&-1	\\
359	&432964	&432964	&324	&-132	&4	&-340	&-20	&-20	&0	&-8	&-8	&0	&0	\\
375	&588966	&588966	&-378	&150	&6	&390	&24	&24	&0	&-2	&6	&0	&0	\\
391	&794178	&794178	&450	&-190	&-6	&-430	&0	&0	&0	&10	&10	&0	&0	\\
407	&1067220	&1067220	&-508	&220	&-12	&500	&-30	&-30	&2	&0	&-4	&0	&0	\\
423	&1423884	&1423884	&572	&-228	&4	&-588	&30	&30	&2	&-12	&-8	&0	&0	\\
439	&1893138	&1893138	&-654	&266	&10	&666	&0	&0	&0	&-2	&6	&2	&2	\\
455	&2501434	&2501434	&762	&-326	&-6	&-742	&-32	&-32	&0	&10	&10	&-2	&-2	\\
471	&3294256	&3294256	&-864	&360	&-8	&848	&40	&40	&0	&4	&-8	&0	&0	\\
487	&4314912	&4314912	&960	&-392	&8	&-984	&0	&0	&0	&-12	&-12	&0	&0	\\
503	&5633596	&5633596	&-1092	&452	&12	&1108	&-44	&-44	&0	&0	&8	&-4	&-4	\\
519	&7320670	&7320670	&1262	&-522	&-10	&-1234	&46	&46	&2	&18	&14	&0	&0	\\
535	&9483336	&9483336	&-1416	&592	&-16	&1400	&0	&0	&0	&4	&-8	&2	&2	\\
551	&12233330	&12233330	&1570	&-646	&10	&-1598	&-58	&-58	&-2	&-18	&-14	&4	&4	\\
567	&15734606	&15734606	&-1778	&726	&14	&1798	&62	&62	&-2	&-6	&10	&-1	&-1	\\
583	&20161302	&20161302	&2022	&-850	&-10	&-1994	&0	&0	&0	&18	&14	&0	&0	\\
599	&25761288	&25761288	&-2264	&944	&-16	&2240	&-72	&-72	&3	&4	&-12	&0	&0		\\\bottomrule
\end{tabular}
\end{table}

\begin{table}
\vspace{-16pt}
\subsection{Lambency Five}
\vspace{-16pt}
\centering
\caption{\label{tab:coeffs:5_1}McKay--Thompson series $H^{(5)}_{g,1}$}\smallskip
\begin{tabular}{r|rrrrrrrrrrr}\toprule
$[g]$	&1A	&2A	&2B	&2C	&3A	&6A	&5A	&10A	&4AB	&4CD	&12AB	\\
	\midrule
$\G_g$&		$1$&	$1|4$&	$2|2$&	$2$&		$3|3$&		$3|12$&	$5$&	$5|4$&		$2|8$&		$4$&	$6|24$	\\	
	\midrule
-1	&-2	&-2	&-2	&-2	&-2	&-2	&-2	&-2	&-2	&-2	&-2	\\
19	&8	&8	&0	&0	&2	&2	&-2	&-2	&4	&0	&-2	\\
39	&18	&18	&2	&2	&0	&0	&-2	&-2	&-6	&2	&0	\\
59	&40	&40	&0	&0	&-2	&-2	&0	&0	&4	&0	&-2	\\
79	&70	&70	&-2	&-2	&4	&4	&0	&0	&-6	&-2	&0	\\
99	&120	&120	&0	&0	&0	&0	&0	&0	&12	&0	&0	\\
119	&208	&208	&0	&0	&-2	&-2	&-2	&-2	&-8	&0	&-2	\\
139	&328	&328	&0	&0	&4	&4	&-2	&-2	&12	&0	&0	\\
159	&510	&510	&-2	&-2	&0	&0	&0	&0	&-18	&-2	&0	\\
179	&792	&792	&0	&0	&-6	&-6	&2	&2	&20	&0	&2	\\
199	&1180	&1180	&4	&4	&4	&4	&0	&0	&-24	&4	&0	\\
219	&1728	&1728	&0	&0	&0	&0	&-2	&-2	&24	&0	&0	\\
239	&2518	&2518	&-2	&-2	&-8	&-8	&-2	&-2	&-30	&-2	&0	\\
259	&3600	&3600	&0	&0	&6	&6	&0	&0	&40	&0	&-2	\\
279	&5082	&5082	&2	&2	&0	&0	&2	&2	&-42	&2	&0	\\
299	&7120	&7120	&0	&0	&-8	&-8	&0	&0	&48	&0	&0	\\
319	&9838	&9838	&-2	&-2	&10	&10	&-2	&-2	&-58	&-2	&2	\\
339	&13488	&13488	&0	&0	&0	&0	&-2	&-2	&72	&0	&0	\\
359	&18380	&18380	&4	&4	&-10	&-10	&0	&0	&-80	&4	&-2	\\
379	&24792	&24792	&0	&0	&12	&12	&2	&2	&84	&0	&0	\\
399	&33210	&33210	&-6	&-6	&0	&0	&0	&0	&-102	&-6	&0	\\
419	&44248	&44248	&0	&0	&-14	&-14	&-2	&-2	&116	&0	&2	\\
439	&58538	&58538	&2	&2	&14	&14	&-2	&-2	&-130	&2	&2	\\
459	&76992	&76992	&0	&0	&0	&0	&2	&2	&144	&0	&0	\\
479	&100772	&100772	&-4	&-4	&-16	&-16	&2	&2	&-168	&-4	&0	\\
499	&131160	&131160	&0	&0	&18	&18	&0	&0	&196	&0	&-2	\\
519	&169896	&169896	&8	&8	&0	&0	&-4	&-4	&-216	&8	&0	\\
539	&219128	&219128	&0	&0	&-22	&-22	&-2	&-2	&236	&0	&2	\\
559	&281322	&281322	&-6	&-6	&24	&24	&2	&2	&-270	&-6	&0	\\
579	&359712	&359712	&0	&0	&0	&0	&2	&2	&312	&0	&0	\\
599	&458220	&458220	&4	&4	&-24	&-24	&0	&0	&-336	&4	&0	\\
619	&581416	&581416	&0	&0	&28	&28	&-4	&-4	&372	&0	&0	\\
639	&735138	&735138	&-6	&-6	&0	&0	&-2	&-2	&-426	&-6	&0	\\
659	&926472	&926472	&0	&0	&-30	&-30	&2	&2	&476	&0	&2	\\
679	&1163674	&1163674	&10	&10	&34	&34	&4	&4	&-526	&10	&2	\\
699	&1457040	&1457040	&0	&0	&0	&0	&0	&0	&576	&0	&0	\\
719	&1819056	&1819056	&-8	&-8	&-42	&-42	&-4	&-4	&-644	&-8	&-2	\\
739	&2264376	&2264376	&0	&0	&42	&42	&-4	&-4	&724	&0	&-2		\\\bottomrule
\end{tabular}
\end{table}

\begin{table}
\centering
\caption{\label{tab:coeffs:5_2}McKay--Thompson series $H^{(5)}_{g,2}$}\smallskip
\begin{tabular}{r|rrrrrrrrrrr}\toprule
$[g]$	&1A	&2A	&2B	&2C	&3A	&6A	&5A	&10A	&4AB	&4CD	&12AB	\\
	\midrule
$\G_g$&		$1$&	$1|4$&	$2|2$&	$2$&		$3|3$&		$3|12$&	$5$&	$5|4$&		$2|8$&		$4$&	$6|24$	\\	
	\midrule
16	&10	&-10	&2	&-2	&-2	&2	&0	&0	&0	&0	&0	\\
36	&30	&-30	&-2	&2	&0	&0	&0	&0	&0	&0	&0	\\
56	&52	&-52	&4	&-4	&4	&-4	&2	&-2	&0	&0	&0	\\
76	&108	&-108	&-4	&4	&0	&0	&-2	&2	&0	&0	&0	\\
96	&180	&-180	&4	&-4	&0	&0	&0	&0	&0	&0	&0	\\
116	&312	&-312	&-8	&8	&0	&0	&2	&-2	&0	&0	&0	\\
136	&488	&-488	&8	&-8	&-4	&4	&-2	&2	&0	&0	&0	\\
156	&792	&-792	&-8	&8	&0	&0	&2	&-2	&0	&0	&0	\\
176	&1180	&-1180	&12	&-12	&4	&-4	&0	&0	&0	&0	&0	\\
196	&1810	&-1810	&-14	&14	&-2	&2	&0	&0	&0	&0	&0	\\
216	&2640	&-2640	&16	&-16	&0	&0	&0	&0	&0	&0	&0	\\
236	&3868	&-3868	&-20	&20	&4	&-4	&-2	&2	&0	&0	&0	\\
256	&5502	&-5502	&22	&-22	&-6	&6	&2	&-2	&0	&0	&0	\\
276	&7848	&-7848	&-24	&24	&0	&0	&-2	&2	&0	&0	&0	\\
296	&10912	&-10912	&32	&-32	&4	&-4	&2	&-2	&0	&0	&0	\\
316	&15212	&-15212	&-36	&36	&-4	&4	&2	&-2	&0	&0	&0	\\
336	&20808	&-20808	&40	&-40	&0	&0	&-2	&2	&0	&0	&0	\\
356	&28432	&-28432	&-48	&48	&4	&-4	&2	&-2	&0	&0	&0	\\
376	&38308	&-38308	&52	&-52	&-8	&8	&-2	&2	&0	&0	&0	\\
396	&51540	&-51540	&-60	&60	&0	&0	&0	&0	&0	&0	&0	\\
416	&68520	&-68520	&72	&-72	&12	&-12	&0	&0	&0	&0	&0	\\
436	&90928	&-90928	&-80	&80	&-8	&8	&-2	&2	&0	&0	&0	\\
456	&119544	&-119544	&88	&-88	&0	&0	&4	&-4	&0	&0	&0	\\
476	&156728	&-156728	&-104	&104	&8	&-8	&-2	&2	&0	&0	&0	\\
496	&203940	&-203940	&116	&-116	&-12	&12	&0	&0	&0	&0	&0	\\
516	&264672	&-264672	&-128	&128	&0	&0	&2	&-2	&0	&0	&0	\\
536	&341188	&-341188	&148	&-148	&16	&-16	&-2	&2	&0	&0	&0	\\
556	&438732	&-438732	&-164	&164	&-12	&12	&2	&-2	&0	&0	&0	\\
576	&560958	&-560958	&182	&-182	&0	&0	&-2	&2	&0	&0	&0	\\
596	&715312	&-715312	&-208	&208	&16	&-16	&2	&-2	&0	&0	&0	\\
616	&907720	&-907720	&232	&-232	&-20	&20	&0	&0	&0	&0	&0	\\
636	&1148928	&-1148928	&-256	&256	&0	&0	&-2	&2	&0	&0	&0	\\
656	&1447904	&-1447904	&288	&-288	&20	&-20	&4	&-4	&0	&0	&0	\\
676	&1820226	&-1820226	&-318	&318	&-18	&18	&-4	&4	&0	&0	&0	\\
696	&2279520	&-2279520	&352	&-352	&0	&0	&0	&0	&0	&0	&0	\\
716	&2847812	&-2847812	&-396	&396	&20	&-20	&2	&-2	&0	&0	&0	\\
736	&3545636	&-3545636	&436	&-436	&-28	&28	&-4	&4	&0	&0	&0	\\
756	&4404384	&-4404384	&-480	&480	&0	&0	&4	&-4	&0	&0	&0		\\\bottomrule
\end{tabular}
\end{table}

\begin{table}
\centering
\caption{\label{tab:coeffs:5_3}McKay--Thompson series $H^{(5)}_{g,3}$}\smallskip
\begin{tabular}{r|rrrrrrrrrrr}\toprule
$[g]$	&1A	&2A	&2B	&2C	&3A	&6A	&5A	&10A	&4AB	&4CD	&12AB	\\
	\midrule
$\G_g$&		$1$&	$1|4$&	$2|2$&	$2$&		$3|3$&		$3|12$&	$5$&	$5|4$&		$2|8$&		$4$&	$6|24$	\\	
	\midrule
11	&8	&8	&0	&0	&2	&2	&-2	&-2	&-4	&0	&2	\\
31	&22	&22	&-2	&-2	&-2	&-2	&2	&2	&2	&-2	&2	\\
51	&48	&48	&0	&0	&0	&0	&-2	&-2	&0	&0	&0	\\
71	&90	&90	&2	&2	&0	&0	&0	&0	&6	&2	&0	\\
91	&160	&160	&0	&0	&-2	&-2	&0	&0	&-8	&0	&-2	\\
111	&270	&270	&-2	&-2	&0	&0	&0	&0	&6	&-2	&0	\\
131	&440	&440	&0	&0	&2	&2	&0	&0	&-4	&0	&2	\\
151	&700	&700	&4	&4	&-2	&-2	&0	&0	&8	&4	&2	\\
171	&1080	&1080	&0	&0	&0	&0	&0	&0	&-12	&0	&0	\\
191	&1620	&1620	&-4	&-4	&6	&6	&0	&0	&16	&-4	&-2	\\
211	&2408	&2408	&0	&0	&-4	&-4	&-2	&-2	&-12	&0	&0	\\
231	&3522	&3522	&2	&2	&0	&0	&2	&2	&18	&2	&0	\\
251	&5048	&5048	&0	&0	&2	&2	&-2	&-2	&-28	&0	&2	\\
271	&7172	&7172	&-4	&-4	&-4	&-4	&2	&2	&24	&-4	&0	\\
291	&10080	&10080	&0	&0	&0	&0	&0	&0	&-24	&0	&0	\\
311	&13998	&13998	&6	&6	&6	&6	&-2	&-2	&34	&6	&-2	\\
331	&19272	&19272	&0	&0	&-6	&-6	&2	&2	&-44	&0	&-2	\\
351	&26298	&26298	&-6	&-6	&0	&0	&-2	&-2	&42	&-6	&0	\\
371	&35600	&35600	&0	&0	&8	&8	&0	&0	&-48	&0	&0	\\
391	&47862	&47862	&6	&6	&-6	&-6	&2	&2	&62	&6	&2	\\
411	&63888	&63888	&0	&0	&0	&0	&-2	&-2	&-72	&0	&0	\\
431	&84722	&84722	&-6	&-6	&8	&8	&2	&2	&78	&-6	&0	\\
451	&111728	&111728	&0	&0	&-10	&-10	&-2	&-2	&-80	&0	&-2	\\
471	&146520	&146520	&8	&8	&0	&0	&0	&0	&96	&8	&0	\\
491	&191080	&191080	&0	&0	&10	&10	&0	&0	&-124	&0	&2	\\
511	&248008	&248008	&-8	&-8	&-14	&-14	&-2	&-2	&128	&-8	&2	\\
531	&320424	&320424	&0	&0	&0	&0	&4	&4	&-132	&0	&0	\\
551	&412088	&412088	&8	&8	&14	&14	&-2	&-2	&160	&8	&-2	\\
571	&527800	&527800	&0	&0	&-14	&-14	&0	&0	&-188	&0	&-2	\\
591	&673302	&673302	&-10	&-10	&0	&0	&2	&2	&198	&-10	&0	\\
611	&855616	&855616	&0	&0	&16	&16	&-4	&-4	&-216	&0	&0	\\
631	&1083444	&1083444	&12	&12	&-18	&-18	&4	&4	&248	&12	&2	\\
651	&1367136	&1367136	&0	&0	&0	&0	&-4	&-4	&-288	&0	&0	\\
671	&1719362	&1719362	&-14	&-14	&20	&20	&2	&2	&314	&-14	&-4	\\
691	&2155592	&2155592	&0	&0	&-22	&-22	&2	&2	&-332	&0	&-2	\\
711	&2694276	&2694276	&12	&12	&0	&0	&-4	&-4	&384	&12	&0	\\
731	&3357664	&3357664	&0	&0	&28	&28	&4	&4	&-440	&0	&4	\\
751	&4172746	&4172746	&-14	&-14	&-26	&-26	&-4	&-4	&470	&-14	&2		\\\bottomrule
\end{tabular}
\end{table}

\begin{table}
\centering
\caption{\label{tab:coeffs:5_4}McKay--Thompson series $H^{(5)}_{g,4}$}\smallskip
\begin{tabular}{r|rrrrrrrrrrr}\toprule
$[g]$	&1A	&2A	&2B	&2C	&3A	&6A	&5A	&10A	&4AB	&4CD	&12AB	\\
	\midrule
$\G_g$&		$1$&	$1|4$&	$2|2$&	$2$&		$3|3$&		$3|12$&	$5$&	$5|4$&		$2|8$&		$4$&	$6|24$	\\	
	\midrule
4	&2	&-2	&2	&-2	&2	&-2	&2	&-2	&0	&0	&0	\\
24	&12	&-12	&-4	&4	&0	&0	&2	&-2	&0	&0	&0	\\
44	&20	&-20	&4	&-4	&-4	&4	&0	&0	&0	&0	&0	\\
64	&50	&-50	&-6	&6	&2	&-2	&0	&0	&0	&0	&0	\\
84	&72	&-72	&8	&-8	&0	&0	&2	&-2	&0	&0	&0	\\
104	&152	&-152	&-8	&8	&-4	&4	&2	&-2	&0	&0	&0	\\
124	&220	&-220	&12	&-12	&4	&-4	&0	&0	&0	&0	&0	\\
144	&378	&-378	&-14	&14	&0	&0	&-2	&2	&0	&0	&0	\\
164	&560	&-560	&16	&-16	&-4	&4	&0	&0	&0	&0	&0	\\
184	&892	&-892	&-20	&20	&4	&-4	&2	&-2	&0	&0	&0	\\
204	&1272	&-1272	&24	&-24	&0	&0	&2	&-2	&0	&0	&0	\\
224	&1940	&-1940	&-28	&28	&-4	&4	&0	&0	&0	&0	&0	\\
244	&2720	&-2720	&32	&-32	&8	&-8	&0	&0	&0	&0	&0	\\
264	&3960	&-3960	&-40	&40	&0	&0	&0	&0	&0	&0	&0	\\
284	&5500	&-5500	&44	&-44	&-8	&8	&0	&0	&0	&0	&0	\\
304	&7772	&-7772	&-52	&52	&8	&-8	&2	&-2	&0	&0	&0	\\
324	&10590	&-10590	&62	&-62	&0	&0	&0	&0	&0	&0	&0	\\
344	&14668	&-14668	&-68	&68	&-8	&8	&-2	&2	&0	&0	&0	\\
364	&19728	&-19728	&80	&-80	&12	&-12	&-2	&2	&0	&0	&0	\\
384	&26772	&-26772	&-92	&92	&0	&0	&2	&-2	&0	&0	&0	\\
404	&35624	&-35624	&104	&-104	&-16	&16	&4	&-4	&0	&0	&0	\\
424	&47592	&-47592	&-120	&120	&12	&-12	&2	&-2	&0	&0	&0	\\
444	&62568	&-62568	&136	&-136	&0	&0	&-2	&2	&0	&0	&0	\\
464	&82568	&-82568	&-152	&152	&-16	&16	&-2	&2	&0	&0	&0	\\
484	&107502	&-107502	&174	&-174	&18	&-18	&2	&-2	&0	&0	&0	\\
504	&140172	&-140172	&-196	&196	&0	&0	&2	&-2	&0	&0	&0	\\
524	&180940	&-180940	&220	&-220	&-20	&20	&0	&0	&0	&0	&0	\\
544	&233576	&-233576	&-248	&248	&20	&-20	&-4	&4	&0	&0	&0	\\
564	&298968	&-298968	&280	&-280	&0	&0	&-2	&2	&0	&0	&0	\\
584	&382632	&-382632	&-312	&312	&-24	&24	&2	&-2	&0	&0	&0	\\
604	&486124	&-486124	&348	&-348	&28	&-28	&4	&-4	&0	&0	&0	\\
624	&617112	&-617112	&-392	&392	&0	&0	&2	&-2	&0	&0	&0	\\
644	&778768	&-778768	&432	&-432	&-32	&32	&-2	&2	&0	&0	&0	\\
664	&981548	&-981548	&-484	&484	&32	&-32	&-2	&2	&0	&0	&0	\\
684	&1230732	&-1230732	&540	&-540	&0	&0	&2	&-2	&0	&0	&0	\\
704	&1541244	&-1541244	&-596	&596	&-36	&36	&4	&-4	&0	&0	&0	\\
724	&1921240	&-1921240	&664	&-664	&40	&-40	&0	&0	&0	&0	&0	\\
744	&2391456	&-2391456	&-736	&736	&0	&0	&-4	&4	&0	&0	&0	\\
\bottomrule
\end{tabular}
\end{table}

\clearpage

\begin{table}
\vspace{-16pt}
\subsection{Lambency Seven}
\vspace{-22pt}
\begin{minipage}[t]{0.49\linewidth}
\centering
\caption{\label{tab:coeffs:7_1}$H^{(7)}_{g,1}$}\smallskip
\begin{tabular}{r|rrrrrrrrrrr}\toprule
$[g]$	&1A	&2A	&4A	&3AB	&6AB	\\
\midrule
$\G_g$&		$1$&	$1|4$&	$2|8$&	$3$&	$3|4$\\
\midrule
-1	&-2	&-2	&-2	&-2	&-2	\\
27	&4	&4	&4	&1	&1	\\
55	&6	&6	&-2	&0	&0	\\
83	&10	&10	&2	&-2	&-2	\\
111	&20	&20	&-4	&2	&2	\\
139	&30	&30	&6	&0	&0	\\
167	&42	&42	&-6	&0	&0	\\
195	&68	&68	&4	&2	&2	\\
223	&96	&96	&-8	&0	&0	\\
251	&130	&130	&10	&-2	&-2	\\
279	&188	&188	&-12	&2	&2	\\
307	&258	&258	&10	&0	&0	\\
335	&350	&350	&-10	&-4	&-4	\\
363	&474	&474	&18	&3	&3	\\
391	&624	&624	&-16	&0	&0	\\
419	&826	&826	&18	&-2	&-2	\\
447	&1090	&1090	&-22	&4	&4	\\
475	&1410	&1410	&26	&0	&0	\\
503	&1814	&1814	&-26	&-4	&-4	\\
531	&2338	&2338	&26	&4	&4	\\
559	&2982	&2982	&-34	&0	&0	\\
587	&3774	&3774	&38	&-6	&-6	\\
615	&4774	&4774	&-42	&4	&4	\\
643	&5994	&5994	&42	&0	&0	\\
671	&7494	&7494	&-50	&-6	&-6	\\
699	&9348	&9348	&60	&6	&6	\\
727	&11586	&11586	&-62	&0	&0	\\
755	&14320	&14320	&64	&-8	&-8	\\
783	&17654	&17654	&-74	&8	&8	\\
811	&21654	&21654	&86	&0	&0	\\
839	&26488	&26488	&-88	&-8	&-8	\\
867	&32334	&32334	&94	&9	&9	\\
895	&39324	&39324	&-108	&0	&0	\\
923	&47680	&47680	&120	&-8	&-8	\\
951	&57688	&57688	&-128	&10	&10	\\
979	&69600	&69600	&136	&0	&0	\\
1007	&83760	&83760	&-152	&-12	&-12	\\
1035	&100596	&100596	&172	&12	&12	\\
\bottomrule
\end{tabular}
\end{minipage}
\begin{minipage}[t]{0.49\linewidth}
\centering
\caption{\label{tab:coeffs:7_2}$H^{(7)}_{g,2}$}\smallskip
\begin{tabular}{r|rrrrrrrrrrr}
\toprule
$[g]$	&1A	&2A	&4A	&3AB	&6AB	\\
\midrule
$\G_g$&		$1$&	$1|4$&	$2|8$&	$3$&	$3|4$\\
\midrule
24	&4	&-4	&0	&-2	&2	\\
52	&12	&-12	&0	&0	&0	\\
80	&20	&-20	&0	&2	&-2	\\
108	&32	&-32	&0	&-1	&1	\\
136	&48	&-48	&0	&0	&0	\\
164	&80	&-80	&0	&2	&-2	\\
192	&108	&-108	&0	&-3	&3	\\
220	&168	&-168	&0	&0	&0	\\
248	&232	&-232	&0	&4	&-4	\\
276	&328	&-328	&0	&-2	&2	\\
304	&444	&-444	&0	&0	&0	\\
332	&620	&-620	&0	&2	&-2	\\
360	&812	&-812	&0	&-4	&4	\\
388	&1104	&-1104	&0	&0	&0	\\
416	&1444	&-1444	&0	&4	&-4	\\
444	&1904	&-1904	&0	&-4	&4	\\
472	&2460	&-2460	&0	&0	&0	\\
500	&3208	&-3208	&0	&4	&-4	\\
528	&4080	&-4080	&0	&-6	&6	\\
556	&5244	&-5244	&0	&0	&0	\\
584	&6632	&-6632	&0	&8	&-8	\\
612	&8400	&-8400	&0	&-6	&6	\\
640	&10524	&-10524	&0	&0	&0	\\
668	&13224	&-13224	&0	&6	&-6	\\
696	&16408	&-16408	&0	&-8	&8	\\
724	&20436	&-20436	&0	&0	&0	\\
752	&25216	&-25216	&0	&10	&-10	\\
780	&31120	&-31120	&0	&-8	&8	\\
808	&38148	&-38148	&0	&0	&0	\\
836	&46784	&-46784	&0	&8	&-8	\\
864	&56976	&-56976	&0	&-12	&12	\\
892	&69432	&-69432	&0	&0	&0	\\
920	&84144	&-84144	&0	&12	&-12	\\
948	&101904	&-101904	&0	&-12	&12	\\
976	&122868	&-122868	&0	&0	&0	\\
1004	&148076	&-148076	&0	&14	&-14	\\
1032	&177656	&-177656	&0	&-16	&16	\\
1060	&213072	&-213072	&0	&0	&0	\\
\bottomrule
\end{tabular}
\end{minipage}
\end{table}

\begin{table}
\begin{minipage}[t]{0.49\linewidth}
\centering
\caption{\label{tab:coeffs:7_3}$H^{(7)}_{g,3}$}\smallskip
\begin{tabular}{r|rrrrrrrrrrr}\toprule
$[g]$	&1A	&2A	&4A	&3AB	&6AB	\\
\midrule
$\G_g$&		$1$&	$1|4$&	$2|8$&	$3$&	$3|4$\\
\midrule
19	&6	&6	&-2	&0	&0	\\
47	&12	&12	&4	&0	&0	\\
75	&22	&22	&-2	&1	&1	\\
103	&36	&36	&4	&0	&0	\\
131	&58	&58	&-6	&-2	&-2	\\
159	&90	&90	&2	&0	&0	\\
187	&132	&132	&-4	&0	&0	\\
215	&190	&190	&6	&-2	&-2	\\
243	&274	&274	&-6	&1	&1	\\
271	&384	&384	&8	&0	&0	\\
299	&528	&528	&-8	&0	&0	\\
327	&722	&722	&10	&2	&2	\\
355	&972	&972	&-12	&0	&0	\\
383	&1300	&1300	&12	&-2	&-2	\\
411	&1724	&1724	&-12	&2	&2	\\
439	&2256	&2256	&16	&0	&0	\\
467	&2938	&2938	&-22	&-2	&-2	\\
495	&3806	&3806	&22	&2	&2	\\
523	&4890	&4890	&-22	&0	&0	\\
551	&6244	&6244	&28	&-2	&-2	\\
579	&7940	&7940	&-28	&2	&2	\\
607	&10038	&10038	&30	&0	&0	\\
635	&12620	&12620	&-36	&-4	&-4	\\
663	&15814	&15814	&38	&4	&4	\\
691	&19722	&19722	&-46	&0	&0	\\
719	&24490	&24490	&50	&-2	&-2	\\
747	&30310	&30310	&-50	&4	&4	\\
775	&37362	&37362	&58	&0	&0	\\
803	&45908	&45908	&-68	&-4	&-4	\\
831	&56236	&56236	&68	&4	&4	\\
859	&68646	&68646	&-74	&0	&0	\\
887	&83556	&83556	&84	&-6	&-6	\\
915	&101436	&101436	&-92	&6	&6	\\
943	&122790	&122790	&102	&0	&0	\\
971	&148254	&148254	&-106	&-6	&-6	\\
999	&178566	&178566	&118	&6	&6	\\
1027	&214548	&214548	&-132	&0	&0	\\
1055	&257190	&257190	&142	&-6	&-6	\\
\bottomrule
\end{tabular}
\end{minipage}
\begin{minipage}[t]{0.49\linewidth}
\centering
\caption{\label{tab:coeffs:7_4}$H^{(7)}_{g,4}$}\smallskip
\begin{tabular}{r|rrrrrrrrrrr}\toprule
$[g]$	&1A	&2A	&4A	&3AB	&6AB	\\
\midrule
$\G_g$&		$1$&	$1|4$&	$2|8$&	$3$&	$3|4$\\
\midrule
12	&4	&-4	&0	&1	&-1	\\
40	&12	&-12	&0	&0	&0	\\
68	&16	&-16	&0	&-2	&2	\\
96	&36	&-36	&0	&0	&0	\\
124	&48	&-48	&0	&0	&0	\\
152	&84	&-84	&0	&0	&0	\\
180	&116	&-116	&0	&2	&-2	\\
208	&180	&-180	&0	&0	&0	\\
236	&244	&-244	&0	&-2	&2	\\
264	&360	&-360	&0	&0	&0	\\
292	&480	&-480	&0	&0	&0	\\
320	&676	&-676	&0	&-2	&2	\\
348	&896	&-896	&0	&2	&-2	\\
376	&1224	&-1224	&0	&0	&0	\\
404	&1588	&-1588	&0	&-2	&2	\\
432	&2128	&-2128	&0	&1	&-1	\\
460	&2736	&-2736	&0	&0	&0	\\
488	&3588	&-3588	&0	&0	&0	\\
516	&4576	&-4576	&0	&4	&-4	\\
544	&5904	&-5904	&0	&0	&0	\\
572	&7448	&-7448	&0	&-4	&4	\\
600	&9500	&-9500	&0	&2	&-2	\\
628	&11892	&-11892	&0	&0	&0	\\
656	&14992	&-14992	&0	&-2	&2	\\
684	&18628	&-18628	&0	&4	&-4	\\
712	&23256	&-23256	&0	&0	&0	\\
740	&28688	&-28688	&0	&-4	&4	\\
768	&35532	&-35532	&0	&3	&-3	\\
796	&43560	&-43560	&0	&0	&0	\\
824	&53528	&-53528	&0	&-4	&4	\\
852	&65256	&-65256	&0	&6	&-6	\\
880	&79656	&-79656	&0	&0	&0	\\
908	&96564	&-96564	&0	&-6	&6	\\
936	&117196	&-117196	&0	&4	&-4	\\
964	&141360	&-141360	&0	&0	&0	\\
992	&170600	&-170600	&0	&-4	&4	\\
1020	&204848	&-204848	&0	&8	&-8	\\
1048	&245988	&-245988	&0	&0	&0	\\
\bottomrule
\end{tabular}
\end{minipage}
\end{table}

\begin{table}
\begin{minipage}[t]{0.49\linewidth}
\centering
\caption{\label{tab:coeffs:7_5}$H^{(7)}_{g,5}$}\smallskip
\begin{tabular}{r|rrrrrrrrrrr}
\toprule
$[g]$	&1A	&2A	&4A	&3AB	&6AB	\\
\midrule
$\G_g$&		$1$&	$1|4$&	$2|8$&	$3$&	$3|4$\\
\midrule
3	&2	&2	&2	&-1	&-1	\\
31	&6	&6	&-2	&0	&0	\\
59	&14	&14	&-2	&2	&2	\\
87	&22	&22	&-2	&-2	&-2	\\
115	&36	&36	&4	&0	&0	\\
143	&56	&56	&0	&2	&2	\\
171	&82	&82	&2	&-2	&-2	\\
199	&126	&126	&-2	&0	&0	\\
227	&182	&182	&6	&2	&2	\\
255	&250	&250	&-6	&-2	&-2	\\
283	&354	&354	&2	&0	&0	\\
311	&490	&490	&-6	&4	&4	\\
339	&656	&656	&8	&-4	&-4	\\
367	&882	&882	&-6	&0	&0	\\
395	&1180	&1180	&4	&4	&4	\\
423	&1550	&1550	&-10	&-4	&-4	\\
451	&2028	&2028	&12	&0	&0	\\
479	&2638	&2638	&-10	&4	&4	\\
507	&3394	&3394	&10	&-5	&-5	\\
535	&4362	&4362	&-14	&0	&0	\\
563	&5562	&5562	&18	&6	&6	\\
591	&7032	&7032	&-16	&-6	&-6	\\
619	&8886	&8886	&14	&0	&0	\\
647	&11166	&11166	&-18	&6	&6	\\
675	&13940	&13940	&28	&-7	&-7	\\
703	&17358	&17358	&-26	&0	&0	\\
731	&21536	&21536	&24	&8	&8	\\
759	&26594	&26594	&-30	&-10	&-10	\\
787	&32742	&32742	&38	&0	&0	\\
815	&40180	&40180	&-36	&10	&10	\\
843	&49124	&49124	&36	&-10	&-10	\\
871	&59916	&59916	&-44	&0	&0	\\
899	&72852	&72852	&52	&12	&12	\\
927	&88296	&88296	&-56	&-12	&-12	\\
955	&106788	&106788	&52	&0	&0	\\
983	&128816	&128816	&-64	&14	&14	\\
1011	&154948	&154948	&76	&-14	&-14	\\
\bottomrule
\end{tabular}
\end{minipage}
\begin{minipage}[t]{0.49\linewidth}
\centering
\caption{\label{tab:coeffs:7_6}$H^{(7)}_{g,6}$}\smallskip
\begin{tabular}{r|rrrrrrrrrrr}
\toprule
$[g]$	&1A	&2A	&4A	&3AB	&6AB	\\
\midrule
$\G_g$&		$1$&	$1|4$&	$2|8$&	$3$&	$3|4$\\
\midrule
20	&4	&-4	&0	&-2	&2	\\
48	&4	&-4	&0	&1	&-1	\\
76	&12	&-12	&0	&0	&0	\\
104	&12	&-12	&0	&0	&0	\\
132	&32	&-32	&0	&2	&-2	\\
160	&36	&-36	&0	&0	&0	\\
188	&64	&-64	&0	&-2	&2	\\
216	&80	&-80	&0	&2	&-2	\\
244	&132	&-132	&0	&0	&0	\\
272	&160	&-160	&0	&-2	&2	\\
300	&252	&-252	&0	&3	&-3	\\
328	&312	&-312	&0	&0	&0	\\
356	&448	&-448	&0	&-2	&2	\\
384	&572	&-572	&0	&2	&-2	\\
412	&792	&-792	&0	&0	&0	\\
440	&992	&-992	&0	&-4	&4	\\
468	&1348	&-1348	&0	&4	&-4	\\
496	&1680	&-1680	&0	&0	&0	\\
524	&2220	&-2220	&0	&-6	&6	\\
552	&2776	&-2776	&0	&4	&-4	\\
580	&3600	&-3600	&0	&0	&0	\\
608	&4460	&-4460	&0	&-4	&4	\\
636	&5712	&-5712	&0	&6	&-6	\\
664	&7044	&-7044	&0	&0	&0	\\
692	&8892	&-8892	&0	&-6	&6	\\
720	&10932	&-10932	&0	&6	&-6	\\
748	&13656	&-13656	&0	&0	&0	\\
776	&16672	&-16672	&0	&-8	&8	\\
804	&20672	&-20672	&0	&8	&-8	\\
832	&25116	&-25116	&0	&0	&0	\\
860	&30856	&-30856	&0	&-8	&8	\\
888	&37352	&-37352	&0	&8	&-8	\\
916	&45564	&-45564	&0	&0	&0	\\
944	&54884	&-54884	&0	&-10	&10	\\
972	&66572	&-66572	&0	&11	&-11	\\
1000	&79848	&-79848	&0	&0	&0	\\
1028	&96256	&-96256	&0	&-14	&14	\\
\bottomrule
\end{tabular}
\end{minipage}
\end{table}

\clearpage

\begin{table}
\vspace{-12pt}
\subsection{Lambency Thirteen}
\vspace{-12pt}
\begin{minipage}[t]{0.33\linewidth}
\centering
\caption{\label{tab:coeffs:13_1}$H^{(13)}_{g,1}$}\smallskip
\begin{tabular}{r|rrr}
\toprule
$[g]$	&1A	&2A	&4AB	\\
	\midrule
$\G_g$&		$1$&$1|4$&${2|8}$\\	
	\midrule
-1	&-2	&-2	&-2	\\
51	&2	&2	&2	\\
103	&2	&2	&-2	\\
155	&0	&0	&0	\\
207	&2	&2	&-2	\\
259	&2	&2	&2	\\
311	&4	&4	&0	\\
363	&6	&6	&2	\\
415	&6	&6	&-2	\\
467	&8	&8	&4	\\
519	&12	&12	&-4	\\
571	&14	&14	&2	\\
623	&14	&14	&-2	\\
675	&20	&20	&4	\\
727	&24	&24	&-4	\\
779	&28	&28	&4	\\
831	&36	&36	&-4	\\
883	&42	&42	&6	\\
935	&50	&50	&-6	\\
987	&62	&62	&6	\\
1039	&70	&70	&-6	\\
1091	&84	&84	&8	\\
1143	&102	&102	&-6	\\
1195	&118	&118	&6	\\
1247	&136	&136	&-8	\\
1299	&162	&162	&10	\\
1351	&190	&190	&-10	\\
1403	&216	&216	&8	\\
1455	&254	&254	&-10	\\
1507	&292	&292	&12	\\
1559	&336	&336	&-12	\\
1611	&392	&392	&12	\\
1663	&446	&446	&-14	\\
1715	&510	&510	&14	\\
1767	&592	&592	&-16	\\
1819	&672	&672	&16	\\
1871	&764	&764	&-16	\\
1923	&876	&876	&20	\\
\bottomrule
\end{tabular}
\end{minipage}
\begin{minipage}[t]{0.33\linewidth}
\centering
\caption{\label{tab:coeffs:13_2}$H^{(13)}_{g,2}$}\smallskip
\begin{tabular}{r|rrr}
\toprule
$[g]$	&1A	&2A	&4AB	\\
	\midrule
$\G_g$&		$1$&$1|4$&${2|8}$\\	
	\midrule
48	&0	&0	&0	\\
100	&2	&-2	&0	\\
152	&4	&-4	&0	\\
204	&4	&-4	&0	\\
256	&6	&-6	&0	\\
308	&8	&-8	&0	\\
360	&8	&-8	&0	\\
412	&12	&-12	&0	\\
464	&16	&-16	&0	\\
516	&20	&-20	&0	\\
568	&24	&-24	&0	\\
620	&32	&-32	&0	\\
672	&36	&-36	&0	\\
724	&48	&-48	&0	\\
776	&56	&-56	&0	\\
828	&68	&-68	&0	\\
880	&80	&-80	&0	\\
932	&100	&-100	&0	\\
984	&112	&-112	&0	\\
1036	&140	&-140	&0	\\
1088	&164	&-164	&0	\\
1140	&192	&-192	&0	\\
1192	&224	&-224	&0	\\
1244	&268	&-268	&0	\\
1296	&306	&-306	&0	\\
1348	&364	&-364	&0	\\
1400	&420	&-420	&0	\\
1452	&488	&-488	&0	\\
1504	&560	&-560	&0	\\
1556	&656	&-656	&0	\\
1608	&744	&-744	&0	\\
1660	&864	&-864	&0	\\
1712	&988	&-988	&0	\\
1764	&1134	&-1134	&0	\\
1816	&1292	&-1292	&0	\\
1868	&1484	&-1484	&0	\\
1920	&1676	&-1676	&0	\\
1972	&1920	&-1920	&0	\\
\bottomrule
\end{tabular}
\end{minipage}
\begin{minipage}[t]{0.33\linewidth}
\centering
\caption{\label{tab:coeffs:13_3}$H^{(13)}_{g,3}$}\smallskip
\begin{tabular}{r|rrr}
\toprule
$[g]$	&1A	&2A	&4AB	\\
	\midrule
$\G_g$&		$1$&$1|4$&${2|8}$\\	
	\midrule
43	&2	&2	&-2	\\
95	&2	&2	&2	\\
147	&4	&4	&0	\\
199	&6	&6	&2	\\
251	&8	&8	&-4	\\
303	&10	&10	&2	\\
355	&14	&14	&-2	\\
407	&18	&18	&2	\\
459	&22	&22	&-2	\\
511	&26	&26	&2	\\
563	&34	&34	&-2	\\
615	&44	&44	&4	\\
667	&52	&52	&-4	\\
719	&64	&64	&4	\\
771	&78	&78	&-2	\\
823	&96	&96	&4	\\
875	&114	&114	&-6	\\
927	&136	&136	&4	\\
979	&164	&164	&-4	\\
1031	&194	&194	&6	\\
1083	&230	&230	&-6	\\
1135	&270	&270	&6	\\
1187	&318	&318	&-6	\\
1239	&374	&374	&6	\\
1291	&434	&434	&-10	\\
1343	&506	&506	&10	\\
1395	&592	&592	&-8	\\
1447	&686	&686	&10	\\
1499	&792	&792	&-12	\\
1551	&914	&914	&10	\\
1603	&1054	&1054	&-10	\\
1655	&1214	&1214	&14	\\
1707	&1394	&1394	&-14	\\
1759	&1594	&1594	&14	\\
1811	&1822	&1822	&-14	\\
1863	&2084	&2084	&16	\\
1915	&2374	&2374	&-18	\\
1967	&2698	&2698	&18	\\
\bottomrule
\end{tabular}
\end{minipage}
\end{table}

\begin{table}
\begin{minipage}[t]{0.33\linewidth}
\centering
\caption{\label{tab:coeffs:13_4}$H^{(13)}_{g,4}$}\smallskip
\begin{tabular}{r|rrr}
\toprule
$[g]$	&1A	&2A	&4AB	\\
	\midrule
$\G_g$&		$1$&$1|4$&${2|8}$\\	
	\midrule
36	&2	&-2	&0	\\
88	&4	&-4	&0	\\
140	&4	&-4	&0	\\
192	&8	&-8	&0	\\
244	&8	&-8	&0	\\
296	&12	&-12	&0	\\
348	&16	&-16	&0	\\
400	&22	&-22	&0	\\
452	&24	&-24	&0	\\
504	&36	&-36	&0	\\
556	&40	&-40	&0	\\
608	&52	&-52	&0	\\
660	&64	&-64	&0	\\
712	&80	&-80	&0	\\
764	&92	&-92	&0	\\
816	&116	&-116	&0	\\
868	&136	&-136	&0	\\
920	&168	&-168	&0	\\
972	&196	&-196	&0	\\
1024	&238	&-238	&0	\\
1076	&272	&-272	&0	\\
1128	&332	&-332	&0	\\
1180	&384	&-384	&0	\\
1232	&456	&-456	&0	\\
1284	&528	&-528	&0	\\
1336	&620	&-620	&0	\\
1388	&712	&-712	&0	\\
1440	&840	&-840	&0	\\
1492	&960	&-960	&0	\\
1544	&1120	&-1120	&0	\\
1596	&1280	&-1280	&0	\\
1648	&1484	&-1484	&0	\\
1700	&1688	&-1688	&0	\\
1752	&1952	&-1952	&0	\\
1804	&2216	&-2216	&0	\\
1856	&2544	&-2544	&0	\\
1908	&2888	&-2888	&0	\\
1960	&3304	&-3304	&0	\\
\bottomrule
\end{tabular}
\end{minipage}
\begin{minipage}[t]{0.33\linewidth}
\centering
\caption{\label{tab:coeffs:13_5}$H^{(13)}_{g,5}$}\smallskip
\begin{tabular}{r|rrr}
\toprule
$[g]$	&1A	&2A	&4AB	\\
	\midrule
$\G_g$&		$1$&$1|4$&${2|8}$\\	
	\midrule
27	&2	&2	&2	\\
79	&4	&4	&0	\\
131	&6	&6	&2	\\
183	&6	&6	&-2	\\
235	&10	&10	&2	\\
287	&14	&14	&-2	\\
339	&16	&16	&0	\\
391	&22	&22	&-2	\\
443	&30	&30	&2	\\
495	&36	&36	&-4	\\
547	&46	&46	&2	\\
599	&58	&58	&-2	\\
651	&68	&68	&4	\\
703	&86	&86	&-2	\\
755	&106	&106	&2	\\
807	&124	&124	&-4	\\
859	&152	&152	&4	\\
911	&184	&184	&-4	\\
963	&216	&216	&4	\\
1015	&258	&258	&-6	\\
1067	&308	&308	&4	\\
1119	&362	&362	&-6	\\
1171	&426	&426	&6	\\
1223	&502	&502	&-6	\\
1275	&584	&584	&8	\\
1327	&684	&684	&-8	\\
1379	&798	&798	&6	\\
1431	&920	&920	&-8	\\
1483	&1070	&1070	&10	\\
1535	&1238	&1238	&-10	\\
1587	&1422	&1422	&10	\\
1639	&1638	&1638	&-10	\\
1691	&1884	&1884	&12	\\
1743	&2156	&2156	&-12	\\
1795	&2468	&2468	&12	\\
1847	&2822	&2822	&-14	\\
1899	&3212	&3212	&16	\\
1951	&3660	&3660	&-16	\\
\bottomrule
\end{tabular}
\end{minipage}
\begin{minipage}[t]{0.33\linewidth}
\centering
\caption{\label{tab:coeffs:13_6}$H^{(13)}_{g,6}$}\smallskip
\begin{tabular}{r|rrr}
\toprule
$[g]$	&1A	&2A	&4AB	\\
	\midrule
$\G_g$&		$1$&$1|4$&${2|8}$\\	
	\midrule
16	&2	&-2	&0	\\
68	&4	&-4	&0	\\
120	&4	&-4	&0	\\
172	&8	&-8	&0	\\
224	&8	&-8	&0	\\
276	&16	&-16	&0	\\
328	&16	&-16	&0	\\
380	&24	&-24	&0	\\
432	&28	&-28	&0	\\
484	&38	&-38	&0	\\
536	&44	&-44	&0	\\
588	&60	&-60	&0	\\
640	&68	&-68	&0	\\
692	&88	&-88	&0	\\
744	&104	&-104	&0	\\
796	&132	&-132	&0	\\
848	&152	&-152	&0	\\
900	&190	&-190	&0	\\
952	&220	&-220	&0	\\
1004	&268	&-268	&0	\\
1056	&312	&-312	&0	\\
1108	&376	&-376	&0	\\
1160	&432	&-432	&0	\\
1212	&520	&-520	&0	\\
1264	&596	&-596	&0	\\
1316	&708	&-708	&0	\\
1368	&812	&-812	&0	\\
1420	&956	&-956	&0	\\
1472	&1092	&-1092	&0	\\
1524	&1280	&-1280	&0	\\
1576	&1460	&-1460	&0	\\
1628	&1696	&-1696	&0	\\
1680	&1932	&-1932	&0	\\
1732	&2236	&-2236	&0	\\
1784	&2536	&-2536	&0	\\
1836	&2924	&-2924	&0	\\
1888	&3308	&-3308	&0	\\
1940	&3792	&-3792	&0	\\
\bottomrule
\end{tabular}
\end{minipage}
\end{table}

\begin{table}
\begin{minipage}[t]{0.33\linewidth}
\centering
\caption{\label{tab:coeffs:13_7}$H^{(13)}_{g,7}$}\smallskip
\begin{tabular}{r|rrr}
\toprule
$[g]$	&1A	&2A	&4AB	\\
	\midrule
$\G_g$&		$1$&$1|4$&${2|8}$\\	
	\midrule
3	&2	&2	&-2	\\
55	&2	&2	&2	\\
107	&4	&4	&0	\\
159	&8	&8	&0	\\
211	&10	&10	&-2	\\
263	&12	&12	&0	\\
315	&16	&16	&0	\\
367	&22	&22	&2	\\
419	&26	&26	&-2	\\
471	&34	&34	&2	\\
523	&44	&44	&0	\\
575	&54	&54	&2	\\
627	&68	&68	&-4	\\
679	&82	&82	&2	\\
731	&102	&102	&-2	\\
783	&124	&124	&4	\\
835	&148	&148	&-4	\\
887	&176	&176	&4	\\
939	&214	&214	&-2	\\
991	&256	&256	&4	\\
1043	&300	&300	&-4	\\
1095	&356	&356	&4	\\
1147	&420	&420	&-4	\\
1199	&494	&494	&6	\\
1251	&580	&580	&-8	\\
1303	&674	&674	&6	\\
1355	&786	&786	&-6	\\
1407	&918	&918	&6	\\
1459	&1060	&1060	&-8	\\
1511	&1226	&1226	&6	\\
1563	&1418	&1418	&-6	\\
1615	&1632	&1632	&8	\\
1667	&1874	&1874	&-10	\\
1719	&2150	&2150	&10	\\
1771	&2464	&2464	&-8	\\
1823	&2816	&2816	&12	\\
1875	&3214	&3214	&-14	\\
\bottomrule
\end{tabular}
\end{minipage}
\begin{minipage}[t]{0.33\linewidth}
\centering
\caption{\label{tab:coeffs:13_8}$H^{(13)}_{g,8}$}\smallskip
\begin{tabular}{r|rrr}
\toprule
$[g]$	&1A	&2A	&4AB	\\
	\midrule
$\G_g$&		$1$&$1|4$&${2|8}$\\	
	\midrule
40	&4	&-4	&0	\\
92	&4	&-4	&0	\\
144	&6	&-6	&0	\\
196	&6	&-6	&0	\\
248	&12	&-12	&0	\\
300	&12	&-12	&0	\\
352	&20	&-20	&0	\\
404	&24	&-24	&0	\\
456	&32	&-32	&0	\\
508	&36	&-36	&0	\\
560	&52	&-52	&0	\\
612	&56	&-56	&0	\\
664	&76	&-76	&0	\\
716	&88	&-88	&0	\\
768	&112	&-112	&0	\\
820	&128	&-128	&0	\\
872	&164	&-164	&0	\\
924	&184	&-184	&0	\\
976	&232	&-232	&0	\\
1028	&268	&-268	&0	\\
1080	&324	&-324	&0	\\
1132	&372	&-372	&0	\\
1184	&452	&-452	&0	\\
1236	&512	&-512	&0	\\
1288	&616	&-616	&0	\\
1340	&704	&-704	&0	\\
1392	&832	&-832	&0	\\
1444	&950	&-950	&0	\\
1496	&1120	&-1120	&0	\\
1548	&1268	&-1268	&0	\\
1600	&1486	&-1486	&0	\\
1652	&1688	&-1688	&0	\\
1704	&1956	&-1956	&0	\\
1756	&2220	&-2220	&0	\\
1808	&2568	&-2568	&0	\\
1860	&2896	&-2896	&0	\\
1912	&3336	&-3336	&0	\\
\bottomrule
\end{tabular}
\end{minipage}
\begin{minipage}[t]{0.33\linewidth}
\centering
\caption{\label{tab:coeffs:13_9}$H^{(13)}_{g,9}$}\smallskip
\begin{tabular}{r|rrr}
\toprule
$[g]$	&1A	&2A	&4AB	\\
	\midrule
$\G_g$&		$1$&$1|4$&${2|8}$\\	
	\midrule
23	&2	&2	&-2	\\
75	&4	&4	&0	\\
127	&4	&4	&0	\\
179	&6	&6	&2	\\
231	&8	&8	&0	\\
283	&12	&12	&0	\\
335	&14	&14	&-2	\\
387	&20	&20	&0	\\
439	&26	&26	&-2	\\
491	&30	&30	&2	\\
543	&40	&40	&0	\\
595	&50	&50	&2	\\
647	&60	&60	&0	\\
699	&74	&74	&2	\\
751	&90	&90	&-2	\\
803	&108	&108	&4	\\
855	&134	&134	&-2	\\
907	&158	&158	&2	\\
959	&188	&188	&-4	\\
1011	&226	&226	&2	\\
1063	&266	&266	&-2	\\
1115	&314	&314	&2	\\
1167	&372	&372	&-4	\\
1219	&436	&436	&4	\\
1271	&508	&508	&-4	\\
1323	&596	&596	&4	\\
1375	&692	&692	&-4	\\
1427	&802	&802	&6	\\
1479	&932	&932	&-4	\\
1531	&1074	&1074	&6	\\
1583	&1238	&1238	&-6	\\
1635	&1430	&1430	&6	\\
1687	&1640	&1640	&-8	\\
1739	&1878	&1878	&6	\\
1791	&2150	&2150	&-6	\\
1843	&2456	&2456	&8	\\
1895	&2800	&2800	&-8	\\
\bottomrule
\end{tabular}
\end{minipage}
\end{table}

\begin{table}
\begin{minipage}[t]{0.33\linewidth}
\centering
\caption{\label{tab:coeffs:13_10}$H^{(13)}_{g,10}$}\smallskip
\begin{tabular}{r|rrr}
\toprule
$[g]$	&1A	&2A	&4AB	\\
	\midrule
$\G_g$&		$1$&$1|4$&${2|8}$\\	
	\midrule
4	&2	&-2	&0	\\
56	&0	&0	&0	\\
108	&4	&-4	&0	\\
160	&4	&-4	&0	\\
212	&8	&-8	&0	\\
264	&8	&-8	&0	\\
316	&12	&-12	&0	\\
368	&12	&-12	&0	\\
420	&20	&-20	&0	\\
472	&20	&-20	&0	\\
524	&32	&-32	&0	\\
576	&34	&-34	&0	\\
628	&48	&-48	&0	\\
680	&52	&-52	&0	\\
732	&72	&-72	&0	\\
784	&78	&-78	&0	\\
836	&104	&-104	&0	\\
888	&116	&-116	&0	\\
940	&148	&-148	&0	\\
992	&164	&-164	&0	\\
1044	&208	&-208	&0	\\
1096	&232	&-232	&0	\\
1148	&288	&-288	&0	\\
1200	&324	&-324	&0	\\
1252	&396	&-396	&0	\\
1304	&444	&-444	&0	\\
1356	&536	&-536	&0	\\
1408	&604	&-604	&0	\\
1460	&720	&-720	&0	\\
1512	&812	&-812	&0	\\
1564	&960	&-960	&0	\\
1616	&1080	&-1080	&0	\\
1668	&1268	&-1268	&0	\\
1720	&1428	&-1428	&0	\\
1772	&1664	&-1664	&0	\\
1824	&1872	&-1872	&0	\\
\bottomrule
\end{tabular}
\end{minipage}
\begin{minipage}[t]{0.33\linewidth}
\centering
\caption{\label{tab:coeffs:13_11}$H^{(13)}_{g,11}$}\smallskip
\begin{tabular}{r|rrr}
\toprule
$[g]$	&1A	&2A	&4AB	\\
	\midrule
$\G_g$&		$1$&$1|4$&${2|8}$\\	
	\midrule
35	&2	&2	&2	\\
87	&2	&2	&2	\\
139	&2	&2	&-2	\\
191	&4	&4	&0	\\
243	&4	&4	&0	\\
295	&8	&8	&0	\\
347	&10	&10	&-2	\\
399	&10	&10	&2	\\
451	&16	&16	&0	\\
503	&20	&20	&0	\\
555	&22	&22	&-2	\\
607	&28	&28	&0	\\
659	&36	&36	&0	\\
711	&44	&44	&0	\\
763	&54	&54	&-2	\\
815	&64	&64	&0	\\
867	&76	&76	&0	\\
919	&94	&94	&2	\\
971	&114	&114	&-2	\\
1023	&130	&130	&2	\\
1075	&156	&156	&0	\\
1127	&188	&188	&0	\\
1179	&216	&216	&-4	\\
1231	&254	&254	&2	\\
1283	&300	&300	&0	\\
1335	&346	&346	&2	\\
1387	&404	&404	&-4	\\
1439	&470	&470	&2	\\
1491	&542	&542	&-2	\\
1543	&630	&630	&2	\\
1595	&724	&724	&-4	\\
1647	&828	&828	&4	\\
1699	&954	&954	&-2	\\
1751	&1100	&1100	&4	\\
1803	&1250	&1250	&-6	\\
1855	&1428	&1428	&4	\\
\bottomrule
\end{tabular}
\end{minipage}
\begin{minipage}[t]{0.33\linewidth}
\centering
\caption{\label{tab:coeffs:13_12}$H^{(13)}_{g,12}$}\smallskip
\begin{tabular}{r|rrr}
\toprule
$[g]$	&1A	&2A	&4AB	\\
	\midrule
$\G_g$&		$1$&$1|4$&${2|8}$\\	
	\midrule
12	&0	&0	&0	\\
64	&2	&-2	&0	\\
116	&0	&0	&0	\\
168	&4	&-4	&0	\\
220	&0	&0	&0	\\
272	&4	&-4	&0	\\
324	&2	&-2	&0	\\
376	&8	&-8	&0	\\
428	&4	&-4	&0	\\
480	&12	&-12	&0	\\
532	&8	&-8	&0	\\
584	&16	&-16	&0	\\
636	&16	&-16	&0	\\
688	&24	&-24	&0	\\
740	&20	&-20	&0	\\
792	&36	&-36	&0	\\
844	&32	&-32	&0	\\
896	&48	&-48	&0	\\
948	&48	&-48	&0	\\
1000	&68	&-68	&0	\\
1052	&68	&-68	&0	\\
1104	&96	&-96	&0	\\
1156	&98	&-98	&0	\\
1208	&128	&-128	&0	\\
1260	&136	&-136	&0	\\
1312	&176	&-176	&0	\\
1364	&184	&-184	&0	\\
1416	&240	&-240	&0	\\
1468	&252	&-252	&0	\\
1520	&312	&-312	&0	\\
1572	&340	&-340	&0	\\
1624	&416	&-416	&0	\\
1676	&448	&-448	&0	\\
1728	&548	&-548	&0	\\
1780	&592	&-592	&0	\\
1832	&708	&-708	&0	\\
\bottomrule
\end{tabular}
\end{minipage}
\end{table}

\clearpage

\section{Decompositions}\label{sec:decompositions}

As explained in \S\ref{sec:mckay} (see also \S\ref{sec:conj:mod}) our conjectural proposals for the umbral McKay--Thompson series $H^{(\ll)}_{g,r}(\t)=\sum_dc^{(\ll)}_{g,r}(d)q^{d}$ (cf. \S\S\ref{sec:mckay},\ref{sec:coeffs}) determine the $G^{(\ll)}$-modules $K^{(\ll)}_{r,d}$ up to isomorphism for $d>0$, at least for those values of $d$ for which we can identify all the Fourier coefficients $c^{(\ll)}_{g,r}(d)$. In this section we furnish tables of explicit decompositions into irreducible representations of $G^{(\ll)}$ for $K^{(\ll)}_{r,d}$, for the first few values of $d$. The coefficient $c^{(\ll)}_r(d)$ of $H^{(\ll)}_r=H^{(\ll)}_{e,r}$ is non-zero only when $d=n-r^2/4\ll$ for some integer $n\geq 0$. For each of the tables in this section the rows are labelled by the values $4\ll d$, so that the entry in row $m$ and column $\chi_i$ indicates the multiplicity of the irreducible representation of $G^{(\ll)}$ with character $\chi_i$ (in the notation of the character tables of \S\ref{sec:chars:irr}) appearing in the $G^{(\ll)}$-module $K^{(\ll)}_{r,m/4\ll}$. One can observe that these tables support Conjectures \ref{conj:conj:mod:Kell}, \ref{conj:conj:disc:dualpair} and \ref{conj:conj:disc:doub}, and also give evidence in support of the hypothesis that $K^{(\ll)}_{r,d}$ has a decomposition into irreducible representations that factor through $\bar{G}^{(\ll)}$ when $r$ is odd, and has a decomposition into faithful irreducible representations of $G^{(\ll)}$ when $r$ is even. 
   
\clearpage

 \begin{sidewaystable}

\subsection{Lambency Two}

\caption{Decomposition of $K^{(2)}_1$}
 \begin{center}
 \begin{small}
 \begin{tabular}{c@{ }|@{ }R@{ }R@{ }R@{ }R@{ }R@{ }R@{ }R@{ }R@{ }R@{ }R@{ }R@{ }R@{ }R@{ }R@{ }R@{ }R@{ }R@{ }R@{ }R@{ }R@{ }R@{ }R@{ }R@{ }R@{ }R@{ }R}
 \toprule  & \chi_{1}& \chi_{2}& \chi_{3}& \chi_{4}& \chi_{5}& \chi_{6}& \chi_{7}& \chi_{8}& \chi_{9}&\chi_{10}& \chi_{11}& \chi_{12}& \chi_{13}&\chi_{14}& \chi_{15}& \chi_{16}& \chi_{17}& \chi_{18}& \chi_{19}& \chi_{20}& \chi_{21}& \chi_{22}& \chi_{23}& \chi_{24}& \chi_{25}& \chi_{26} \\\midrule
 -1&-2&0&0&0&0&0&0&0&0&0&0&0&0&0&0&0&0&0&0&0&0&0&0&0&0&0\\ 
7&0&0&1&1&0&0&0&0&0&0&0&0&0&0&0&0&0&0&0&0&0&0&0&0&0&0\\ 
15&0&0&0&0&1&1&0&0&0&0&0&0&0&0&0&0&0&0&0&0&0&0&0&0&0&0\\ 
23&0&0&0&0&0&0&0&0&0&1&1&0&0&0&0&0&0&0&0&0&0&0&0&0&0&0\\ 
31&0&0&0&0&0&0&0&0&0&0&0&0&0&0&0&0&0&0&0&2&0&0&0&0&0&0\\ 
39&0&0&0&0&0&0&0&0&0&0&0&0&0&0&0&0&0&0&0&0&0&0&0&0&2&0\\ 
47&0&0&0&0&0&0&0&0&0&0&0&0&0&0&0&0&0&0&0&0&0&2&0&0&0&2\\ 
55&0&0&0&0&0&0&0&0&0&0&0&0&0&0&0&0&0&2&2&0&0&0&2&2&2&2\\
63&0&0&0&0&0&0&0&0&0&0&0&1&1&0&1&1&2&0&0&2&2&2&4&2&2&6\\ 
71&0&0&0&0&0&0&0&0&2&2&2&0&0&2&2&2&0&2&2&2&4&4&4&8&8&10\\
79& 0& 0& 0& 0& 2& 2& 0& 2& 2& 0& 0& 2& 2& 2& 2& 2& 4& 4& 4& 6& 6& 8& 12& 10& 10& 24\\
 87& 0& 0& 0& 0& 0& 0& 0&  0& 0& 4& 4& 4& 4& 6& 4& 4& 2& 8&   10& 8& 14& 12& 22& 24& 26& 40\\
 95& 0& 2& 0& 0& 2& 2& 2& 4& 4& 6& 6& 8& 8& 4& 8& 8& 12& 12& 12& 18& 26& 30& 40&   38& 40& 80\\
 103& 0& 0&  2& 2& 2& 2& 4& 2& 6& 10& 10& 14& 14& 18& 14& 14& 16& 26& 30& 28& 44& 44& 70& 80& 84&   136\\
 111& 0& 0& 0& 0& 8& 8& 4& 6& 14& 16&  16& 24& 24& 22& 24& 24& 34& 38& 46& 58& 80& 86& 128& 126& 132& 254\\
 119& 0& 0& 2& 2& 8& 8& 12& 8& 18& 38& 38& 40& 40&   46& 44& 44& 46& 78& 86& 88& 138& 144& 218& 238& 246& 424\\
 127& 0& 2&  2& 2& 18& 18& 18& 22& 36& 50& 50& 72& 72& 68& 72& 72& 100& 122& 140& 170& 232& 252& 378& 382& 400& 742\\
 135& 0& 2& 8& 8& 25& 25& 30&  26& 54& 94& 94& 116& 116& 130& 124& 124& 140& 212& 246& 262& 392& 410& 630& 670& 704& 1222\\
 143& 0& 6& 6& 6& 50& 50& 50& 58& 100& 148& 148&  194& 194& 192& 202& 202&   256& 342& 388& 454& 654& 704& 1044& 1074& 1120& 2058\\
 151& 0& 4& 18& 18& 68& 68& 80& 72& 150& 252& 252& 318&  318& 346& 332& 332&   394& 582& 664& 722& 1062& 1116& 1702& 1800& 1880& 3320\\
 159& 0& 14& 20& 20& 126& 126& 128& 138& 254& 390& 390& 516&  516& 520& 536&   536& 676& 904& 1036& 1196& 1716& 1836& 2764& 2846& 2980& 5408\\
 167& 2& 20& 40& 40& 182& 182& 214& 200& 396& 652& 652& 814&  814& 872&  860& 860& 1020& 1476& 1684& 1862& 2742& 2902& 4384& 4622& 4828& 8572\\
 175& 2& 32& 55& 55& 314& 314& 328& 346& 640& 988& 988& 1298&  1298&1336& 1348& 1348& 1686& 2302& 2630& 3000& 4324& 4616& 6950& 7204& 7532& 13620\\
 183& 2& 40& 98& 98& 460& 460& 512& 496& 972& 1590& 1590& 2020&  2020& 2144& 2118& 2118& 2546& 3638& 4162& 4624& 6768& 7166& 10856& 11376& 11898& 21204\\
\bottomrule
\end{tabular}
\end{small}
\end{center}
\end{sidewaystable}

 \begin{table} \vspace{-2pt}  \subsection{Lambency Three}    \vspace{-2pt}
 \begin{center}
     \caption{Decomposition of $K^{(3)}_1$}
\vspace{.2cm}
 \begin{tabular}{c|RRRRRRRRRRRRRRRRR}
 \toprule  & \chi_{1}& \chi_{2}& \chi_{3}& \chi_{4}& \chi_{5}& \chi_{6}& \chi_{7}& \chi_{8}& \chi_{9}&\chi_{10}& \chi_{11}& \chi_{12}& \chi_{13} &\chi_{14}& \chi_{15} \\\midrule
-1 & -2 &0& 0& 0& 0& 0& 0& 0& 0& 0& 0& 0& 0& 0& 0\\
11&0&0&0& 1& 1& 0& 0& 0& 0& 0& 0& 0& 0& 0& 0\\
23&0& 0& 0& 0& 0& 0& 0& 0& 2&0&0& 0& 0& 0& 0\\
35&0& 0& 0& 0& 0&0& 0& 0& 0& 0&0& 0& 0& 2& 0 \\
47&0& 0& 0& 0& 0& 0& 0& 0& 0&2& 0& 2& 0& 0& 2\\
59&0& 0& 0& 0& 0& 2& 2& 0& 0& 0& 2& 2& 2& 2& 2\\
71&0& 0& 2& 0& 0& 0& 0& 2& 2& 2& 2& 2& 4& 4& 6\\
83&0& 0& 0& 2& 2& 4& 4& 2& 2& 2& 4& 6& 6& 8& 10\\
95&0& 0& 2& 0& 0& 4& 4& 8& 8& 6& 6& 10& 12& 14& 18\\
107&0& 2& 2& 4& 4& 8& 12& 8& 8& 8& 14& 16& 22& 28& 30\\
\bottomrule
\end{tabular}\end{center}
 
 \begin{center}                 
 \caption{Decomposition of $K^{(3)}_2$}    
\vspace{.2cm}
 \begin{tabular}{c|RRRRRRRRRRRRR}\toprule
& \chi_{16}& \chi_{17}& \chi_{18}& \chi_{19}& \chi_{20}& \chi_{21}& \chi_{22}& \chi_{23}& \chi_{24}& \chi_{25}& \chi_{26}\\\midrule                  
 8&1&1 &0& 0& 0& 0&0& 0& 0& 0& 0\\
 20&0&0 &0& 0&  1&1&0& 0& 0& 0& 0\\
   32&0&0 &0& 0& 0& 0& 1& 1&0& 0& 0\\
     44& 0& 0& 0& 0& 0& 0& 0& 0& 2& 1& 1\\
     56& 0& 0& 0& 2& 0& 0& 2& 2& 0& 2& 2\\
     68&0& 0& 2& 0& 2& 2& 2& 2& 4& 4& 4\\
     80& 2& 2& 0& 0& 1& 1& 6& 6& 4& 8& 8\\
     92& 0& 0& 2& 4& 6& 6& 8& 8& 12& 14&14\\ 
     104& 2& 2& 0& 4& 6& 6& 20& 20& 16& 24&24 \\
     116& 2& 2& 6& 8& 12& 12& 26& 26& 36& 44& 44\\                 
                  \bottomrule
                  \end{tabular}
                  \end{center}
   \end{table}

\begin{table}\vspace{-2pt}  \subsection{Lambency Four}    \vspace{-16pt} 
 \begin{center}
     \caption{Decomposition of $K^{(4)}_1$}
    \vspace{0.2cm}

 \begin{tabular}{cc|RRRRRRRRRRRR}
 \toprule  && \chi_{1}& \chi_{2}& \chi_{3}& \chi_{4}& \chi_{5}& \chi_{6}& \chi_{7}& \chi_{8}& \chi_{9}&\chi_{10}& \chi_{11} \\\midrule
&-1 & -2 &0& 0& 0& 0& 0& 0& 0& 0& 0& 0\\
      &15&0& 0& 0& 0& 0& 0& 2& 0& 0& 0& 0\\
      &31&0& 0& 0& 0& 0& 0& 0& 0& 0& 0& 2\\
      &47&0& 0& 0& 0& 0& 2& 0& 0& 2& 2& 0& \\
      &63&0& 1& 1& 0& 2& 0& 0& 2& 2& 2& 4& \\
      &79	&0& 0& 0& 2& 2& 2& 4& 0& 4& 6& 4\\
      &95	&0& 2& 2& 2& 2& 4& 2& 4& 6& 8& 12\\
      &111	&2& 2& 2& 6& 6& 6& 8& 6& 10& 18& 14	\\
      &127	&0& 4& 4& 6&10&10&6&12&18& 26& 30\\
&143	&2& 6& 6& 14& 14& 18& 18& 10& 32& 46& 40\\
&159	&4& 10& 10& 18& 24& 26& 20& 26& 44& 68& 76\\
\bottomrule
\end{tabular}\end{center}
 \begin{center}
     \caption{Decomposition of $K^{(4)}_3$}
    \vspace{0.2cm}

 \begin{tabular}{cc|RRRRRRRRRRRR}
 \toprule  && \chi_{1}& \chi_{2}& \chi_{3}& \chi_{4}& \chi_{5}& \chi_{6}& \chi_{7}& \chi_{8}& \chi_{9}&\chi_{10}& \chi_{11} \\\midrule
 & 7	&0& 1& 1& 0& 0& 0& 0& 0& 0& 0& 0\\
&23	&0& 0& 0& 0& 0& 0& 0& 0& 2& 0& 0\\
&39	&0& 0& 0& 0& 2& 0& 0& 0& 0& 2& 0\\
&55	&0& 0& 0& 2& 0& 0& 0& 2& 2& 2& 2\\
&71	&0& 2& 2& 0& 2& 2& 0& 0& 2& 4& 4\\
&87	&2& 0& 0& 2& 2& 2& 4& 4& 6& 6& 8\\
&103	& 0& 2& 2& 2& 6& 6& 4& 2& 6& 14& 12\\
&119	& 2& 2& 2& 8& 4& 8& 6& 8& 18& 20& 22\\
&135	&2& 8& 8& 8& 14& 14& 12& 10& 20& 36& 34\\
&151	&4& 6& 6& 18& 16& 20& 20& 22& 42& 54& 56\\\bottomrule
\end{tabular}\end{center}

 \begin{center}
  \caption{Decomposition of $K^{(4)}_2$}
    \vspace{0.2cm}

 \begin{tabular}{c|RRRRR}
\toprule & \chi_{12}& \chi_{13}& \chi_{14}& \chi_{15}& \chi_{16} \\ \midrule
12 &0& 1& 1& 0& 0\\
28& 0& 0& 0& 1& 1\\
44& 2& 0& 0& 2& 2\\
60& 0& 2& 2& 4& 4\\
76&2& 2& 2& 8& 8\\
92&6& 4& 4& 14& 14\\
108&6& 9& 9& 24& 24\\
124&14& 14& 14& 40& 40\\
140&24& 20& 20& 66& 66\\
156&32& 36& 36& 104& 104
\\\bottomrule\end{tabular}\end{center}
\end{table}

\begin{table} \vspace{-2pt}  \subsection{Lambency Five}    \vspace{-16pt}\begin{minipage}[t]{0.5\linewidth}
 \begin{center}
     \caption{Decomposition of $K^{(5)}_1$}
    \vspace{0.2cm}

 \begin{tabular}{cc|RRRRRRRRRRRR}
 \toprule  && \chi_{1}& \chi_{2}& \chi_{3}& \chi_{4}& \chi_{5}& \chi_{6}& \chi_{7} \\\midrule
 &-1& -2& 0& 0& 0& 0& 0& 0\\ 
     &19& 0& 0& 2& 0& 0& 0& 0\\ 
     &39& 0& 0& 0& 2& 0& 2& 0\\ 
     &59& 0& 0& 2& 0& 2& 2& 2\\ 
     &79& 0& 2& 2& 4& 2& 2& 4\\ 
     &99& 2& 0& 6& 2& 6& 4& 6\\ 
     &119& 0& 2& 6& 8& 8& 10& 10\\ 
     &139& 4& 2& 14& 10& 14& 12& 16\\ 
     &159& 2& 6& 14& 20& 20& 22& 26\\ 
     &179& 8&4& 28& 22& 36& 32& 40\\ 
\bottomrule
\end{tabular}\end{center}
 \begin{center}
     \caption{Decomposition of $K^{(5)}_3$}
    \vspace{0.2cm}

 \begin{tabular}{cc|RRRRRRRRRRRR}
 \toprule  && \chi_{1}& \chi_{2}& \chi_{3}& \chi_{4}& \chi_{5}& \chi_{6}& \chi_{7} \\\midrule
 &11& 0& 0& 0& 2& 0& 0& 0\\ 
  &31& 0& 0& 0& 0& 2& 0& 2\\ 
  &51& 0& 0& 2& 2& 2& 2& 2\\ 
  &71& 2& 0& 4& 2& 4& 4& 4\\ 
  &91& 0& 2& 4& 6& 6& 8& 8\\ 
  &111& 2& 2& 10& 8& 12& 10& 14\\ 
  &131& 4& 4& 14& 16& 18& 18& 22\\ 
  &151& 8& 4& 24& 22& 30& 30& 34\\ 
  &171& 8& 10& 34& 38& 44& 46& 54\\ 
  &191& 14& 14& 58& 52& 68& 64& 82\\\bottomrule
\end{tabular}\end{center}
\end{minipage}
\begin{minipage}[t]{0.5\linewidth}
 \begin{center}
  \caption{Decomposition of $K^{(5)}_2$}
    \vspace{0.2cm}

 \begin{tabular}{c|RRRRRRR}
\toprule & \chi_{8}& \chi_{9}&\chi_{10}& \chi_{11}& \chi_{12}& \chi_{13}& \chi_{14} \\ \midrule
16& 0& 0& 0& 0& 1& 1& 0\\ 
36& 0& 0& 1& 1& 1& 1& 2\\ 
 56&  2& 2& 2& 2& 2& 2& 2\\
76&  0& 0& 4& 4& 4& 4& 6\\ 
 96& 2& 2& 6& 6& 8&8& 8\\ 
 116& 2& 2& 10& 10& 12& 12& 18\\ 
 136& 4& 4& 16& 16& 22& 22& 22\\ 
 156& 6& 6& 26& 26& 32& 32& 42\\ 
 176&12& 12& 40& 40& 50& 50& 56\\
196& 13& 13& 60& 60& 74& 74& 94
\\\bottomrule\end{tabular}\end{center}
 \begin{center}
  \caption{Decomposition of $K^{(5)}_4$}
    \vspace{0.2cm}

 \begin{tabular}{c|RRRRRRR}
\toprule & \chi_{8}& \chi_{9}&\chi_{10}& \chi_{11}& \chi_{12}& \chi_{13}& \chi_{14} \\ \midrule
4& 1& 1& 0& 0& 0& 0& 0\\ 
  24& 0& 0& 0& 0& 0& 0& 2\\ 
44& 0& 0& 0& 0& 2& 2& 0\\ 
64& 0& 0& 2& 2& 1& 1& 4\\ 
84& 2& 2& 2& 2& 4& 4& 2\\ 
104& 0& 0& 4& 4& 6& 6& 10\\ 
124& 4& 4& 8& 8& 10& 10& 8\\ 
144&  1& 1& 13& 13& 14& 14& 22\\ 
164&  6& 6& 18& 18& 26& 26& 24\\ 
184&6& 6& 30& 30& 34& 34& 50
\\\bottomrule\end{tabular}\end{center}
\end{minipage}
\end{table}

\begin{table} \vspace{-2pt}  \subsection{Lambency Seven}    \vspace{-16pt}\begin{minipage}[t]{0.5\linewidth}
 \begin{center}
     \caption{Decomposition of $K^{(7)}_1$}
    \vspace{0.2cm}

 \begin{tabular}{cc|RRRR}
 \toprule  && \chi_{1}& \chi_{2}& \chi_{3}& \chi_{4} \\\midrule
 &-1& -2& 0& 0& 0\\
 &27& 2& 1& 1& 0\\ 
&55& 0& 0& 0& 2\\ 
&83& 0& 2& 2& 2\\ 
&111&  2& 0& 0& 6
\\ &139& 4& 4& 4& 6
\\ &167& 2& 2& 2& 12\\ 
&195& 8& 6& 6& 16\\
 &223& 6& 6& 6& 26\\ 
&251&12& 14& 14& 30\\\bottomrule
\end{tabular}\end{center}
 \begin{center}
     \caption{Decomposition of $K^{(7)}_3$}
    \vspace{0.2cm}

 \begin{tabular}{cc|RRRR}
 \toprule  && \chi_{1}& \chi_{2}& \chi_{3}& \chi_{4} \\\midrule
&19& 0& 0& 0& 2\\ 
&47& 2& 2& 2& 2\\ 
&75& 2& 1& 1& 6\\ 
&103& 4& 4& 4& 8\\ 
&131& 2& 4& 4& 16\\ 
&159& 8& 8& 8& 22\\ 
&187& 10& 10& 10& 34\\ 
&215& 16& 18& 18& 46\\ 
&243& 22& 21& 21& 70\\ 
&271& 34& 34& 34& 94\\ \bottomrule
\end{tabular}\end{center}
 \begin{center}
     \caption{Decomposition of $K^{(7)}_5$}
    \vspace{0.2cm}

 \begin{tabular}{c|RRRR}
 \toprule  & \chi_{1}& \chi_{2}& \chi_{3}& \chi_{4} \\\midrule
 3& 0& 1& 1& 0\\ 
 31& 0& 0& 0& 2\\ 
 59& 2& 0& 0& 4\\ 
87& 0& 2& 2& 6\\ 
115& 4& 4& 4& 8\\ 
 143& 6& 4& 4& 14\\ 
171& 6& 8& 8& 20\\ 
 199& 10& 10& 10& 32\\
227& 18& 16& 16& 44\\ 
 255& 18& 20& 20& 64\\
 \bottomrule
\end{tabular}\end{center}
\end{minipage}
\begin{minipage}[t]{0.5\linewidth}
 \begin{center}
     \caption{Decomposition of $K^{(7)}_2$}
    \vspace{0.2cm}

 \begin{tabular}{cc|RRR}
 \toprule  && \chi_{5}& \chi_{6}& \chi_{7} \\\midrule
   &24& 2& 0& 0\\ 
&52& 2& 2& 2\\ 
&80& 2& 4& 4\\
&108& 6& 5& 5\\
&136& 8& 8& 8\\
&164& 12& 14& 14\\
&192&20& 17& 17\\
&220& 28& 28& 28\\
&248&  36& 40& 40\\
&276&  56& 54& 54   \\\bottomrule
\end{tabular}\end{center}
 \begin{center}
     \caption{Decomposition of $K^{(7)}_4$}
    \vspace{0.2cm}

  \begin{tabular}{c|RRR}
 \toprule  & \chi_{5}& \chi_{6}& \chi_{7} \\\midrule
12& 0& 1& 1\\
 40& 2& 2& 2\\ 
 68& 4& 2& 2\\ 
 96& 6& 6& 6\\ 
124& 8& 8& 8\\ 
152& 14& 14& 14\\
 180& 18& 20& 20\\ 
208 &30&  30& 30\\
 236& 42& 40& 40\\ 
 264&  60& 60& 60  \\\bottomrule
\end{tabular}\end{center}
 \begin{center}
     \caption{Decomposition of $K^{(7)}_6$}
    \vspace{0.2cm}

 \begin{tabular}{c|RRR}
 \toprule  & \chi_{5}& \chi_{6}& \chi_{7} \\\midrule
 20& 2& 0& 0\\ 
 48&  0& 1& 1\\ 
76&  2& 2& 2\\ 
 104&  2& 2& 2\\ 
 132& 4& 6& 6\\ 
 160& 6& 6& 6\\ 
188&12& 10& 10\\ 
216& 12& 14& 14\\ 
 244& 22& 22& 22\\ 
 272&  28& 26& 26\\ \bottomrule \end{tabular}\end{center}
\end{minipage}\end{table}

\begin{table}\vspace{-2pt}  \subsection{Lambency Thirteen}    \vspace{-16pt}
\begin{minipage}[t]{0.5\linewidth}
 \begin{center}
     \caption{Decomposition of $K^{(13)}_1$}
    \vspace{0.2cm}
 \begin{tabular}{cc|RR}
 \toprule  && \chi_{1}& \chi_{2} \\\midrule
& -1& -2& 0\\
&51& 2& 0\\ 
&103& 0& 2\\\ 
&155& 0& 0\\ 
&207& 0& 2\\ 
&259& 2& 0\\ 
&311& 2& 2\\ 
&363& 4& 2\\ 
&415& 2& 4\\ 
&467& 6& 2 \\\bottomrule
\end{tabular}\end{center}
 \begin{center}
     \caption{Decomposition of $K^{(13)}_3$}
    \vspace{0.2cm}

 \begin{tabular}{cc|RR}
 \toprule  && \chi_{1}& \chi_{2} \\\midrule
 &43& 0& 2\\
&95&2&0\\
&147& 2& 2\\ 
&199& 4& 2\\ 
&251& 2& 6\\ 
&303& 6& 4\\ 
&355& 6& 8\\ 
&407& 10& 8\\ 
&459& 10& 12\\ 
&511& 14& 12 \\\bottomrule
\end{tabular}\end{center}
 \begin{center}
     \caption{Decomposition of $K^{(13)}_5$}
    \vspace{0.2cm}
 \begin{tabular}{c|RR}
 \toprule  & \chi_{1}& \chi_{2} \\\midrule
 27& 2& 0\\79& 2& 2\\131& 4& 2\\183& 2& 4\\235& 6& 4\\287& 6& 8\\339& 8& 8\\391& 10& 12\\443& 16& 14\\495& 16& 20\\
 \bottomrule
\end{tabular}\end{center}
\end{minipage}
\begin{minipage}[t]{0.5\linewidth}
 \begin{center}
     \caption{Decomposition of $K^{(13)}_2$}
    \vspace{0.2cm}
 \begin{tabular}{c|RRR}
 \toprule  & \chi_{3}& \chi_{4} \\\midrule
48&  0& 0\\
 100& 1& 1\\
 152& 2& 2\\
 204& 2& 2\\
 256& 3& 3\\
 308& 4& 4\\
 360& 4& 4\\
 412& 6& 6\\
 464& 8& 8\\
 516& 10& 10
  \\\bottomrule
\end{tabular}\end{center}
 \begin{center}
     \caption{Decomposition of $K^{(13)}_4$}
    \vspace{0.2cm}

  \begin{tabular}{c|RR}
   \toprule & \chi_{3}& \chi_{4} \\\midrule
    36& 1& 1\\
   88& 2& 2\\
   140& 2& 2\\
   192& 4& 4\\
   244& 4& 4\\
   296& 6& 6\\
   348& 8& 8\\
   400& 11& 11\\
   452& 12& 12\\
   504& 18& 18
 \\\bottomrule
\end{tabular}\end{center}
 \begin{center}
     \caption{Decomposition of $K^{(13)}_6$}
    \vspace{0.2cm}

 \begin{tabular}{c|RR}
 \toprule  & \chi_{3}& \chi_{4}  \\\midrule
16& 1& 1\\68& 2& 2\\120& 2& 2\\172& 4& 4\\224& 4& 4\\276& 8& 8\\328& 8& 8\\380& 12& 12\\
432& 14& 14\\484& 19& 19
\\\bottomrule   \end{tabular}\end{center}
\end{minipage}\end{table}

\begin{table}
\begin{minipage}[t]{0.5\linewidth}
 \begin{center}
     \caption{Decomposition of $K^{(13)}_7$}
    \vspace{0.2cm}
 \begin{tabular}{c|RR} \toprule  & \chi_{1}& \chi_{2} \\\midrule
 3& 0& 2\\55& 2& 0\\107& 2& 2\\159& 4& 4\\211& 4& 6\\263& 6& 6\\315& 8& 8\\367& 12& 10\\419& 12& 14\\471& 18& 16 
 \\\bottomrule
\end{tabular}\end{center}
 \begin{center}
     \caption{Decomposition of $K^{(13)}_9$}
    \vspace{0.2cm}

 \begin{tabular}{c|RR}
 \toprule  & \chi_{1}& \chi_{2} \\\midrule
 23& 0& 2\\ 75& 2& 2\\127& 2& 2\\179& 4& 2\\231& 4& 4\\283& 6& 6\\335& 6& 8\\387& 10& 10\\439& 12& 14\\491& 16& 14
\\\bottomrule
\end{tabular}\end{center}
 \begin{center}
     \caption{Decomposition of $K^{(13)}_{11}$}
    \vspace{0.2cm}
 \begin{tabular}{c|RR}
 \toprule  & \chi_{1}& \chi_{2} \\\midrule
 35& 2& 0\\ 87& 2& 0\\139& 0& 2\\191& 2& 2\\243& 2& 2\\295& 4& 4\\347& 4& 6\\399& 6& 4\\451& 8& 8\\503& 10& 10\\
 \bottomrule
\end{tabular}\end{center}
\end{minipage}
\begin{minipage}[t]{0.5\linewidth}
 \begin{center}
     \caption{Decomposition of $K^{(13)}_8$}
    \vspace{0.2cm}
 \begin{tabular}{c|RRR}
 \toprule  & \chi_{3}& \chi_{4} \\\midrule
  40& 2& 2\\ 92& 2& 2\\144& 3& 3\\196& 3& 3\\248& 6& 6\\300& 6& 6\\352& 10& 10\\ 
 404& 12& 12\\456& 16& 16\\508& 18& 18
   \\\bottomrule
\end{tabular}\end{center}
 \begin{center}
     \caption{Decomposition of $K^{(13)}_{10}$}
    \vspace{0.2cm}

  \begin{tabular}{c|RR}
   \toprule & \chi_{3}& \chi_{4} \\\midrule
   4& 1&   1\\ 56&0&0\\ 108&  2& 2\\ 160&  2& 2\\ 212& 4& 4\\   264& 4& 4\\ 316& 6& 6\\ 368& 6& 6\\ 420& 10& 10\\ 472& 10& 10\\
   \bottomrule
   \end{tabular}\end{center}
 \begin{center}
     \caption{Decomposition of $K^{(13)}_{12}$}
    \vspace{0.2cm}

 \begin{tabular}{c|RR}
 \toprule  & \chi_{3}& \chi_{4}  \\\midrule
12& 0&0\\
 64& 1& 1\\
116& 0&0\\ 
 168& 2& 2\\ 
 220& 0&0\\ 
 272& 2& 2\\ 
324& 1& 1\\ 
 376& 4& 4\\ 
 428& 2& 2\\ 
 480& 6& 6\\\bottomrule
   \end{tabular}\end{center}
\end{minipage}\end{table}

\clearpage

\addcontentsline{toc}{section}{References}

\begin{thebibliography}{100}

\bibitem{conway_norton}
J.~H. Conway and S.~P. Norton, ``{Monstrous Moonshine},'' {\em Bull. London
  Math. Soc.} {\bf 11} (1979)  308~339.

\bibitem{Tho_NmrlgyMonsEllModFn}
J.~G. Thompson, ``Some numerology between the {F}ischer-{G}riess {M}onster and
  the elliptic modular function,'' {\em Bull. London Math. Soc.} {\bf 11}
  (1979) no.~3, 352--353.

\bibitem{Tho_FinGpsModFns}
J.~G. Thompson, ``Finite groups and modular functions,'' {\em Bull. London
  Math. Soc.} {\bf 11} (1979) no.~3, 347--351.

\bibitem{MR604633}
P.~Fong, ``Characters arising in the {M}onster-modular connection,'' in {\em
  The {S}anta {C}ruz {C}onference on {F}inite {G}roups ({U}niv. {C}alifornia,
  {S}anta {C}ruz, {C}alif., 1979)}, vol.~37 of {\em Proc. Sympos. Pure Math.},
  pp.~557--559.
\newblock Amer. Math. Soc., Providence, R.I., 1980.

\bibitem{MR822245}
S.~D. Smith, \href{http://dx.doi.org/10.1090/conm/045/822245}{``On the head
  characters of the {M}onster simple group,''} in {\em Finite groups---coming
  of age ({M}ontreal, {Q}ue., 1982)}, vol.~45 of {\em Contemp. Math.},
  pp.~303--313.
\newblock Amer. Math. Soc., Providence, RI, 1985.
\newblock \url{http://dx.doi.org/10.1090/conm/045/822245}.

\bibitem{FLMPNAS}
I.~B. Frenkel, J.~Lepowsky, and A.~Meurman, ``A natural representation of the
  {F}ischer-{G}riess {M}onster with the modular function {$J$} as character,''
  {\em Proc. Nat. Acad. Sci. U.S.A.} {\bf 81} (1984) no.~10, Phys. Sci.,
  3256--3260.

\bibitem{FLMBerk}
I.~B. Frenkel, J.~Lepowsky, and A.~Meurman, ``A moonshine module for the
  {M}onster,'' in {\em Vertex operators in mathematics and physics (Berkeley,
  Calif., 1983)}, vol.~3 of {\em Math. Sci. Res. Inst. Publ.}, pp.~231--273.
\newblock Springer, New York, 1985.

\bibitem{FreKac_AffLieDualRes}
I.~B. Frenkel and V.~G. Kac, ``Basic representations of affine {L}ie algebras
  and dual resonance models,'' {\em Invent. Math.} {\bf 62} (1980/81) no.~1,
  23--66.

\bibitem{Seg_UtyRpsInfDimGps}
G.~Segal, ``Unitary representations of some infinite-dimensional groups,'' {\em
  Comm. Math. Phys.} {\bf 80} (1981) no.~3, 301--342.
  \url{http://projecteuclid.org/getRecord?id=euclid.cmp/1103919978}.

\bibitem{Gri_FG}
R.~L. Griess, Jr., ``The friendly giant,'' {\em Invent. Math.} {\bf 69} (1982)
  no.~1, 1--102.

\bibitem{Bor_PNAS}
R.~Borcherds, ``Vertex algebras, {Kac}-{Moody} algebras, and the {Monster},''
  {\em Proceedings of the National Academy of Sciences, U.S.A.} {\bf 83} (1986)
  no.~10, 3068--3071.

\bibitem{FLM}
I.~Frenkel, J.~Lepowsky, and A.~Meurman, {\em Vertex operator algebras and the
  {M}onster}, vol.~134 of {\em Pure and Applied Mathematics}.
\newblock Academic Press Inc., Boston, MA, 1988.

\bibitem{Bor_GKM}
R.~Borcherds, ``Generalized {K}ac-{M}oody algebras,''
  \href{http://dx.doi.org/10.1016/0021-8693(88)90275-X}{{\em J. Algebra} {\bf
  115} (1988) no.~2, 501--512}.
  \url{http://dx.doi.org/10.1016/0021-8693(88)90275-X}.

\bibitem{borcherds_monstrous}
R.~E. Borcherds, ``{Monstrous moonshine and monstrous Lie superalgebras},''
  {\em Invent. Math.} {\bf 109, No.2} (1992)  405--444.

\bibitem{Eguchi2010}
T.~Eguchi, H.~Ooguri, and Y.~Tachikawa, ``{Notes on the K3 Surface and the
  Mathieu group $M_{24}$},'' {\em Exper.Math.} {\bf 20} (2011)  91--96,
\href{http://arxiv.org/abs/1004.0956}{{\tt arXiv:1004.0956 [hep-th]}}.

\bibitem{Eguchi1987}
T.~Eguchi and A.~Taormina, ``Unitary representations of the {$N=4$}
  superconformal algebra,''
  \href{http://dx.doi.org/10.1016/0370-2693(87)91679-0}{{\em Phys. Lett. B}
  {\bf 196} (1987) no.~1, 75--81}.
  \url{http://dx.doi.org/10.1016/0370-2693(87)91679-0}.

\bibitem{Eguchi1988}
T.~Eguchi and A.~Taormina, ``Character formulas for the {$N=4$} superconformal
  algebra,'' \href{http://dx.doi.org/10.1016/0370-2693(88)90778-2}{{\em Phys.
  Lett. B} {\bf 200} (1988) no.~3, 315--322}.
  \url{http://dx.doi.org/10.1016/0370-2693(88)90778-2}.

\bibitem{MR2060475}
V.~G. Kac and M.~Wakimoto, ``Quantum reduction and representation theory of
  superconformal algebras,''
  \href{http://dx.doi.org/10.1016/j.aim.2003.12.005}{{\em Adv. Math.} {\bf 185}
  (2004) no.~2, 400--458}. \url{http://dx.doi.org/10.1016/j.aim.2003.12.005}.

\bibitem{Cheng2010_1}
M.~C.~N. Cheng, ``{$K3$} {S}urfaces, {$N=4$} {D}yons, and the {M}athieu {G}roup
  {$M_{24}$},'' \href{http://arxiv.org/abs/1005.5415}{{\tt 1005.5415}}.

\bibitem{Gaberdiel2010}
M.~R. Gaberdiel, S.~Hohenegger, and R.~Volpato, ``{Mathieu twining characters
  for K3},'' \href{http://dx.doi.org/10.1007/JHEP09(2010)058}{{\em JHEP} {\bf
  1009} (2010)  058}, \href{http://arxiv.org/abs/1006.0221}{{\tt
  arXiv:1006.0221 [hep-th]}}.
19 pages.

\bibitem{Gaberdiel2010a}
M.~R. Gaberdiel, S.~Hohenegger, and R.~Volpato, ``{Mathieu Moonshine in the
  elliptic genus of K3},''
  \href{http://dx.doi.org/10.1007/JHEP10(2010)062}{{\em JHEP} {\bf 1010} (2010)
   062},
\href{http://arxiv.org/abs/1008.3778}{{\tt arXiv:1008.3778 [hep-th]}}.

\bibitem{Eguchi2010a}
T.~Eguchi and K.~Hikami, ``{Note on Twisted Elliptic Genus of K3 Surface},''
  \href{http://dx.doi.org/10.1016/j.physletb.2010.10.017}{{\em Phys.Lett.} {\bf
  B694} (2011)  446--455},
\href{http://arxiv.org/abs/1008.4924}{{\tt arXiv:1008.4924 [hep-th]}}.

\bibitem{Gannon:2012ck}
T.~Gannon, ``{Much ado about Mathieu},''
\href{http://arxiv.org/abs/1211.5531}{{\tt arXiv:1211.5531 [math.RT]}}.

\bibitem{Cheng2011}
M.~C.~N. Cheng and J.~F.~R. Duncan, ``{On Rademacher Sums, the Largest Mathieu
  Group, and the Holographic Modularity of Moonshine},'' {\em Commun. Number
  Theory Phys.} {\bf 6} (2012) no.~3, ,
  \href{http://arxiv.org/abs/1110.3859}{{\tt 1110.3859}}.

\bibitem{DunFre_RSMG}
J.~F.~R. Duncan and I.~B. Frenkel, ``Rademacher sums, moonshine and gravity,''
  {\em Commun. Number Theory Phys.} {\bf 5} (2011) no.~4, 1--128.

\bibitem{Rad_FuncEqnModInv}
H.~Rademacher, ``The {F}ourier {S}eries and the {F}unctional {E}quation of the
  {A}bsolute {M}odular {I}nvariant {J}({$\tau$}),'' {\em Amer. J. Math.} {\bf
  61} (1939) no.~1, 237--248.

\bibitem{Witten2007}
E.~Witten, ``{Three-Dimensional Gravity Revisited},''
  \href{http://arxiv.org/abs/0706.3359}{{\tt 0706.3359}}.

\bibitem{MalWit_QGPtnFn3D}
A.~Maloney and E.~Witten, ``Quantum gravity partition functions in three
  dimensions,'' \href{http://dx.doi.org/10.1007/JHEP02(2010)029}{{\em J. High
  Energy Phys.} (2010) no.~2, 029, 58}.
  \url{http://dx.doi.org/10.1007/JHEP02(2010)029}.

\bibitem{LiSonStr_ChGrav3D}
W.~Li, W.~Song, and A.~Strominger, ``Chiral gravity in three dimensions,''
  \href{http://dx.doi.org/10.1088/1126-6708/2008/04/082}{{\em J. High Energy
  Phys.} (2008) no.~4, 082, 15}.
  \url{http://dx.doi.org/10.1088/1126-6708/2008/04/082}.

\bibitem{MaldacenaAdv.Theor.Math.Phys.2:231-2521998}
J.~M. Maldacena, ``The large {N} limit of superconformal field theories and
  supergravity,'' {\em Adv.Theor.Math.Phys.} {\bf 2} (1998)  231--252,
  \href{http://arxiv.org/abs/hep-th/9711200}{{\tt hep-th/9711200}}.

\bibitem{Gubser:1998bc}
S.~Gubser, I.~R. Klebanov, and A.~M. Polyakov, ``{Gauge theory correlators from
  noncritical string theory},''
  \href{http://dx.doi.org/10.1016/S0370-2693(98)00377-3}{{\em Phys.Lett.} {\bf
  B428} (1998)  105--114},
\href{http://arxiv.org/abs/hep-th/9802109}{{\tt arXiv:hep-th/9802109
  [hep-th]}}.

\bibitem{Witten:1998qj}
E.~Witten, ``{Anti-de Sitter space and holography},'' {\em
  Adv.Theor.Math.Phys.} {\bf 2} (1998)  253--291,
\href{http://arxiv.org/abs/hep-th/9802150}{{\tt arXiv:hep-th/9802150
  [hep-th]}}.

\bibitem{Dijkgraaf2007}
R.~Dijkgraaf, J.~Maldacena, G.~Moore, and E.~Verlinde, ``A black hole farey
  tail,'' \href{http://arxiv.org/abs/hep-th/0005003}{{\tt hep-th/0005003}}.

\bibitem{Moore2007}
G.~W. Moore, ``Les {H}ouches {L}ectures on {S}trings and {A}rithmetic,''
  \href{http://arxiv.org/abs/hep-th/0401049}{{\tt hep-th/0401049}}.

\bibitem{BoerJHEP0611:0242006}
J.~de~Boer, M.~C.~N. Cheng, R.~Dijkgraaf, J.~Manschot, and E.~Verlinde, ``A
  farey tail for attractor black holes,'' {\em JHEP} {\bf 0611:024,2006}
  (JHEP0611:024,2006)  , \href{http://arxiv.org/abs/hep-th/0608059}{{\tt
  hep-th/0608059}}.

\bibitem{KrausJHEP0701:0022007}
P.~Kraus and F.~Larsen, ``Partition functions and elliptic genera from
  supergravity,'' {\em JHEP} {\bf 0701:002,2007} (JHEP 0701:002,2007)  ,
  \href{http://arxiv.org/abs/hep-th/0607138}{{\tt hep-th/0607138}}.

\bibitem{Denef2007}
F.~Denef and G.~W. Moore, ``{Split states, entropy enigmas, holes and halos},''
  \href{http://dx.doi.org/10.1007/JHEP11(2011)129}{{\em JHEP} {\bf 1111} (2011)
   129}, \href{http://arxiv.org/abs/hep-th/0702146}{{\tt arXiv:hep-th/0702146
  [HEP-TH]}}.
149 pages, 21 figures.

\bibitem{Sen2007c}
A.~Sen, ``Black hole entropy function, attractors and precision counting of
  microstates,''{\em Gen.Rel.Grav.} {\bf 40:2249-2431,2008} (Aug., 2007)  ,
  \href{http://arxiv.org/abs/0708.1270}{{\tt 0708.1270}}.

\bibitem{Dabholkar:2011ec}
A.~Dabholkar, J.~Gomes, and S.~Murthy, ``{Localization and Exact Holography},''
\href{http://arxiv.org/abs/1111.1161}{{\tt arXiv:1111.1161 [hep-th]}}.

\bibitem{Manschot2007}
J.~Manschot and G.~W. Moore, ``A modern fareytail,''{\em
  Commun.Num.Theor.Phys.} {\bf 4,} (Dec., 2007)  103--159,2010,
  \href{http://arxiv.org/abs/0712.0573}{{\tt 0712.0573}}.

\bibitem{MR2985326}
M.~C.~N. Cheng and J.~F.~R. Duncan,
  \href{http://dx.doi.org/10.1090/pspum/085/1374}{``The largest {M}athieu group
  and (mock) automorphic forms,''} in {\em String-{M}ath 2011}, vol.~85 of {\em
  Proc. Sympos. Pure Math.}, pp.~53--82.
\newblock Amer. Math. Soc., Providence, RI, 2012.
\newblock \url{http://dx.doi.org/10.1090/pspum/085/1374}.

\bibitem{MR2176151}
J.~S. Ellenberg, ``On the error term in {D}uke's estimate for the average
  special value of {$L$}-functions,''
  \href{http://dx.doi.org/10.4153/CMB-2005-049-8}{{\em Canad. Math. Bull.} {\bf
  48} (2005) no.~4, 535--546}. \url{http://dx.doi.org/10.4153/CMB-2005-049-8}.

\bibitem{Gri_EllGenCYMnflds}
V.~Gritsenko, ``Elliptic genus of {C}alabi-{Y}au manifolds and {J}acobi and
  {S}iegel modular forms,'' {\em Algebra i Analiz} {\bf 11} (1999) no.~5,
  100--125.

\bibitem{zwegers}
S.~Zwegers, {\em {Mock Theta Functions}}.
\newblock PhD thesis, Utrecht University, 2002.

\bibitem{Dabholkar:2012nd}
A.~Dabholkar, S.~Murthy, and D.~Zagier, ``{Quantum Black Holes, Wall Crossing,
  and Mock Modular Forms},''
\href{http://arxiv.org/abs/1208.4074}{{\tt arXiv:1208.4074 [hep-th]}}.

\bibitem{BringmannOno2006}
K.~Bringmann and K.~Ono, ``The {$f(q)$} mock theta function conjecture and
  partition ranks,'' \href{http://dx.doi.org/10.1007/s00222-005-0493-5}{{\em
  Invent. Math.} {\bf 165} (2006) no.~2, 243--266}.
  \url{http://dx.doi.org/10.1007/s00222-005-0493-5}.

\bibitem{BringmannOno2010}
K.~Bringmann and K.~Ono, ``Dyson's ranks and {M}aass forms,''
  \href{http://dx.doi.org/10.4007/annals.2010.171.419}{{\em Ann. of Math. (2)}
  {\bf 171} (2010) no.~1, 419--449}.
  \url{http://dx.doi.org/10.4007/annals.2010.171.419}.

\bibitem{Eguchi2008}
T.~Eguchi and K.~Hikami, ``Superconformal algebras and mock theta
  functions,''{\em J.Phys.A} {\bf 42:304010,2009} (Dec., 2008)  ,
  \href{http://arxiv.org/abs/0812.1151}{{\tt 0812.1151}}.

\bibitem{Eguchi2009a}
T.~Eguchi and K.~Hikami, ``{Superconformal Algebras and Mock Theta Functions 2.
  Rademacher Expansion for K3 Surface},''{\em Communications in Number Theory
  and Physics} {\bf 3,} (Apr., 2009)  531--554,
  \href{http://arxiv.org/abs/0904.0911}{{\tt 0904.0911}}.

\bibitem{VafaWitten1994}
C.~Vafa and E.~Witten, ``A strong coupling test of {$S$}-duality,''
  \href{http://dx.doi.org/10.1016/0550-3213(94)90097-3}{{\em Nuclear Phys. B}
  {\bf 431} (1994) no.~1-2, 3--77}.
  \url{http://dx.doi.org/10.1016/0550-3213(94)90097-3}.

\bibitem{Gottsche_Zagier}
L.~G{\"o}ttsche and D.~Zagier, ``{J}acobi forms and the structure of
  {D}onaldson invariants for {$4$}-manifolds with {$b_+=1$},''
  \href{http://dx.doi.org/10.1007/s000290050025}{{\em Selecta Math. (N.S.)}
  {\bf 4} (1998) no.~1, 69--115}.
  \url{http://dx.doi.org/10.1007/s000290050025}.

\bibitem{Troost:2010ud}
J.~Troost, ``{The non-compact elliptic genus: mock or modular},''
  \href{http://dx.doi.org/10.1007/JHEP06(2010)104}{{\em JHEP} {\bf 1006} (2010)
   104},
\href{http://arxiv.org/abs/1004.3649}{{\tt arXiv:1004.3649 [hep-th]}}.

\bibitem{Manschot2011}
J.~Manschot, ``The {B}etti numbers of the moduli space of stable sheaves of
  rank 3 on {$\mathbb P^2$},''
  \href{http://dx.doi.org/10.1007/s11005-011-0490-0}{{\em Lett. Math. Phys.}
  {\bf 98} (2011) no.~1, 65--78}.
  \url{http://dx.doi.org/10.1007/s11005-011-0490-0}.

\bibitem{Lawrence_Zagier}
R.~Lawrence and D.~Zagier, ``Modular forms and quantum invariants of
  {$3$}-manifolds,'' {\em Asian J. Math.} {\bf 3} (1999) no.~1, 93--107. Sir
  Michael Atiyah: a great mathematician of the twentieth century.

\bibitem{zagier_mock}
D.~Zagier, ``Ramanujan's mock theta functions and their applications (after
  {Z}wegers and {O}no-{B}ringmann),'' {\em Ast\'erisque} (2009) no.~326, Exp.
  No. 986, vii--viii, 143--164 (2010). S{\'e}minaire Bourbaki. Vol. 2007/2008.

\bibitem{Folsom_what}
A.~Folsom, ``What is {$\dots$} a mock modular form?,'' {\em Notices Amer. Math.
  Soc.} {\bf 57} (2010) no.~11, 1441--1443.

\bibitem{Ono_unearthing}
K.~Ono, ``Unearthing the visions of a master: harmonic {M}aass forms and number
  theory,'' in {\em Current developments in mathematics, 2008}, pp.~347--454.
\newblock Int. Press, Somerville, MA, 2009.

\bibitem{Ramanujan_lost}
S.~Ramanujan, {\em The lost notebook and other unpublished papers}.
\newblock Springer-Verlag, Berlin, 1988.
\newblock With an introduction by George E. Andrews.

\bibitem{Gordon_Mcintosh}
B.~Gordon and R.~J. McIntosh, ``Some eighth order mock theta functions,''
  \href{http://dx.doi.org/10.1112/S0024610700008735}{{\em J. London Math. Soc.
  (2)} {\bf 62} (2000) no.~2, 321--335}.
  \url{http://dx.doi.org/10.1112/S0024610700008735}.

\bibitem{eichler_zagier}
M.~Eichler and D.~Zagier, {\em {The theory of Jacobi forms}}.
\newblock Birkh{\"a}user, 1985.

\bibitem{Gri_CplxVBsJacFrms}
V.~Gritsenko, ``Complex vector bundles and {J}acobi forms,'' {\em
  S\=urikaisekikenky\=usho K\=oky\=uroku} (1999) no.~1103, 71--85. Automorphic
  forms and $L$-functions (Kyoto, 1999).

\bibitem{Sko_Thesis}
N.-P. Skoruppa, {\em {\"U}ber den {Z}usammenhang zwischen {J}acobiformen und
  {M}odulformen halbganzen {G}ewichts}.
\newblock Bonner Mathematische Schriften [Bonn Mathematical Publications], 159.
  Universit{\"a}t Bonn Mathematisches Institut, Bonn, 1985.
\newblock Dissertation, Rheinische Friedrich-Wilhelms-Universit{\"a}t, Bonn,
  1984.

\bibitem{MR1193029}
A.~W. Knapp, {\em Elliptic curves}, vol.~40 of {\em Mathematical Notes}.
\newblock Princeton University Press, Princeton, NJ, 1992.

\bibitem{Shi_IntThyAutFns}
G.~Shimura, {\em Introduction to the arithmetic theory of automorphic
  functions}.
\newblock Publications of the Mathematical Society of Japan, No. 11. Iwanami
  Shoten, Publishers, Tokyo, 1971.
\newblock Kan{\^o} Memorial Lectures, No. 1.

\bibitem{MR0179168}
B.~J. Birch and H.~P.~F. Swinnerton-Dyer, ``Notes on elliptic curves. {II},''
  {\em J. Reine Angew. Math.} {\bf 218} (1965)  79--108.

\bibitem{MR2238272}
A.~Wiles, ``The {B}irch and {S}winnerton-{D}yer conjecture,'' in {\em The
  millennium prize problems}, pp.~31--41.
\newblock Clay Math. Inst., Cambridge, MA, 2006.

\bibitem{BruFun_TwoGmtThtLfts}
J.~H. Bruinier and J.~Funke, ``On two geometric theta lifts,''
  \href{http://dx.doi.org/10.1215/S0012-7094-04-12513-8}{{\em Duke Math. J.}
  {\bf 125} (2004) no.~1, 45--90}.
  \url{http://dx.doi.org/10.1215/S0012-7094-04-12513-8}.

\bibitem{MR1074485}
N.-P. Skoruppa, ``Explicit formulas for the {F}ourier coefficients of {J}acobi
  and elliptic modular forms,''
  \href{http://dx.doi.org/10.1007/BF01233438}{{\em Invent. Math.} {\bf 102}
  (1990) no.~3, 501--520}. \url{http://dx.doi.org/10.1007/BF01233438}.

\bibitem{MR1116103}
N.-P. Skoruppa, ``Heegner cycles, modular forms and {J}acobi forms,'' {\em
  S\'em. Th\'eor. Nombres Bordeaux (2)} {\bf 3} (1991) no.~1, 93--116.
  \url{http://jtnb.cedram.org/item?id=JTNB_1991__3_1_93_0}.

\bibitem{MR1309975}
W.~Duke, ``The critical order of vanishing of automorphic {$L$}-functions with
  large level,'' \href{http://dx.doi.org/10.1007/BF01245178}{{\em Invent.
  Math.} {\bf 119} (1995) no.~1, 165--174}.
  \url{http://dx.doi.org/10.1007/BF01245178}.

\bibitem{mfd}
W.~Stein, ``The {M}odular {F}orms {D}atabase,''.
  \newline\url{http://wstein.org/Tables}.

\bibitem{Hoehn2007}
G.~Hoehn, ``Selbstduale vertexoperatorsuperalgebren und das babymonster
  (self-dual vertex operator super algebras and the baby monster),''{\em Bonner
  Mathematische Schriften, Vol.} {\bf 286,} (June, 2007)  1--85,Bonn1996,
  \href{http://arxiv.org/abs/0706.0236}{{\tt 0706.0236}}.

\bibitem{Gaberdiel:2008xb}
M.~R. Gaberdiel, S.~Gukov, C.~A. Keller, G.~W. Moore, and H.~Ooguri, ``Extremal
  {$N=(2,2)$} 2{D} conformal field theories and constraints of modularity,''
  {\em Commun. Number Theory Phys.} {\bf 2} (2008) no.~4, 743--801.

\bibitem{Maloney:2009ck}
A.~Maloney, W.~Song, and A.~Strominger, ``{Chiral Gravity, Log Gravity and
  Extremal CFT},'' \href{http://dx.doi.org/10.1103/PhysRevD.81.064007}{{\em
  Phys.Rev.} {\bf D81} (2010)  064007},
\href{http://arxiv.org/abs/0903.4573}{{\tt arXiv:0903.4573 [hep-th]}}.

\bibitem{Carlip:2008eq}
S.~Carlip, S.~Deser, A.~Waldron, and D.~Wise, ``{Topologically Massive AdS
  Gravity},'' \href{http://dx.doi.org/10.1016/j.physletb.2008.07.057}{{\em
  Phys.Lett.} {\bf B666} (2008)  272--276},
\href{http://arxiv.org/abs/0807.0486}{{\tt arXiv:0807.0486 [hep-th]}}.

\bibitem{Carlip:2008jk}
S.~Carlip, S.~Deser, A.~Waldron, and D.~Wise, ``{Cosmological Topologically
  Massive Gravitons and Photons},''
  \href{http://dx.doi.org/10.1088/0264-9381/26/7/075008}{{\em
  Class.Quant.Grav.} {\bf 26} (2009)  075008},
\href{http://arxiv.org/abs/0803.3998}{{\tt arXiv:0803.3998 [hep-th]}}.

\bibitem{Giribet:2008bw}
G.~Giribet, M.~Kleban, and M.~Porrati, ``{Topologically Massive Gravity at the
  Chiral Point is Not Chiral},''
  \href{http://dx.doi.org/10.1088/1126-6708/2008/10/045}{{\em JHEP} {\bf 0810}
  (2008)  045},
\href{http://arxiv.org/abs/0807.4703}{{\tt arXiv:0807.4703 [hep-th]}}.

\bibitem{Grumiller:2008pr}
D.~Grumiller, R.~Jackiw, and N.~Johansson, ``{Canonical analysis of
  cosmological topologically massive gravity at the chiral point},''
  \href{http://arxiv.org/abs/0806.4185}{{\tt arXiv:0806.4185 [hep-th]}}.
Published in `Fundamental Interactions: A Memorial Volume for Wolfgang Kummer,'
  Editors: Daniel Grumiller, Anton Rebhan and Dimitri Vassilevich, World
  Scientific, 2010, pp.363-374.

\bibitem{Strominger:2008dp}
A.~Strominger, ``{A Simple Proof of the Chiral Gravity Conjecture},''
\href{http://arxiv.org/abs/0808.0506}{{\tt arXiv:0808.0506 [hep-th]}}.

\bibitem{Carlip:2008qh}
S.~Carlip, ``{The Constraint Algebra of Topologically Massive AdS Gravity},''
  \href{http://dx.doi.org/10.1088/1126-6708/2008/10/078}{{\em JHEP} {\bf 0810}
  (2008)  078},
\href{http://arxiv.org/abs/0807.4152}{{\tt arXiv:0807.4152 [hep-th]}}.

\bibitem{Li:2008yz}
W.~Li, W.~Song, and A.~Strominger, ``{Comment on 'Cosmological Topological
  Massive Gravitons and Photons'},''
\href{http://arxiv.org/abs/0805.3101}{{\tt arXiv:0805.3101 [hep-th]}}.

\bibitem{Myung:2008dm}
Y.~S. Myung, ``{Logarithmic conformal field theory approach to topologically
  massive gravity},''
  \href{http://dx.doi.org/10.1016/j.physletb.2008.11.002}{{\em Phys.Lett.} {\bf
  B670} (2008)  220--223},
\href{http://arxiv.org/abs/0808.1942}{{\tt arXiv:0808.1942 [hep-th]}}.

\bibitem{Sachs:2008gt}
I.~Sachs and S.~N. Solodukhin, ``{Quasi-Normal Modes in Topologically Massive
  Gravity},'' \href{http://dx.doi.org/10.1088/1126-6708/2008/08/003}{{\em JHEP}
  {\bf 0808} (2008)  003},
\href{http://arxiv.org/abs/0806.1788}{{\tt arXiv:0806.1788 [hep-th]}}.

\bibitem{Gibbons:2008vi}
G.~Gibbons, C.~Pope, and E.~Sezgin, ``{The General Supersymmetric Solution of
  Topologically Massive Supergravity},''
  \href{http://dx.doi.org/10.1088/0264-9381/25/20/205005}{{\em
  Class.Quant.Grav.} {\bf 25} (2008)  205005},
\href{http://arxiv.org/abs/0807.2613}{{\tt arXiv:0807.2613 [hep-th]}}.

\bibitem{Skenderis:2009nt}
K.~Skenderis, M.~Taylor, and B.~C. van Rees, ``{Topologically Massive Gravity
  and the AdS/CFT Correspondence},''
  \href{http://dx.doi.org/10.1088/1126-6708/2009/09/045}{{\em JHEP} {\bf 0909}
  (2009)  045},
\href{http://arxiv.org/abs/0906.4926}{{\tt arXiv:0906.4926 [hep-th]}}.

\bibitem{feingold_frenkel}
A.~J. Feingold and I.~B. Frenkel, ``{A hyperbolic Kac-Moody algebra and the
  theory of Siegel modular forms of genus 2},'' {\em J. Math. Ann.} {\bf 263}
  (1983)  87--144.

\bibitem{GriNik_AutFrmLorKMAlgs_II}
V.~A. Gritsenko and V.~V. Nikulin, ``Automorphic forms and {L}orentzian
  {K}ac-{M}oody algebras. {II},''
  \href{http://dx.doi.org/10.1142/S0129167X98000117}{{\em Internat. J. Math.}
  {\bf 9} (1998) no.~2, 201--275}.
  \url{http://dx.doi.org/10.1142/S0129167X98000117}.

\bibitem{GriNik_AutFrmLorKMAlgs_I}
V.~A. Gritsenko and V.~V. Nikulin, ``Automorphic forms and {L}orentzian
  {K}ac-{M}oody algebras. {I},''
  \href{http://dx.doi.org/10.1142/S0129167X98000105}{{\em Internat. J. Math.}
  {\bf 9} (1998) no.~2, 153--199}.
  \url{http://dx.doi.org/10.1142/S0129167X98000105}.

\bibitem{GrNik_SieAutFrmCorrLorKMAlgs}
V.~A. Gritsenko and V.~V. Nikulin, ``Siegel automorphic form corrections of
  some {L}orentzian {K}ac-{M}oody {L}ie algebras,'' {\em Amer. J. Math.} {\bf
  119} (1997) no.~1, 181--224.
  \url{http://muse.jhu.edu/journals/american_journal_of_mathematics/v119/119.1gritsenko.pdf}.

\bibitem{ATLAS}
J.~Conway, R.~Curtis, S.~Norton, R.~Parker, and R.~Wilson, {\em {Atlas of
  finite groups. Maximal subgroups and ordinary characters for simple groups.
  With comput. assist. from J. G. Thackray.}}
\newblock Oxford: Clarendon Press, 1985.

\bibitem{Ple_UnqGolCdes}
V.~Pless, ``On the uniqueness of the {G}olay codes,'' {\em J. Combinatorial
  Theory} {\bf 5} (1968)  215--228.

\bibitem{sphere_packing}
J.~H. Conway and N.~J.~A. Sloane, {\em {Sphere Packings, Lattices and Groups}}.
\newblock Springer-Verlag, 1999.

\bibitem{McKay_Corr}
J.~McKay, ``Graphs, singularities, and finite groups,'' in {\em The {S}anta
  {C}ruz {C}onference on {F}inite {G}roups ({U}niv. {C}alifornia, {S}anta
  {C}ruz, {C}alif., 1979)}, vol.~37 of {\em Proc. Sympos. Pure Math.},
  pp.~183--186.
\newblock Amer. Math. Soc., Providence, R.I., 1980.

\bibitem{Slo_SmpSngSmpAlgGps}
P.~Slodowy, {\em Simple singularities and simple algebraic groups}, vol.~815 of
  {\em Lecture Notes in Mathematics}.
\newblock Springer, Berlin, 1980.

\bibitem{GonVer_GeomCnstMcKCorr}
G.~Gonzalez-Sprinberg and J.-L. Verdier, ``Construction g\'eom\'etrique de la
  correspondance de {M}c{K}ay,'' {\em Ann. Sci. \'Ecole Norm. Sup. (4)} {\bf
  16} (1983) no.~3, 409--449 (1984).
  \url{http://www.numdam.org/item?id=ASENS_1983_4_16_3_409_0}.

\bibitem{Con_CnstM}
J.~H. Conway, ``A simple construction for the {F}ischer-{G}riess monster
  group,'' {\em Invent. Math.} {\bf 79} (1985) no.~3, 513--540.

\bibitem{GlaNor_McKE8M}
G.~Glauberman and S.~P. Norton, ``On {M}c{K}ay's connection between the affine
  {$E_8$} diagram and the {M}onster,'' in {\em Proceedings on {M}oonshine and
  related topics ({M}ontr\'eal, {QC}, 1999)}, vol.~30 of {\em CRM Proc. Lecture
  Notes}, pp.~37--42.
\newblock Amer. Math. Soc., Providence, RI, 2001.

\bibitem{GAP4}
The GAP~Group, {\em {GAP -- Groups, Algorithms, and Programming, Version 4.4}},
  2005.
\newblock \verb+(http://www.gap-system.org)+.

\bibitem{Kos_GphTrnIcoGal}
B.~Kostant, ``The graph of the truncated icosahedron and the last letter of
  {G}alois,'' {\em Notices Amer. Math. Soc.} {\bf 42} (1995) no.~9, 959--968.

\bibitem{Kos_StrcTrncIcos}
B.~Kostant, ``Structure of the truncated icosahedron (e.g.\ {F}ullerene or
  {${\rm C}_{60}$}, viral coatings) and a {$60$}-element conjugacy class in
  {${\rm PSL}(2,11)$},'' \href{http://dx.doi.org/10.1007/BF01614076}{{\em
  Selecta Math. (N.S.)} {\bf 1} (1995) no.~1, 163--195}.
  \url{http://dx.doi.org/10.1007/BF01614076}.

\bibitem{CumDun_E8M24}
C.~J. Cummins and J.~F. Duncan, ``An {$E_8$} {C}orrespondence for
  {M}ultiplicative {E}ta-{P}roducts,''
  \href{http://dx.doi.org/10.4153/CMB-2011-068-9}{{\em Canad. Math. Bull.} {\bf
  55} (2012) no.~1, 67--72}. \url{http://dx.doi.org/10.4153/CMB-2011-068-9}.

\bibitem{Dun_ArthGrpsAffE8Dyn}
J.~F. Duncan, ``Arithmetic groups and the affine {$E_8$} {D}ynkin diagram,'' in
  {\em Groups and symmetries}, vol.~47 of {\em CRM Proc. Lecture Notes},
  pp.~135--163.
\newblock Amer. Math. Soc., Providence, RI, 2009.

\bibitem{MR2347298}
S.~Sakuma, ``6-transposition property of {$\tau$}-involutions of vertex
  operator algebras,'' \href{http://dx.doi.org/10.1093/imrn/rnm030}{{\em Int.
  Math. Res. Not. IMRN} (2007) no.~9, Art. ID rnm 030, 19}.
  \url{http://dx.doi.org/10.1093/imrn/rnm030}.

\bibitem{MR2309178}
C.~H. Lam, H.~Yamada, and H.~Yamauchi, ``Vertex operator algebras, extended
  {$E_8$} diagram, and {M}c{K}ay's observation on the {M}onster simple group,''
  \href{http://dx.doi.org/10.1090/S0002-9947-07-04002-0}{{\em Trans. Amer.
  Math. Soc.} {\bf 359} (2007) no.~9, 4107--4123 (electronic)}.
  \url{http://dx.doi.org/10.1090/S0002-9947-07-04002-0}.

\bibitem{MR2160172}
C.~H. Lam, H.~Yamada, and H.~Yamauchi, ``Mc{K}ay's observation and vertex
  operator algebras generated by two conformal vectors of central charge 1/2,''
  {\em IMRP Int. Math. Res. Pap.} (2005) no.~3, 117--181.

\bibitem{MR2874931}
G.~H{\"o}hn, C.~H. Lam, and H.~Yamauchi, ``Mc{K}ay's {$E_7$} observation on the
  {B}aby {M}onster,'' {\em Int. Math. Res. Not. IMRN} (2012) no.~1, 166--212.

\bibitem{MR2890302}
G.~H{\"o}hn, C.~H. Lam, and H.~Yamauchi, ``Mc{K}ay's {$E_6$} observation on the
  largest {F}ischer group,''
  \href{http://dx.doi.org/10.1007/s00220-011-1413-8}{{\em Comm. Math. Phys.}
  {\bf 310} (2012) no.~2, 329--365}.
  \url{http://dx.doi.org/10.1007/s00220-011-1413-8}.

\bibitem{Eguchi2011}
T.~Eguchi and K.~Hikami, ``{Twisted elliptic genus for $K3$ and Borcherds
  product},'' \href{http://arxiv.org/abs/1112.5928}{{\tt 1112.5928}}.

\bibitem{Govindarajan2010a}
S.~Govindarajan, ``Brewing moonshine for mathieu,''
  \href{http://arxiv.org/abs/1012.5732}{{\tt 1012.5732}}.

\bibitem{Volpato:2012qe}
R.~Volpato, ``{Mathieu Moonshine and symmetries of K3 sigma models},''
\href{http://arxiv.org/abs/1201.6172}{{\tt arXiv:1201.6172 [hep-th]}}.

\bibitem{Gaberdiel2011}
M.~R. Gaberdiel, S.~Hohenegger, and R.~Volpato, ``{Symmetries of $K3$ sigma
  models},'' \href{http://arxiv.org/abs/1106.4315}{{\tt 1106.4315}}.

\bibitem{Taormina:2011rr}
A.~Taormina and K.~Wendland, ``{The overarching finite symmetry group of Kummer
  surfaces in the Mathieu group $M_{24}$},''
\href{http://arxiv.org/abs/1107.3834}{{\tt arXiv:1107.3834 [hep-th]}}.

\bibitem{GorMcI_SvyMckTht}
B.~Gordon and R.~J. McIntosh, ``A survey of classical mock theta functions,''
  in {\em Partitions, q-Series, and Modular Forms}, K.~Alladi and F.~Garvan,
  eds., vol.~23 of {\em Developments in Mathematics}, pp.~95--144.
\newblock Springer New York, 2012.
\newblock \url{http://dx.doi.org/10.1007/978-1-4614-0028-8$_9$}.
\newblock 10.1007/978-1-4614-0028-8$_9$.

\bibitem{Zhu_ModInv}
Y.~Zhu, ``Modular invariance of characters of vertex operator algebras,'' {\em
  Journal of the American Mathematical Society} {\bf 9} (1996) no.~1, 237--302.

\bibitem{Dong2000}
C.~Dong, H.~Li, and G.~Mason, ``Modular invariance of trace functions in
  orbifold theory and generalized {M}oonshine,'' {\em Communications in
  Mathematical Physics} {\bf 214} (2000)  1--56,
  \href{http://arxiv.org/abs/q-alg/9703016}{{\tt q-alg/9703016}}.

\bibitem{Ben_SchInd}
M.~Benard, ``Schur indexes of sporadic simple groups,''
  \href{http://dx.doi.org/10.1016/0021-8693(79)90177-7}{{\em J. Algebra} {\bf
  58} (1979) no.~2, 508--522}.
  \url{http://dx.doi.org/10.1016/0021-8693(79)90177-7}.

\bibitem{Mar_M12}
R.~S. Margolin, ``Representations of {$M_{12}$},''
  \href{http://dx.doi.org/10.1006/jabr.1993.1078}{{\em J. Algebra} {\bf 156}
  (1993) no.~2, 362--369}. \url{http://dx.doi.org/10.1006/jabr.1993.1078}.

\bibitem{Mukai}
S.~Mukai, ``Finite groups of automorphisms of {$K3$} surfaces and the {M}athieu
  group,'' \href{http://dx.doi.org/10.1007/BF01394352}{{\em Invent. Math.} {\bf
  94} (1988) no.~1, 183--221}. \url{http://dx.doi.org/10.1007/BF01394352}.

\bibitem{Kondo}
S.~Kond{\=o}, ``Niemeier lattices, {M}athieu groups, and finite groups of
  symplectic automorphisms of {$K3$} surfaces,''
  \href{http://dx.doi.org/10.1215/S0012-7094-98-09217-1}{{\em Duke Math. J.}
  {\bf 92} (1998) no.~3, 593--603}.
  \url{http://dx.doi.org/10.1215/S0012-7094-98-09217-1}. With an appendix by
  Shigeru Mukai.

\bibitem{GriNik_K3SrfsLorKMAlgsMrrSym}
V.~A. Gritsenko and V.~V. Nikulin, ``{$K3$} surfaces, {L}orentzian
  {K}ac-{M}oody algebras and mirror symmetry,'' {\em Math. Res. Lett.} {\bf 3}
  (1996) no.~2, 211--229.

\bibitem{GriNik_ArthMrrSymCYMfds}
V.~A. Gritsenko and V.~V. Nikulin, ``The arithmetic mirror symmetry and
  {C}alabi-{Y}au manifolds,''
  \href{http://dx.doi.org/10.1007/s002200050769}{{\em Comm. Math. Phys.} {\bf
  210} (2000) no.~1, 1--11}. \url{http://dx.doi.org/10.1007/s002200050769}.

\bibitem{Dol_MrrSymK3Srfs}
I.~V. Dolgachev, ``Mirror symmetry for lattice polarized {$K3$} surfaces,''
  \href{http://dx.doi.org/10.1007/BF02362332}{{\em J. Math. Sci.} {\bf 81}
  (1996) no.~3, 2599--2630}. \url{http://dx.doi.org/10.1007/BF02362332}.
  Algebraic geometry, 4.

\bibitem{DVV}
R.~Dijkgraaf, E.~Verlinde, and H.~Verlinde, ``{Counting Dyons in N=4 String
  Theory},'' {\em Nucl.Phys.B} {\bf 484:543-561,1997}
  (Nucl.Phys.B484:543-561,1997)  ,
  \href{http://arxiv.org/abs/hep-th/9607026}{{\tt hep-th/9607026}}.

\bibitem{SenJHEP0705:0392007}
A.~Sen, ``{Walls of Marginal Stability and Dyon Spectrum in N=4 Supersymmetric
  String Theories},'' {\em JHEP} {\bf 0705:039,2007} (JHEP 0705:039,2007)  ,
  \href{http://arxiv.org/abs/hep-th/0702141}{{\tt hep-th/0702141}}.

\bibitem{Cheng2007a}
M.~C.~N. Cheng and E.~Verlinde, ``{Dying Dyons Don't Count},''
  \href{http://dx.doi.org/10.1088/1126-6708/2007/09/070}{{\em JHEP} {\bf 09}
  (2007)  070},
\href{http://arxiv.org/abs/0706.2363}{{\tt arXiv:0706.2363 [hep-th]}}.

\bibitem{Cheng2008a}
M.~C. Cheng and E.~P. Verlinde, ``{Wall Crossing, Discrete Attractor Flow, and
  Borcherds Algebra},'' \href{http://dx.doi.org/10.3842/SIGMA.2008.068}{{\em
  SIGMA} {\bf 4} (2008)  068},
\href{http://arxiv.org/abs/0806.2337}{{\tt arXiv:0806.2337 [hep-th]}}.

\bibitem{Dabholkar2008}
A.~Dabholkar, D.~Gaiotto, and S.~Nampuri, ``{Comments on the spectrum of CHL
  dyons},'' \href{http://dx.doi.org/10.1088/1126-6708/2008/01/023}{{\em JHEP}
  {\bf 01} (2008)  023},
\href{http://arxiv.org/abs/hep-th/0702150}{{\tt arXiv:hep-th/0702150}}.

\bibitem{Dabholkar2007}
A.~Dabholkar, ``{Cargese lectures on black holes, dyons, and modular forms},''
\href{http://dx.doi.org/10.1016/j.nuclphysbps.2007.06.003}{{\em Nucl. Phys.
  Proc. Suppl.} {\bf 171} (2007)  2--15}.

\bibitem{Kachru:1995wm}
S.~Kachru and C.~Vafa, ``{Exact results for N=2 compactifications of heterotic
  strings},'' \href{http://dx.doi.org/10.1016/0550-3213(95)00307-E}{{\em
  Nucl.Phys.} {\bf B450} (1995)  69--89},
\href{http://arxiv.org/abs/hep-th/9505105}{{\tt arXiv:hep-th/9505105
  [hep-th]}}.

\bibitem{Ferrara:1995yx}
S.~Ferrara, J.~A. Harvey, A.~Strominger, and C.~Vafa, ``{Second quantized
  mirror symmetry},''
  \href{http://dx.doi.org/10.1016/0370-2693(95)01074-Z}{{\em Phys.Lett.} {\bf
  B361} (1995)  59--65},
\href{http://arxiv.org/abs/hep-th/9505162}{{\tt arXiv:hep-th/9505162
  [hep-th]}}.

\bibitem{Harvey:1995fq}
J.~A. Harvey and G.~W. Moore, ``{Algebras, BPS states, and strings},''
  \href{http://dx.doi.org/10.1016/0550-3213(95)00605-2}{{\em Nucl.Phys.} {\bf
  B463} (1996)  315--368},
\href{http://arxiv.org/abs/hep-th/9510182}{{\tt arXiv:hep-th/9510182
  [hep-th]}}.

\bibitem{Bor_AutFmsSngGrs}
R.~E. Borcherds, ``Automorphic forms with singularities on {G}rassmannians,''
  \href{http://dx.doi.org/10.1007/s002220050232}{{\em Invent. Math.} {\bf 132}
  (1998) no.~3, 491--562}. \url{http://dx.doi.org/10.1007/s002220050232}.

\bibitem{Kon_PdtFmlsMdlrFmsAftBor}
M.~Kontsevich, ``Product formulas for modular forms on {${\rm O}(2,n)$} (after
  {R}. {B}orcherds),'' {\em Ast\'erisque} (1997) no.~245, Exp.\ No.\ 821, 3,
  41--56. S{\'e}minaire Bourbaki, Vol. 1996/97.

\bibitem{LopesCardoso:1996zj}
G.~Lopes~Cardoso, ``{Perturbative gravitational couplings and Siegel modular
  forms in D = 4, N=2 string models},'' {\em Nucl.Phys.Proc.Suppl.} {\bf 56B}
  (1997)  94--101,
\href{http://arxiv.org/abs/hep-th/9612200}{{\tt arXiv:hep-th/9612200
  [hep-th]}}.

\bibitem{LopesCardoso:1996nc}
G.~Lopes~Cardoso, G.~Curio, and D.~Lust, ``{Perturbative couplings and modular
  forms in N=2 string models with a Wilson line},''
  \href{http://dx.doi.org/10.1016/S0550-3213(97)00047-3}{{\em Nucl.Phys.} {\bf
  B491} (1997)  147--183},
\href{http://arxiv.org/abs/hep-th/9608154}{{\tt arXiv:hep-th/9608154
  [hep-th]}}.

\bibitem{Govindarajan2010}
S.~Govindarajan, ``{BKM Lie superalgebras from counting twisted CHL dyons},''
  \href{http://dx.doi.org/10.1007/JHEP05(2011)089}{{\em JHEP} {\bf 1105} (2011)
   089}, \href{http://arxiv.org/abs/1006.3472}{{\tt arXiv:1006.3472 [hep-th]}}.
Dedicated to the memories of Jaydeep Majumder and Alok Kumar.

\bibitem{modi}
\url{http://modi.countnumber.de/index.php?chap=ell.newforms/ell.newforms.html}.

\bibitem{from_number}
C.~Itzykson, ed., {\em {From Number Theory to Physics}}.
\newblock Springer, 1992.

\bibitem{aoki}
H.~Aoki and T.~Ibukiyama, ``{Simple graded rings of Siegel modular forms,
  differential operators and Borcherds products},'' {\em Int. J. Math.} {\bf
  16, No. 3} (2005)  249--279.

\end{thebibliography}
\providecommand{\href}[2]{#2}\begingroup\raggedright\endgroup

\end{document}